\documentclass[10pt]{article}

\usepackage{fullpage}
\usepackage{natbib}
\usepackage{algorithm}
\usepackage{algorithmic}
\newcommand{\squeeze}{}

\usepackage{threeparttable}

\usepackage{multirow}
\usepackage{colortbl}
\definecolor{bgcolor}{rgb}{0.8,1,1}
\definecolor{bgcolor2}{rgb}{0.8,1,0.8}

\usepackage{graphicx} 
\graphicspath{{figs/}}

\usepackage{nicefrac}       

\newcommand{\eqdef}{\; { := }\;}

\newcommand{\R}{\mathbb{R}}

\def\<#1,#2>{\langle #1,#2\rangle}

\newcommand{\prox}{\operatorname{prox}}

\usepackage{tcolorbox}
\usepackage{pifont}
\definecolor{mydarkgreen}{RGB}{39,130,67}
\definecolor{mydarkred}{RGB}{192,47,25}

\usepackage{amsmath,amsfonts,amssymb,amsthm,array}

\usepackage{mdframed} 
\usepackage{thmtools}
\usepackage{textcomp}

\declaretheorem[within=section]{definition}
\declaretheorem[sibling=definition]{theorem}

\declaretheorem[sibling=definition]{assumption}

\declaretheorem[sibling=definition]{lemma}

\usepackage{microtype}
\usepackage{subfigure}
\usepackage{booktabs} 

\usepackage{grffile}

\usepackage{hyperref}

\title{\bf Error Compensated Loopless SVRG, Quartz, and SDCA for Distributed Optimization}
\author{Xun Qian \\ KAUST\thanks{King Abdullah University of Science and Technology, Thuwal, Saudi Arabia.} \and Hanze Dong \\ HKUST\thanks{Hong Kong University of Science and Technology, Hong Kong} \and Peter Richt\'{a}rik\\KAUST \and Tong  Zhang\\HKUST}
\date{September 21, 2021}

\begin{document}

\maketitle

\begin{abstract}
The communication of gradients is a key bottleneck in distributed training of large scale machine learning models. In order to reduce the communication cost, gradient compression (e.g., sparsification and quantization) and error compensation techniques are often used. In this paper, we propose and study three new efficient methods in this space: error compensated loopless SVRG method (EC-LSVRG), error compensated Quartz (EC-Quartz), and error compensated SDCA (EC-SDCA). Our method is capable of working with any contraction compressor (e.g., TopK compressor), and we perform analysis for convex optimization problems in the composite case and smooth case for EC-LSVRG. We prove linear convergence rates for both cases and show that in the smooth case the rate has a better dependence on the parameter associated with the contraction compressor. Further, we show that in the smooth case, and under some certain conditions, error compensated loopless SVRG has the same convergence rate as the vanilla loopless SVRG method. Then we show that the convergence rates of EC-Quartz and EC-SDCA in the composite case are as good as EC-LSVRG in the smooth case. Finally, numerical experiments are presented to illustrate the efficiency of our methods.
\end{abstract}

{
\tableofcontents
}


\section{Introduction}

In this work we consider the composite finite-sum optimization problem

\begin{equation}\label{primal-LSVRG}
\squeeze \min \limits_{x\in \mathbb{R}^d} P(x) \eqdef \frac{1}{n} \sum\limits_{\tau=1}^n  f^{(\tau)}(x) + \psi(x),
\end{equation}
where $f(x)\eqdef\frac{1}{n}\sum_{\tau}  f^{(\tau)}(x)$ is an average of $n$ smooth convex functions $f^{(\tau)}:\R^d\to \R$ distributed over $n$ nodes (devices, computers), and  $\psi:\R^d \to \R\cup \{+\infty\}$ is a proper, closed and convex function representing a possibly nonsmooth regularizer. On each node, $f^{(\tau)}(x)$ is an average of $m$ smooth convex functions 
\[
\squeeze f^{(\tau)}(x) \eqdef  \frac{1}{m} \sum \limits_{i=1}^m f^{(\tau)}_i(x) , 
\]
representing the average loss over the training data stored on node $\tau$. We assume that problem (\ref{primal-LSVRG}) has at least one optimal solution $x^*$. 

\vskip 2mm 

For large scale machine learning problems, distributed training and parallel training are often used. While in such settings, communication is generally much slower than the computation, which makes the communication overhead become a key bottleneck. There are many ways to tackle this issue. For instance, there are large mini-batches strategy \citep{Goyal17, You17}, local SGD \citep{COCOA+journal, localSGD-Stich}, asynchronous learning \citep{Agarwal11, Lian15, Recht11}, quantization and error compensation \citep{Alistarh17, Bernstein18, Mish19, Seide14, Wen17}. 

For quantization, there are mainly two types, i.e., contraction compressor and unbiased compressor, which are defined as follows. 

A possibly randomized map $Q: \R^d \to \R^d$ is called a contraction compressor if there is a $0< \delta \leq 1$ such that  
\begin{equation}\label{eq:contractor}
\mathbb{E} \left[ \|x - Q(x) \|^2 \right] \leq (1-\delta)\|x\|^2,  \quad \forall x\in \R^d.
\end{equation}
${\tilde Q}: \R^d \to \R^d$ is called an unbiased compressor if there is $\omega \geq 0$ such that
\begin{equation} \label{eq:unbiased}
\mathbb{E} \left[{\tilde Q}(x) \right] = x \quad \text{and} \quad  \mathbb{E} \left[ \|{\tilde Q}(x)\|^2  \right] \leq (\omega + 1)\|x\|^2, \quad \forall x\in \R^d.
\end{equation}

Quantization can reduce the communicated bits to improve the communication efficiency, but it will slow down the convergence rate generally. Hence, error feedback (or called error compensation) scheme is often used to improve the performance of quantization algorithms \citep{Seide14}. For the unbiased compressor, if the accumulated quantization error is assumed to be bounded, the convergence rate of error compensated SGD is the same as vanilla SGD \citep{Tang18}. However, if the bounded stochastic gradient is assumed instead, in order to bound the accumulated quantization error, some decaying factors need to be involved generally, and the error compensated SGD is proved to have some advantage over QSGD \citet{Alistarh17} in some perspective for convex quadratic problems \citep{Wu18}. On the other hand, for the contraction compressor (for example TopK compressor \citep{Alistarh18}), the error compensated SGD actually has the same convergence rate as Vanilla SGD \citep{Stich18, Stich19, Tang19}. The above results are all for the smooth case. If $f$ is non-smooth and $\psi\equiv0$, error compensated SGD was studied by \citet{karimireddy2019error} in the single node case, and the convergence rate is of order $O\left(  \nicefrac{1}{\sqrt{\delta k} } \right)$ for $f$ being non-strongly convex. 

For variance-reduced methods, there are QSVRG \citep{Alistarh17} for the smooth case where $\psi \equiv 0$ in problem (\ref{primal-LSVRG}), and VR-DIANA \citep{Samuel19} for the composite or regularized case. However, the compressor of both algorithms need to be unbiased. Recently, an error compensated method called EC-LSVRG-DIANA which can achieve linear convergence for the strongly convex and smooth case was proposed by \citet{gorbunov2020linearly}, but besides the contraction compressor, the unbiased compressor is also needed in the algorithm. In this paper, we study the error compensated methods for loopless SVRG (L-SVRG) \citep{LSVRG}, Quartz \citep{Quartz}, and SDCA \citep{SDCA}, where only contraction compressors are needed. 

\subsection{Contributions}

We now outline our main theoretical contributions.

\subsubsection{Strongly convex case for EC-LSVRG} Denote the smoothness constants of functions $f$, $f^{(\tau)}$, and $f^{(\tau)}_i$ by $L_f$, ${\bar L}$, and $L$, respectively. Let $\mu$ be the strong convexity parameter, $p$ be the updating frequency of the reference point, and $\delta$, $\delta_1$ be the contraction compressor parameters. Let $p\leq O(\delta_1)$.

\paragraph{Composite case.}
In the composite case, the iteration complexity of error compensated L-SVRG (EC-LSVRG) is 
$$
\squeeze
O\left(  \left(  \frac{1}{\delta} + \frac{1}{p} + \frac{L_f}{\mu} + \frac{L}{n \mu}  + \frac{(1-\delta){{\color{blue}\bar L}}}{\delta^2 \mu} + \frac{(1-\delta)L}{\delta \mu} \right) \ln \frac{1}{\epsilon} \right). 
$$
Under an additional assumption (Assumption \ref{as:expcompressor}) on the contraction compressor, the iteration complexity is improved to 
$$
\squeeze
O\left(  \left(  \frac{1}{\delta} + \frac{1}{p} + \frac{L_f}{\mu} + \frac{L}{n \mu} + \frac{(1-\delta){\color{red}L_f}}{\delta^2 \mu}  + \frac{(1-\delta)L}{{\color{red}n} \delta \mu}  \right) \ln \frac{1}{\epsilon}  \right). 
$$

\paragraph{Smooth case.}
In the smooth case, the iteration complexity of EC-LSVRG is 
$$
\squeeze
O\left( \left( \frac{1}{\delta} + \frac{1}{p}  + \frac{L_f}{\mu} + \frac{L}{n \mu}  +  \frac{\sqrt{(1-\delta) L_f{\bar L}}}{\mu \delta}  + \frac{\sqrt{(1-\delta)L_f L}}{\mu \sqrt{\delta}}  \right)   \ln \frac{1}{\epsilon} \right). 
$$
The iteration complexity of EC-LSVRG-DIANA \citep{gorbunov2020linearly} is $O((\omega + m + \frac{L}{\mu \delta}) \ln \frac{1}{\epsilon})$. If the compressor $Q_1$ in EC-LSVRG is obtained by scaling the unbiased compressor in EC-LSVRG-DIANA and we choose $p = \min \{ \frac{1}{m}, \frac{1}{\omega + 1}   \}$, then our iteration complexity becomes 
$$
\squeeze
O\left( \left( \frac{1}{\delta} + \omega + m + \frac{L_f}{\mu} + \frac{L}{n \mu}  +  \frac{\sqrt{(1-\delta) L_f{\bar L}}}{\mu \delta}  + \frac{\sqrt{(1-\delta)L_f L}}{\mu \sqrt{\delta}}  \right)   \ln \frac{1}{\epsilon} \right), $$ which is better than that of EC-LSVRG-DIANA since $L_f \leq {\bar L} \leq L$. 

Under an additional assumption (Assumption \ref{as:expcompressor}) on the contraction compressor, the iteration complexity of EC-LSVRG is improved to 
$$
\squeeze
O\left( \left( \frac{1}{\delta} + \frac{1}{p} + \frac{L_f}{\mu} + \frac{L}{n \mu} + \frac{\sqrt{(1-\delta)}L_f}{\mu \delta}  \right)   \ln \frac{1}{\epsilon} \right). 
$$
In particular, if $\frac{L_f}{\delta} \leq \frac{L}{n}$, then the above iteration complexity becomes 
$$
\squeeze
O\left( \left( \frac{1}{p} + \frac{L_f}{\mu} + \frac{L}{n \mu}  \right)   \ln \frac{1}{\epsilon} \right), 
$$
which is actually the iteration complexity of the uncompressed L-SVRG \citep{qian2019svrg}. Noticing that $L_f \leq L \leq mnL_f$, this means that in the extreme case: $L=mnL_f$, the error compensated L-SVRG has the same convergence rate as the uncompressed L-SVRG as long as $\frac{1}{\delta} \leq m$ and $p\leq O(\delta_1)$. 

\subsubsection{Non-strongly convex case for EC-LSVRG} Let $p\leq O(\delta_1)$. 

\paragraph{Composite case.} In the composite case, the iteration complexity of EC-LSVRG is 
$$
\squeeze
O\left(  \left( \frac{1}{p} + L_f + \frac{L}{n} + \frac{(1-\delta){\bar L}}{\delta^2} + \frac{(1-\delta) L}{\delta}   \right) \frac{1}{\epsilon} \right). 
$$
Under an additional assumption (Assumption \ref{as:expcompressor}) on the contraction compressor, the iteration complexity is improved to 
$$
\squeeze
O\left(  \left( \frac{1}{p} + L_f + \frac{L}{n} + \frac{(1-\delta){\color{red}L_f}}{\delta^2} + \frac{(1-\delta) L}{{\color{red}n} \delta}   \right) \frac{1}{\epsilon} \right). 
$$

\paragraph{Smooth case.} 
In the smooth case, the iteration complexity of EC-LSVRG is 
$$
\squeeze
O\left(  \left( \frac{1}{p} + L_f + \frac{L}{n} +  \frac{\sqrt{(1-\delta) L_f{\bar L}}}{\delta}  + \frac{\sqrt{(1-\delta)L_f L}}{\sqrt{\delta}}    \right) \frac{1}{\epsilon}  \right). 
$$
Under an additional assumption (Assumption \ref{as:expcompressor}) on the contraction compressor, the iteration complexity is improved to 
$$
\squeeze
O\left( \left(  \frac{1}{p} +  L_f + \frac{L}{n} + \frac{\sqrt{(1-\delta)}L_f}{\delta} \right)  \frac{1}{\epsilon}  \right).  
$$

\subsubsection{Strongly convex and composite case for EC-Quartz and EC-SDCA} We consider problem (\ref{primal-sdca}) for EC-Quartz and EC-SDCA in the strongly convex and composite case. The iteration complexities of EC-Quartz and EC-SDCA are 
$$
O \left( \left(  \frac{1}{\delta} + m + \frac{R_m^2}{n\lambda \gamma} + \frac{R^2}{\lambda \gamma}   +   \frac{\sqrt{1-\delta} R{\bar R}}{\delta \lambda \gamma}  + \frac{\sqrt{1-\delta} RR_m}{\lambda \gamma \sqrt{\delta}}  \right) \ln\frac{1}{\epsilon}  \right). 
$$ 
Under an additional assumption (Assumption~\ref{as:expcompressor}) on the contraction compressor, the iteration complexities are improved to 
$$
O \left( \left(  \frac{1}{\delta} + m + \frac{R_m^2}{n\lambda \gamma} + \frac{R^2}{\lambda \gamma}    +  \frac{\sqrt{1-\delta}}{\delta} \frac{R^2}{\lambda \gamma}   \right) \ln\frac{1}{\epsilon}  \right). 
$$
Noticing that for problem (\ref{primal-sdca}), we usually use $\frac{R^2}{\gamma}$, $\frac{{\bar R}^2}{\gamma}$, and $\frac{R_m^2}{\gamma}$ ( ${R}$, ${\bar R}$, and $R_m$ are defined in Algorithm~\ref{alg:ec-quartz}) to estimate the smoothness constants of $\frac{1}{mn} \sum_{\tau=1}^n \sum_{i=1}^m \phi_{i \tau}(A_{i\tau}^\top x)$, $\frac{1}{m} \sum_{i=1}^m \phi_{i \tau}(A_{i\tau}^\top x)$, and $\phi_{i \tau}(A_{i\tau}^\top x)$, respectively. In this sense, the iteration complexities of EC-Quartz and EC-SDCA in the composite case are as good as EC-LSVRG in the smooth case. Let $\delta$ go to $1$. Then the iteration complexity becomes $O \left( \left(   \frac{mn}{n} + \frac{R_m^2}{n\lambda \gamma} + \frac{R^2}{\lambda \gamma}   \right) \ln\frac{1}{\epsilon}  \right)$, which has linear speed up with respect to the number of nodes $n$ when $n \leq \frac{R_m^2}{R^2}$. Noticing that there is no linear speed up with respect to $n$ for Quartz with fully dense data, our result is actually better than Quartz in this case. This improvement for fully dense data benefits from the better estimation of expected eparable overapproximation (ESO). The ESO estimation for arbitrary sampling for Quartz can be found in the appendix.

\subsection{Compression methods}

We now give a few examples of contraction compressors:

\noindent {\bf TopK compressor.} For a parameter $1\leq K \leq d$,  the TopK compressor \citep{Stich18} is defined as
$$
({\rm TopK}(x))_{\pi(i)} = \left\{ \begin{array}{rl}
(x)_{\pi(i)}  &\mbox{ if $i\leq K$, } \\
0  \quad \quad &\mbox{ otherwise, }
\end{array} \right.
$$
where $\pi$ is a permutation of $[d] \eqdef \{1, 2, \dots, d\}$ such that $(|x|)_{\pi(i)} \geq (|x|)_{\pi(i+1)}$ for $i = 1, \dots, d-1$.

\noindent {\bf RandK compressor.} For a parameter $1\leq K \leq d$,  the {\rm RandK} compressor is defined as
$$
({\rm RandK}(x))_{i} = \left\{ \begin{array}{rl}
(x)_i  &\mbox{ if $i \in S$, } \\
0  \quad \quad &\mbox{ otherwise, }
\end{array} \right.
$$
where $S$ is chosen uniformly from the set of all $K$ element subsets of $[d]$.  \\ 

\noindent It is known that after appropriate scaling, any unbiased compressor satisfying \eqref{eq:unbiased} becomes a contraction compressor \citep{biased2020}. Indeed, it is easy to verify that for any ${\tilde Q}$ satisfying \eqref{eq:unbiased}, $\frac{1}{\omega+1}{\tilde Q}$ is  a contraction compressor satisfying \eqref{eq:contractor} with $\delta = \nicefrac{1}{(\omega+1)}$ as follows. 
\begin{eqnarray*}
	\mathbb{E} \left[ \left\|\frac{1}{\omega+1}{\tilde Q}(x) - x \right\|^2 \right] &=& \frac{1}{(\omega+1)^2} \mathbb{E} \left[ \|{\tilde Q}(x)\|^2 \right] + \|x\|^2 - \frac{2}{\omega+1}\mathbb{E} \left[ \langle {\tilde Q}(x), x\rangle \right]  \\ 
	&\leq& \frac{1}{\omega+1}\|x\|^2 + \|x\|^2 - \frac{2}{\omega+1}\|x\|^2 = \left(1 - \frac{1}{\omega+1} \right) \|x\|^2. 
\end{eqnarray*}

For more examples of contraction and unbiased compressors, we refer the reader to \citep{biased2020}. For the {\rm TopK} and {\rm RandK} compressors, we have the following property. 

\begin{lemma}[Lemma A.1 in \citep{Stich18}]
	For the {\rm TopK} and {\rm RandK} compressors with $1\leq K \leq d$, we have 
	$$
	\mathbb{E} \left[ \|{\rm TopK}(x) - x \|^2 \right] \leq \left(  1 - \frac{K}{d}  \right) \|x\|^2, 
\qquad 	\mathbb{E} \left[ \|{\rm RandK}(x) - x \|^2 \right] \leq \left(  1 - \frac{K}{d}  \right) \|x\|^2.  
	$$
\end{lemma}

\noindent Now we propose some new contraction compressors. \\ 
{\bf Composition of the unbiased compressor and contraction compressor.} Let $S \in [d]$, and denote $S_i \in S$ such that $S_1 < S_2 < \cdots < S_{|S|}$, where $|S|$ represents the cardinality of $S$. For any $x \in \R^d$, define $x_{S} \in \R^{|S|}$ such that $(x_{S})_i = x_{S_i}$ for $1\leq i \leq |S|$. For any $y \in \R^{|S|}$, define $y_{S^{-1}} \in \R^{d}$ such that $(y_{S^{-1}})_{S_i} = y_i$ for $1\leq i \leq |S|$ and $(y_{S^{-1}})_j = 0$ for $j\notin S$. Then we have following result. 

\begin{theorem}\label{th:compositioncomp}
For any unbiased compressor ${\tilde Q}$ with parameter $\omega$,  and any contraction compressor $Q$ with parameter $\delta$, Define ${\tilde Q} \circ Q:$ $ x \to \left(  \frac{1}{\omega + 1} {\tilde Q} (Q(x)_S)  \right)_{S^{-1}}$, where $x\in \R^d$, and $S$ is any subset of $[d]$ such that $Q(x)_j = 0$ for $j\notin S$ ($S$ can depend on $Q(x)$). Then ${\tilde Q} \circ Q$ is a contraction compressor with parameter $\frac{\delta}{\omega+1}$. 
\end{theorem}

\begin{proof}
For any $x\in \R^d$, we have 
\begin{align*}
\mathbb{E} \| x - {\tilde Q} \circ Q(x)\|^2 & =  \mathbb{E} \left\|  x - \left(  \frac{1}{\omega + 1} {\tilde Q} (Q(x)_S)  \right)_{S^{-1}}  \right\|^2 \\ 
& = \|x\|^2 + \mathbb{E} \left\| \left(  \frac{1}{\omega + 1} {\tilde Q} (Q(x)_S)  \right)_{S^{-1}}  \right\|^2 - 2\mathbb{E} \left\langle x, \left(  \frac{1}{\omega + 1} {\tilde Q} (Q(x)_S)  \right)_{S^{-1}}  \right\rangle \\ 
& = \|x\|^2 + \mathbb{E} \left\| \frac{1}{\omega + 1} {\tilde Q} (Q(x)_S)   \right\|^2 - 2\mathbb{E} \left\langle x, \frac{Q(x)}{\omega+1}  \right\rangle \\  
& \leq \|x\|^2 + \frac{1}{\omega+1} \mathbb{E}\|Q(x)_S\|^2 - 2\mathbb{E} \left\langle x, \frac{Q(x)}{\omega+1}  \right\rangle, 
\end{align*}
where in the third equality we use $\|y_{S^{-1}}\|^2 = \|y\|^2$ for any $y\in \R^{|S|}$, 
$$
\mathbb{E} \left[ \left(  \frac{1}{\omega + 1} {\tilde Q} (Q(x)_S)  \right)_{S^{-1}} \right] = \mathbb{E} \left[ \left(  \frac{1}{\omega + 1} Q(x)_S  \right)_{S^{-1}} \right], 
$$ 
and the fact that for any $x\in \R^d$, if $x_j = 0$ for $j\notin S$, then $(x_S)_{S^{-1}} = x$. Since $\|Q(x)_S\|^2 = \|Q(x)\|^2$, we further have 
\begin{align*}
\mathbb{E} \| x - {\tilde Q} \circ Q(x)\|^2 & \leq \|x\|^2 + \frac{1}{\omega+1} \mathbb{E}\|Q(x)\|^2 - 2\mathbb{E} \left\langle x, \frac{Q(x)}{\omega+1}  \right\rangle \\ 
& = \left(  1 - \frac{1}{\omega+1}  \right)\|x\|^2 + \frac{1}{\omega+1} \mathbb{E} \left[  \|x\|^2 + \|Q(x)\|^2 - 2 \langle x, Q(x)\rangle  \right] \\ 
& = \left(  1 - \frac{1}{\omega+1}  \right)\|x\|^2  + \frac{1}{\omega+1} \mathbb{E}\|x-Q(x)\|^2 \\ 
& \leq \left(  1 - \frac{\delta}{\omega+1}  \right) \|x\|^2. 
\end{align*}
\end{proof}

\noindent If we let $Q$ be the TopK compressor, ${\tilde Q}$ be the general exponential dithering operator \citep{biased2020}, and $S$ be $[d]$, then we recover the TopK combined with exponential dithering compressor in \citep{biased2020}. 

We can construct some concrete contraction compressors based on Theorem \ref{th:compositioncomp}. For example, let RTopK be the composition of TopK and random dithering \citep{Alistarh17}, and NTopK be the composition of TopK and natural compression \citep{horvath2019natural}, where $S$ is composed of the $K$ indexes in ${\rm TopK}(x)$, i.e., $S = \{  \pi(1), ..., \pi(K)  \}$.

We may use the following assumptions for the contraction compressor in some cases; and this will allow us to obtained improved results. 

\begin{assumption}\label{as:expcompressor}
	$\mathbb{E}[Q(x)] = \delta x$. $\mathbb{E}[Q_1(x)] = \delta_1 x$. 
\end{assumption}

It is easy to verify that {\rm RandK} compressor satisfies Assumption \ref{as:expcompressor} with $\delta = \nicefrac{K}{d}$, and ${\tilde Q}/(\omega+1)$, where ${\tilde Q}$ is any unbiased compressor, also satisfies Assumption \ref{as:expcompressor} with $\delta = \nicefrac{1}{(\omega+1)}$.

\section{Error Compensated L-SVRG }

We now describe the error compensated L-SVRG algorithm (Algotirhm \ref{alg:ec-lsvrg}). Roughly speaking, EC-LSVRG is a combination of L-SVRG, the error feedback technique, and the learning scheme method in VR-DIANA \citep{Samuel19}. The search direction in L-SVRG is 
\begin{equation}\label{eq:sdinLsvrg}
\squeeze
\frac{1}{n} \sum_{\tau=1}^n  \nabla f_{i_k^\tau}^{(\tau)}(x^k) - \nabla f_{i_k^\tau}^{(\tau)}(w^k) + \nabla f^{(\tau)}(w^k) , 
\end{equation}
where $i_k^\tau$ is sampled uniformly and independently from $[m] \eqdef \{ 1, 2, \dots, m  \}$ on $\tau$-th node for $1\leq \tau \leq n$, $x^k$ is the current iteration, and $w^k$ is the reference point. Since when $\psi$ is nonzero in problem (\ref{primal-LSVRG}), $\nabla f(x^*)$ is nonzero in general, and so is $\nabla f^{(\tau)}(x^*)$. Thus, compressing the direction 
$$
\nabla f_{i_k^\tau}^{(\tau)}(x^k) - \nabla f_{i_k^\tau}^{(\tau)}(w^k) + \nabla f^{(\tau)}(w^k)
$$ 
directly on each node would cause nonzero noise even when $x^k$ and $w^k$ goes to the optimal solution $x^*$. We adopt a learning scheme method in VR-DIANA \citep{Samuel19}, but with the contraction compressor rather than the unbiased compressor. We maintain a vector $h_\tau^k\in \R^d$ on each node, and use it to learn $\nabla f^{(\tau)}(w^k)$ iteratively. The same copy of $h^k$, which is the average of $h_\tau^k$, is also maintained on each node. We substract $h_\tau^k$ from the search direction in L-SVRG on each node, and then add $h^k$ back after the aggregation. The search direction of EC-LSVRG becomes 
$$
g^k_\tau = \nabla f_{i_k^\tau}^{(\tau)}(x^k) - \nabla f_{i_k^\tau}^{(\tau)}(w^k) + \nabla f^{(\tau)}(w^k) - h^k_{\tau},
$$
that could be small if  $w^k$ is close to $x^k$ and $h_\tau^k$ close to $\nabla f^{(\tau)}(w^k)$. Before compressing $g_\tau^k$, we apply the error feedback technique. The accumulated error vector $e_\tau^k$ is maintained on each node and will be added to $\eta g_\tau^k$, where $\eta$ represents the stepsize, before the compression. $e^{k+1}_\tau$ is updated by the compression error at iteration $k$ for each node. On each node, a scalar $u^k_\tau$ is also maintained, and only $u^k_1$ will be updated. The summation of $u^k_\tau$ is $u^k$, and we use $u^k$ to control the updating frequency of the reference point $w^k$. All nodes maintain the same copies of $x^k$, $w^k$, $y^k$, $z^k$, and $u^k$. Each node sends their compressed vectors $y^k_{\tau}$, $z_\tau^k$, and $u^{k+1}_\tau$ to the other nodes. At last, the proximal step is taken on each node, where we use the standard proximal operator: 
$$
\prox_{\eta \psi} (x) \eqdef \arg\min_y \left\{  \frac{1}{2}\|x-y\|^2 + \eta \psi(y)  \right\}. 
$$
The reference point $w^k$ will be updated if $u^{k+1}=1$. It is easy to see that $w^k$ will be updated with probability $p$ at each iteration. 

\begin{algorithm}[tb]
	\caption{Error compensated loopless SVRG (EC-LSVRG)}
	\label{alg:ec-lsvrg}
	\begin{algorithmic}[1]
		\STATE {\bfseries Parameters:} stepsize $\eta >0$; probability $p \in (0, 1]$
		\STATE {\bfseries Initialization:}
		$x^0 = w^0 \in \R^d$; $e^0_\tau = 0 \in \R^d$; $u^0=1 \in \R$; $h^0_{\tau} \in \R^d$; $h^0 = \frac{1}{n} \sum_{\tau=1}^n h^0_\tau$
		\FOR{ $k = 0, 1, 2, \dots$} 
		\FOR{ $\tau = 1, \dots, n$} 
		\STATE Sample $i_k^\tau$ uniformly and independently in $[m]$ on each node 
		\STATE $g^k_{\tau} = \nabla f_{i_k^\tau}^{(\tau)}(x^k) - \nabla f_{i_k^\tau}^{(\tau)}(w^k) + \nabla f^{(\tau)}(w^k) - h^k_{\tau}$ 
		\STATE $y^k_{\tau} = Q(\eta g^k_{\tau} + e^k_{\tau})$ 
		\STATE $e^{k+1}_{\tau} = e^k_{\tau} + \eta g^k_{\tau} - y^k_{\tau}$
		\STATE $z^k_{\tau} = Q_1(\nabla f^{(\tau)}(w^k) - h^k_\tau)$ 
		\STATE $h^{k+1}_{\tau} = h^k_{\tau} + z^k_{\tau}$
		\STATE $u^{k+1}_\tau = 0$ for $\tau = 2, \dots, n$ 
		\STATE $
		u^{k+1}_1 = \left\{ \begin{array}{rl}
		1 & \mbox{ with probability $p$} \\
		0 &\mbox{ with probability $1-p$}
		\end{array} \right.
		$
		\STATE Send $y^k_{\tau}$, $z^k_{\tau}$, and $u^{k+1}_\tau$ to the other nodes 
		\STATE Receive $y^k_{\tau}$, $z^k_{\tau}$, and $u^{k+1}_\tau$  from the other nodes
		\STATE $y^k = \frac{1}{n} \sum_{\tau=1}^n y^k_{\tau}$ 
		\STATE $z^k = \frac{1}{n} \sum_{\tau=1}^n z^k_{\tau}$
		\STATE $u^{k+1} = \sum_{\tau=1}^n u^{k+1}_\tau$ 
		\STATE $x^{k+0.5} = x^k - (y^k + \eta h^k)$
		\STATE $x^{k+1} = \prox_{\eta \psi} (x^{k+0.5})$
		\STATE $
		w^{k+1} = \left\{ \begin{array}{rl}
		x^k & \mbox{ if $u^{k+1} = 1$} \\
		w^k &\mbox{ otherwise }
		\end{array} \right.
		$
		\STATE $h^{k+1} = h^k + z^k$
		\ENDFOR
		\ENDFOR
	\end{algorithmic}
\end{algorithm}

In algorithm \ref{alg:ec-lsvrg}, let $e^k = \frac{1}{n}\sum_{\tau=1}^n e^k_\tau$, $g^k = \frac{1}{n} \sum_{\tau=1}^n g^k_\tau$, and ${\tilde x}^k = x^k - e^k$ for $k\geq 0$. Then we have 
$$
e^{k+1} = \frac{1}{n} \sum_{\tau=1}^n \left(  e^k_\tau + \eta g^k_\tau - y^k_\tau  \right) = e^k + \eta g^k - y^k, 
$$
and 
\begin{eqnarray*}
	{\tilde x}^{k+1} &=& x^{k+1} - e^{k+1} \\
	&=& x^{k+0.5} - \eta \partial \psi(x^{k+1}) - e^{k+1} \\ 
	&=& x^k - y^k - \eta h^k - \eta \partial \psi(x^{k+1}) - e^k  - \eta g^k + y^k \\ 
	&=& {\tilde x}^k - \eta(g^k + h^k + \partial \psi(x^{k+1})). 
\end{eqnarray*}

\subsection{Composite case ($\psi\neq 0$) of EC-LSVRG}

We need the following assumptions in this subsection. 

\begin{assumption}\label{as:contracQQ1}
The two compressors $Q$ and $Q_1$ in Algorithm \ref{alg:ec-lsvrg}  are contraction compressors with parameters $\delta$ and $\delta_1$, respectively. 
\end{assumption}

\begin{assumption}\label{as:eclsvrg}
	$f_i^{(\tau)}$ is $L$-smooth, $f^{(\tau)}$ is ${\bar L}$-smooth, $f$ is $L_f$-smooth, and $\psi$ is $\mu$-strongly convex. $L_f \geq \mu\geq 0$. 
\end{assumption}

The following are the main results. We use two Lyapunov functions for two cases: with or without Assumption \ref{as:expcompressor} in the following two theorems.

\begin{theorem}\label{th:eclsvrg-1}
	Let Assumption \ref{as:contracQQ1} and Assumption \ref{as:eclsvrg} hold. Define 
	\begin{eqnarray*}
	\squeeze
		\Phi^k_1 &\eqdef&  \left(1 + \frac{\eta \mu}{2} \right)\mathbb{E} \left[ \|\tilde{x}^{k} - x^*\|^2 \right] + \frac{9}{\delta n} \sum_{\tau=1}^n\mathbb{E} \left[ \|e^{k}_{\tau}\|^2\right] + \frac{164(1-\delta)\eta^2}{\delta^2 \delta_1 n} \sum_{\tau=1}^n \mathbb{E} \left[ \| h^{k}_\tau - \nabla f^{(\tau)}(w^{k})\|^2 \right] \\ 
		&& + \frac{4\eta^2}{3p} \left(  \frac{41(1-\delta)}{\delta} \left(  \frac{4{\bar L}}{\delta} + L  + \frac{16{\bar L}p}{\delta \delta_1} \left(  1 + \frac{2p}{\delta_1}  \right) \right) + \frac{16L}{n} \right)  \mathbb{E} [P(w^{k}) - P(x^*)]. 
	\end{eqnarray*}
	If $\eta \leq \frac{1}{4L_f}$, then we have 
	\begin{eqnarray*}
		\mathbb{E} \left[\Phi^{k+1}_1 \right] &\leq& \left(  1 - \min\left\{  \frac{\mu \eta}{3}, \frac{\delta}{4}, \frac{\delta_1}{4}, \frac{p}{4}  \right\}  \right) \mathbb{E} \left[\Phi^k_1 \right] + 2\eta \mathbb{E} \left[P(x^*) - P(x^{k+1}) \right] \\ 
		&& + \left(  \frac{96(1-\delta)}{\delta} \left(  \frac{4{\bar L}}{\delta} + L  + \frac{16{\bar L}p}{\delta \delta_1} \left(  1 + \frac{2p}{\delta_1}  \right) \right) + \frac{38L}{n}  \right) \eta^2 \mathbb{E} \left[P(x^k) - P(x^*) \right]. 
	\end{eqnarray*}
\end{theorem}

\begin{theorem}\label{th:eclsvrg-2}
	Let Assumption \ref{as:contracQQ1} and Assumption \ref{as:eclsvrg} hold. Define 
	\begin{eqnarray*}
		{\Phi_2^{k}} &\eqdef& \left(1 + \frac{\eta \mu}{2} \right)\mathbb{E}\|\tilde{x}^{k} - x^*\|^2 + \frac{9}{\delta} \mathbb{E}\|e^{k}\|^2 + \frac{84(1-\delta)}{\delta n^2} \sum_{\tau=1}^n \mathbb{E}\|e^{k}_\tau\|^2 \\ 
		&&  + \frac{164(1-\delta)\eta^2}{\delta^2 \delta_1} \mathbb{E}\|h^{k} - \nabla f(w^{k})\|^2 + \frac{2000(1-\delta)\eta^2}{\delta^2 \delta_1 n^2} \sum_{\tau=1}^n \mathbb{E} \|h^{k}_\tau - \nabla f^{(\tau)}(w^{k})\|^2 \\ 
		&& + \frac{4\eta^2}{3p}  \left(  \frac{(1-\delta)}{\delta} \left(  \frac{164L_f}{\delta} + \frac{1344{\bar L}}{\delta n} + \frac{541L}{n}  + \frac{(656L_f + \frac{8000{\bar L}}{n})p}{\delta \delta_1}\left(  1 + \frac{2p}{\delta_1}  \right)  \right) + \frac{16L}{n} \right) \mathbb{E} [P(w^{k}) - P(x^*)] 
	\end{eqnarray*}
	Under Assumption \ref{as:expcompressor}, if $\eta \leq \frac{1}{4L_f}$, then we have 
	\begin{eqnarray*}
		\mathbb{E} [{\Phi}_2^{k+1}] &\leq& \left(  1 - \min\left\{  \frac{\mu \eta}{3}, \frac{\delta}{4}, \frac{\delta_1}{4}, \frac{p}{4}  \right\}  \right) \mathbb{E} [{\Phi}_2^k] + 2\eta \mathbb{E}[P(x^*) - P(x^{k+1})] \\ 
		&& +  \left(  \frac{(1-\delta)}{\delta} \left(  \frac{383L_f}{\delta} + \frac{3136{\bar L}}{\delta n} + \frac{1263L}{n} + \frac{(1531L_f + \frac{18667{\bar L}}{n})p}{\delta \delta_1}\left(  1 + \frac{2p}{\delta_1}  \right)  \right) + \frac{38L}{n} \right)  \eta^2  \mathbb{E}[P(x^k)- P(x^*) ]. 
	\end{eqnarray*}
\end{theorem}

\noindent From the above two theorems, we can get the iteration complexity. 

\begin{theorem}\label{th:eclsvrg-11}
	Let Assumption \ref{as:contracQQ1} and Assumption \ref{as:eclsvrg} hold.\\ Let $w_k = \left(  1 - \min\left\{  \frac{\mu \eta}{3}, \frac{\delta}{4}, \frac{\delta_1}{4}, \frac{p}{4}  \right\}  \right)^{-k}$, $W_k = \sum_{i=0}^k w_i$, and ${\bar x}^k = \frac{1}{W_k} \sum_{i=0}^k w_ix^i$. \\ 
	(i) Let $\eta \leq {1}/{\left(  \frac{105(1-\delta)}{\delta} \left(  \frac{4{\bar L}}{\delta} + L  + \frac{16{\bar L}p}{\delta \delta_1} \left(  1 + \frac{2p}{\delta_1}  \right) \right) + 4L_f + \frac{42L}{n}  \right)}$. If $\mu>0$, we have
	$$
	\mathbb{E}[P({\bar x}^k)- P(x^*)] \leq  \frac{\frac{\mu}{2} \|x^0 - x^*\|^2 + \frac{1}{2} (P(x^0) - P(x^*))  + \frac{1}{40{\bar L}n}\sum_{\tau=1}^n \| \nabla f^{(\tau)}(x^0)  - h_\tau^0\|^2 }{1 - (1 - \min\{   \frac{\mu \eta}{3}, \frac{\delta}{4}, \frac{\delta_1}{4},  \frac{p}{4}   \})^{k+1}} \left(  1 - \min\left\{  \frac{\mu \eta}{3}, \frac{\delta}{4}, \frac{\delta_1}{4}, \frac{p}{4}   \right\}  \right)^k. 
	$$
	If $\mu=0$, we have 
	$$
	\mathbb{E}[P({\bar x}^k) - P(x^*)] \leq \frac{1}{k+1} \left( \frac{9}{8\eta} \|x^0-x^*\|^2 + \frac{1}{10{\bar L}p n} \sum_{\tau=1}^n \| \nabla f^{(\tau)}(x^0)  - h_\tau^0\|^2 + \frac{8}{3p} (P(x^0) - P(x^*))  \right). 
	$$
	(ii) Let $\eta = {1}/{\left(  \frac{105(1-\delta)}{\delta} \left(  \frac{4{\bar L}}{\delta} + L  + \frac{16{\bar L}p}{\delta \delta_1} \left(  1 + \frac{2p}{\delta_1}  \right) \right) +4L_f + \frac{42L}{n}  \right)}$ and $p\leq O(\delta_1)$.  If $\mu>0$ and $\epsilon \leq \frac{\mu}{2}\|x^0-x^*\|^2 + \frac{1}{2} (P(x^0) - P(x^*)) + \frac{1}{40{\bar L}n}\sum_{\tau=1}^n \| \nabla f^{(\tau)}(x^0)  - h_\tau^0\|^2$, we have $\mathbb{E}[P({\bar x}^k) - P(x^*)] \leq \epsilon$ as long as 
	$$
	k \geq O\left(  \left(  \frac{1}{\delta} + \frac{1}{p} + \frac{L_f}{\mu} + \frac{L}{n \mu}  +  \frac{(1-\delta){\bar L}}{\delta^2 \mu} + \frac{(1-\delta)L}{\delta \mu} \right)  \ln\frac{1}{\epsilon} \right). 
	$$
	If $\mu=0$, we have $\mathbb{E}[P({\bar x}^k) - P(x^*)] \leq \epsilon$ as long as  
	$$
	k \geq O\left(  \left( \frac{1}{p} + L_f + \frac{L}{n} + \frac{(1-\delta){\bar L}}{\delta^2} + \frac{(1-\delta) L}{\delta}   \right) \frac{1}{\epsilon} \right). 
	$$
\end{theorem}

\begin{theorem}\label{th:eclsvrg-22}
	Let Assumption \ref{as:contracQQ1} and Assumption \ref{as:eclsvrg} hold. Assume the compressors $Q$ and $Q_1$ also satisfy Assumption \ref{as:expcompressor}. 
	Let $w_k = \left(  1 - \min\left\{  \frac{\mu \eta}{3}, \frac{\delta}{4}, \frac{\delta_1}{4}, \frac{p}{4}  \right\}  \right)^{-k}$, $W_k = \sum_{i=0}^k w_i$, and ${\bar x}^k = \frac{1}{W_k} \sum_{i=0}^k w_ix^i$. \\ 
	(i) Let $\eta \leq {1}/{ \left(  \frac{(1-\delta)}{\delta} \left(  \frac{418L_f}{\delta} + \frac{3422{\bar L}}{\delta n} + \frac{1349L}{n} + \frac{(1671L_f + \frac{20364{\bar L}}{n})p}{\delta \delta_1}\left(  1 + \frac{2p}{\delta_1}  \right)  \right) + 4L_f + \frac{42L}{n} \right) }$. If $\mu>0$, we have
	\begin{eqnarray*}
	\mathbb{E}[P({\bar x}^k)- P(x^*)] &\leq&  \frac{  \frac{\mu}{2} \|x^0 - x^*\|^2 + \frac{1}{2} (P(x^0) - P(x^*)) + \frac{1}{40L_f}\|\nabla f(x^0)\|^2 + \frac{3}{10L_f n}\sum_{\tau=1}^n \| \nabla f^{(\tau)}(x^0)  - h_\tau^0\|^2 }{1 - (1 - \min\{   \frac{\mu \eta}{3}, \frac{\delta}{4}, \frac{\delta_1}{4},  \frac{p}{4}   \})^{k+1}} \\ 
	&& \cdot \left(  1 - \min\left\{  \frac{\mu \eta}{3}, \frac{\delta}{4}, \frac{\delta_1}{4}, \frac{p}{4}   \right\}  \right)^k. 
	\end{eqnarray*}
	If $\mu=0$, we have 
	\begin{eqnarray*}
	\mathbb{E}[P({\bar x}^k) - P(x^*)] &\leq& \frac{1}{k+1} \left( \frac{9}{8\eta} \|x^0-x^*\|^2 +  \frac{1}{10L_fp} \|\nabla f(x^0)\|^2  \right. \\ 
	&& \left. +  \frac{6}{5 L_f p n} \sum_{\tau=1}^n \| \nabla f^{(\tau)}(x^0)  - h_\tau^0\|^2  + \frac{8}{3p} \left(P(x^0) - P(x^*) \right)  \right). 
	\end{eqnarray*}
	(ii) Let $\eta = {1}/{ \left(  \frac{(1-\delta)}{\delta} \left(  \frac{418L_f}{\delta} + \frac{3422{\bar L}}{\delta n} + \frac{1349L}{n} + \frac{(1671L_f + \frac{20364{\bar L}}{n})p}{\delta \delta_1}\left(  1 + \frac{2p}{\delta_1}  \right)  \right) + 4L_f + \frac{42L}{n} \right) }$ and $p\leq O(\delta_1)$.  If $\mu>0$ and $\epsilon \leq  \frac{\mu}{2} \|x^0 - x^*\|^2 + \frac{1}{2} (P(x^0) - P(x^*)) + \frac{1}{40L_f}\|\nabla f(x^0)\|^2 + \frac{3}{10L_f n}\sum_{\tau=1}^n \| \nabla f^{(\tau)}(x^0)  - h_\tau^0\|^2$, we have $\mathbb{E}[P({\bar x}^k) - P(x^*)] \leq \epsilon$ as long as 
	$$
	k \geq O\left(  \left(  \frac{1}{\delta} + \frac{1}{p} + \frac{L_f}{\mu} + \frac{L}{n \mu}  +  \frac{(1-\delta)L_f}{\delta^2 \mu} + \frac{(1-\delta)L}{n \delta  \mu} \right)  \ln\frac{1}{\epsilon} \right). 
	$$
	If $\mu=0$, we have $\mathbb{E}[P({\bar x}^k) - P(x^*)] \leq \epsilon$ as long as  
	$$
	k \geq O\left(  \left( \frac{1}{p} + L_f + \frac{L}{n} + \frac{(1-\delta)L_f}{\delta^2} + \frac{(1-\delta) L}{n \delta}   \right) \frac{1}{\epsilon} \right). 
	$$
\end{theorem}

{\bf Remark.} The iteration complexity of EC-LSVRG in the case where $p\geq O(\delta_1)$ can be obtained easily, however, it is no better than the iteration complexity of EC-LSVRG in the case where $p\leq O(\delta_1)$. Thus we omit that case for simplicity. 

Noticing that $L_f\leq {\bar L} \leq nL_f$ and ${\bar L} \leq L \leq m{\bar L}$, the iteration complexity in Theorem \ref{th:eclsvrg-22} could be better than that in Theorem \ref{th:eclsvrg-11}. On the other hand, if $L_f = {\bar L} = L$, then both iteration complexities in Theorem \ref{th:eclsvrg-11} and Theorem \ref{th:eclsvrg-22} in the strongly convex case become 
\begin{equation}\label{eq:iter}
O\left(   \left(  \frac{1}{\delta} + \frac{1}{p} + \frac{L}{\mu} + \frac{(1-\delta) L}{\delta^2 \mu}  \right) \ln \frac{1}{\epsilon} \right). 
\end{equation}

\subsection{Smooth case ($\psi=0$) of EC-LSVRG}

In this subsection, we study the Algorithm \ref{alg:ec-lsvrg} for problem (\ref{primal-LSVRG}) with $\psi =0$. 
We need the following assumption in this subsection. 

\begin{assumption}\label{as:eclsvrgsmooth}
	$f^{(\tau)}_i$ is $L$-smooth, $f^{(\tau)}$ is ${\bar L}$-smooth, $f$ is $L_f$-smooth and $f$ is $\mu$-strongly convex. 
\end{assumption}

We also use two Lyapulov functions for two cases: with or without Assumption \ref{as:expcompressor} in the following two theorems. 

\begin{theorem}\label{th:eclsvrgsmooth-1}
	Let Assumption \ref{as:contracQQ1} and Assumption \ref{as:eclsvrgsmooth} hold. Define 
	\begin{eqnarray*}
	\Phi_3^k &\eqdef&  \mathbb{E} \|{\tilde x}^{k} - x^*\|^2 + \frac{12L_f \eta}{n \delta } \sum_{\tau=1}^n \mathbb{E}\|e^{k}_\tau\|^2 + \frac{192(1-\delta)L_f \eta^3}{\delta^2 \delta_1 n}  \sum_{\tau=1}^n \mathbb{E}\|h^{k}_\tau - \nabla f^{(\tau)}(w^{k})\|^2 \\ 
	&& + \frac{4}{3p}  \left(   \frac{48(1-\delta) L_f \eta^3}{\delta} \left(  \frac{4{\bar L}}{\delta} + L  + \frac{16{\bar L}p}{\delta \delta_1} \left(  1 + \frac{2p}{\delta_1}  \right) \right) + \frac{4L\eta^2}{n}  \right) \mathbb{E} [f(w^{k}) - f(x^*)]. 
	\end{eqnarray*}
	If $\eta\leq \frac{1}{4L_f + 8L/n}$, then 
	\begin{eqnarray*}
		\mathbb{E} [\Phi_3^{k+1}] &\leq& \left(  1 - \min\left\{  \frac{\mu \eta}{2}, \frac{\delta}{4}, \frac{\delta_1}{4}, \frac{p}{4}  \right\}  \right) \mathbb{E} [\Phi_3^k] \\ 
		&& - \frac{\eta}{2} \left(  1 - \frac{224(1-\delta)L_f \eta^2}{\delta} \left(  \frac{4{\bar L}}{\delta} + L + \frac{16{\bar L}p}{\delta \delta_1} \left(  1 + \frac{2p}{\delta_1}  \right)   \right) - \frac{11L\eta}{n}   \right)  \mathbb{E} [f(x^k) - f(x^*)]. 
	\end{eqnarray*}
\end{theorem}

\begin{theorem}\label{th:eclsvrgsmooth-2}
	Let Assumption \ref{as:contracQQ1} and Assumption \ref{as:eclsvrgsmooth} hold. Define 
	\begin{eqnarray*}
		\Phi_4^{k} &\eqdef& \mathbb{E} \|{\tilde x}^{k} - x^*\|^2 + \frac{12L_f \eta}{\delta} \mathbb{E}\|e^{k}\|^2 +  \frac{96(1-\delta) L_f \eta}{n^2 \delta } \sum_{\tau=1}^n \mathbb{E}\|e^{k}_\tau\|^2 \\ 
		&& + \frac{192(1-\delta)L_f \eta^3}{\delta^2 \delta_1} \mathbb{E}\|h^{k} - \nabla f(w^{k})\|^2  + \frac{2304(1-\delta)L_f \eta^3}{\delta^2 \delta_1 n^2} \sum_{\tau=1}^n \mathbb{E}\|h^{k}_\tau - \nabla f^{(\tau)}(w^{k})\|^2  \\ 
		&& + \frac{4}{3p} \left( \frac{48(1-\delta)L_f \eta^3}{\delta} \left(  \frac{4L_f}{\delta} + \frac{32{\bar L}}{n \delta} + \frac{13L}{n}  + \frac{16p(L_f + \frac{12{\bar L}}{n})}{\delta\delta_1} \left(  1 + \frac{2p}{\delta_1}  \right)\right) + \frac{4L\eta^2}{n}  \right) \mathbb{E} [f(w^{k}) - f(x^*)]. 
	\end{eqnarray*}
	Under Assumption \ref{as:expcompressor}, if $\eta\leq \frac{1}{4L_f + 8L/n}$, then
	\begin{eqnarray*}
		\mathbb{E} [\Phi_4^{k+1}] &\leq& \left(  1 - \min\left\{  \frac{\mu\eta}{2}, \frac{\delta}{4}, \frac{\delta_1}{4}, \frac{p}{4}  \right\}  \right) \mathbb{E}[\Phi_4^k]  \\ 
		&& -  \frac{\eta}{2} \left(  1 - \frac{224(1-\delta) L_f \eta^2}{\delta} \left(  \frac{4L_f}{\delta} + \frac{32{\bar L}}{n \delta} + \frac{13L}{n}  + \frac{16p(L_f + \frac{12{\bar L}}{n})}{\delta\delta_1} \left(  1 + \frac{2p}{\delta_1}  \right)\right) -  \frac{11L\eta}{n} \right) \mathbb{E} [f(x^k) - f(x^*)]. 
	\end{eqnarray*}
\end{theorem}

From the above two theorems, we can get the iteration complexity. 

\begin{theorem}\label{th:eclsvrgsmooth-11}
	Let Assumption \ref{as:contracQQ1} and Assumption \ref{as:eclsvrgsmooth} hold. \\ Let $w_k = \left(  1 - \min\left\{  \frac{\mu \eta}{2}, \frac{\delta}{4}, \frac{\delta_1}{4}, \frac{p}{4}  \right\}  \right)^{-k}$, $W_k = \sum_{i=0}^k w_i$, and ${\bar x}^k = \frac{1}{W_k} \sum_{i=0}^k w_ix^i$. \\ 
	(i) Let $\eta \leq \min\left\{  \frac{1}{4L_f + 33L/n}, \frac{\delta}{60\sqrt{(1-\delta)L_f{\bar L}}}, \frac{\sqrt{\delta}}{64\sqrt{(1-\delta) L_fL}}, \frac{\delta \sqrt{\delta_1}}{120\sqrt{(1-\delta)L_f{\bar L}p\left(1 + \frac{2p}{\delta_1} \right)}}  \right\}$. If $\mu>0$, we have
	\begin{eqnarray*}
		 \mathbb{E} [f({\bar x}^k) - f(x^*)] \leq \frac{ 9\mu\|x^0-x^*\|^2 + 2(f(x^0) - f(x^*)) + \frac{1}{15{\bar L}n} \sum \limits_{\tau=1}^n \|\nabla f^{(\tau)}(x^0)  - h_\tau^0\|^2 }{1 - (1 - \min\{   \frac{\mu \eta}{2}, \frac{\delta}{4}, \frac{\delta_1}{4}, \frac{p}{4}   \})^{k+1}} \left(1 - \min\left\{   \frac{\mu \eta}{2}, \frac{\delta}{4}, \frac{\delta_1}{4}, \frac{p}{4}   \right\} \right)^k. 
	\end{eqnarray*}
	
	If $\mu=0$, we have 
	$$
	\mathbb{E}[f({\bar x}^k) - f(x^*)] \leq  \frac{1}{k+1} \left( \frac{18}{\eta} \|x^0 - x^*\|^2 + \frac{1}{3{\bar L}p n} \sum_{\tau=1}^n \|\nabla f^{(\tau)}(x^0)  - h_\tau^0\|^2  + \frac{6}{p} \left(f(x^0) - f(x^*) \right)  \right). 
	$$
	(ii) Let $\eta = \min\left\{  \frac{1}{4L_f + 33L/n}, \frac{\delta}{60\sqrt{(1-\delta)L_f{\bar L}}}, \frac{\sqrt{\delta}}{64\sqrt{(1-\delta) L_fL}}, \frac{\delta \sqrt{\delta_1}}{120\sqrt{(1-\delta)L_f{\bar L}p\left(1 + \frac{2p}{\delta_1} \right)}}  \right\}$ and $p\leq O(\delta_1)$.  If $\mu>0$ and $\epsilon \leq  9\mu\|x^0-x^*\|^2 + 2(f(x^0) - f(x^*)) + \frac{1}{15{\bar L}n} \sum_{\tau=1}^n \|\nabla f^{(\tau)}(x^0)  - h_\tau^0\|^2 $, we have $\mathbb{E}[f({\bar x}^k) - f(x^*)] \leq \epsilon$ as long as 
	$$
	k \geq O\left( \left( \frac{1}{\delta} + \frac{1}{p} + \frac{L_f}{\mu} + \frac{L}{n \mu}  +  \frac{\sqrt{(1-\delta) L_f{\bar L}}}{\delta \mu}  + \frac{\sqrt{(1-\delta)L_f L}}{\sqrt{\delta}\mu}   \right) \ln \frac{1}{\epsilon} \right). 
	$$
	If $\mu=0$, we have $\mathbb{E}[P({\bar x}^k) - P(x^*)] \leq \epsilon$ as long as  
	$$
	k \geq O\left(  \left( \frac{1}{p} + L_f + \frac{L}{n} +  \frac{\sqrt{(1-\delta) L_f{\bar L}}}{\delta}  + \frac{\sqrt{(1-\delta)L_f L}}{\sqrt{\delta}}    \right) \frac{1}{\epsilon}  \right). 
	$$
\end{theorem}

\begin{theorem}\label{th:eclsvrgsmooth-22}
	Let Assumption \ref{as:contracQQ1} and Assumption \ref{as:eclsvrgsmooth} hold. Assume the compressors $Q$ and $Q_1$ also satisfy Assumption \ref{as:expcompressor}. Let $w_k = \left(  1 - \min\left\{  \frac{\mu \eta}{2}, \frac{\delta}{4}, \frac{\delta_1}{4}, \frac{p}{4}  \right\}  \right)^{-k}$, $W_k = \sum_{i=0}^k w_i$, and ${\bar x}^k = \frac{1}{W_k} \sum_{i=0}^k w_ix^i$. \\ 
	(i) Let 
	\begin{equation}\label{eq:eta2-eclsvrgsmooth}
	\eta \leq  \min \left\{  \frac{1}{4L_f + 33L/n}, \frac{\delta}{60\sqrt{1-\delta}L_f}, \frac{\sqrt{n \delta}}{229\sqrt{(1-\delta)L_fL}}, \frac{\sqrt{n} \delta}{360\sqrt{(1-\delta)L_f {\bar L}}} , \frac{\delta\sqrt{\delta_1}}{120\sqrt{(1-\delta)pL_f\left(  L_f + \frac{12{\bar L}}{n}  \right) \left(  1 + \frac{2p}{\delta_1}  \right)}} \right\}.
	\end{equation}
	If $\mu>0$, we have
	\begin{eqnarray*}
		\mathbb{E} [f({\bar x}^k) - f(x^*)] &\leq& \frac{ 18\mu\|x^0-x^*\|^2  + \frac{2}{L_f n}\sum \limits_{\tau=1}^n \|\nabla f^{(\tau)}(x^0)  - h_\tau^0\|^2 +  5(f(x^0) - f(x^*)) }{1 - (1 - \min\{   \frac{\mu \eta}{2}, \frac{\delta}{4}, \frac{\delta_1}{4}, \frac{p}{4}   \})^{k+1}}\\ 
		&& \cdot \left(1 - \min\left\{   \frac{\mu \eta}{2}, \frac{\delta}{4}, \frac{\delta_1}{4}, \frac{p}{4}   \right\} \right)^k, 
	\end{eqnarray*}
	
	If $\mu=0$, we have 
	$$
	\mathbb{E}[f({\bar x}^k) - f(x^*)] \leq \frac{1}{k+1} \left(  \frac{36}{\eta} \|x^0 - x^*\|^2 + \frac{6}{L_f p n} \sum_{\tau=1}^n \|\nabla f^{(\tau)}(x^0)  - h_\tau^0\|^2  + \frac{18}{p} \left(f(x^0) - f(x^*)\right)  \right). 
	$$
	(ii) Let $\eta$ equal to the upper bound in (\ref{eq:eta2-eclsvrgsmooth}) and $p\leq O(\delta_1)$.  If $\mu>0$ and $\epsilon \leq 9\mu\|x^0-x^*\|^2  + \frac{1}{L_f n}\sum_{\tau=1}^n \|\nabla f^{(\tau)}(x^0)  - h_\tau^0\|^2 +  3(f(x^0) - f(x^*))  $, we have $\mathbb{E}[f({\bar x}^k) - f(x^*)] \leq \epsilon$ as long as 
	$$
	k \geq O\left( \left( \frac{1}{\delta} + \frac{1}{p}  + \frac{L_f}{\mu} + \frac{L}{n \mu} + \frac{\sqrt{(1-\delta)}L_f}{\mu \delta}   \right) \ln \frac{1}{\epsilon} \right), 
	$$
	If $\mu=0$, we have $\mathbb{E}[P({\bar x}^k) - P(x^*)] \leq \epsilon$ as long as  
	$$
	k \geq O\left( \left(  \frac{1}{p} +  L_f + \frac{L}{n} + \frac{\sqrt{(1-\delta)}L_f}{\delta} \right)  \frac{1}{\epsilon}  \right). 
	$$
\end{theorem}

Same as the composite case, the iteration complexity in Theorem \ref{th:eclsvrgsmooth-22} could be better than that in Theorem \ref{th:eclsvrgsmooth-11}. On the other hand, if $L_f = {\bar L} = L$, then both iteration complexities in Theorem \ref{th:eclsvrgsmooth-11} and Theorem \ref{th:eclsvrgsmooth-22} in the strongly convex case become 
\begin{equation}\label{eq:itersmooth}
O\left(   \left(  \frac{1}{\delta} + \frac{1}{p} + \frac{L}{\mu} + \frac{\sqrt{(1-\delta)} L}{\delta \mu}  \right)  \ln\frac{1}{\epsilon}  \right). 
\end{equation}

\section{Error Compensated Quartz and Error Compensated SDCA}\label{sec:ecQuartz}

In this section, we study the following problem: 
\begin{equation}\label{primal-sdca}
	\min_{x\in \mathbb{R}^d} P(x) = \frac{1}{N} \sum_{\tau=1}^n \sum_{i=1}^m \phi_{i \tau}(A_{i\tau}^\top x) + \lambda g(x),
\end{equation}
where $N = mn$ and $A_{i\tau} \in \R^{d\times t}$.  

The corresponding dual problem of problem (\ref{primal-sdca}) is 
$$
\max_{\alpha \in \mathbb{R}^{tN}} D(\alpha) = - {\tilde f} (\alpha) - {\tilde \psi}(\alpha), 
$$
where ${\tilde f} (\alpha) \eqdef \lambda g^*\left(\frac{1}{\lambda N} \sum_{\tau=1}^n \sum_{i=1}^m A_{i\tau}\alpha_{i\tau} \right)$, ${\tilde \psi}(\alpha) \eqdef \frac{1}{N} \sum_{\tau=1}^n \sum_{i=1}^m \phi_{i\tau}^*(-\alpha_{i\tau})$, $\alpha_{i\tau} \in \R^t$,  
$$\alpha = (\alpha_{11}^\top, \dots, \alpha_{m1}^\top, \alpha_{12}^\top, \dots, \alpha_{m2}^\top, \dots, \alpha_{n1}^\top, \dots, \alpha_{mn}^\top)^\top \in \R^{tN},$$ 
$\phi_{i\tau}^*$ and $g^*$ are the conjugate functions of $\phi_{i\tau}$ and $g$ respectively. Generally, for any vector $h \in \R^{tN}$, we use $h_{i\tau} \in \R^t$ with $i\in [m]$ and $\tau \in [n]$ to denote the $i+ m(\tau-1)$-th block vector of $h$. 

We need the following assumptions in this section. 

\begin{assumption}\label{as:contracQ}
	The compressor $Q$ in Algorithm \ref{alg:ec-quartz}  is a contraction compressor with parameters $\delta$. 
\end{assumption}

\begin{assumption}\label{as:ec-Quartz}
$\phi_{i\tau}$ is $\frac{1}{\gamma}$-smooth. $g$ is $1$-strongly convex. $\frac{R^2}{\gamma} \geq \lambda >0$. 
\end{assumption}

The error compensated Quartz and error compensated SDCA are described in Algorithm \ref{alg:ec-quartz}. Quartz is a variance reduced primal-dual method, and also a minibatch version of SDCA. The updates of $x^k$ in Quartz and SDCA are slightly different, and the rest steps are the same. In distributed Quartz, each node needs to communicate $\frac{1}{\lambda m}A_{i_k^\tau \tau} \Delta\alpha_{i^\tau_{k} \tau}^{k+1}$ with each other at each step. The error feedback technique can be applied easily in this case. We maintain an accumulated vector $e_\tau^k$ on each node, and add it to $\frac{1}{\lambda m}A_{i_k^\tau \tau} \Delta\alpha_{i^\tau_{k} \tau}^{k+1}$ before compression. $e^{k+1}_\tau$ is updated by the compression error at iteration $k$ for each node. All nodes maintain the same copies of $x^k$ and $u^k$. The rest of Algorithm \ref{alg:ec-quartz} is the same as Quartz and SDCA. 

\begin{algorithm}[tb]
	\caption{Error compensated Quartz (EC-Quartz) and error compensated SDCA (EC-SDCA)}
	\label{alg:ec-quartz}
	\begin{algorithmic}[1]
		\STATE {\bfseries Parameters:} $\theta>0$; $R_m \eqdef \max_{i, \tau} \|A_{i\tau}\|$; ${\bar R}^2 \eqdef \max_{\tau \in [n]} \{  \frac{1}{m}\lambda_{\rm max}(\sum_{i=1}^m A_{i\tau}A_{i\tau}^\top) \}$; $R^2 \eqdef \frac{1}{N} \lambda_{\rm max} (\sum_{\tau=1}^n \sum_{i=1}^m A_{i\tau} A_{i\tau}^\top)$; $p_{i\tau} = \frac{1}{m} \in \R$ for $i\in [m]$ and $\tau \in [n]$; positive constants $v_{i\tau} = R_m^2 + nR^2 \in \R$ for $i\in [m]$ and $\tau \in [n]$ 
		\STATE {\bfseries Initialization:}
		$\alpha^0 \in \R^{tN}$; $x^0 \in \R^d$; $u^0 = \frac{1}{\lambda N} \sum_{\tau=1}^n \sum_{i=1}^m A_{i\tau} \alpha^0_{i\tau} \in \R^d$; $e^0_\tau=0 \in \R^d$ for $\tau \in [n]$
		\FOR{ $k = 0, 1, 2, \dots$} 
		\FOR{ $\tau = 1, \dots, n$} 
		\STATE {\color{blue} EC-Quartz: } \hspace{0.05em} $x^{k+1} = (1-\theta) x^k + \theta \nabla g^*(u^k)$ 
		\STATE {\color{blue} EC-SDCA: }  \quad $x^{k+1} = \nabla g^*(u^k)$ 
		\STATE $\alpha_{i\tau}^{k+1} = \alpha_{i\tau}^k$ for $i \in [m]$
		\STATE Sample $i_k^\tau$ uniformly and independently in $[m]$ on each node 
		\STATE $\Delta\alpha_{i^\tau_k \tau}^{k+1} = -\theta p_{i^\tau_k \tau}^{-1} \alpha_{i^\tau_k \tau}^{k} - \theta p_{i_k^\tau\tau}^{-1} \nabla \phi_{i^\tau_k \tau} (A_{i_k^\tau \tau}^\top x^{k+1}) $
		\STATE $\alpha_{i^\tau_{k} \tau}^{k+1} = \alpha_{i^\tau_{k} \tau}^{k} + \Delta\alpha_{i^\tau_{k} \tau}^{k+1}$
		\STATE $y_\tau^k = Q\left(  \frac{1}{\lambda m}A_{i_k^\tau \tau} \Delta\alpha_{i^\tau_{k} \tau}^{k+1} + e_\tau^k  \right)$
		\STATE $e_{\tau}^{k+1} = e_\tau^k + \frac{1}{\lambda m}A_{i_k^\tau \tau} \Delta\alpha_{i^\tau_{k} \tau}^{k+1} - y_\tau^k$
		\STATE Send $y^k_{\tau}$ to the other nodes 
		\STATE Receive $y^k_{\tau}$ from the other nodes
		\STATE $u^{k+1} = u^k + \frac{1}{n}\sum_{\tau=1}^n y_\tau^k$
		\ENDFOR
		\ENDFOR
	\end{algorithmic}
\end{algorithm}

\subsection{Convergence of EC-Quartz}

\begin{theorem}\label{th:ecQuartz-1}
Let Assumption \ref{as:contracQ} and Assumption \ref{as:ec-Quartz} hold. Assume $\delta<1$. Define 
$$
\Psi_1^{k} \eqdef P(x^k) - D(\alpha^k) + \frac{2(\rho + \theta \lambda)}{\delta n} \sum_{\tau=1}^n \|e_\tau^k\|^2,  
$$
where $\alpha^k \eqdef ((\alpha^k_{11})^\top, \dots, (\alpha_{m1}^k)^\top, (\alpha_{12}^k)^\top, \dots, (\alpha_{m2}^k)^\top, \dots, (\alpha_{n1}^k)^\top, \dots, (\alpha_{mn}^k)^\top)^\top \in \R^{tN}$, $\rho = \frac{\delta \lambda R}{2 \sqrt{a_1}}$ and  $a_1 = (1-\delta) (2{\bar R}^2 + \delta R_m^2)$. Let 
\begin{equation}\label{eq:theta-ecQuartz-1}
\theta = \min\left\{  \frac{2\delta \lambda \gamma}{\delta \lambda \gamma m + \sqrt{\delta^2 \lambda^2 \gamma^2 m^2 + 48\lambda \gamma a_1 }},  \frac{N\lambda \gamma p_{i\tau}}{3v_{i\tau} + N\lambda \gamma}, \frac{\delta \lambda \gamma}{\delta \lambda \gamma m + 12R \sqrt{a_1} }  \right\}.  
\end{equation} 
Then $\mathbb{E}[\Psi_1^k] \leq \left(  1 - \min\left\{  \theta, \frac{\delta}{4}  \right\}  \right)^k \Psi_1^0$, and we have  $\mathbb{E}[\Psi_1^k] \leq \epsilon$ as long as 
\begin{equation}\label{eq:iterc-ecQuartz-1}
k  \geq O \left( \left(  \frac{1}{\delta} + m + \frac{R_m^2}{n\lambda \gamma} + \frac{R^2}{\lambda \gamma}   +   \frac{\sqrt{1-\delta} R{\bar R}}{\delta \lambda \gamma}  + \frac{\sqrt{1-\delta} RR_m}{\lambda \gamma \sqrt{\delta}}  \right) \ln\frac{1}{\epsilon}  \right). 
\end{equation}

\end{theorem}

\begin{theorem}\label{th:ecQuartz-2}
	Let Assumption \ref{as:expcompressor}, Assumption \ref{as:contracQ}, and Assumption \ref{as:ec-Quartz} hold. Assume $\delta<1$. Define 
	$$
	\Psi_2^{k} \eqdef P(x^k) - D(\alpha^k) + \frac{2(\rho + \theta \lambda)}{\delta } \|e^k\|^2 + \frac{16(1-\delta) (\rho + \theta \lambda)}{\delta n^2} \sum_{\tau=1}^n \|e_\tau^k\|^2,  
	$$
	where $e^k = \frac{1}{n}\sum_{\tau=1}^n e_\tau^k$, $\rho = \frac{\delta \lambda R}{2\sqrt{a_2}}$ and $a_2 = (1-\delta) (2R^2 + \frac{16{\bar R}^2}{n} + \frac{9\delta R_m^2}{n})$ . Let 
	\begin{equation}\label{eq:theta-ecQuartz-2}
	\theta = \min\left\{  \frac{2\delta \lambda \gamma}{\delta \lambda \gamma m + \sqrt{\delta^2 \lambda^2 \gamma^2 m^2 + 48\lambda \gamma a_2 }},  \frac{N\lambda \gamma p_{i\tau}}{3v_{i\tau} + N\lambda \gamma}, \frac{\delta \lambda \gamma}{\delta \lambda \gamma m + 12R \sqrt{a_2} }  \right\}.  
	\end{equation} 
	Then $\mathbb{E}[\Psi_2^k] \leq \left(  1 - \min\left\{  \theta, \frac{\delta}{4}  \right\}  \right)^k \Psi_2^0$, and we have  $\mathbb{E}[\Psi_2^k] \leq \epsilon$ as long as 
	\begin{equation}\label{eq:iterc-ecQuartz-2}
	k  \geq O \left( \left(  \frac{1}{\delta} + m + \frac{R_m^2}{n\lambda \gamma} + \frac{R^2}{\lambda \gamma}    +  \frac{\sqrt{1-\delta}}{\delta} \frac{R^2}{\lambda \gamma}   \right) \ln\frac{1}{\epsilon}  \right). 
	\end{equation}
	
\end{theorem}

Noticing that $R^2 \leq {\bar R}^2 \leq nR^2$ and ${\bar R}^2 \leq R_m^2 \leq m{\bar R}^2$, the iteration complexity in Theorem \ref{th:ecQuartz-2} could be better than that in Theorem \ref{th:ecQuartz-1}. On the other hand, if $R^2 = {\bar R}^2 = R_m^2$, then both iteration complexities in Theorem \ref{th:ecQuartz-1} and Theorem \ref{th:ecQuartz-2} become 
\begin{equation}\label{eq:iterQuartz}
O \left( \left(  \frac{1}{\delta} + m  + \frac{R^2}{\lambda \gamma}    +  \frac{\sqrt{1-\delta}}{\delta} \frac{R^2}{\lambda \gamma}   \right) \ln\frac{1}{\epsilon}  \right). 
\end{equation}

\subsection{Convergence of EC-SDCA}

Define $\epsilon_P^k \eqdef P(x^k) - P(x^*)$ and $\epsilon_D^k \eqdef D(\alpha^*) - D(\alpha^k)$ for $k\geq 0$, where $x^*$ and $\alpha^*$ are the optimal solutions of the primal and dual problems. 

\begin{theorem}\label{th:ecSDCA-1}
	Let Assumption \ref{as:contracQ} and Assumption \ref{as:ec-Quartz} hold. Assume $\delta<1$. Define 
	$$
	\Psi_3^{k} \eqdef \epsilon_D^k + \frac{2(\rho + \theta \lambda)}{\delta n} \sum_{\tau=1}^n \|e_\tau^k\|^2,  
	$$
	where $\rho = \frac{\delta \lambda R}{2 \sqrt{a_1}}$ and  $a_1 = (1-\delta) (2{\bar R}^2 + \delta R_m^2)$. Choose $\theta$ as in (\ref{eq:theta-ecQuartz-1}). Then $\mathbb{E}[\Psi_3^k] \leq \left(  1 - \min\left\{  \theta, \frac{\delta}{4}  \right\}  \right)^k \Psi_3^0$, and we have  $\mathbb{E}[\Psi_3^k] \leq \epsilon$ as long as $k$ satisfies (\ref{eq:iterc-ecQuartz-1}). Let $w_k = \left(  1 - \min\left\{  \theta, \frac{\delta}{4}  \right\}  \right)^{-k}$, $W_k = \sum_{i=1}^k w_i$, and ${\bar x}^k = \frac{1}{W_k} \sum_{i=1}^kw_i x^i$. Then we have 
	\begin{equation}\label{eq:P-ecsdca}
	\mathbb{E}[P({\bar x}^k) - P(x^*)] \leq  \frac{\left(  1 - \min\left\{  \theta, \frac{\delta}{4}  \right\}  \right)^k \epsilon_D^0}{1 - \left(  1 - \min\left\{  \theta, \frac{\delta}{4}  \right\}  \right)^k}, 
	\end{equation}
	and for $\epsilon \leq \epsilon_D^0$, $\mathbb{E}[P({\bar x}^k) - P(x^*)] \leq \epsilon$ as  long as $k$ satisfies (\ref{eq:iterc-ecQuartz-1}). 
\end{theorem}

\begin{theorem}\label{th:ecSDCA-2}
	Let Assumption \ref{as:expcompressor}, Assumption \ref{as:contracQ}, and Assumption \ref{as:ec-Quartz} hold. Assume $\delta<1$. Define 
	$$
	\Psi_4^{k} \eqdef \epsilon_D^k + \frac{2(\rho + \theta \lambda)}{\delta } \|e^k\|^2 + \frac{16(1-\delta) (\rho + \theta \lambda)}{\delta n^2} \sum_{\tau=1}^n \|e_\tau^k\|^2,  
	$$
	where $e^k = \frac{1}{n}\sum_{\tau=1}^n e_\tau^k$, $\rho = \frac{\delta \lambda R}{2\sqrt{a_2}}$ and $a_2 = (1-\delta) (2R^2 + \frac{16{\bar R}^2}{n} + \frac{9\delta R_m^2}{n})$ . Choose $\theta$ as in (\ref{eq:theta-ecQuartz-2}). Then $\mathbb{E}[\Psi_4^k] \leq \left(  1 - \min\left\{  \theta, \frac{\delta}{4}  \right\}  \right)^k \Psi_4^0$, and we have  $\mathbb{E}[\Psi_4^k] \leq \epsilon$ as long as $k$ satisfies (\ref{eq:iterc-ecQuartz-2}). Let $w_k = \left(  1 - \min\left\{  \theta, \frac{\delta}{4}  \right\}  \right)^{-k}$, $W_k = \sum_{i=1}^k w_i$, and ${\bar x}^k = \frac{1}{W_k} \sum_{i=1}^kw_i x^i$. Then $\mathbb{E}[P({\bar x}^k) - P(x^*)]$ satisfies (\ref{eq:P-ecsdca}), and for $\epsilon \leq \epsilon_D^0$, $\mathbb{E}[P({\bar x}^k) - P(x^*)] \leq \epsilon$ as  long as $k$ satisfies (\ref{eq:iterc-ecQuartz-2}).
	
\end{theorem}

\begin{figure}[H]
	\vspace{-0.25cm}
	\centering
	\begin{tabular}{cccc}
		\includegraphics[width=3.8cm]{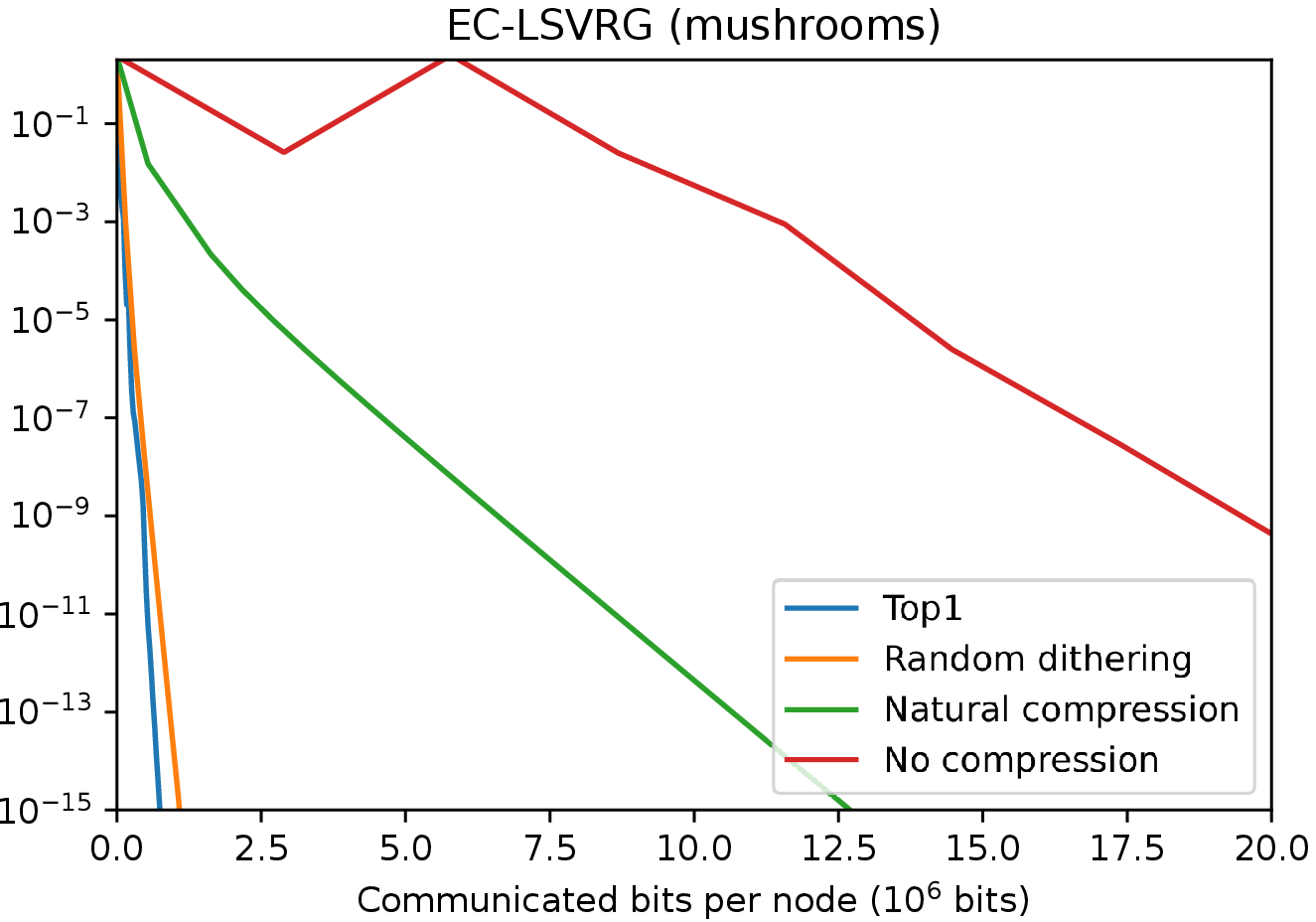}&
		\includegraphics[width=3.8cm]{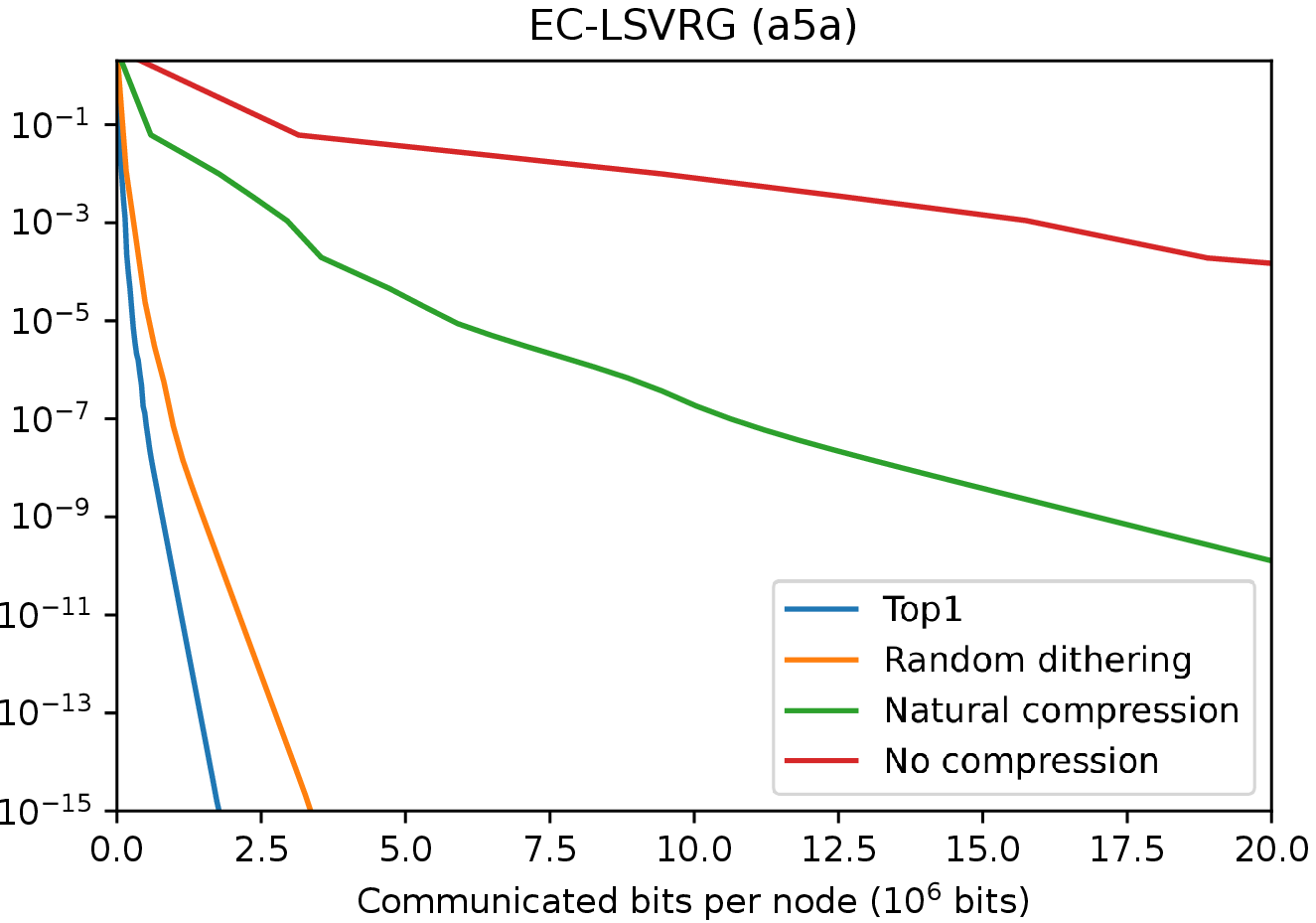}&
		\includegraphics[width=3.8cm]{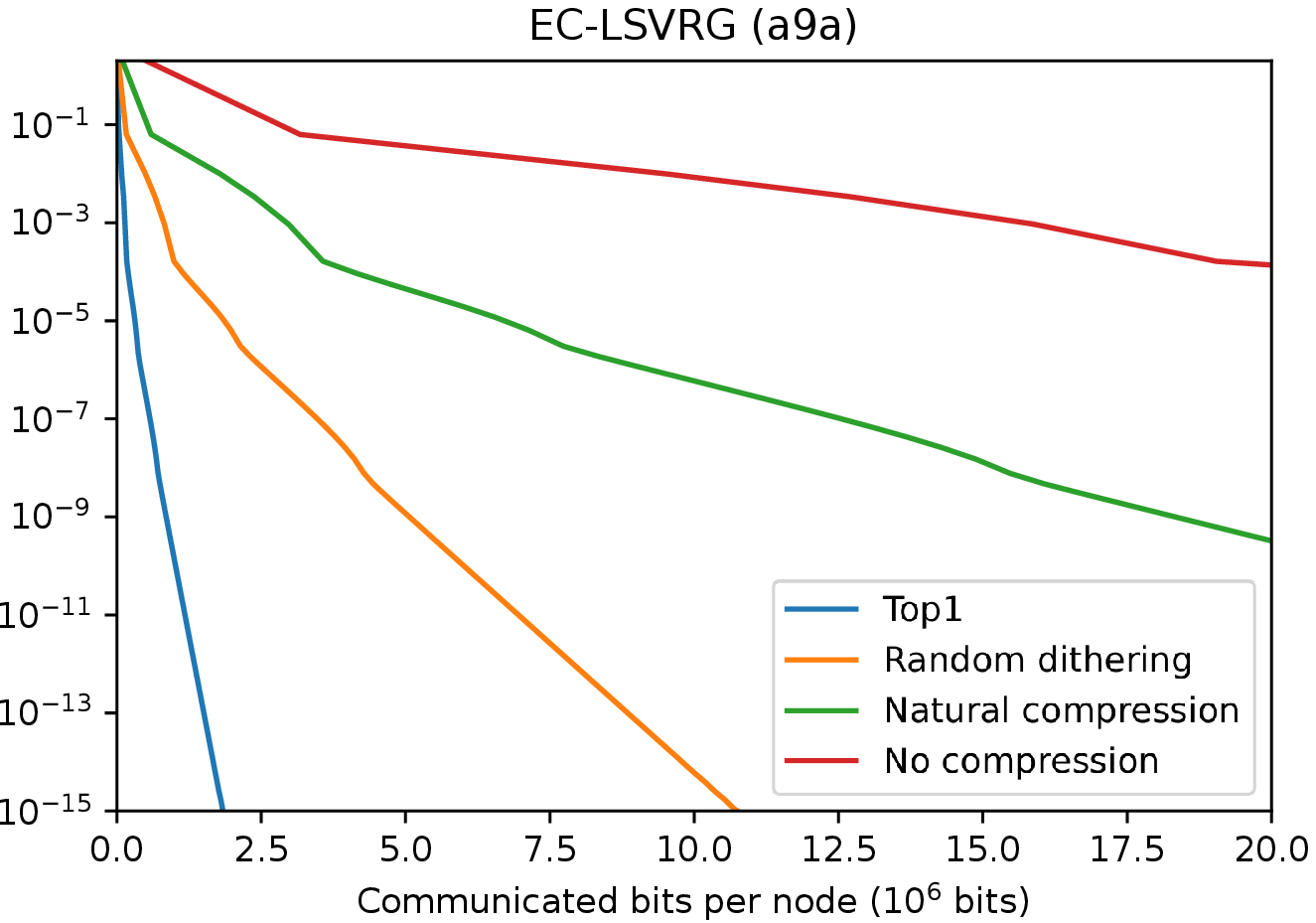}&
		\includegraphics[width=3.8cm]{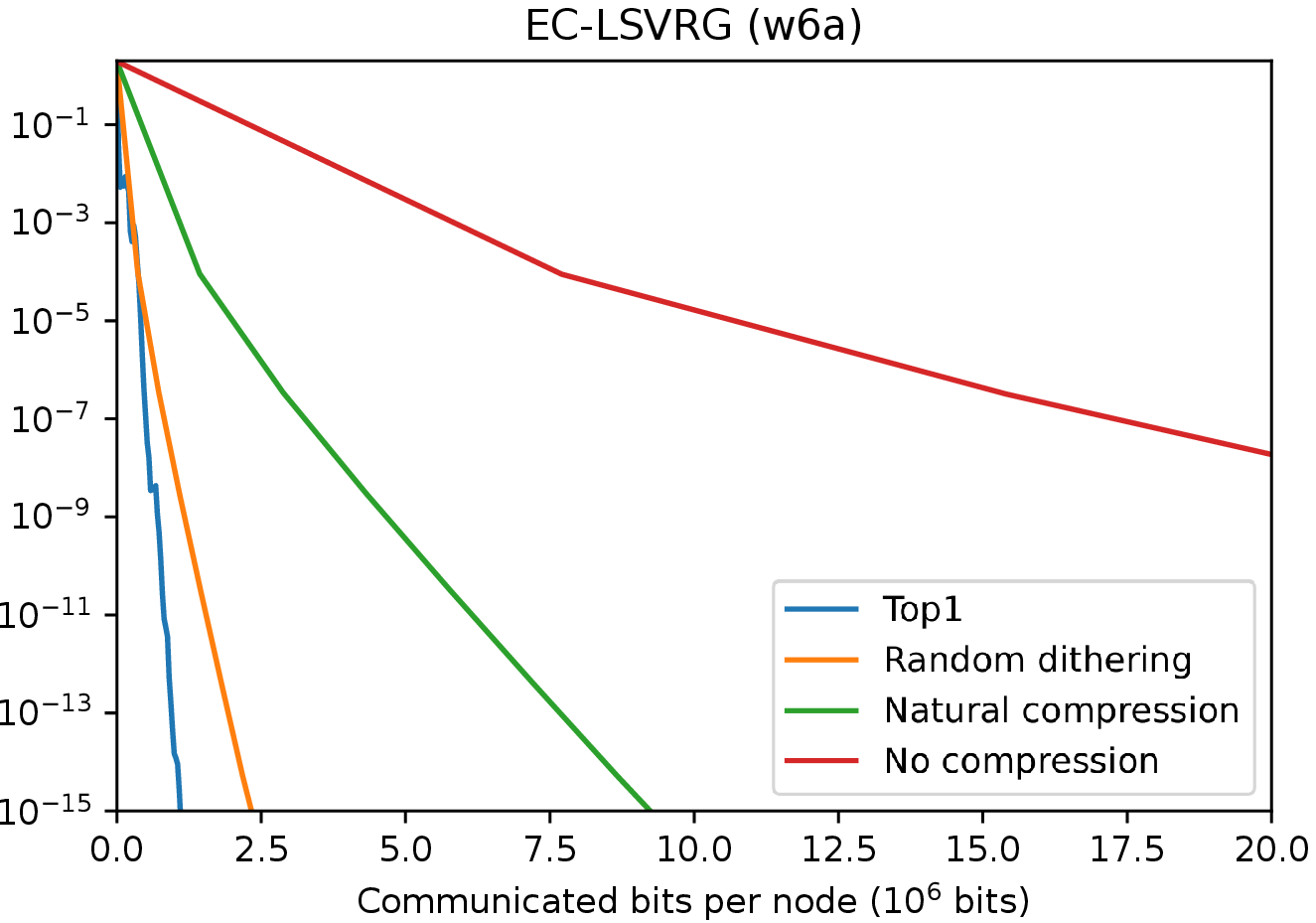}
		\\
		\includegraphics[width=3.8cm]{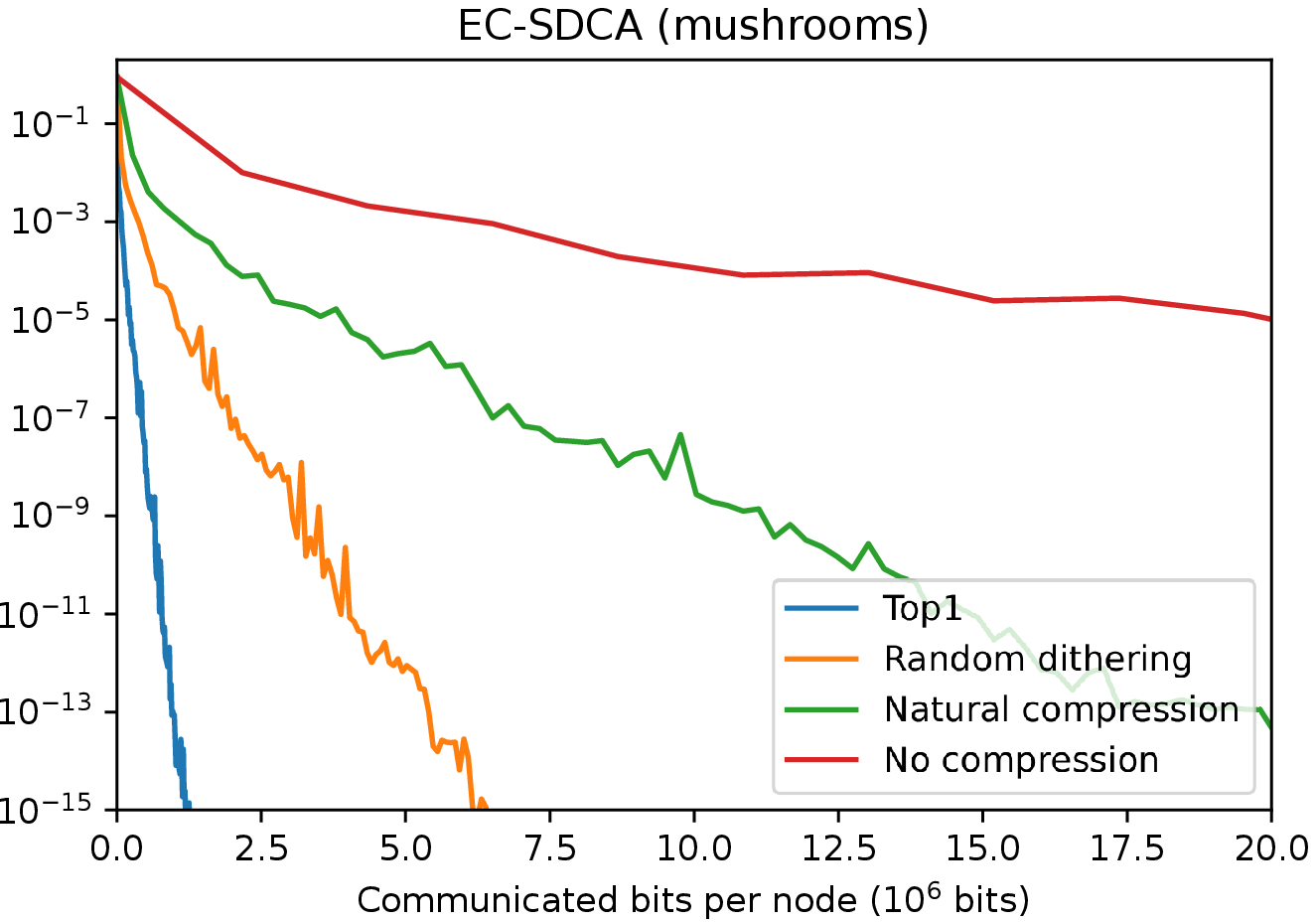}&
		\includegraphics[width=3.8cm]{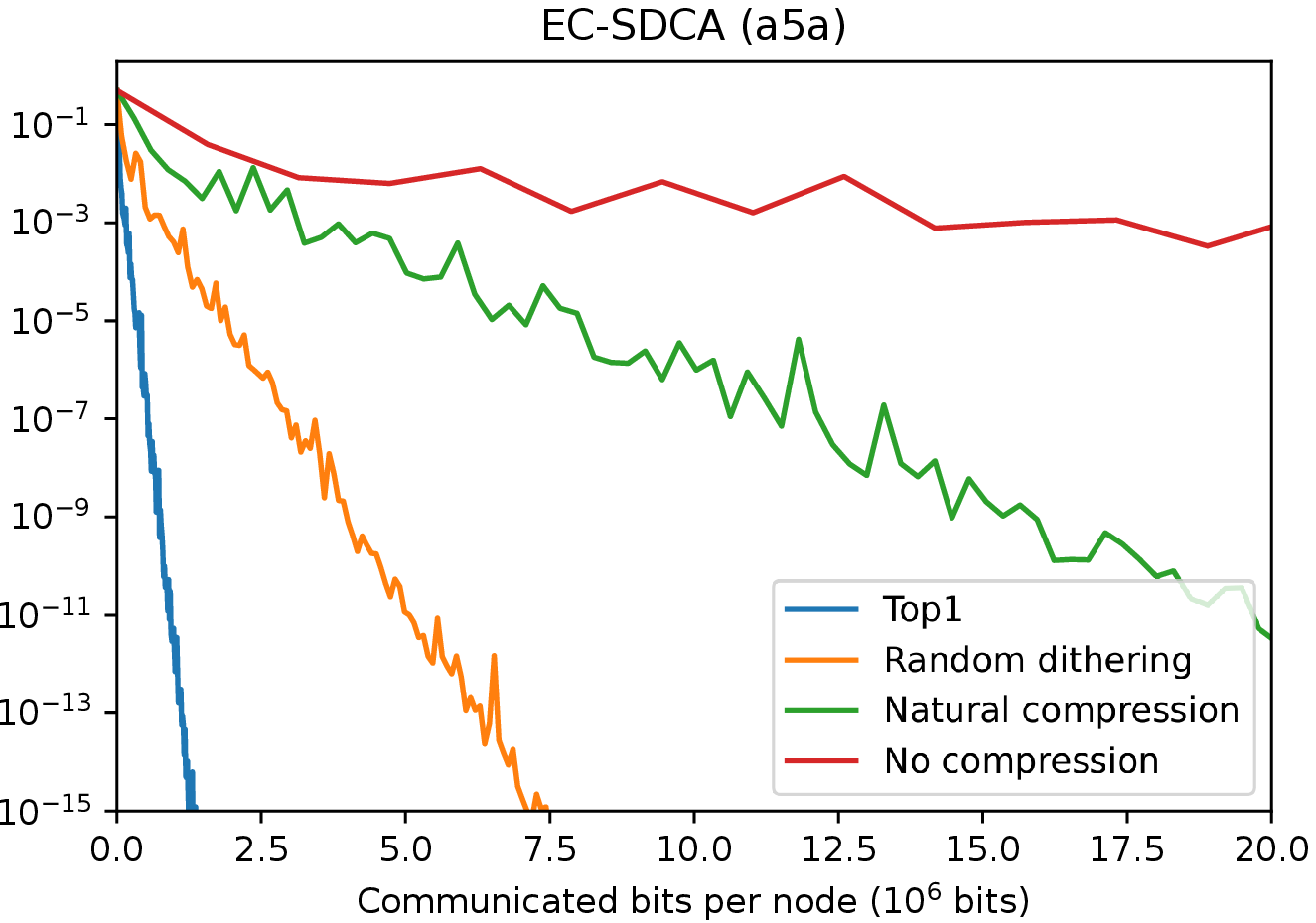}&
		\includegraphics[width=3.8cm]{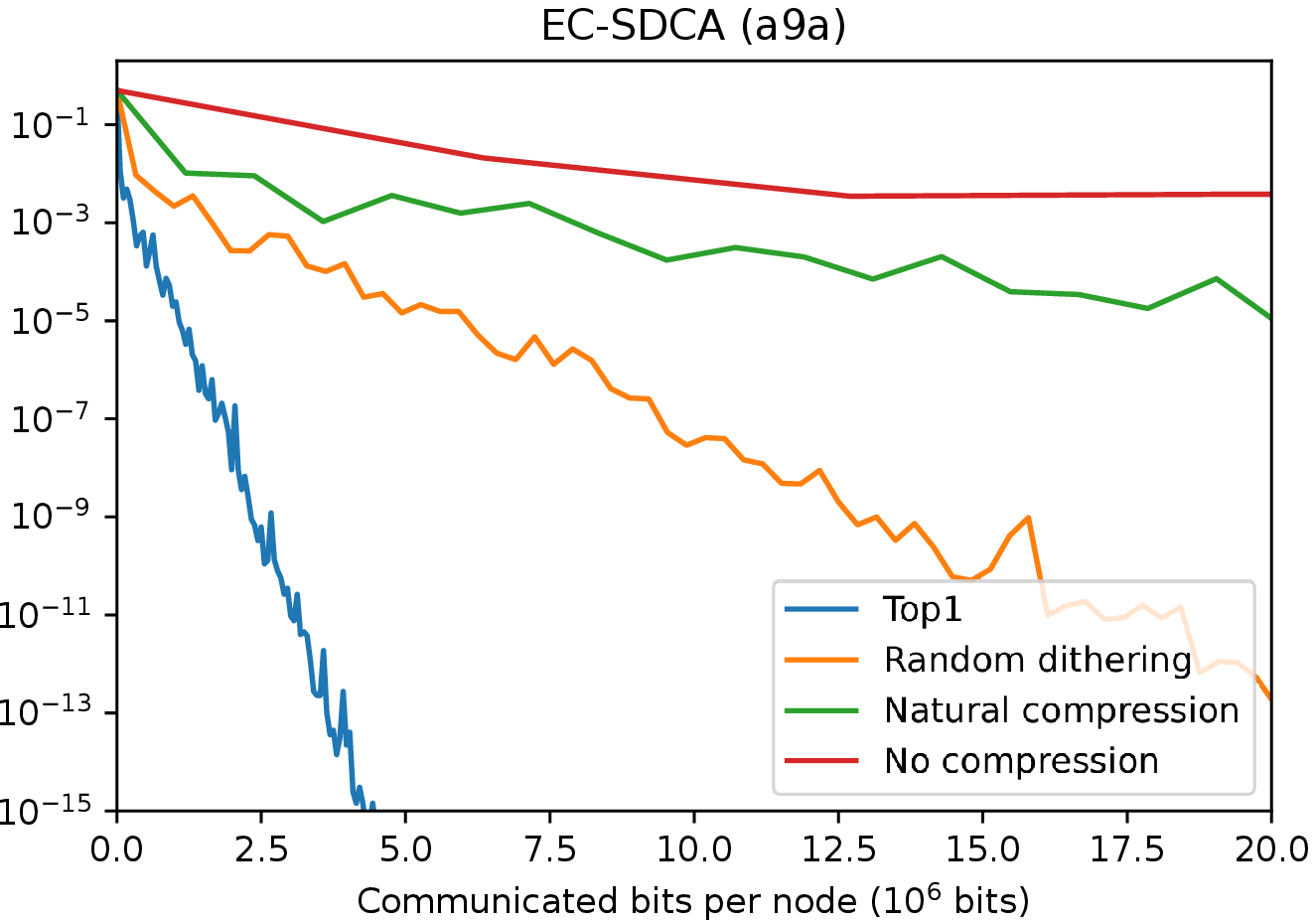}&
		\includegraphics[width=3.8cm]{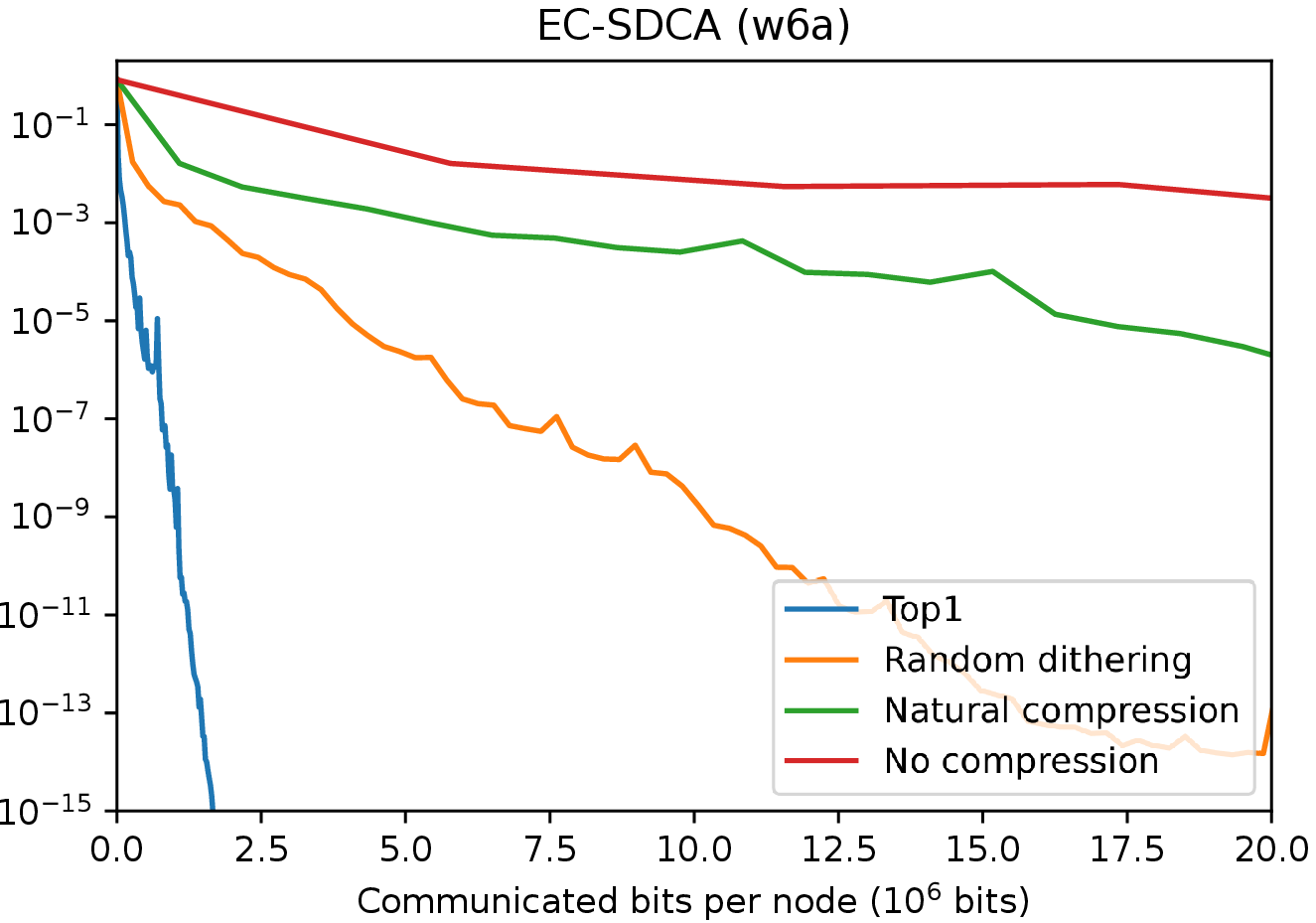}
	\end{tabular}
	\caption{ EC-LSVRG and EC-SDCA with different compressors}\label{fig:compressors}
	\vspace{-0.25cm}
\end{figure}

\section{Experiments}

In this section, we run experiments with EC-LSVRG, EC-Quartz, and EC-SDCA to demonstrate the empirical effectiveness. In particular, we should highlight the linear convergence rate of our algorithm with the biased compressor and non-smooth objective function. Also, the communication complexity performance is competitive to other compressed algorithms.

\paragraph{Setting.} We implement the logistic regression problem with $L_1$-$L_2$ regularization 
\begin{equation*}
\min_{x\in\R^d} \bigg\{\frac{1}{n}\sum \limits_{i=1}^{n}\log \big(1 + \exp(-b_i a_i^\top x)\big) + {\lambda_1}\|x\|_1\bigg\}+ \frac{\lambda_2}{2}\|x\|_2^2,
\end{equation*}
where $\{a_i, b_i\}_{i\in[n]}$ are data samples. 
We use Python 3.7 to perform experiments on a server with 2 processors (Intel Xeon Gold 5120 @ 2.20GHz), 28 cores in total. Library include numpy, sklearn. We search the optimal step size from $\{10^t,3\times10^t\}$, where $t\in\{-4,\cdots,0,1\}$ for all tested algorithms. We choose the same contraction compressors for $Q$ and $Q_1$ for EC-LSVRG. Without more speficiation, we use $L_1$-$L_2$ regularization with $\lambda_1=\lambda_2 = 10^{-3}$ and $p=\delta$ for EC-LSVRG. For EC-LSVRG-DIANA \citep{gorbunov2020linearly}, we choose $p=\nicefrac{1}{m}$. For VR-DIANA \citep{Samuel19}, we use the optimal $\alpha = \nicefrac{1}{(w+1)}$. For smooth experiments, we use $\lambda_1=0$ and $\lambda_2 = 10^{-3}$. 

\paragraph{Datasets.} We have four real data sets: \textbf{a5a}, \textbf{a9a}, \textbf{mushrooms}, and \textbf{w6a}, from the LIBSVM library \citep{chang2011libsvm}. 

\paragraph{Compressors.} We use RandK, TopK, random dithering in \citep{Alistarh17}, natural compression in \citep{horvath2019natural}, RTopK, and NTopK. It should be noticed that the random dithering and natural compression are both unbiased compressors. When we use them as contraction compressors, we mean the ones scaled by $\nicefrac{1}{(\omega+1)}$. RandK is a contraction compressor. When we use it as an unbiased compressor, we mean the one scaled by $\nicefrac{d}{K}$. For RandK and TopK, $\delta = \nicefrac{K}{d}$, and the number of communicated bits for the comprssed vector in $\R^d$ is $(64 + \lceil \log d \rceil)K$. For random dithering, we choose the level $s=\sqrt{d}$, the number of communicated bits for the comprssed vector is $2.8d+64$, and $\omega=1$. For natural compression, the number of communicated bits for the comprssed vector is $12d$, and $\omega = \nicefrac{1}{8}$. For the random dithering in RTopK, we choose $s = \sqrt{K}$. Then the number of communicated bits for the comprssed vector using RTopK is $2.8K+64+ K \lceil \log d \rceil$ and $\delta = \nicefrac{K}{2d}$. For NTopK, the number of communicated bits for the comprssed vector is $12K+K \lceil \log d \rceil$, and $\delta = \nicefrac{8K}{9d}$. 

\subsection{TopK, random dithering, natural compression vs no compression}

We firstly compare our algorithm with different compressors in Figure \ref{fig:compressors}. It shows that, for the communication complexity, EC-LSVRG and EC-SDCA with contraction compressors are superior to the uncompressed ones, especially for Top1 compressor.

\subsection{Comparison with  ECSGD, ECGD, and EC-LSVRG-DIANA}

We also compare our algorithms with the baseline EC-GD and state-of-the-art competitor EC-LSVRG-DIANA \citep{gorbunov2020linearly}, VR-DIANA \citep{Samuel19} in Figures \ref{fig:non-smooth} $-$ \ref{fig:smooth_w6a}. 
 For the subfigures where the title is ``Top1/Rand1'', we use Top1 in EC-GD, EC-LSVRG-DIANA and our algorithms and Rand1 in VR-DIANA. For EC-LSVRG-DIANA, an unbiased compressor is also needed. Thus, for the Top1/Rand1 case, we use Rand1; for random dithering and natural compression cases, we use random dithering and natural compression, respectively, for the unbiased compressor in EC-LSVRG-DIANA. 
 
 Because of the compression error, EC-GD could not converge to the optimal solution. For EC-LSVRG-DIANA, it converges linearly to the optimal solution when the objective function is smooth, and the communication complexity performance of it is almost the same as that of EC-LSVRG. However, EC-LSVRG-DIANA does not support non-smooth objective function well, leading to a biased solution. These figures show that in the most cases, EC-LSVRG or EC-SDCA performs the best. VR-DIANA is not compatible  with Top1 compressor, which is extremely efficient. While our methods, including EC-LSVRG and EC-SDCA, perform well on either smooth or non-smooth case with Top1 compressor.

\begin{figure}[H]
	\vspace{-0.35cm}	
	\centering
	\begin{tabular}{ccc}
		\includegraphics[width=5cm]{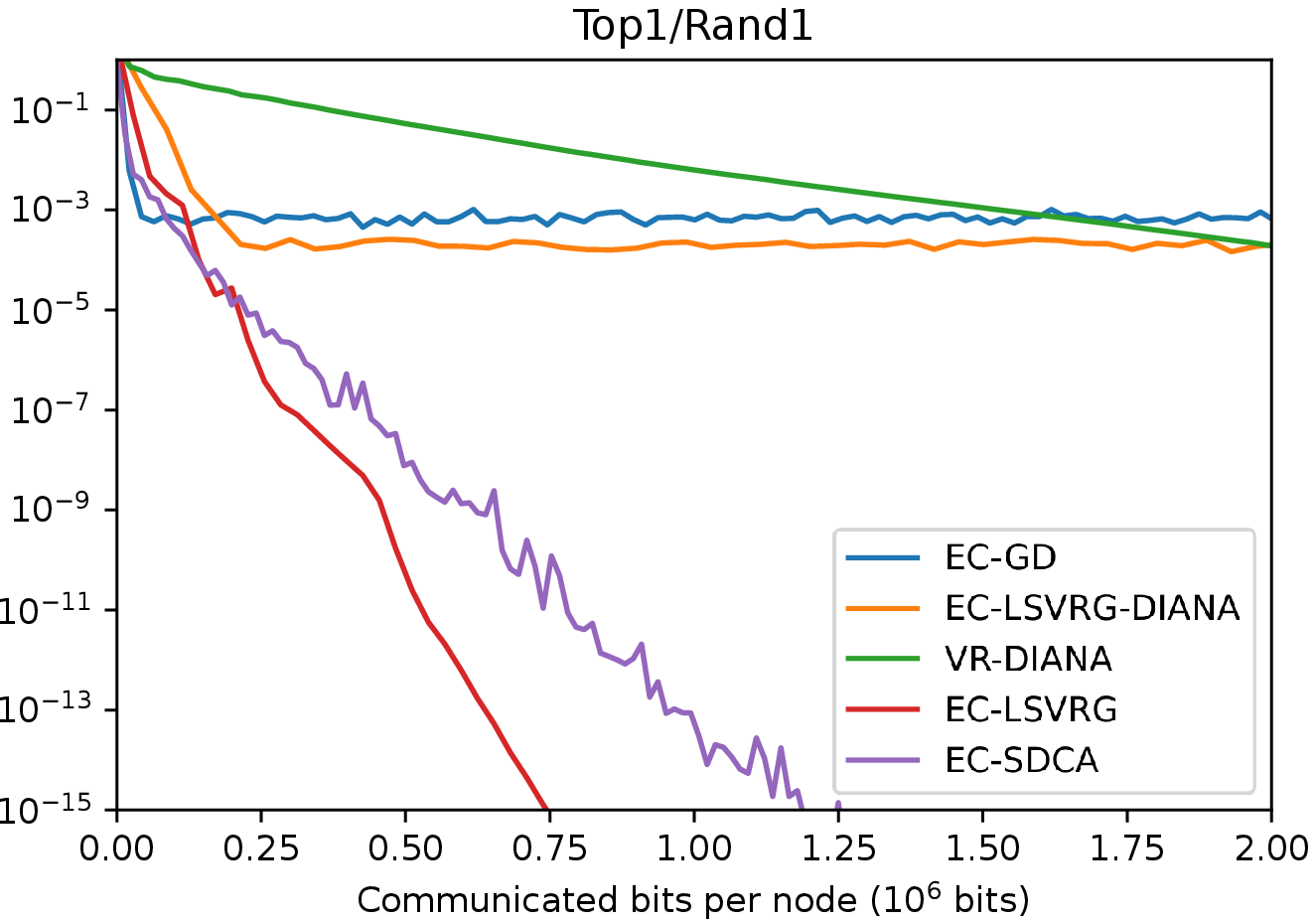}&
		\includegraphics[width=5cm]{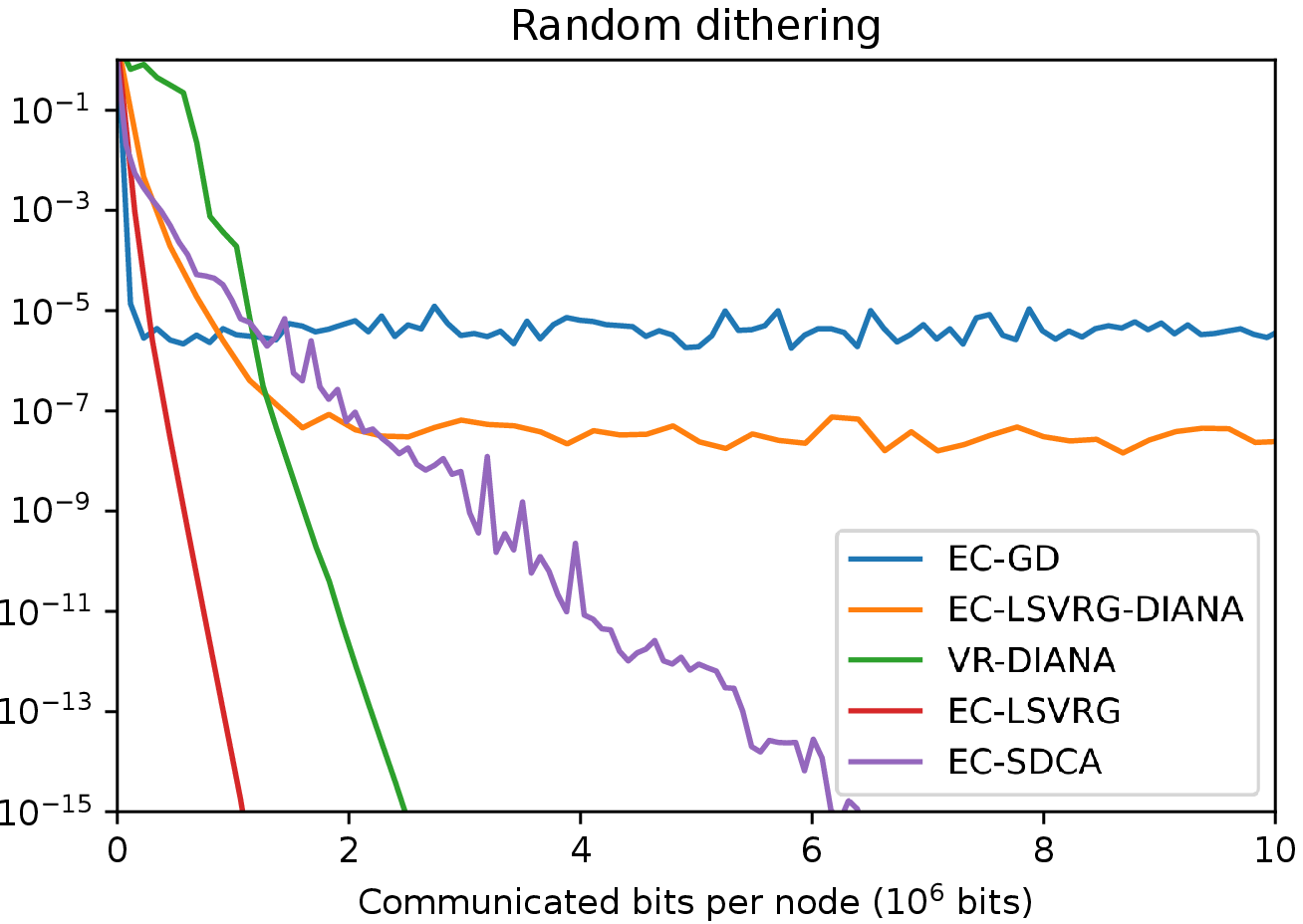}&
		\includegraphics[width=5cm]{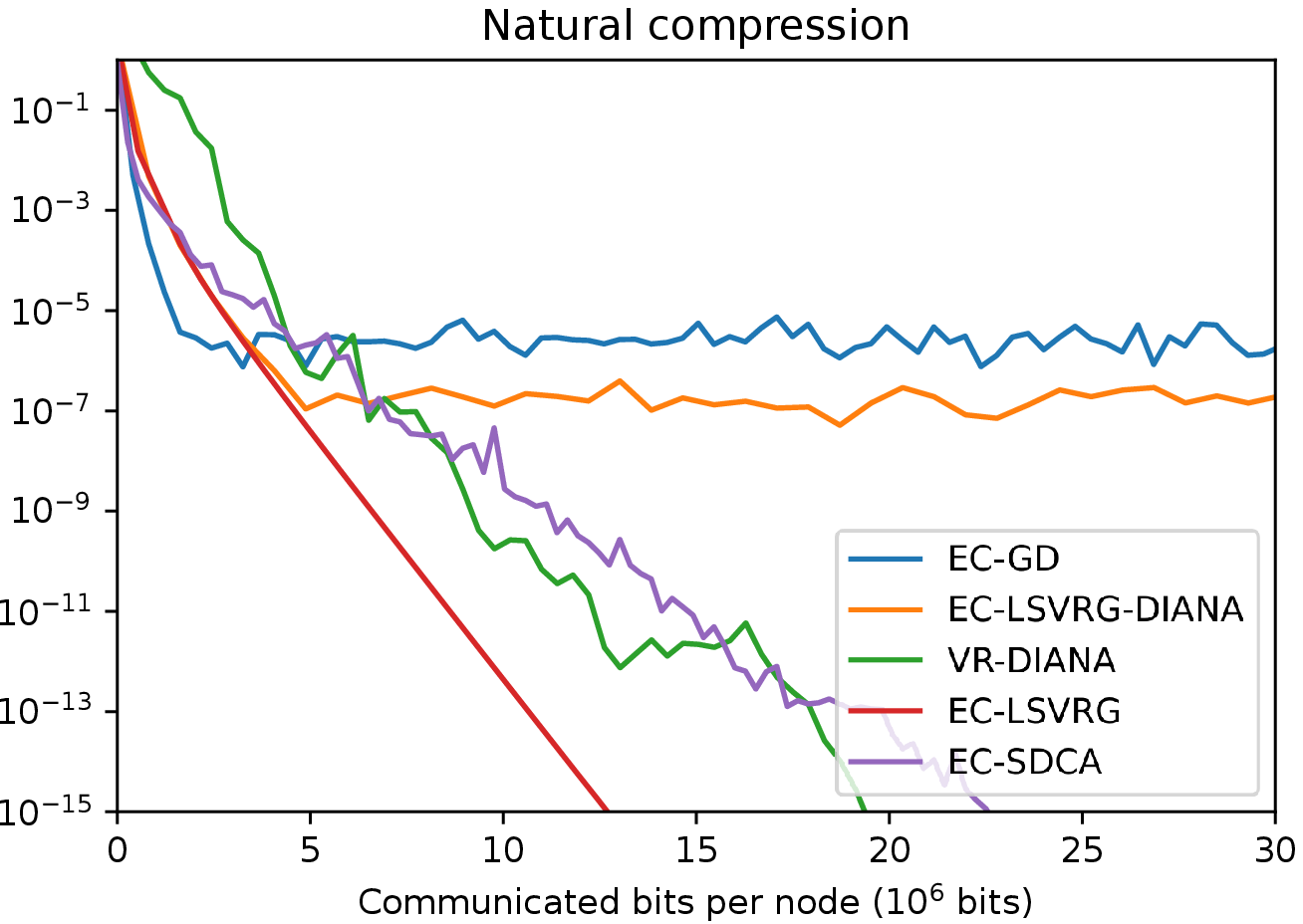}
	\end{tabular}
	\vspace{-0.35cm}
	\caption{Comparison with  ECGD, EC-LSVRG-DIANA, VR-DIANA on  \textbf{mushrooms} (non-smooth case)}\label{fig:non-smooth}
	\vspace{-0.35cm}	
\end{figure}

\begin{figure}[H]
	\vspace{-0.35cm}
	\centering
	\begin{tabular}{ccc}
		\includegraphics[width=5cm]{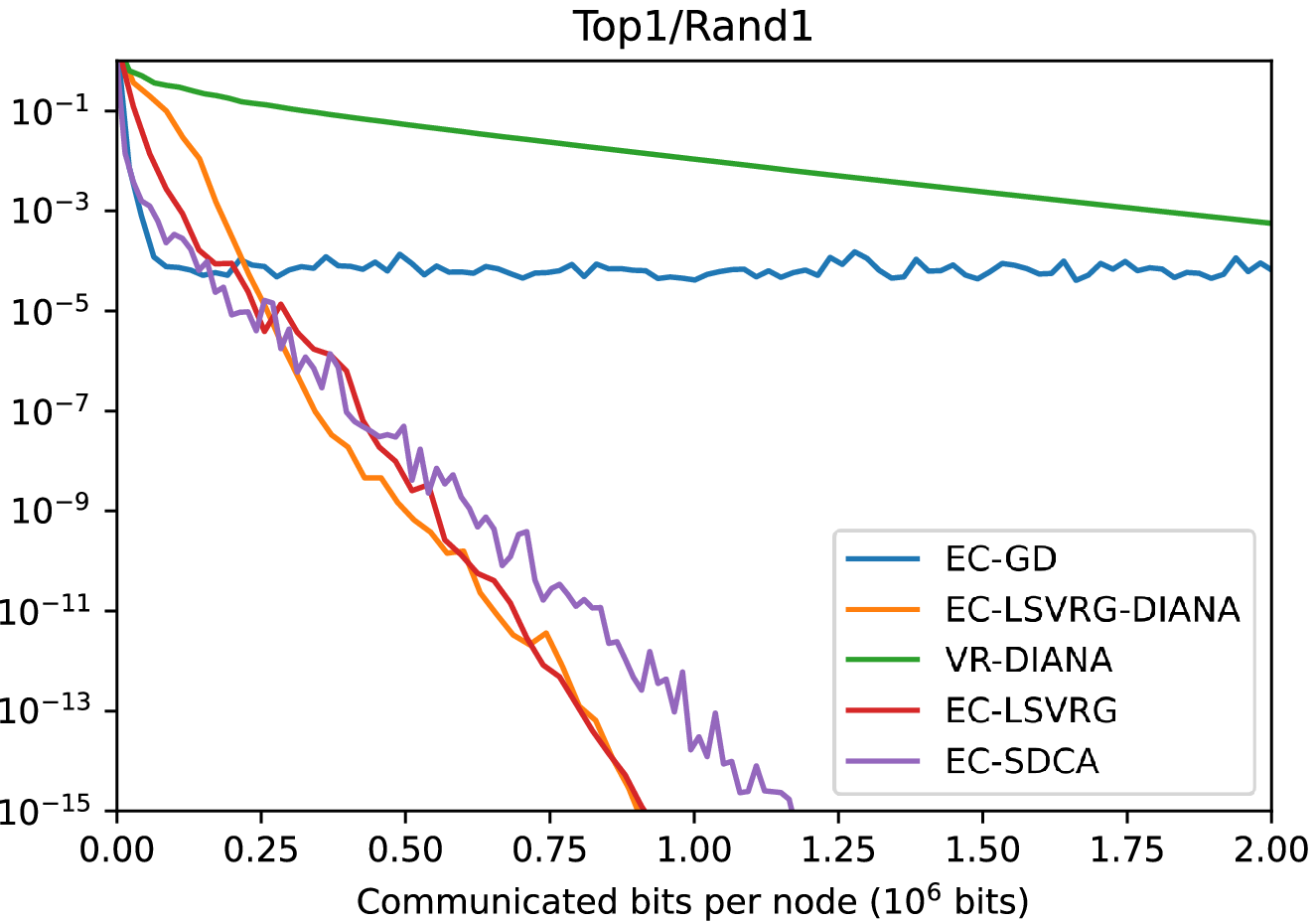}&
		\includegraphics[width=5cm]{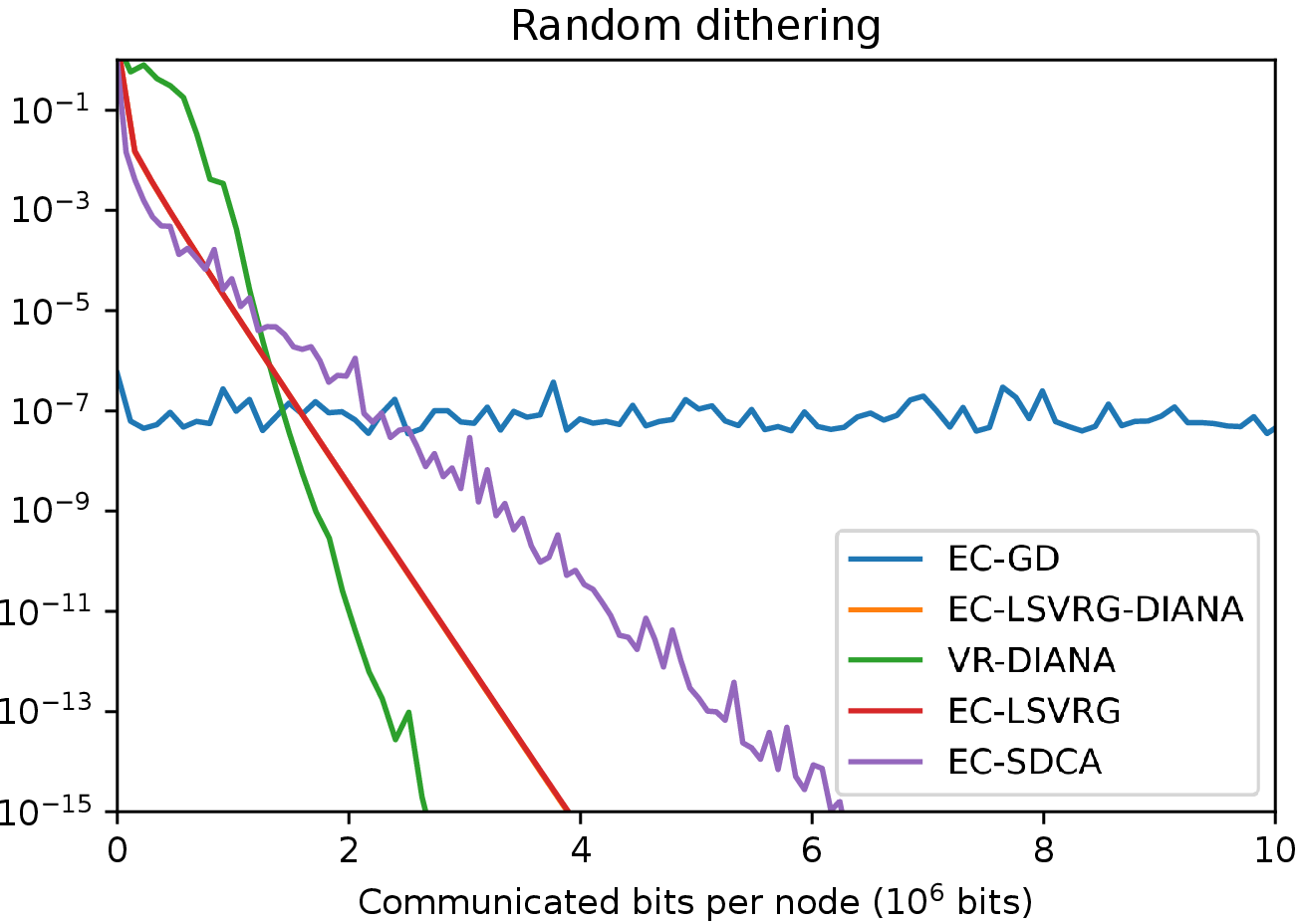}&
		\includegraphics[width=5cm]{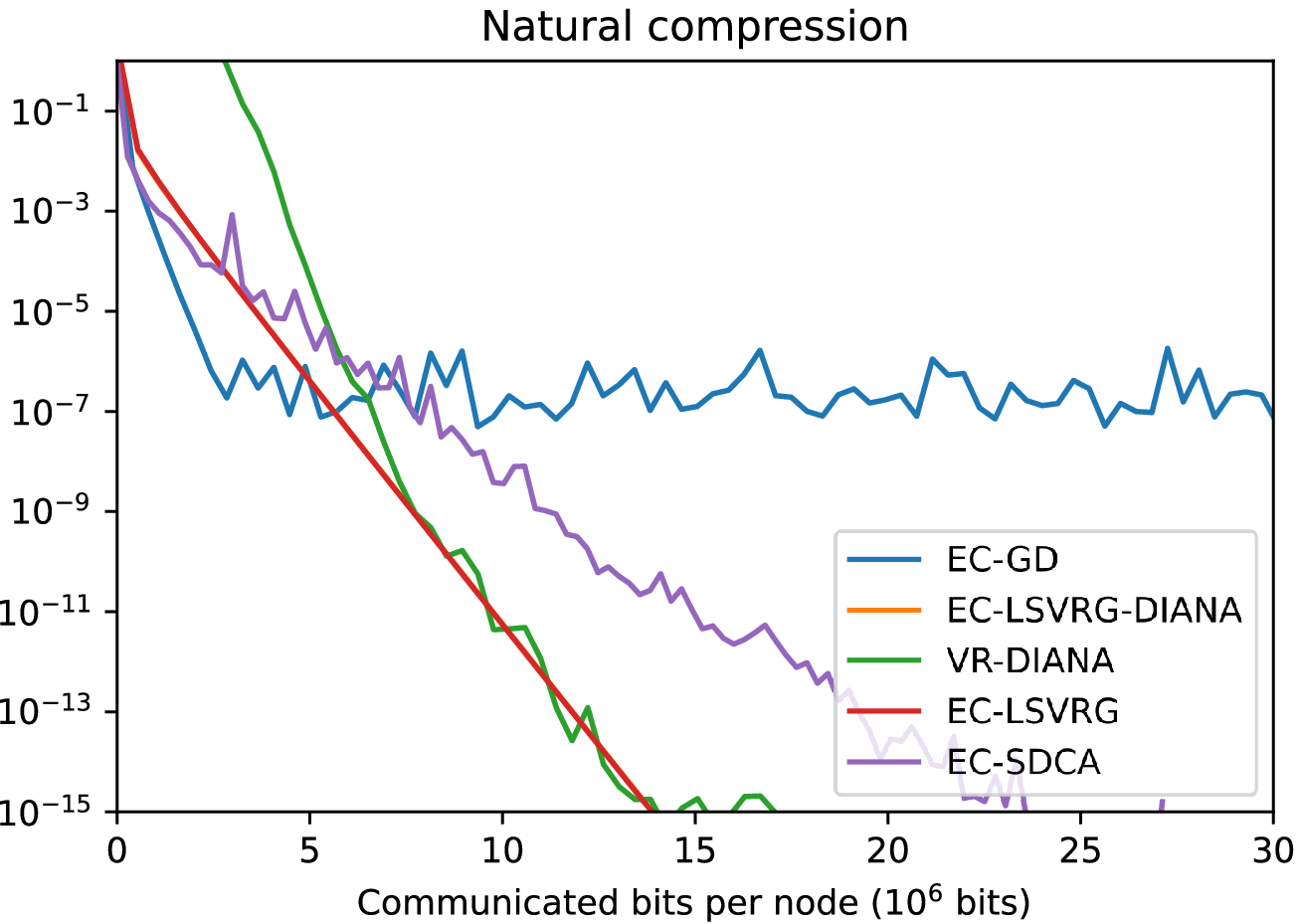}
	\end{tabular}
	\vspace{-0.35cm}
	\caption{Comparison with  ECSGD, ECGD, and EC-LSVRG-DIANA on  \textbf{mushrooms} (smooth case)}\label{fig:smooth}
	\vspace{-0.35cm}
\end{figure}

\begin{figure}[H]
	\vspace{-0.25cm}
	\centering
	\begin{tabular}{ccc}
		\includegraphics[width=5cm]{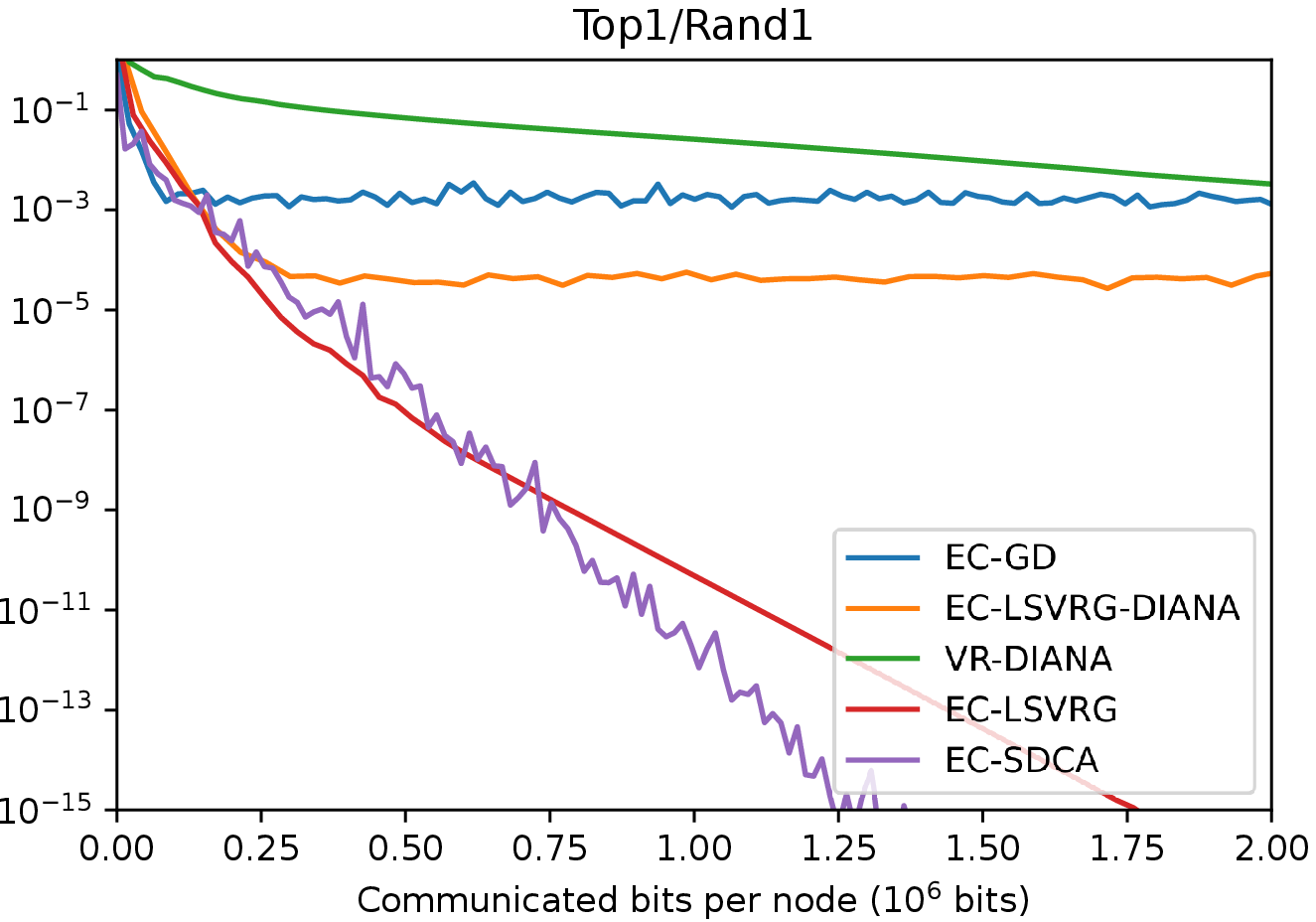}&
		\includegraphics[width=5cm]{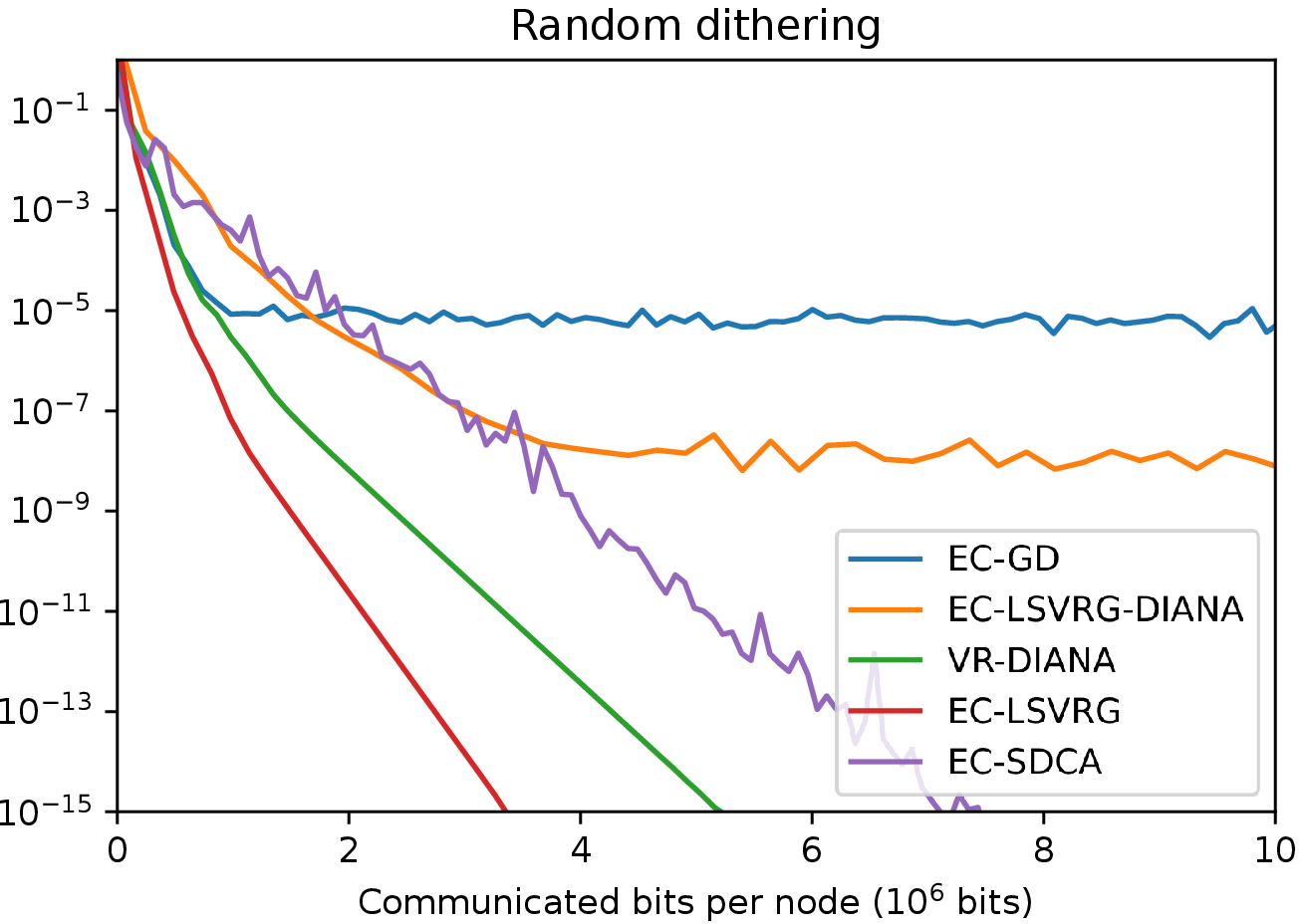}&
		\includegraphics[width=5cm]{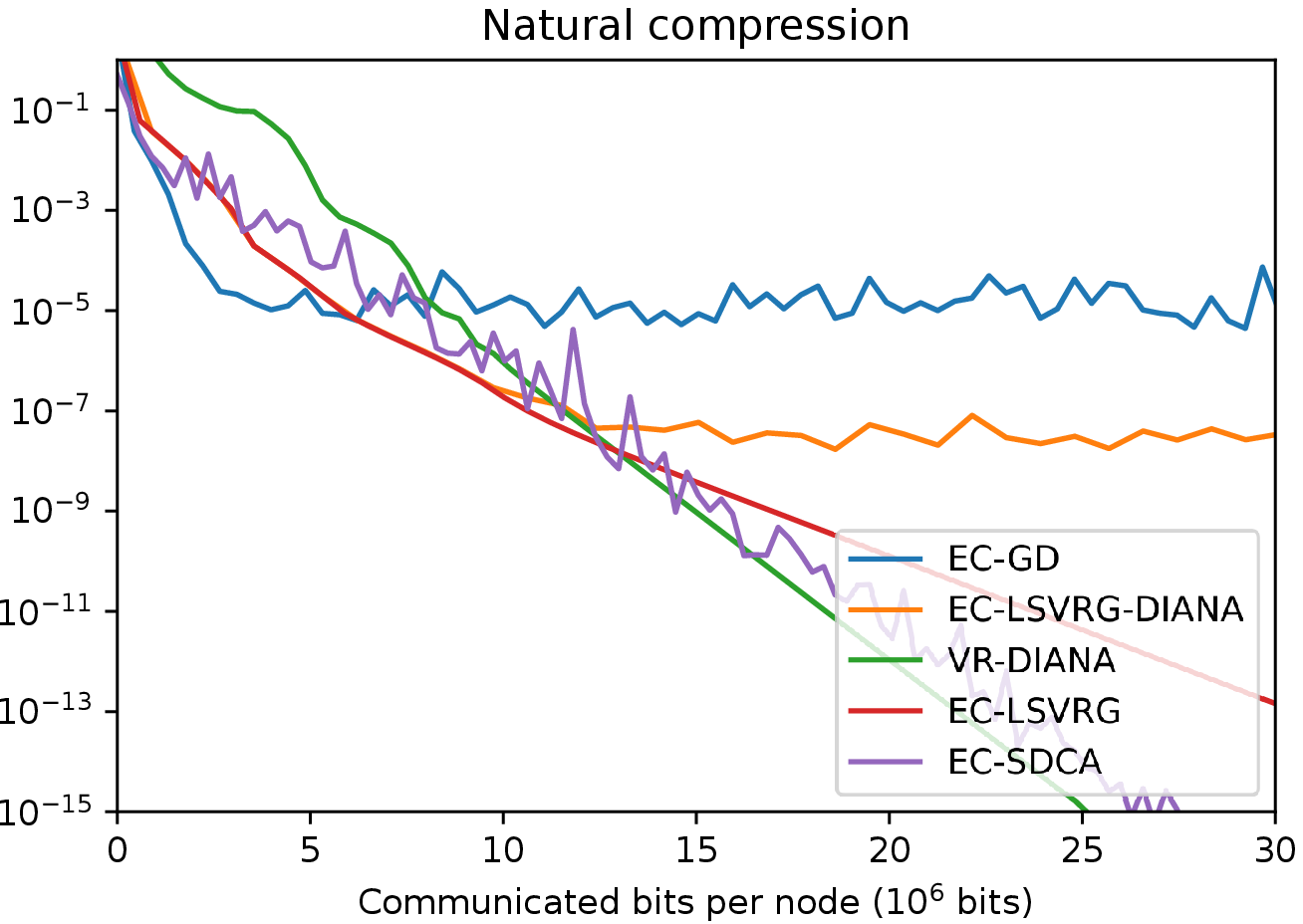}
	\end{tabular}
	\vspace{-0.25cm}
	\caption{Comparison with  ECGD, EC-LSVRG-DIANA, VR-DIANA on  \textbf{a5a} (non-smooth case)}\label{fig:non-smooth_a5a}
	\vspace{-0.25cm}
\end{figure}

\begin{figure}[H]
	\vspace{-0.25cm}
	\centering
	\begin{tabular}{ccc}
		\includegraphics[width=5cm]{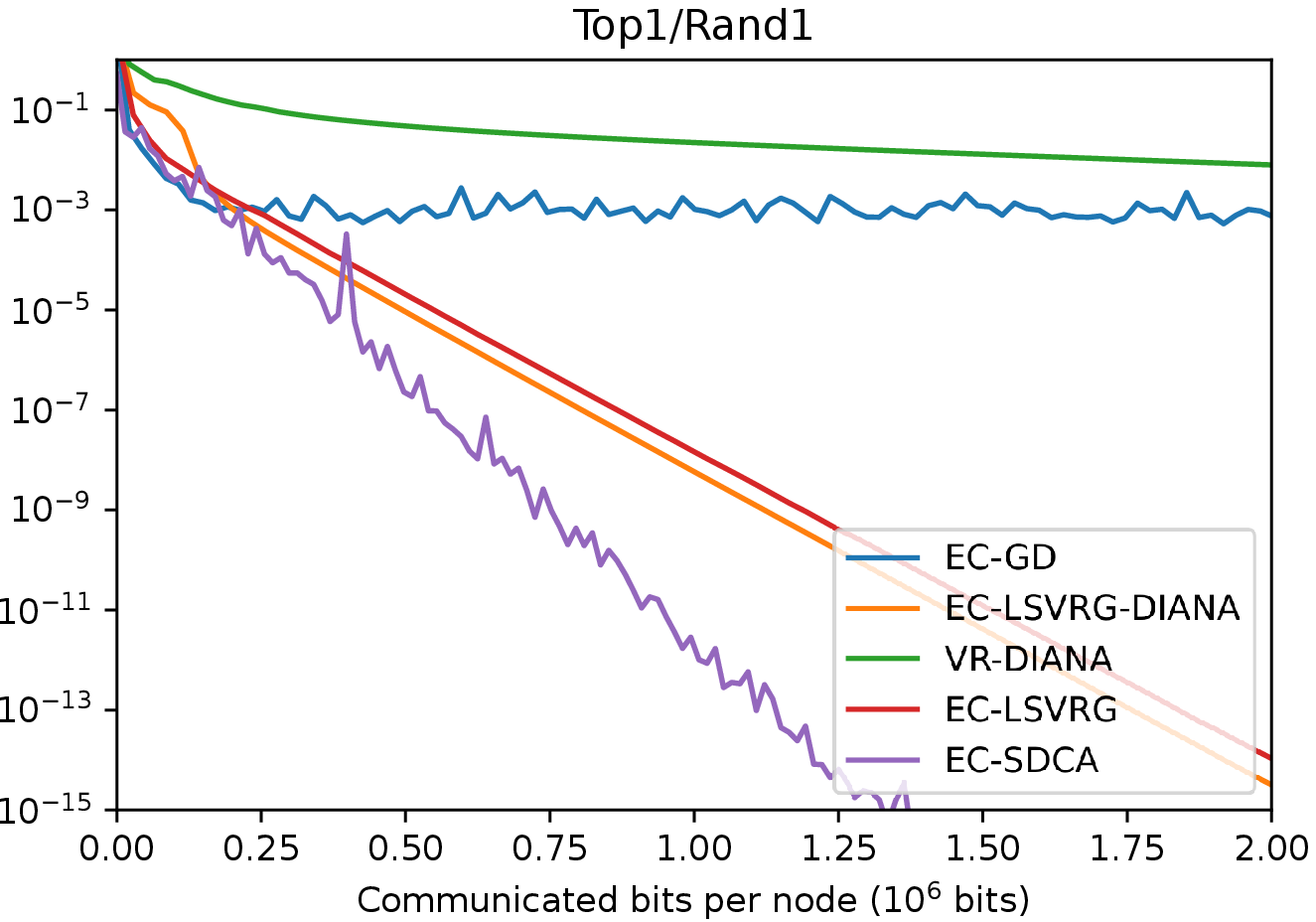}&
		\includegraphics[width=5cm]{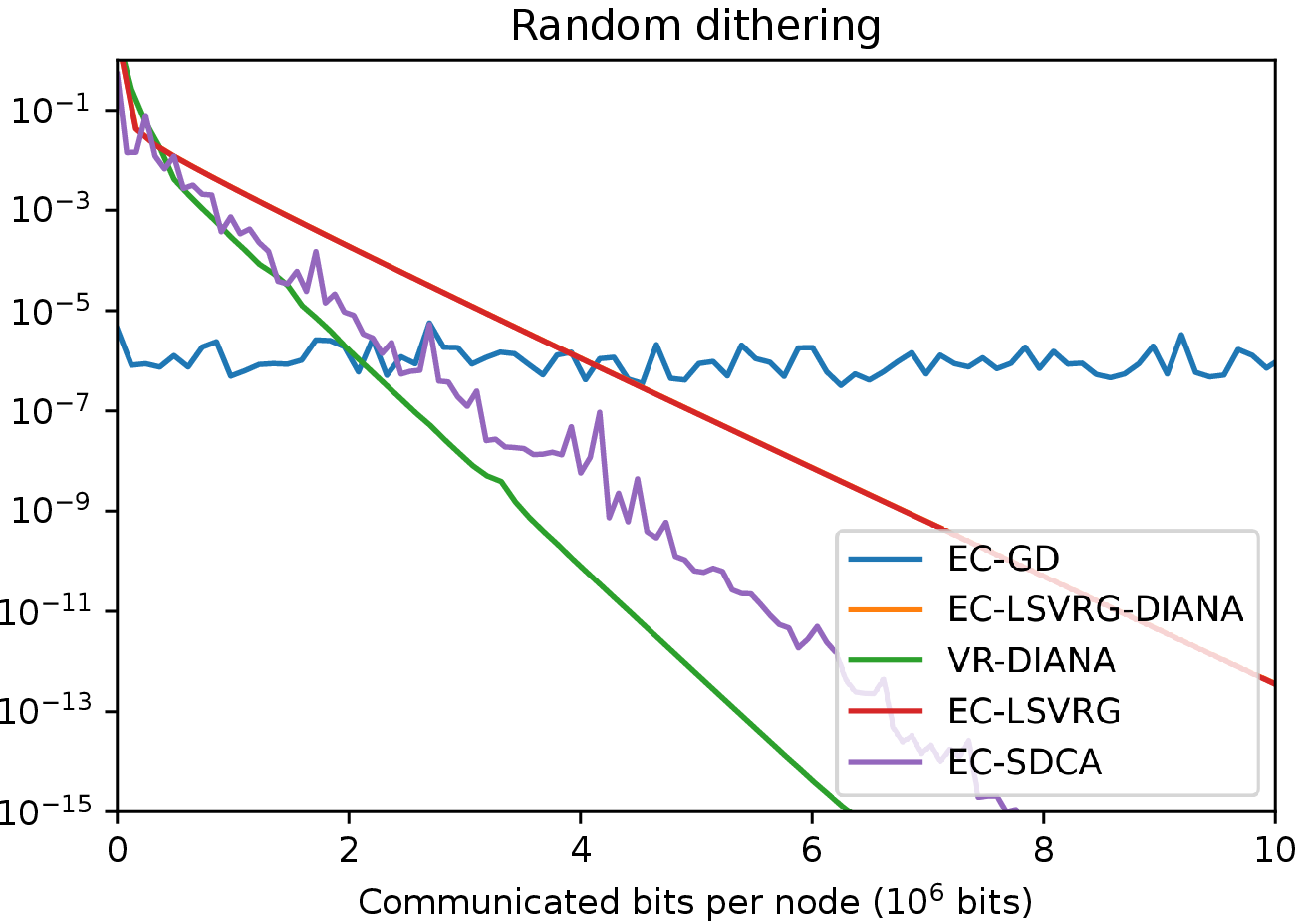}&
		\includegraphics[width=5cm]{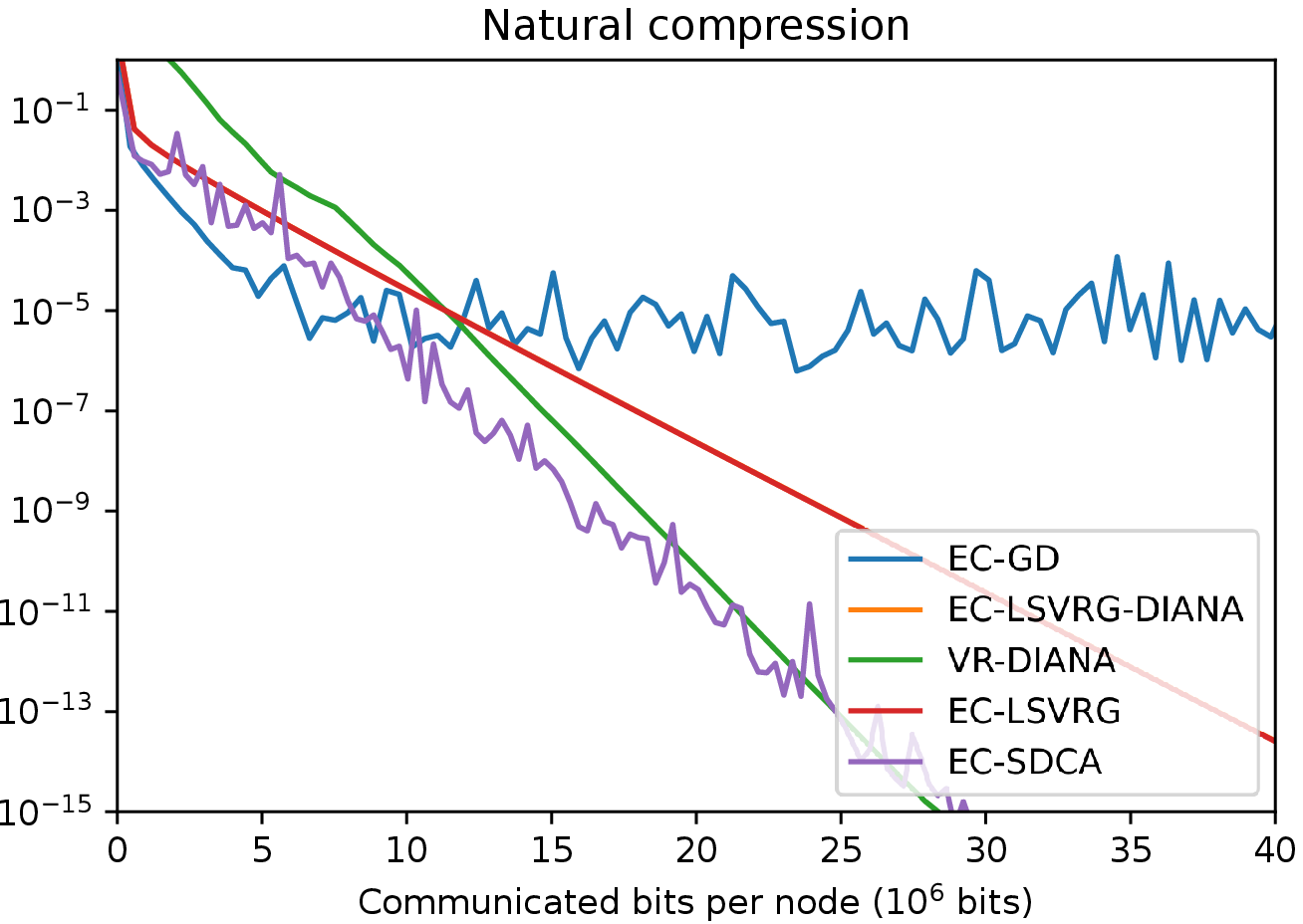}
	\end{tabular}
	\vspace{-0.25cm}
	\caption{Comparison with  ECSGD, ECGD, and EC-LSVRG-DIANA on  \textbf{a5a} (smooth case)}\label{fig:smooth_a5a}
	\vspace{-0.25cm}
\end{figure}

\begin{figure}[H]
	\vspace{-0.25cm}
	\centering
	\begin{tabular}{ccc}
		\includegraphics[width=5cm]{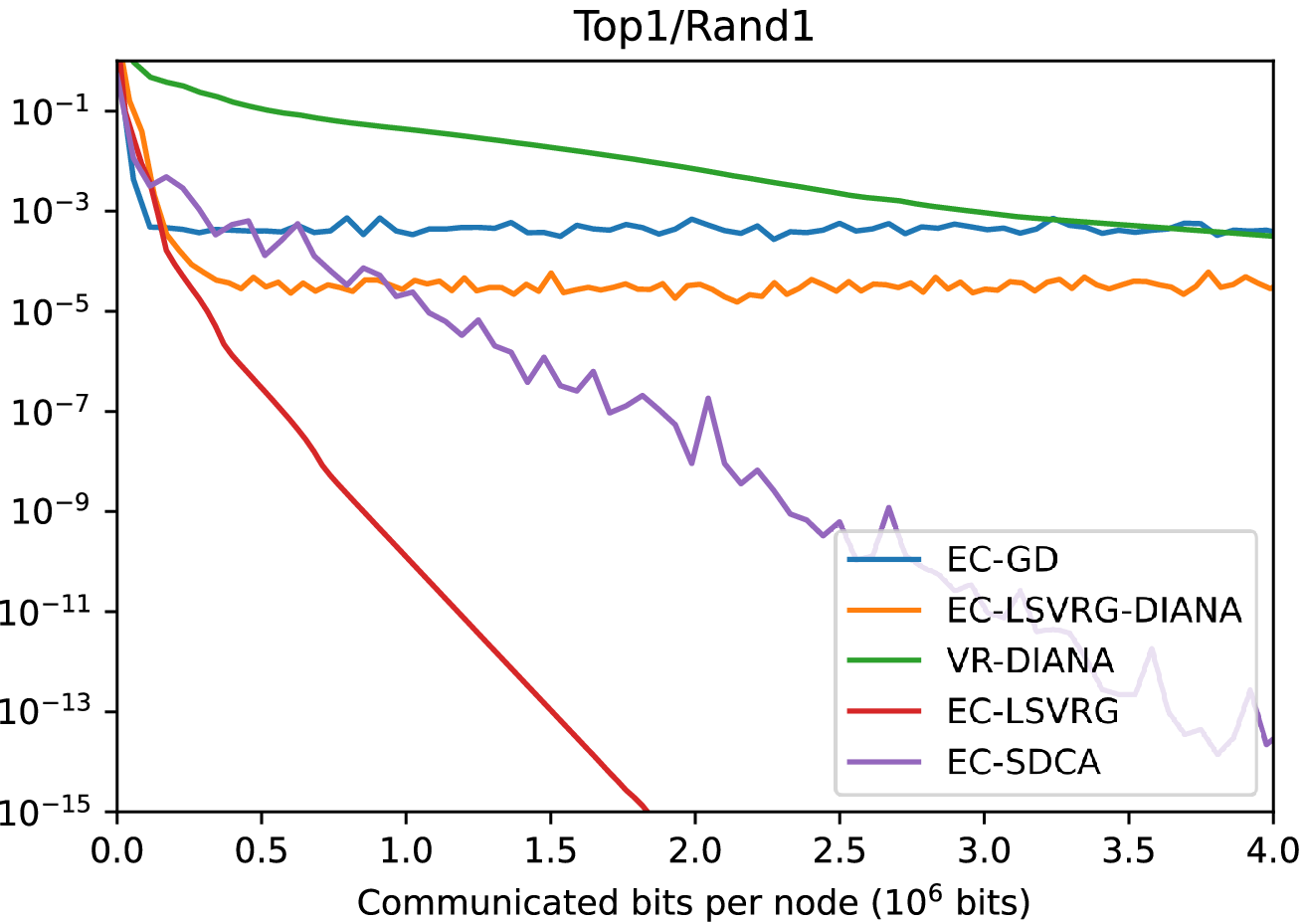}&
		\includegraphics[width=5cm]{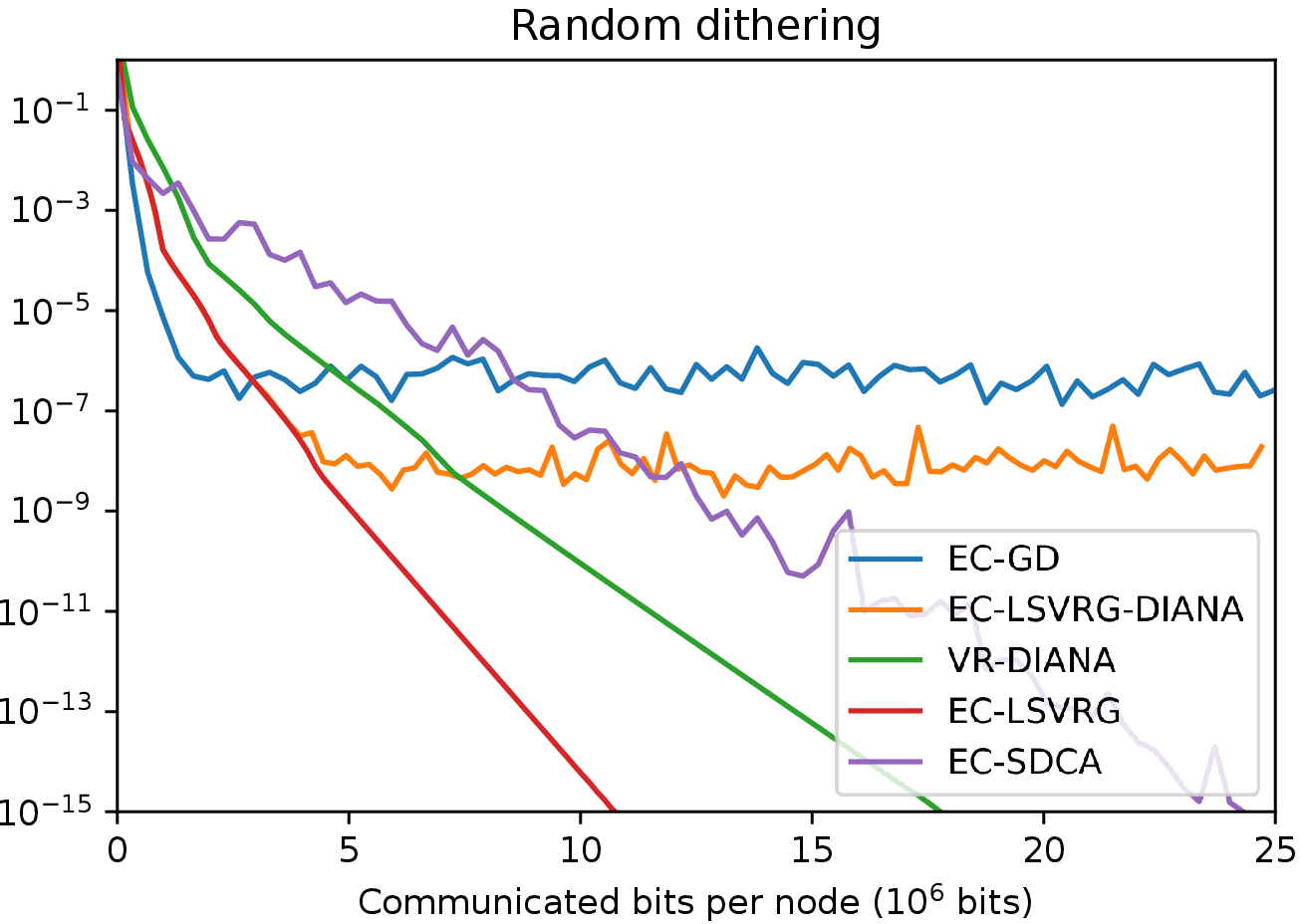}&
		\includegraphics[width=5cm]{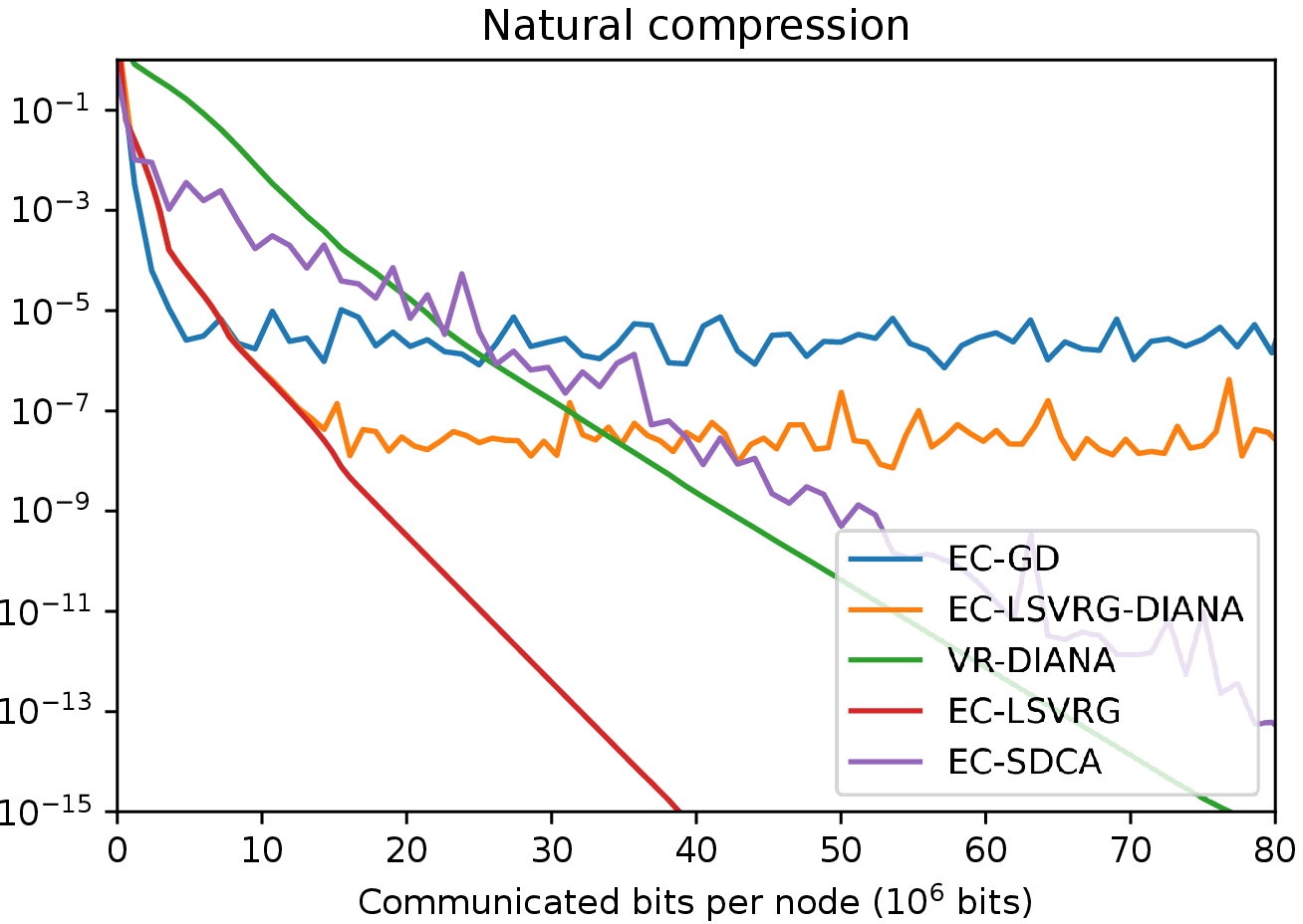}
	\end{tabular}
	\vspace{-0.25cm}
	\caption{Comparison with  ECGD, EC-LSVRG-DIANA, VR-DIANA on  \textbf{a5a} (non-smooth case)}\label{fig:non-smooth_a9a}
	\vspace{-0.25cm}
\end{figure}

\begin{figure}[H]
	\vspace{-0.25cm}
	\centering
	\begin{tabular}{ccc}
		\includegraphics[width=5cm]{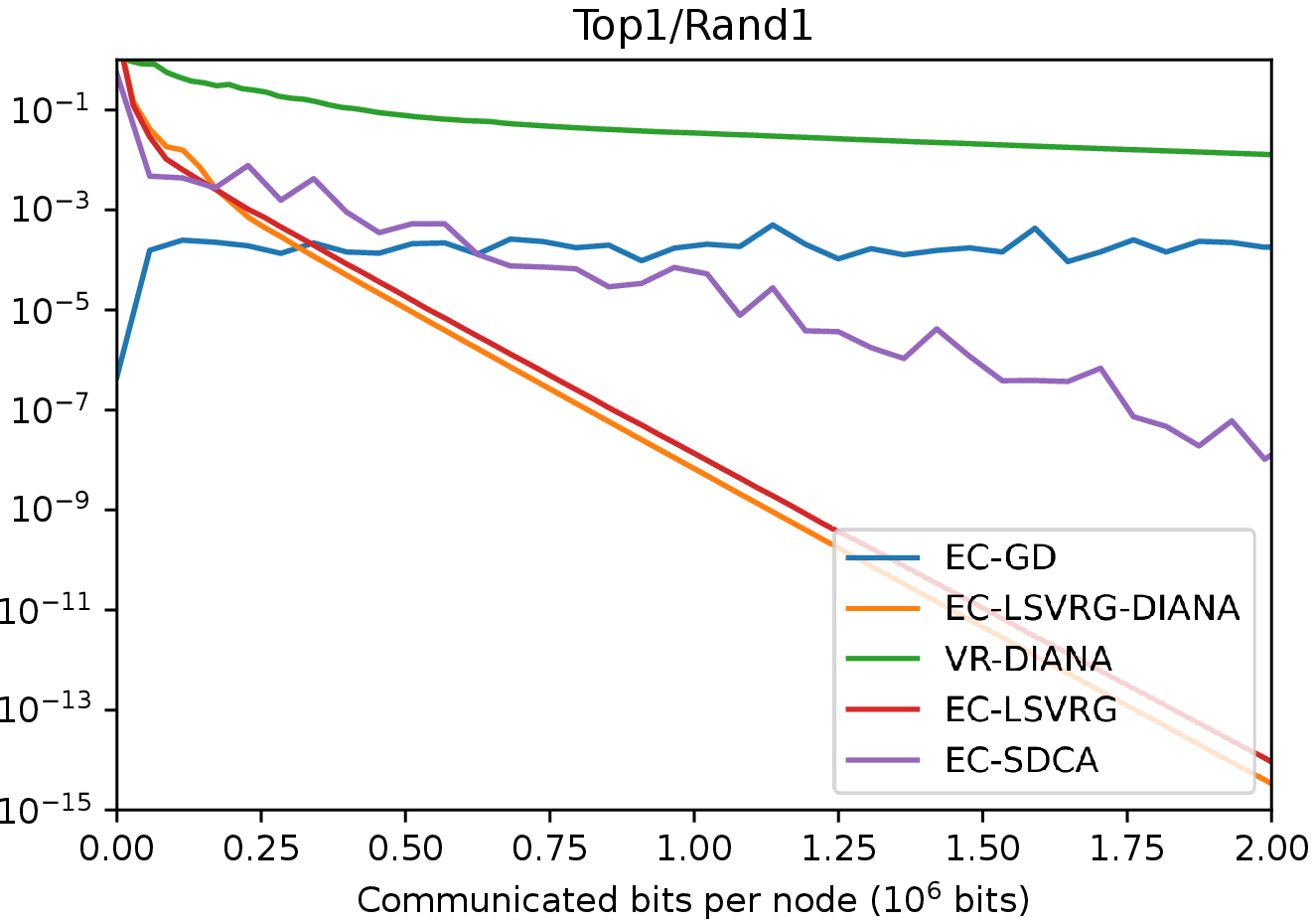}&
		\includegraphics[width=5cm]{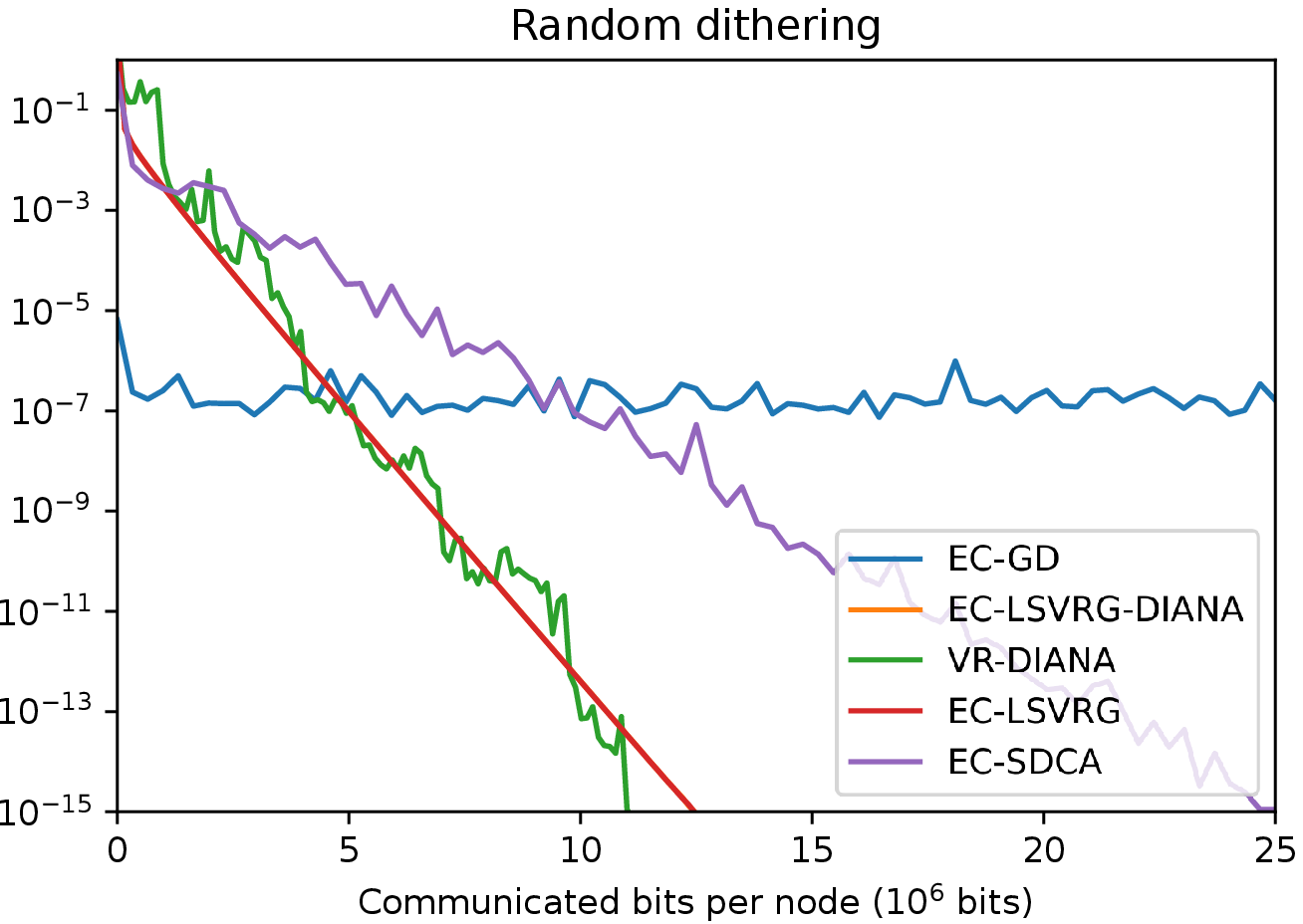}&
		\includegraphics[width=5cm]{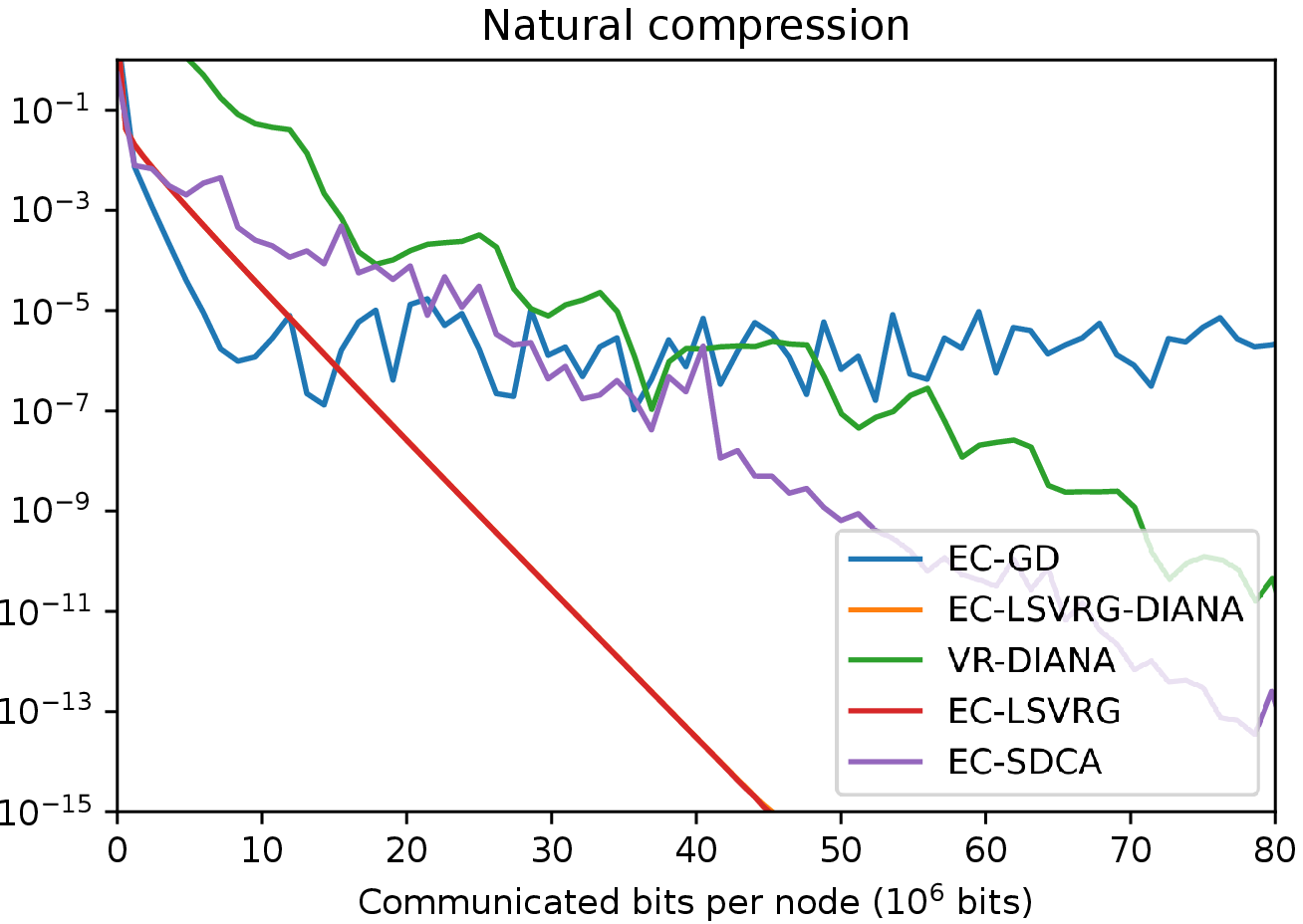}
	\end{tabular}
	\vspace{-0.25cm}
	\caption{Comparison with  ECSGD, ECGD, and EC-LSVRG-DIANA on  \textbf{a9a} (smooth case)}\label{fig:smooth_a9a}
	\vspace{-0.25cm}
\end{figure}

\begin{figure}[H]
	\vspace{-0.25cm}
	\centering
	\begin{tabular}{ccc}
		\includegraphics[width=5cm]{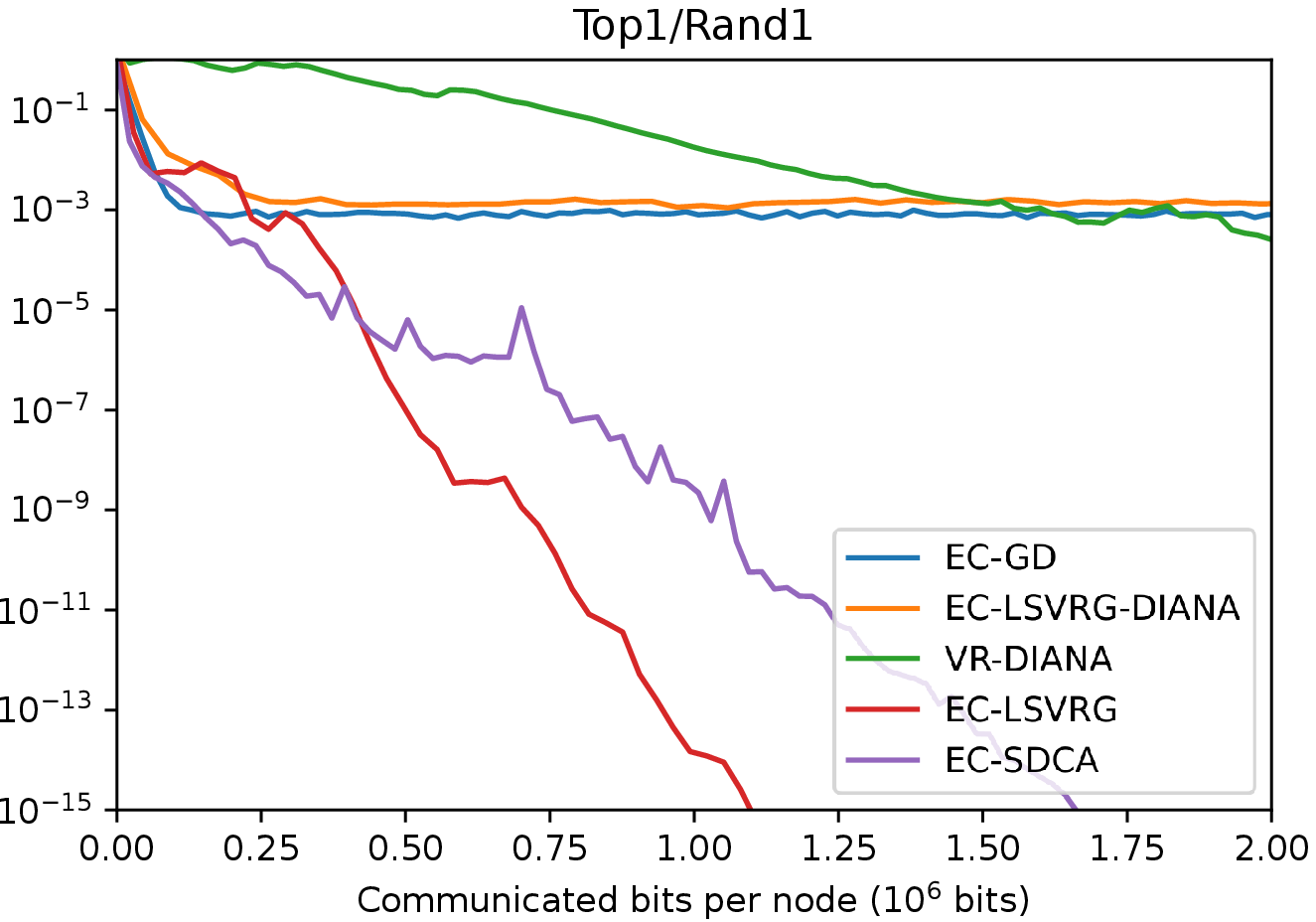}&
		\includegraphics[width=5cm]{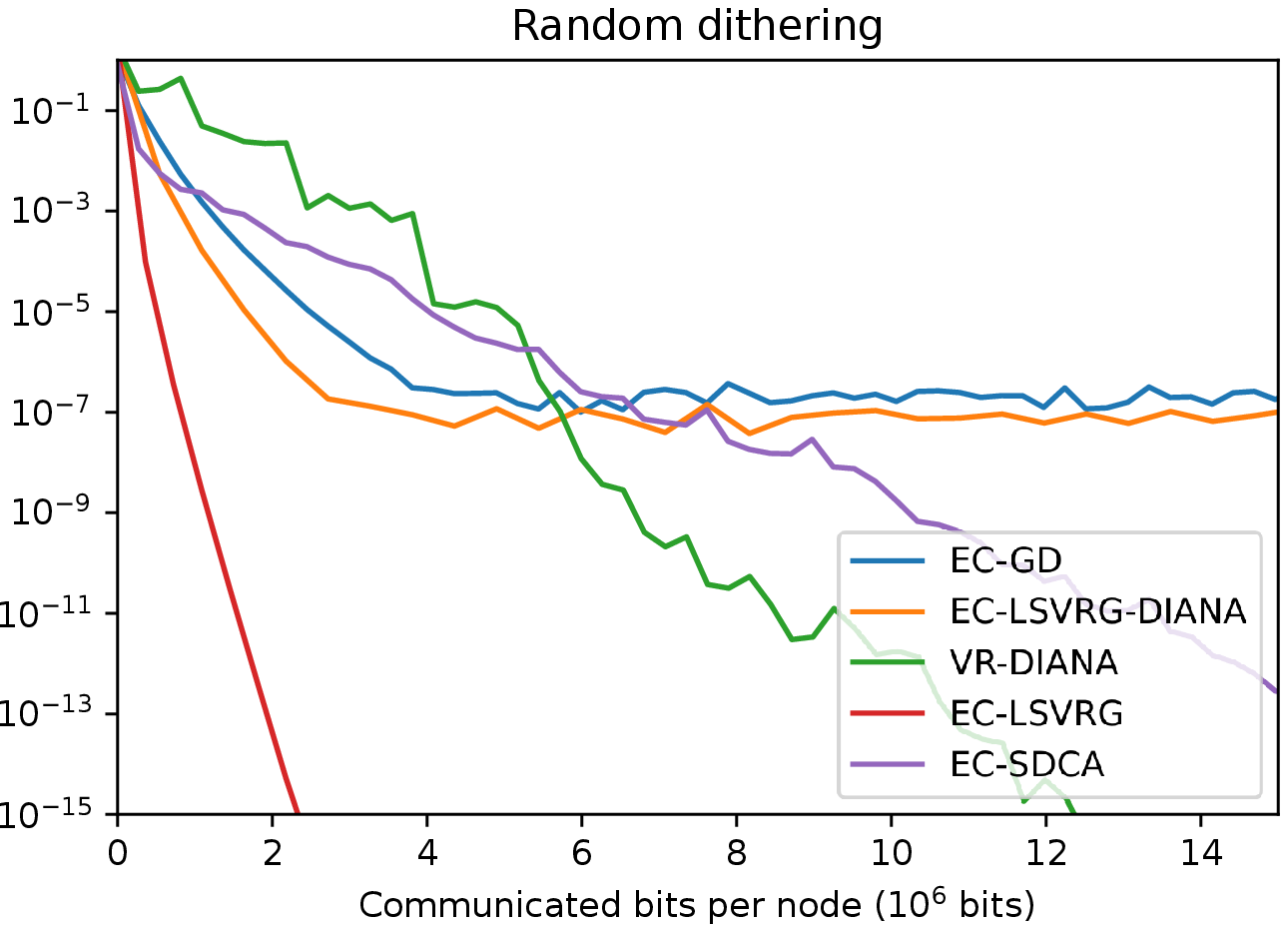}&
		\includegraphics[width=5cm]{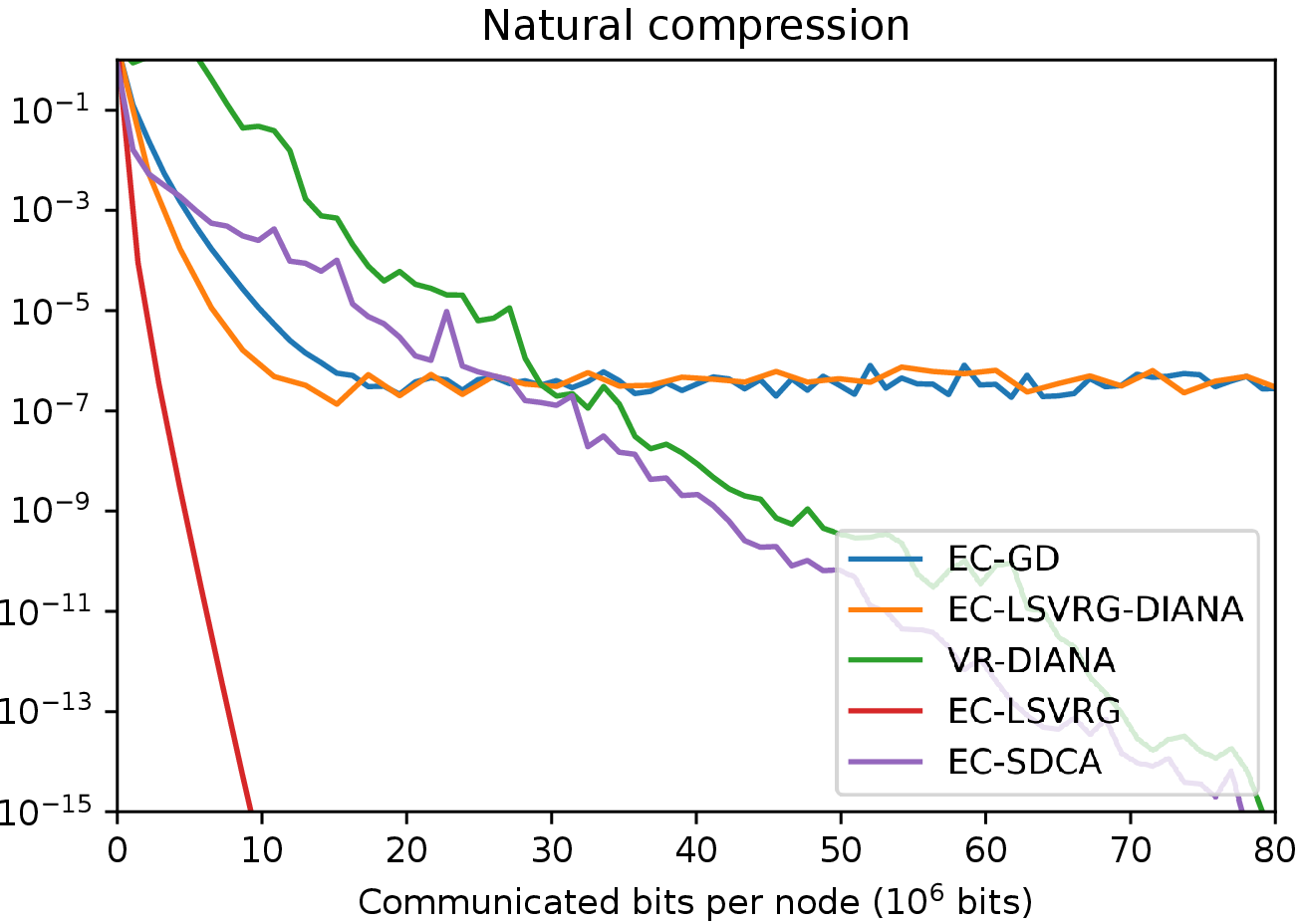}
	\end{tabular}
	\vspace{-0.25cm}
	\caption{Comparison with  ECGD, EC-LSVRG-DIANA, VR-DIANA on  \textbf{w6a} (non-smooth case)}\label{fig:non-smooth_w6a}
	\vspace{-0.25cm}
\end{figure}

\begin{figure}[H]
	\vspace{-0.35cm}
	\centering
	\begin{tabular}{ccc}
		\includegraphics[width=5cm]{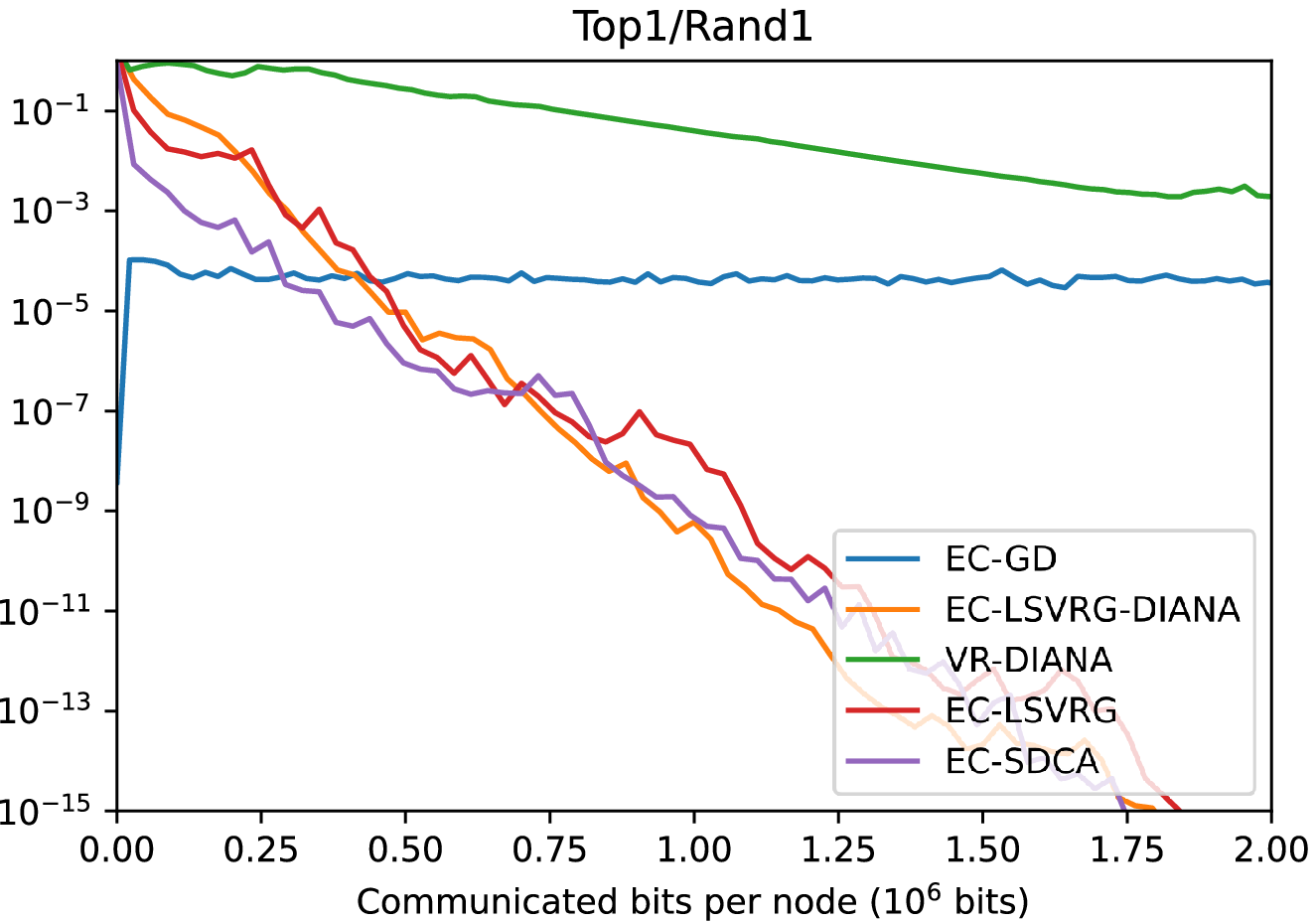}&
		\includegraphics[width=5cm]{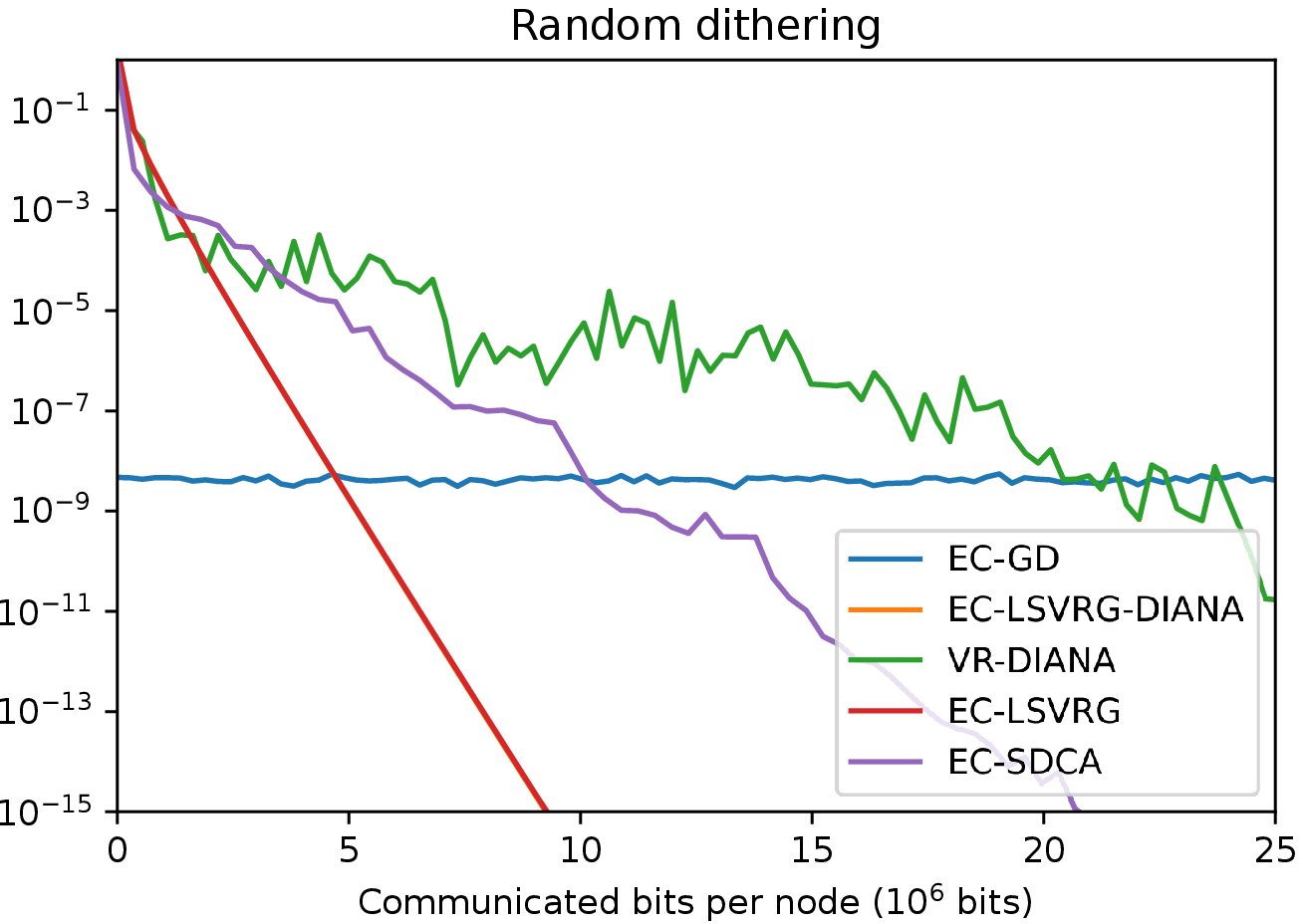}&
		\includegraphics[width=5cm]{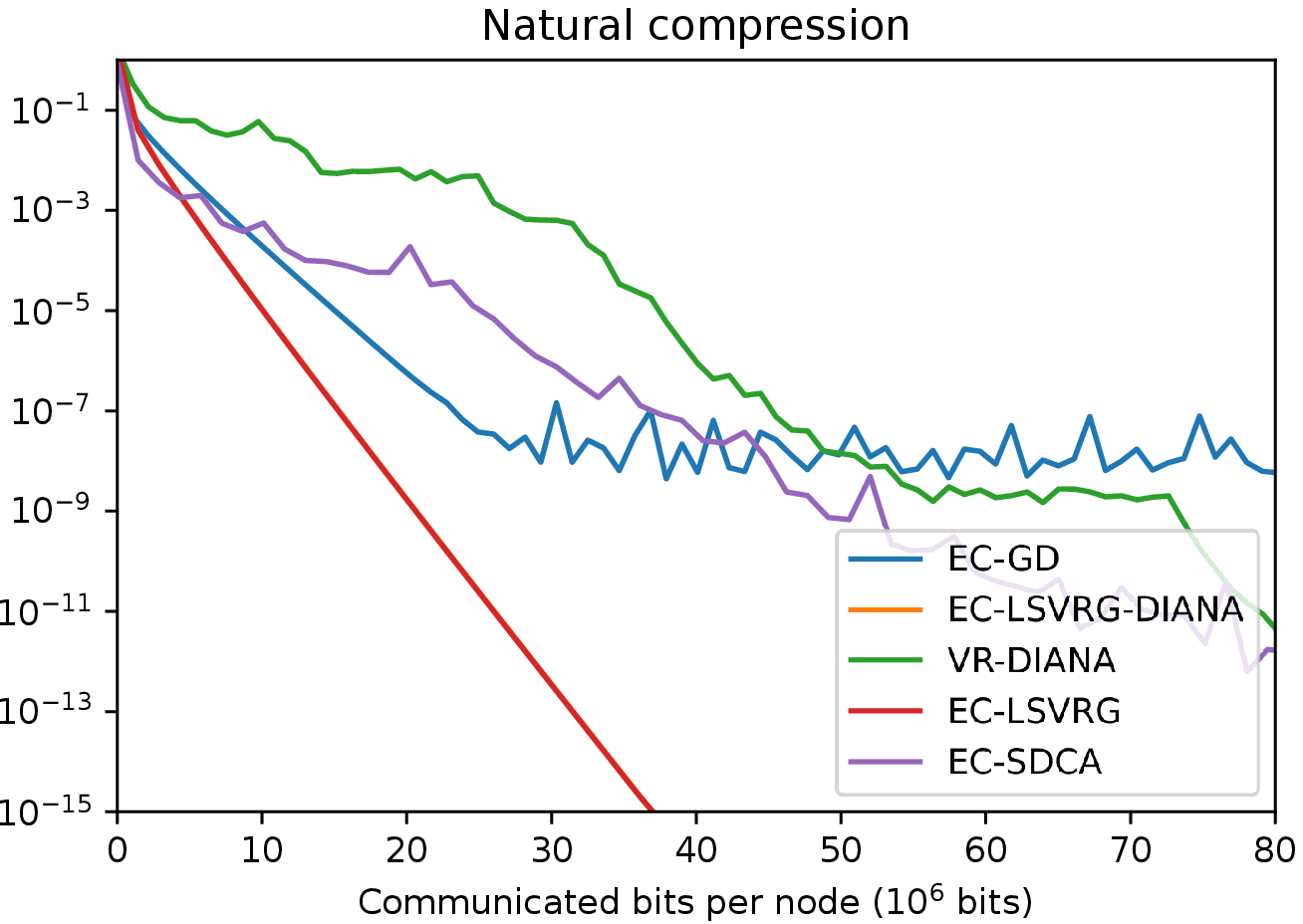}
	\end{tabular}
	\vspace{-0.35cm}
	\caption{Comparison with  ECSGD, ECGD, and EC-LSVRG-DIANA on  \textbf{w6a} (smooth case)}\label{fig:smooth_w6a}
	\vspace{-0.55cm}
\end{figure}

\subsection{TopK vs NTopK vs RTopK}

Previous experiments have shown the efficiency of the contraction compressor. In this context, we consider using random dithering + TopK (RTopK) and natural compression + TopK (NTopK) to further improve the performance. It should be noted that NTopK is suitable for any $K$, while for RTopK, we usually require $K>1$. By Figures \ref{fig:topk} $-$ \ref{fig:topk_w6a}, we can notice that either NTopK or RTopK reduces the communication costs than TopK only.

\begin{figure}[H]
	\vspace{-0.35cm}
	\centering
	\begin{tabular}{ccc}
		\includegraphics[width=5cm]{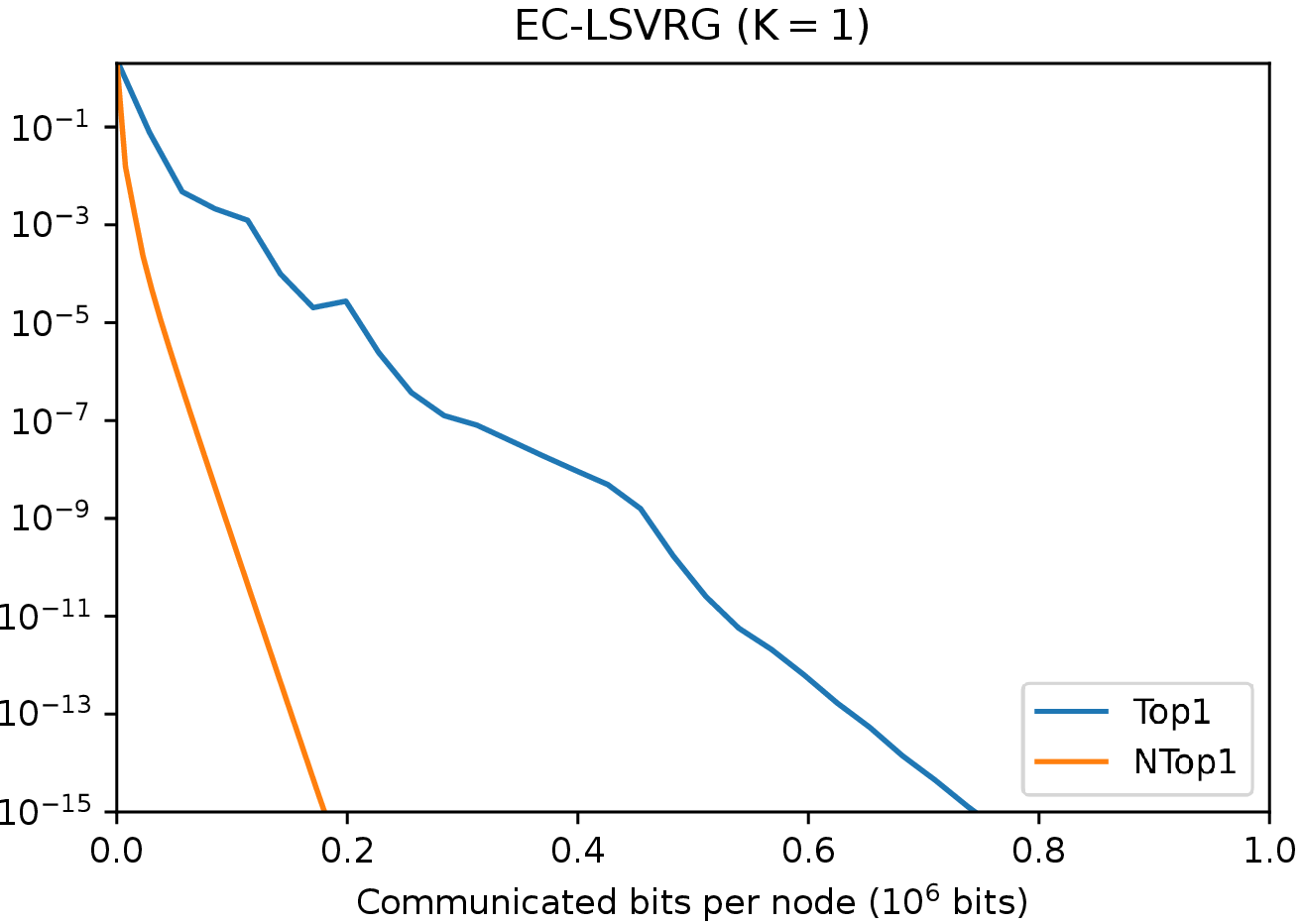}&
		\includegraphics[width=5cm]{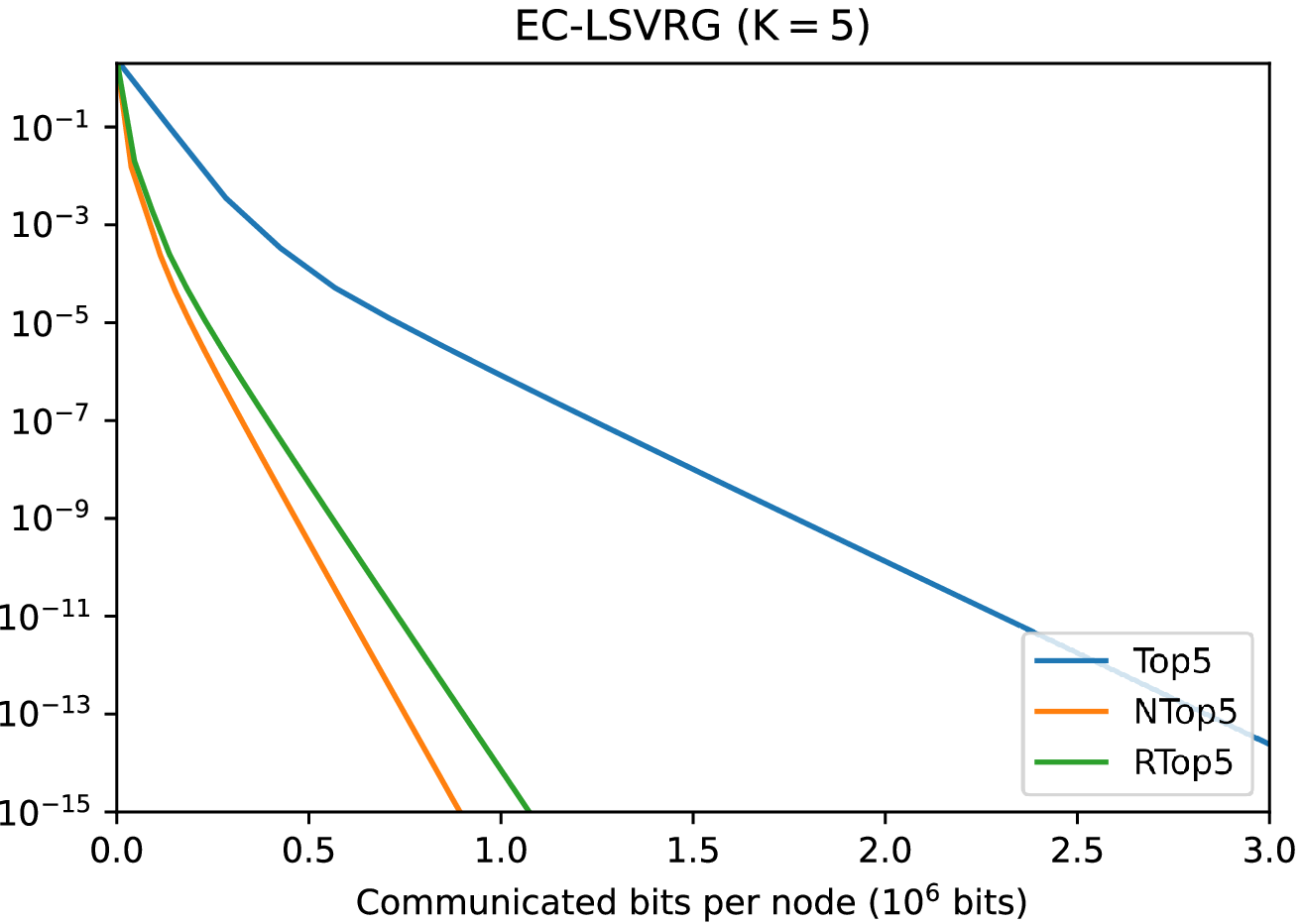}&
		\includegraphics[width=5cm]{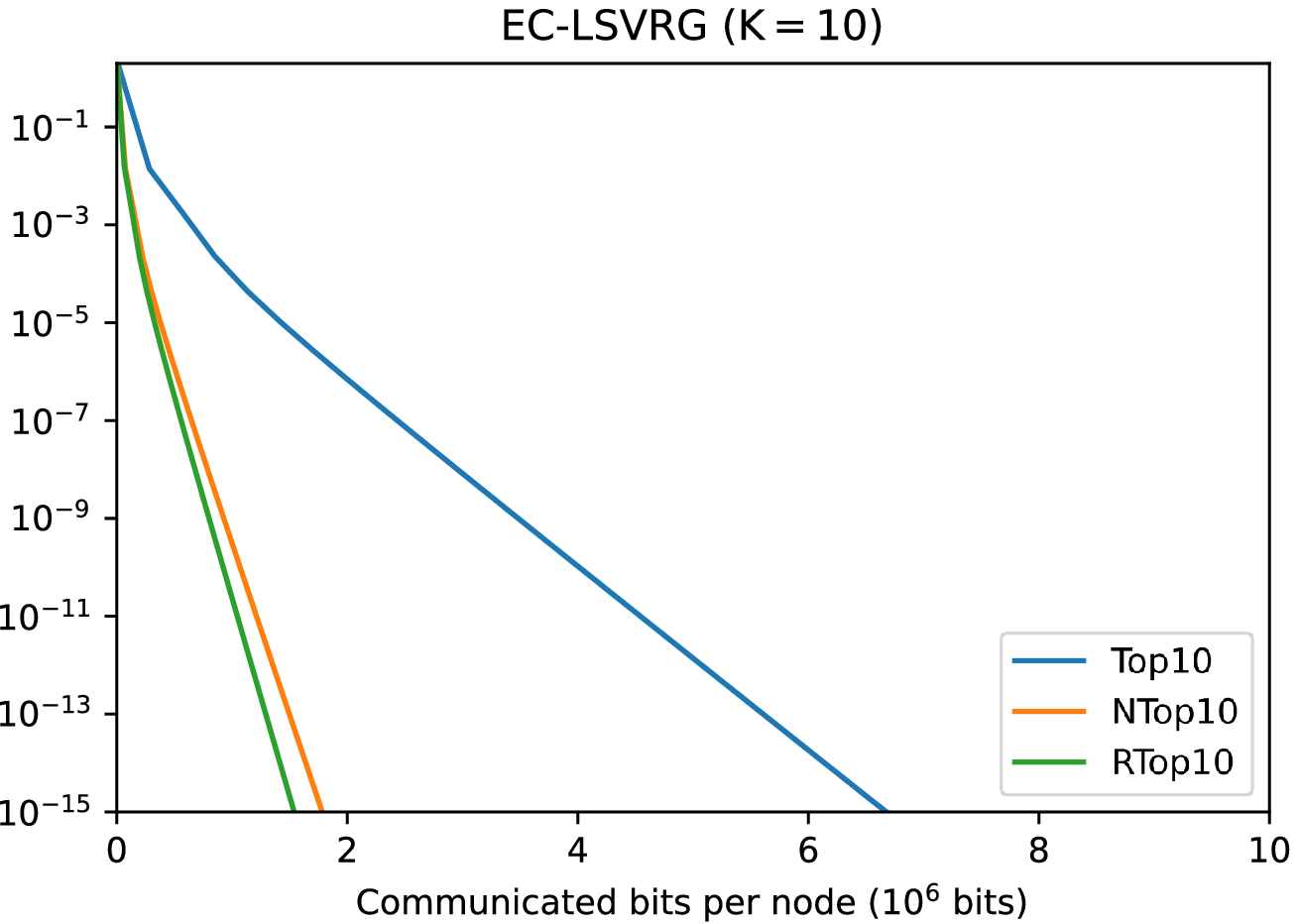}\\
				\includegraphics[width=5cm]{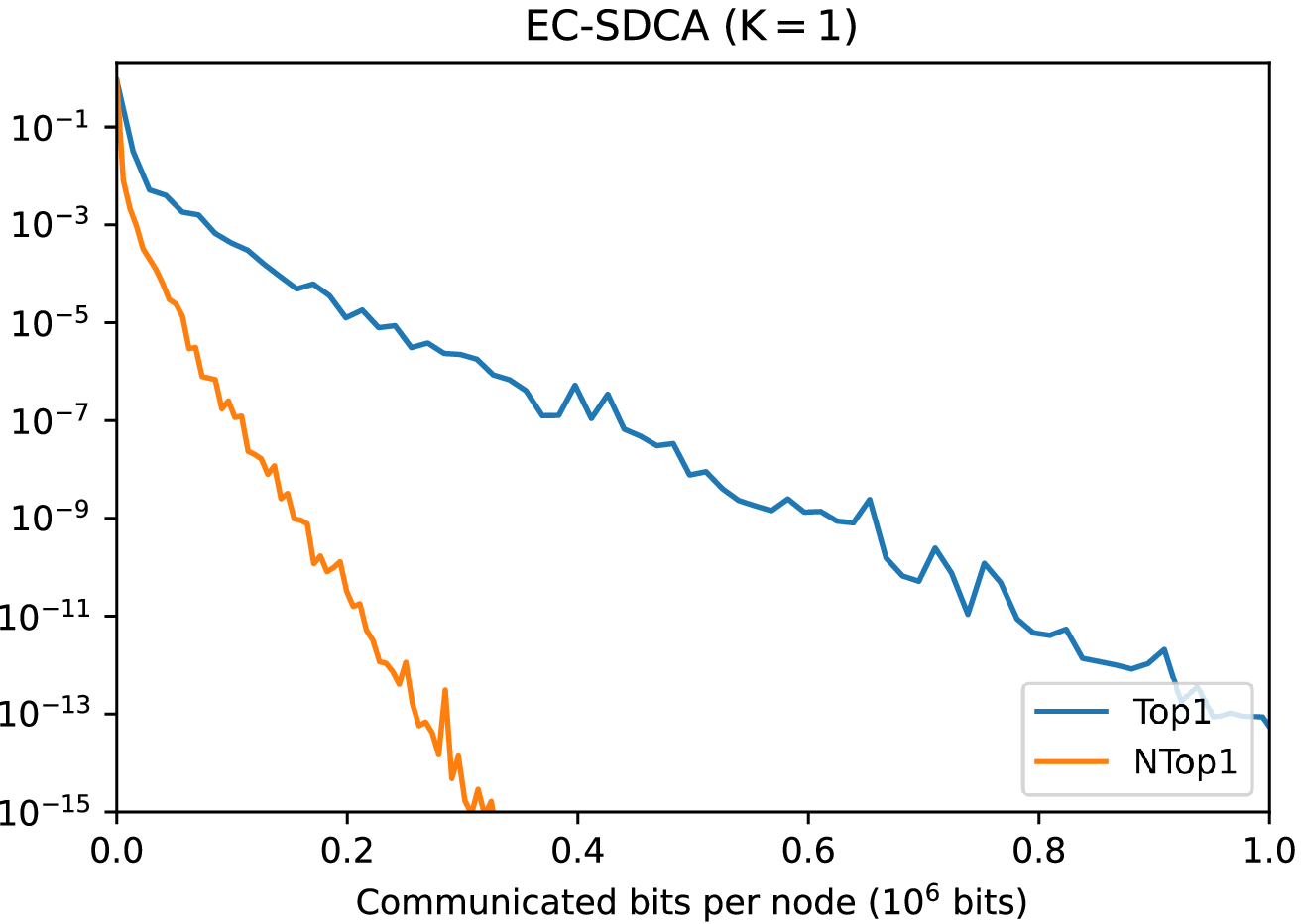}&
		\includegraphics[width=5cm]{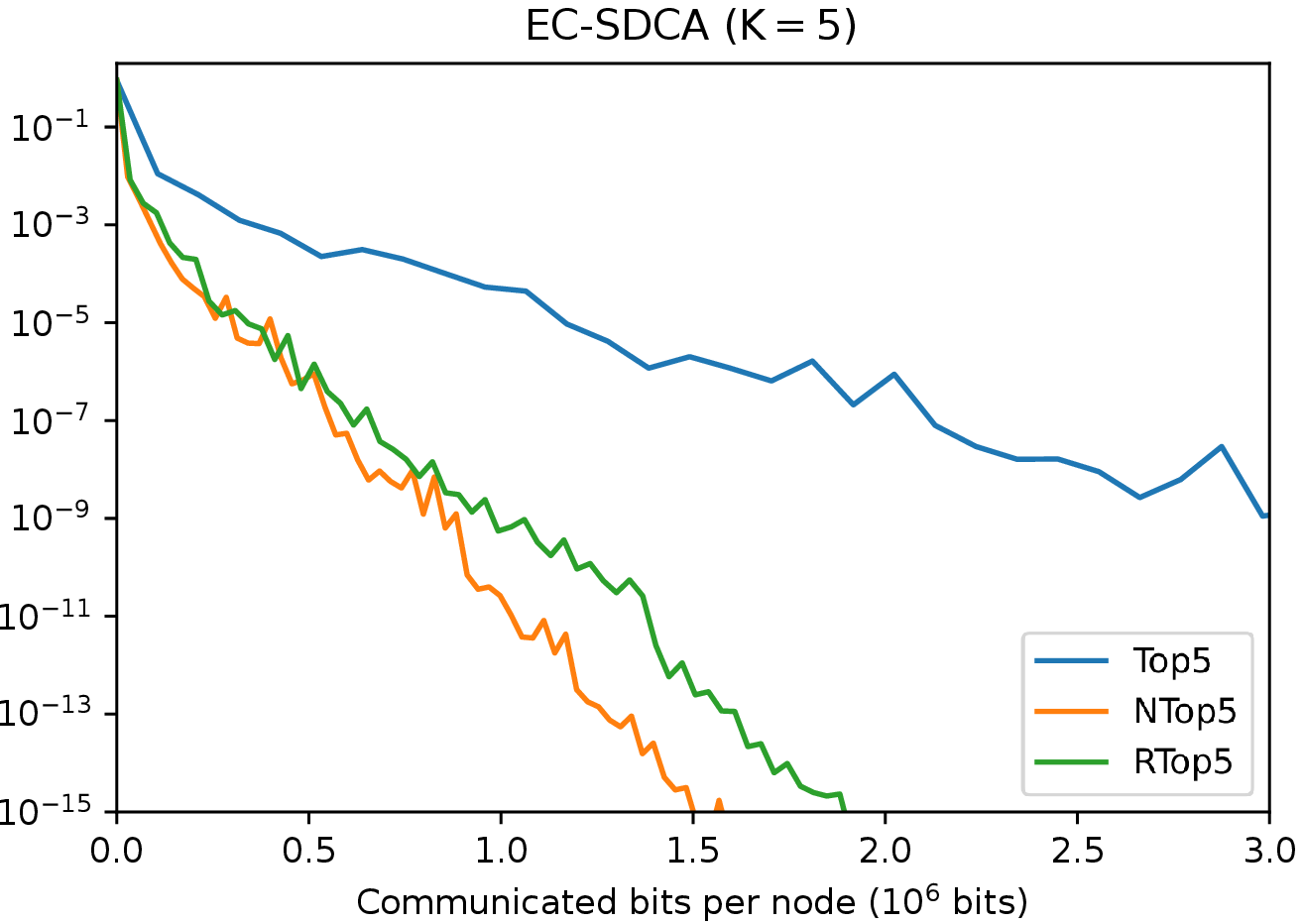}&
		\includegraphics[width=5cm]{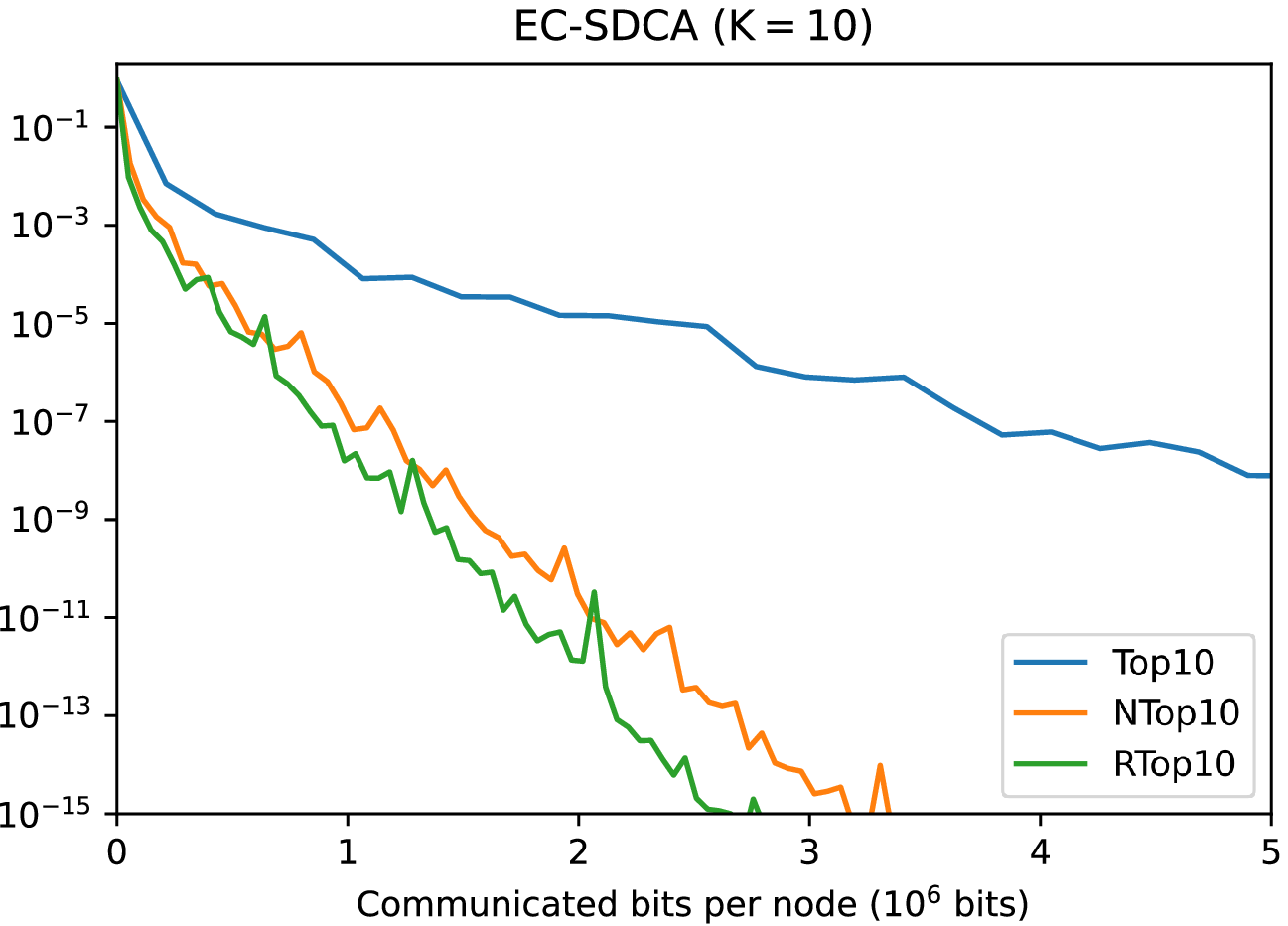}
	\end{tabular}
	\vspace{-0.35cm}
	\caption{Comparison among  TopK, NTopK, and RTopK on \textbf{mushrooms} }\label{fig:topk}
	\vspace{-0.35cm}
\end{figure}

\begin{figure}[H]
	\vspace{-0.35cm}
	\centering
	\begin{tabular}{ccc}
		\includegraphics[width=5cm]{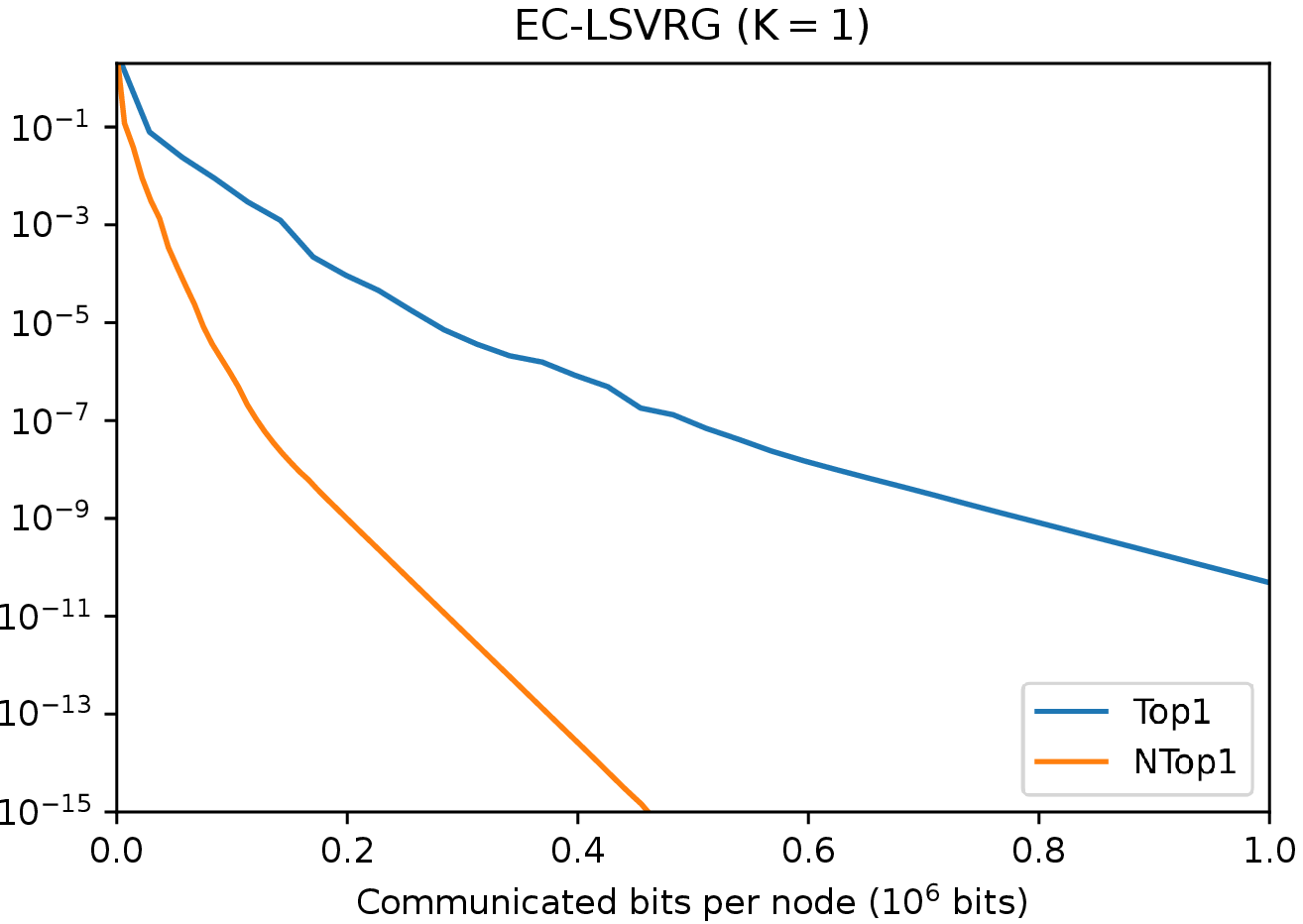}&
		\includegraphics[width=5cm]{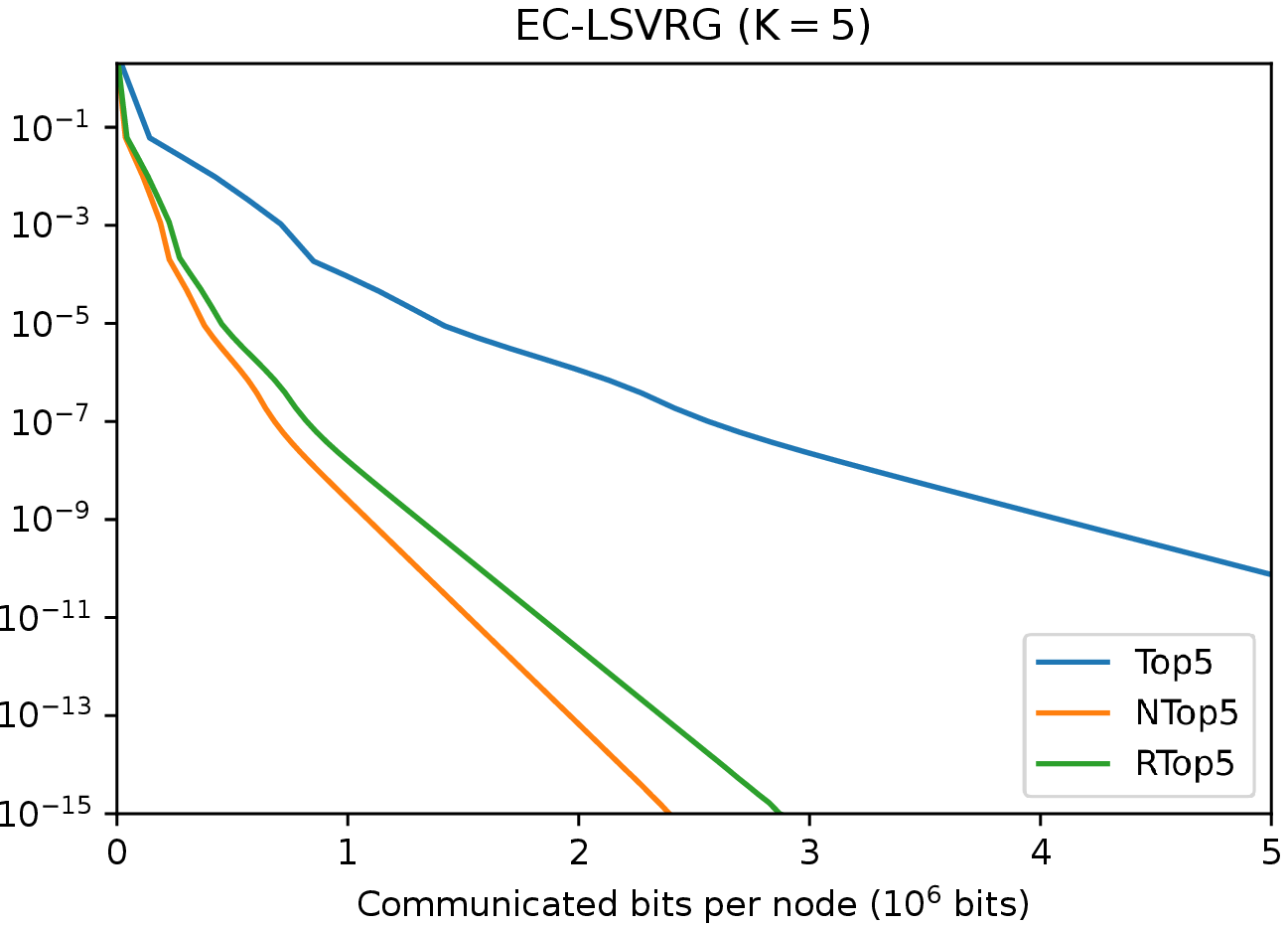}&
		\includegraphics[width=5cm]{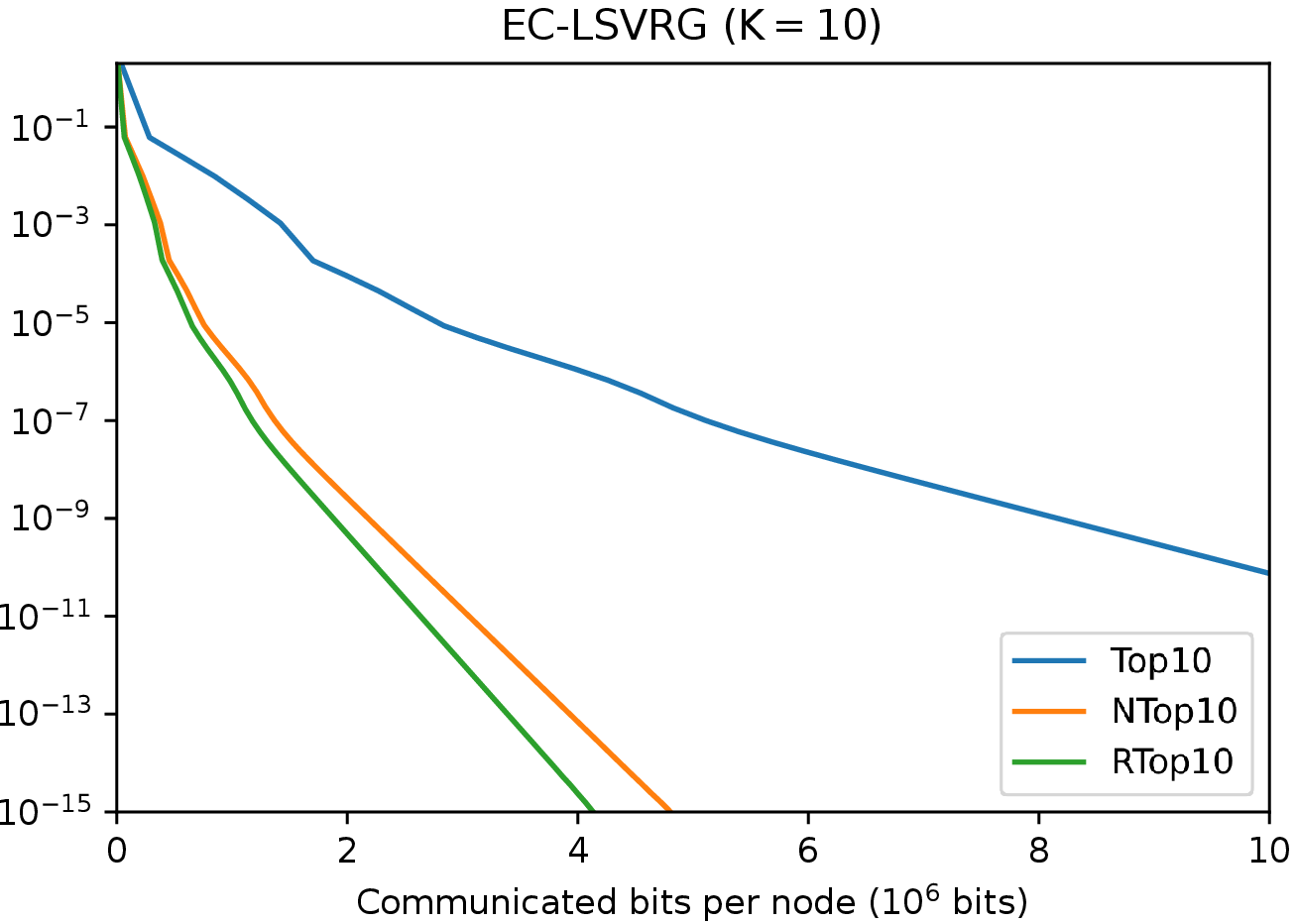}\\
		\includegraphics[width=5cm]{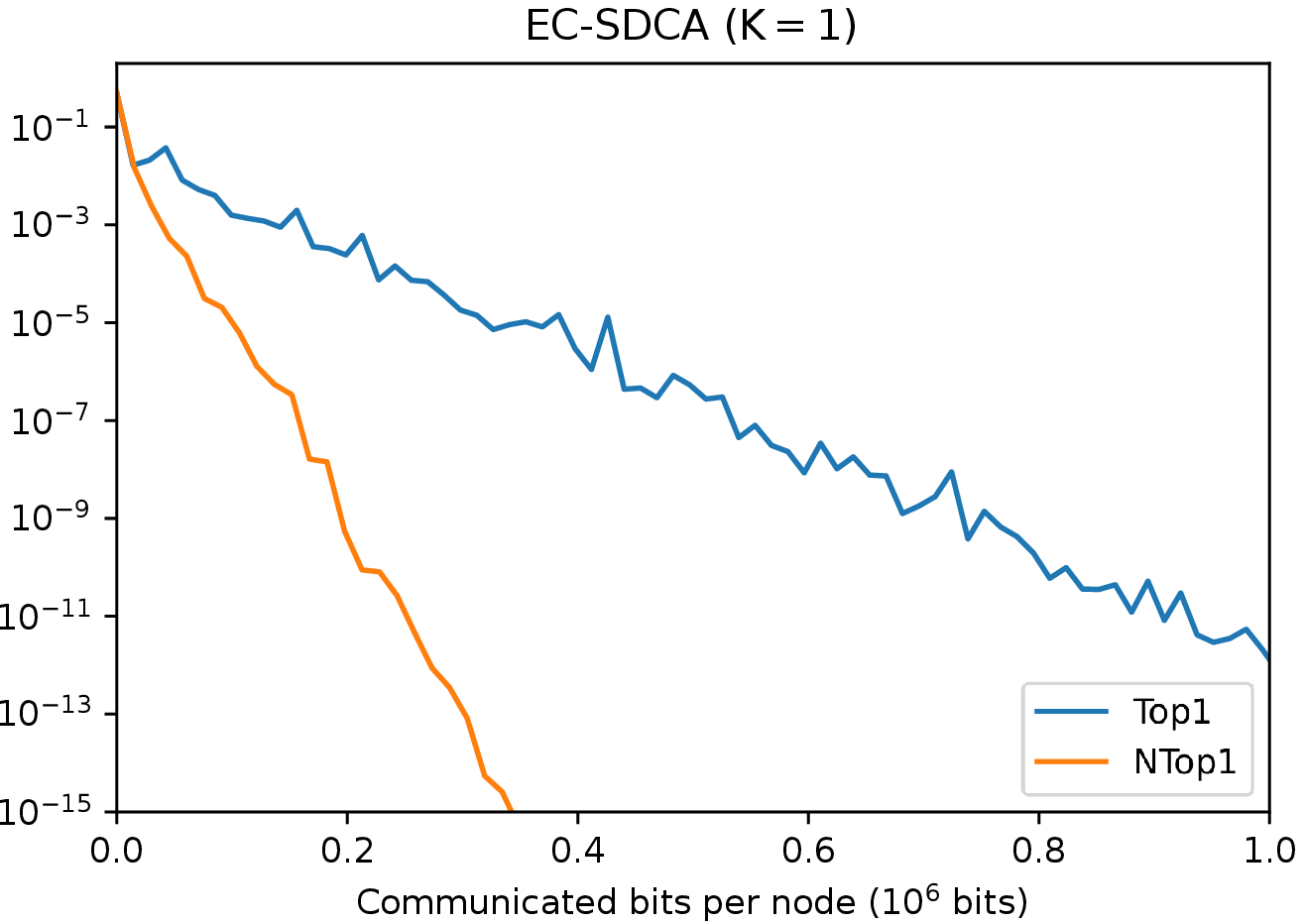}&
		\includegraphics[width=5cm]{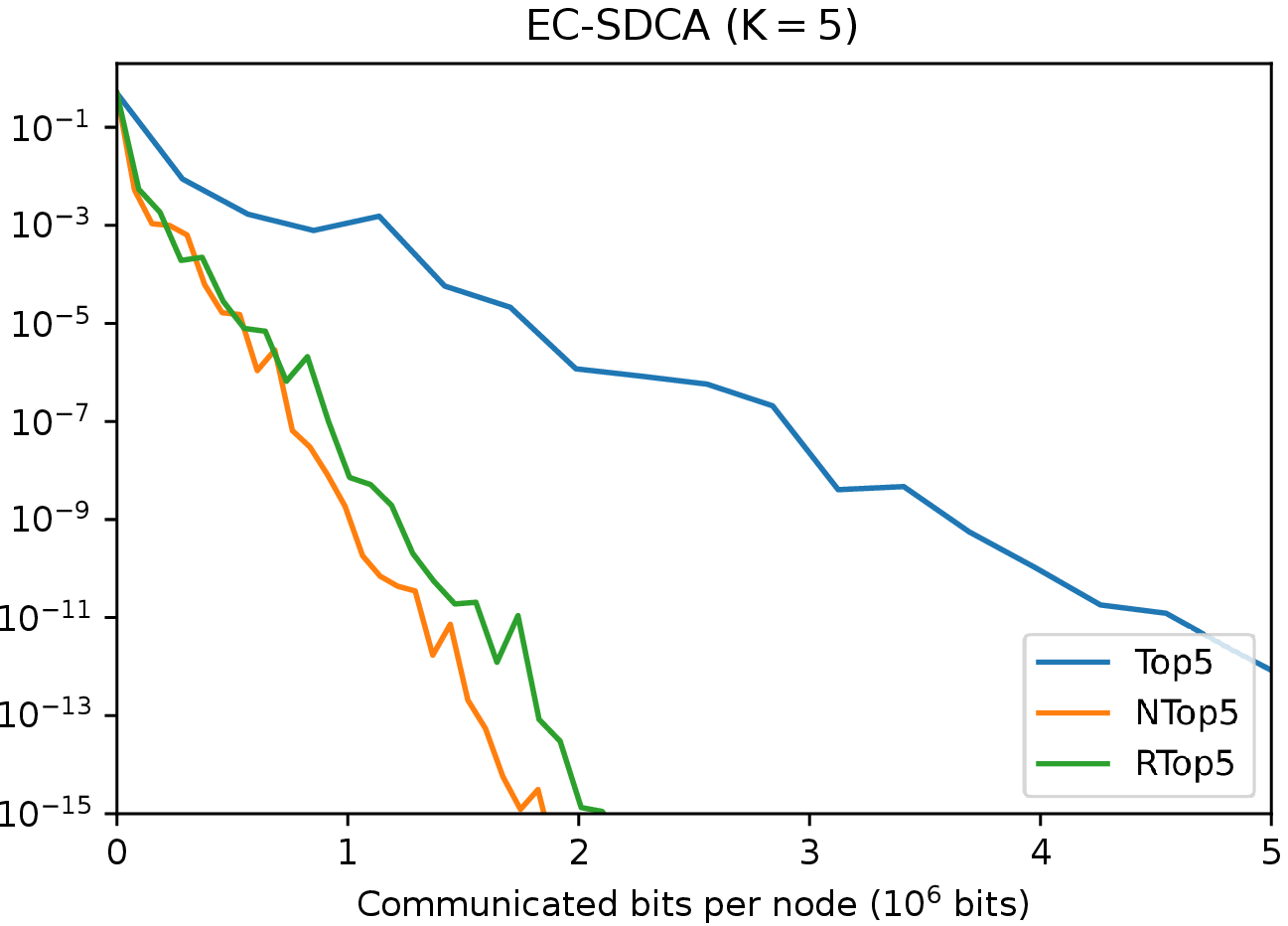}&
		\includegraphics[width=5cm]{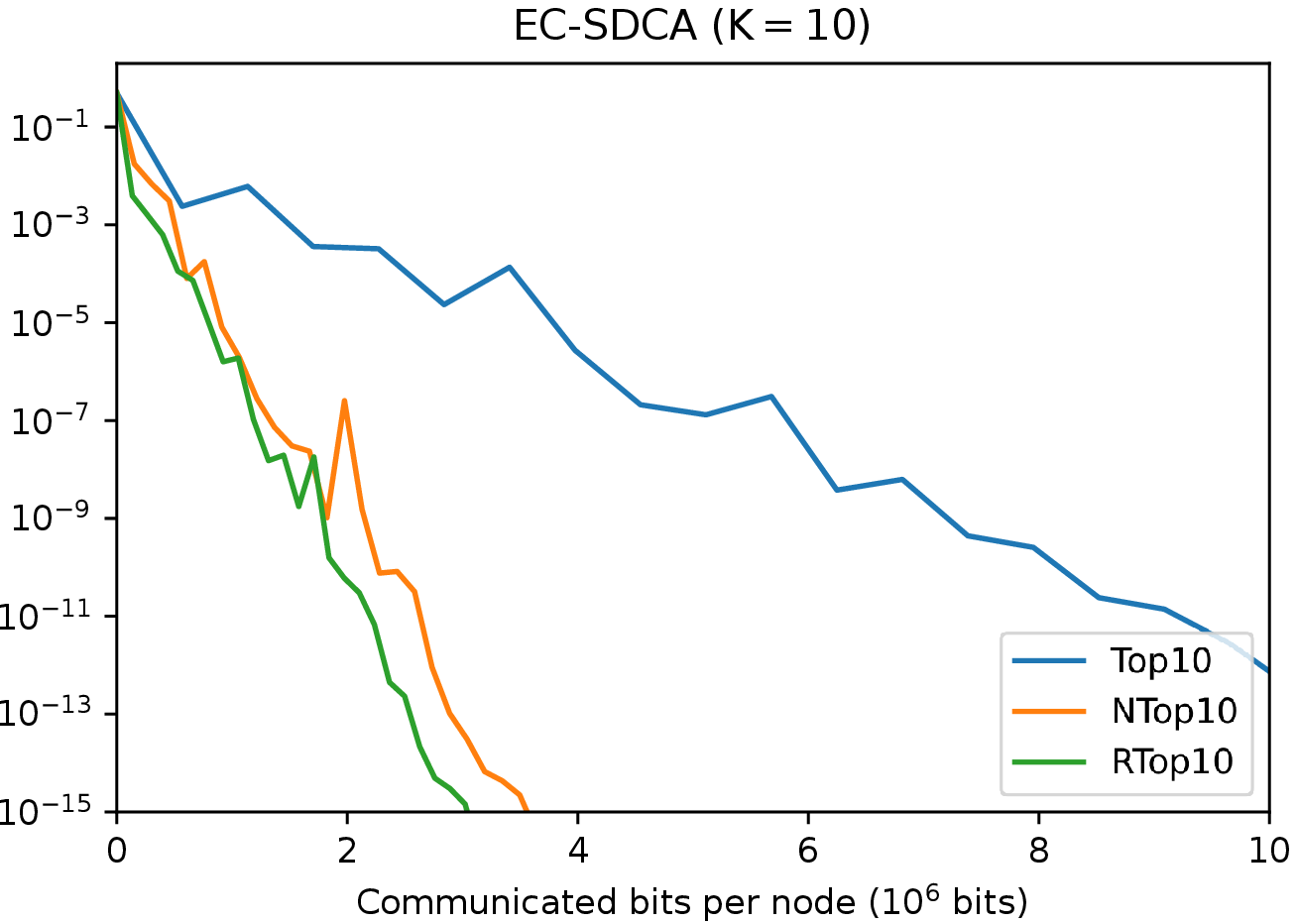}
	\end{tabular}
	\vspace{-0.35cm}
	\caption{Comparison among  TopK, NTopK, and RTopK on \textbf{a5a} }\label{fig:topk_a5a}
	\vspace{-0.35cm}
\end{figure}

\begin{figure}[H]
	\vspace{-0.25cm}
	\centering
	\begin{tabular}{ccc}
		\includegraphics[width=5cm]{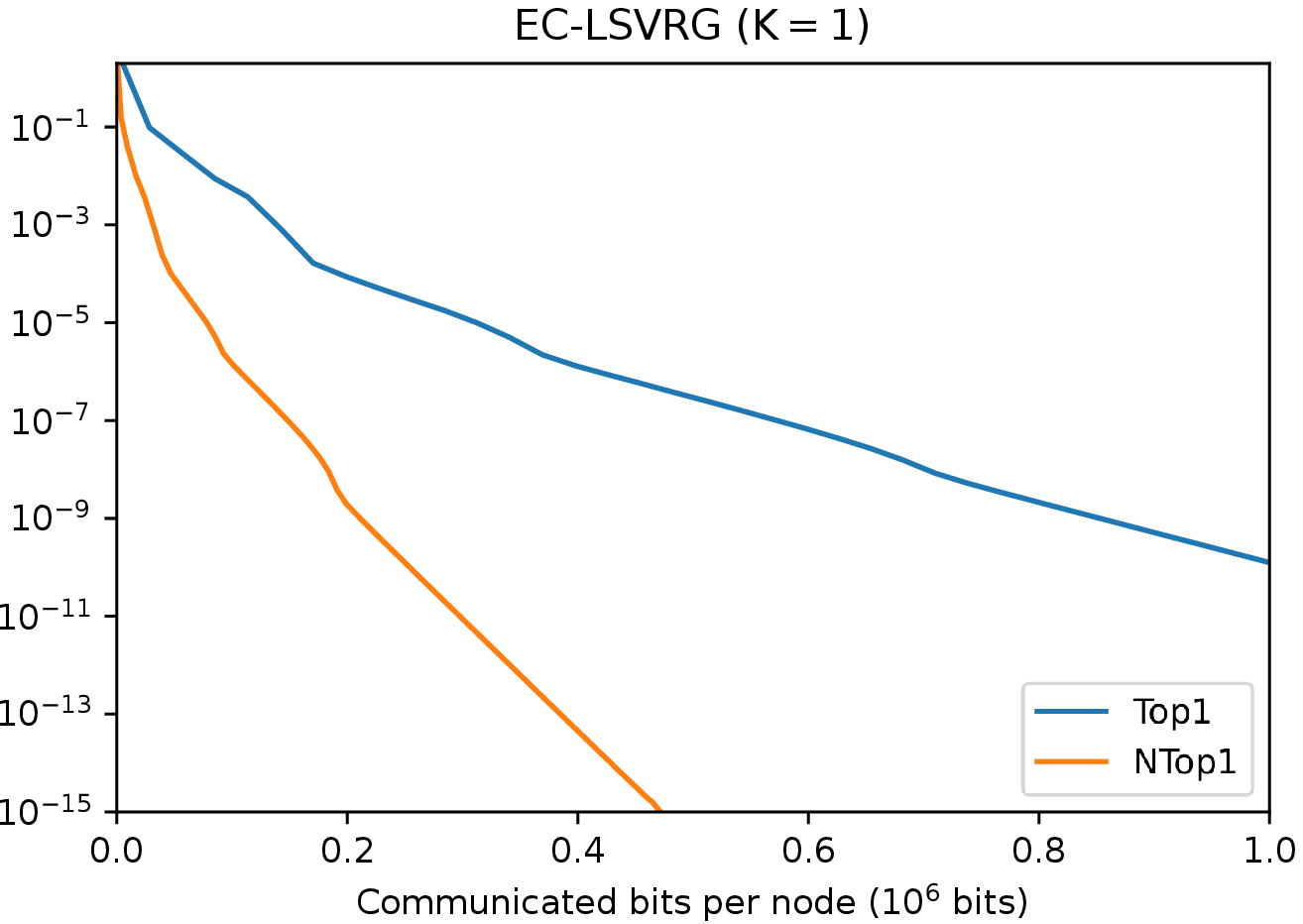}&
		\includegraphics[width=5cm]{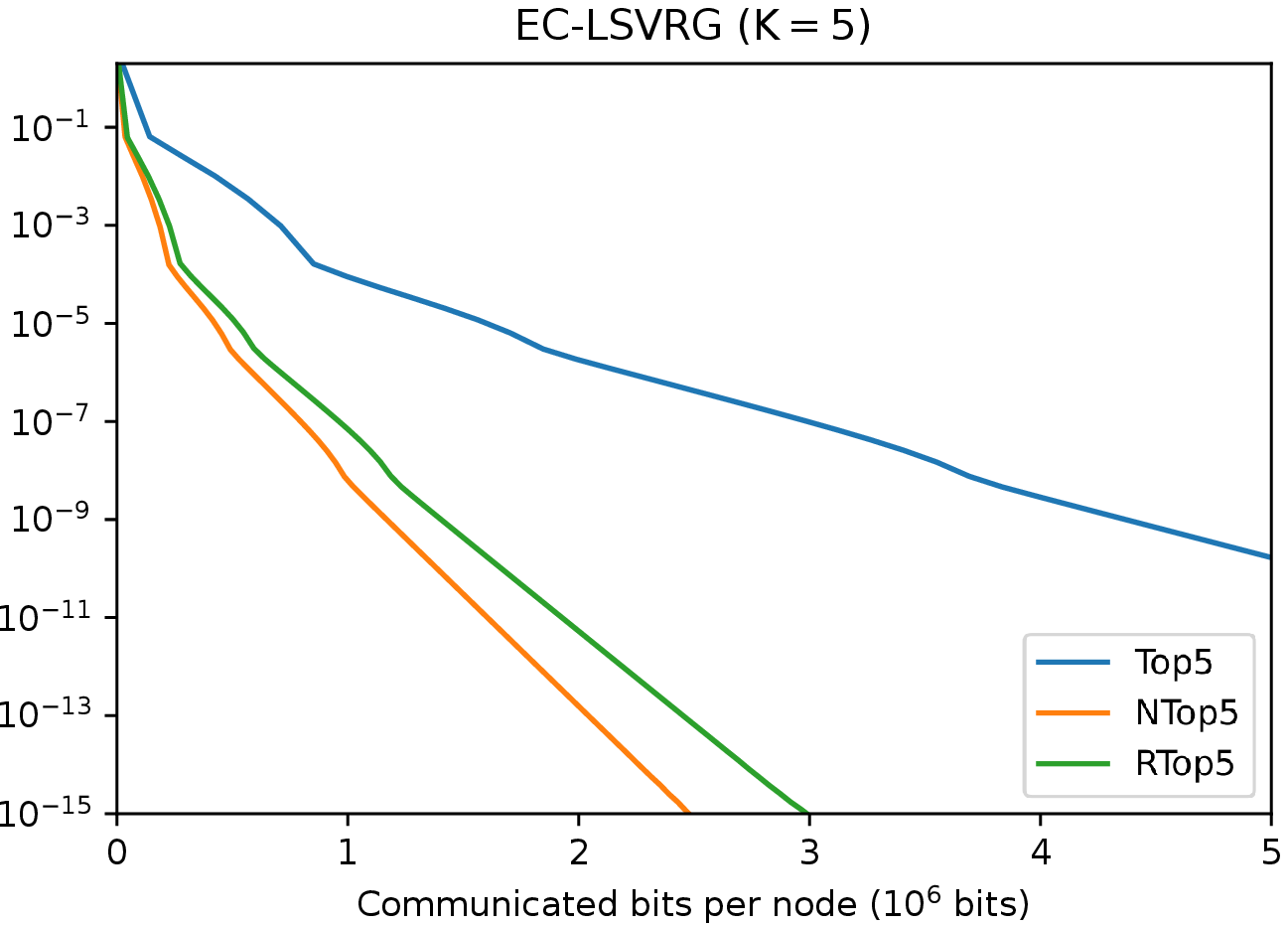}&
		\includegraphics[width=5cm]{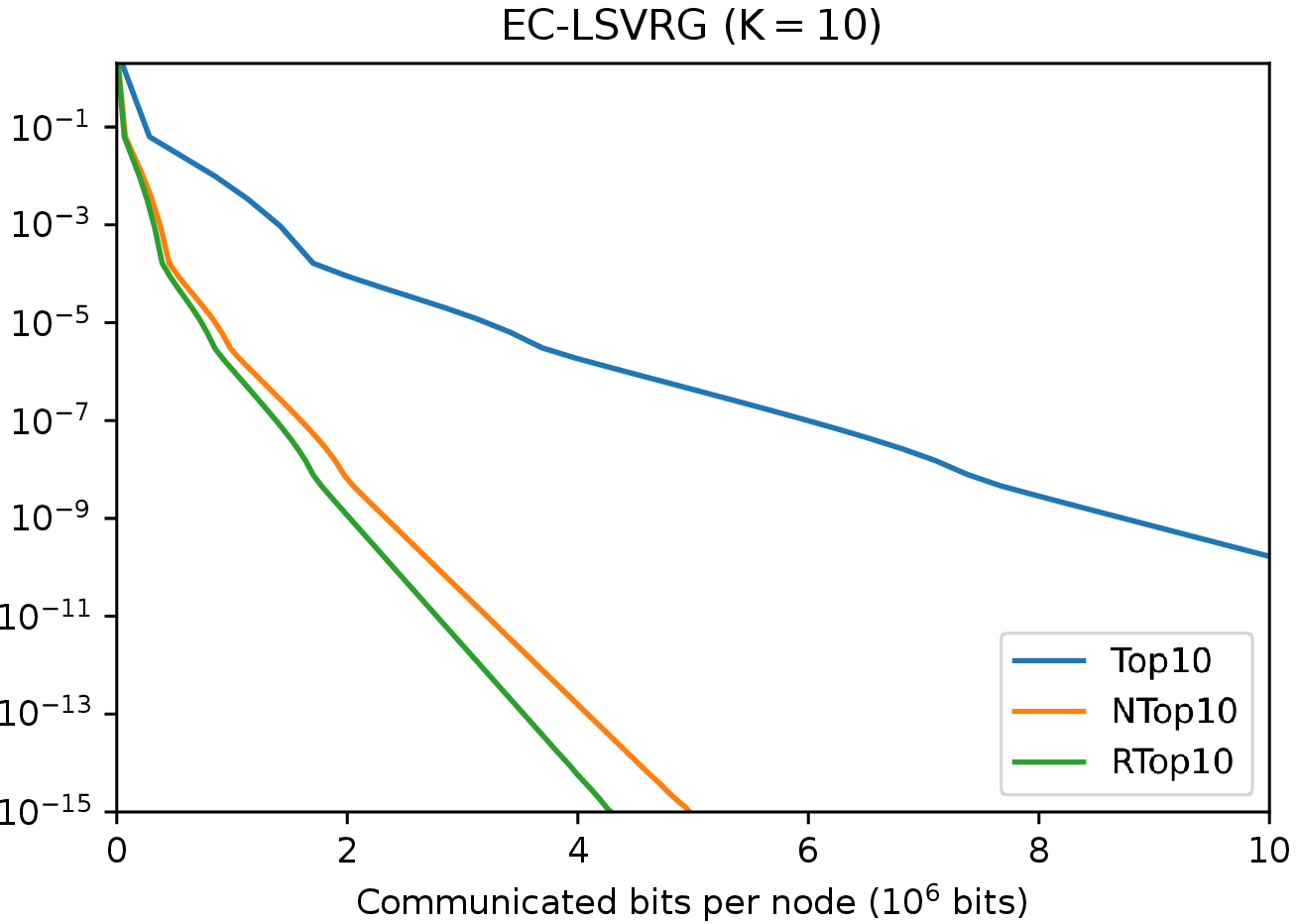}\\
		\includegraphics[width=5cm]{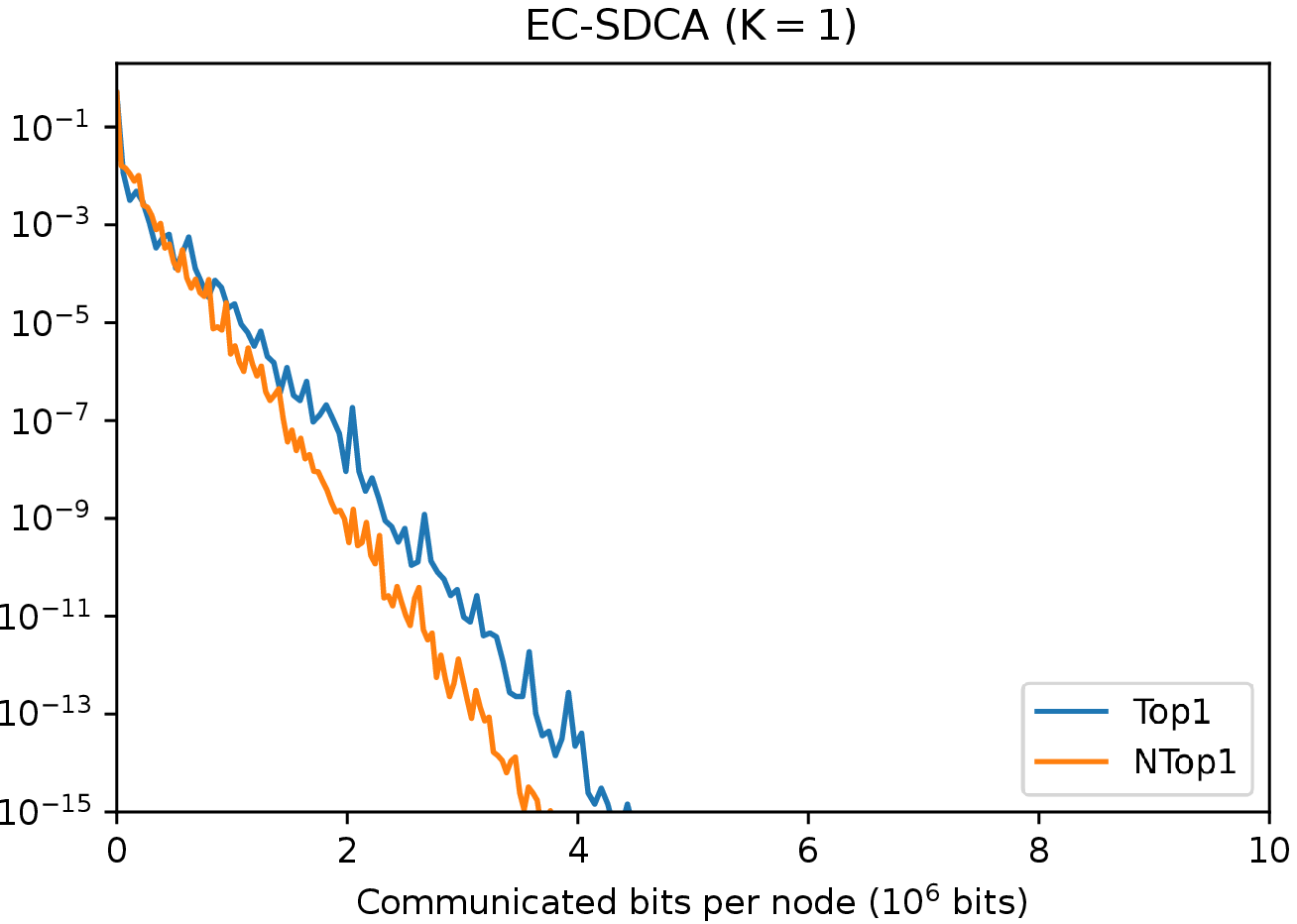}&
		\includegraphics[width=5cm]{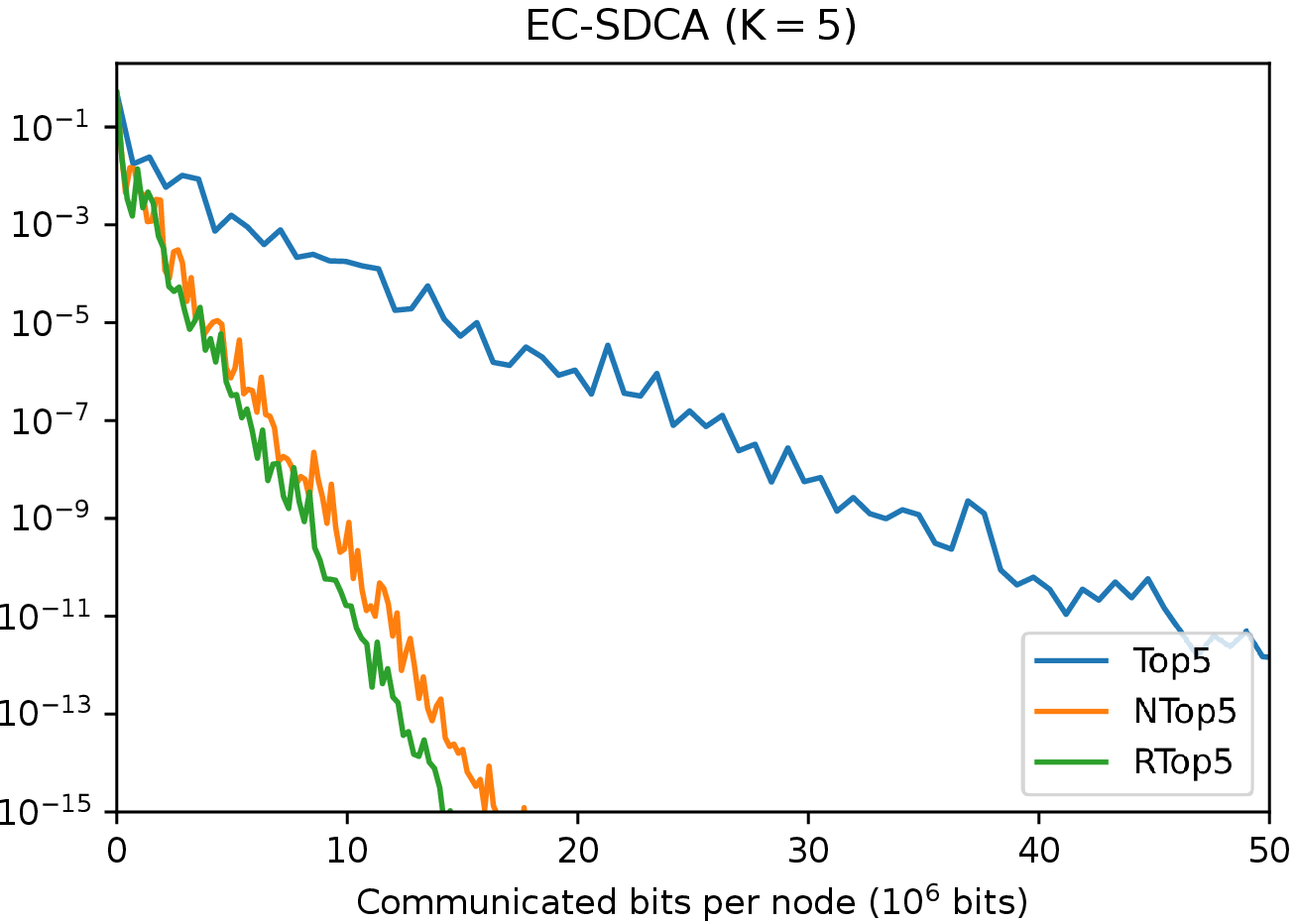}&
		\includegraphics[width=5cm]{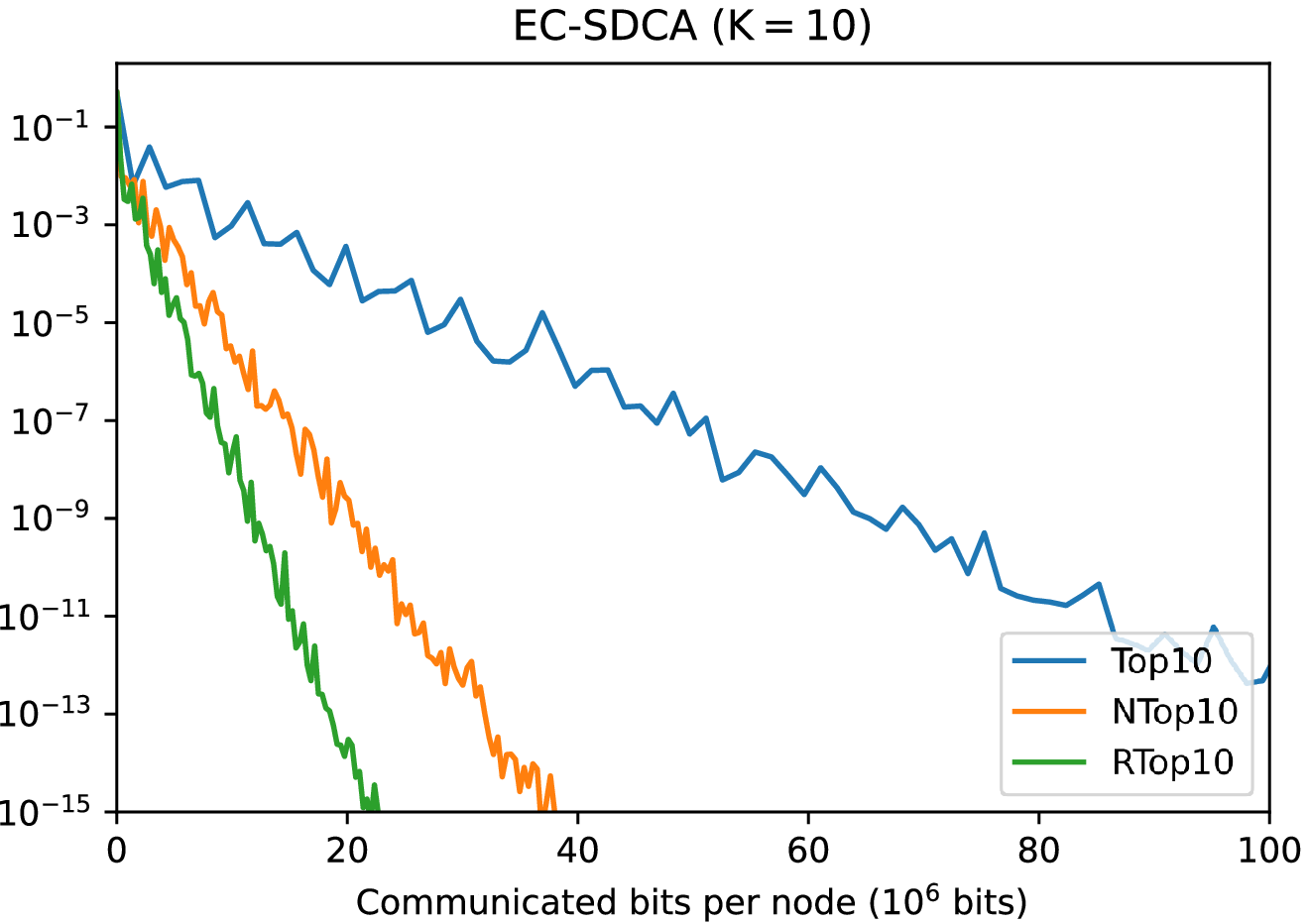}
	\end{tabular}
	\vspace{-0.25cm}
	\caption{Comparison among  TopK, NTopK, and RTopK on \textbf{a9a} }\label{fig:topk_a9a}
	\vspace{-0.25cm}
\end{figure}

\begin{figure}[H]
	\vspace{-0.25cm}
	\centering
	\begin{tabular}{ccc}
		\includegraphics[width=5cm]{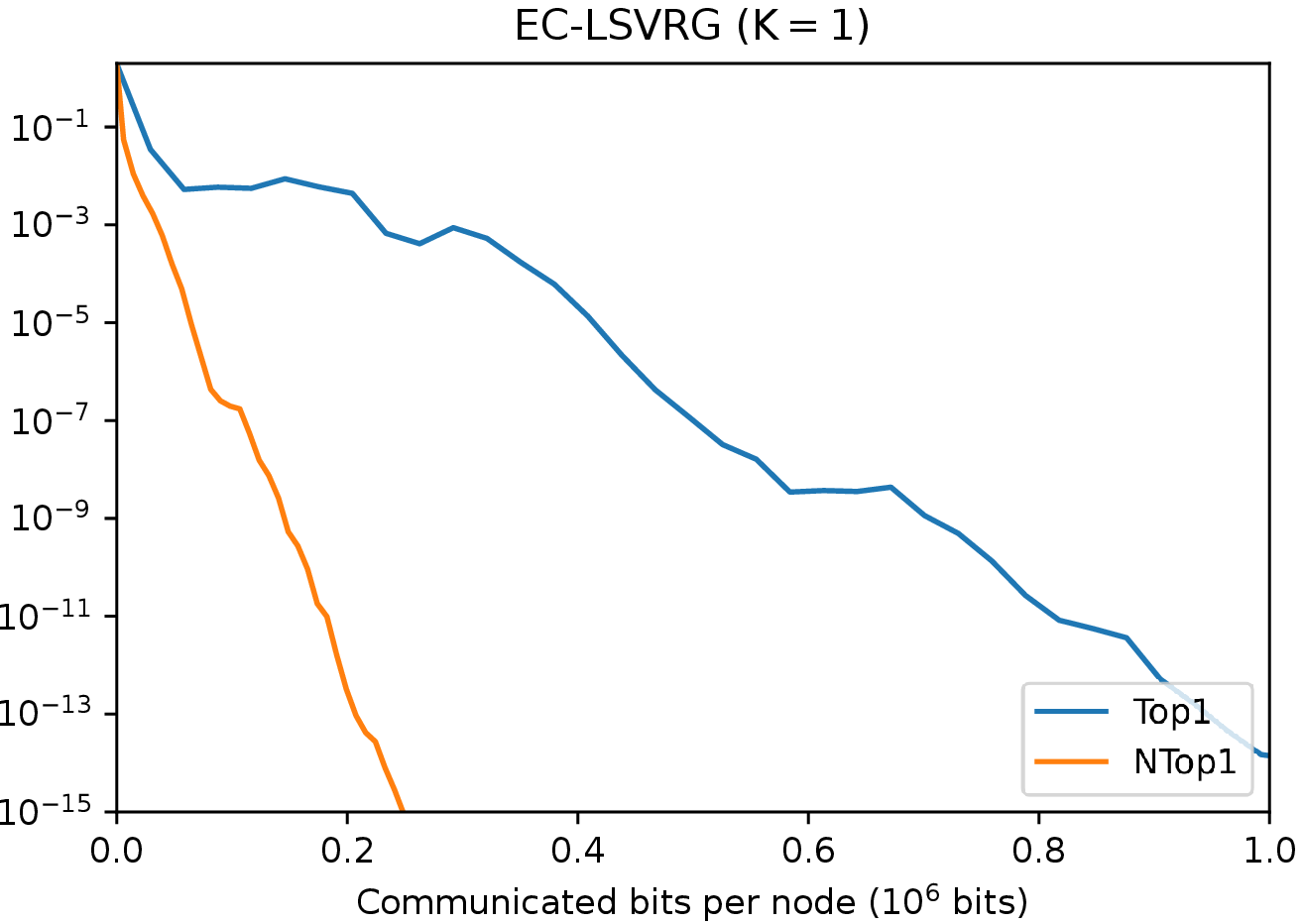}&
		\includegraphics[width=5cm]{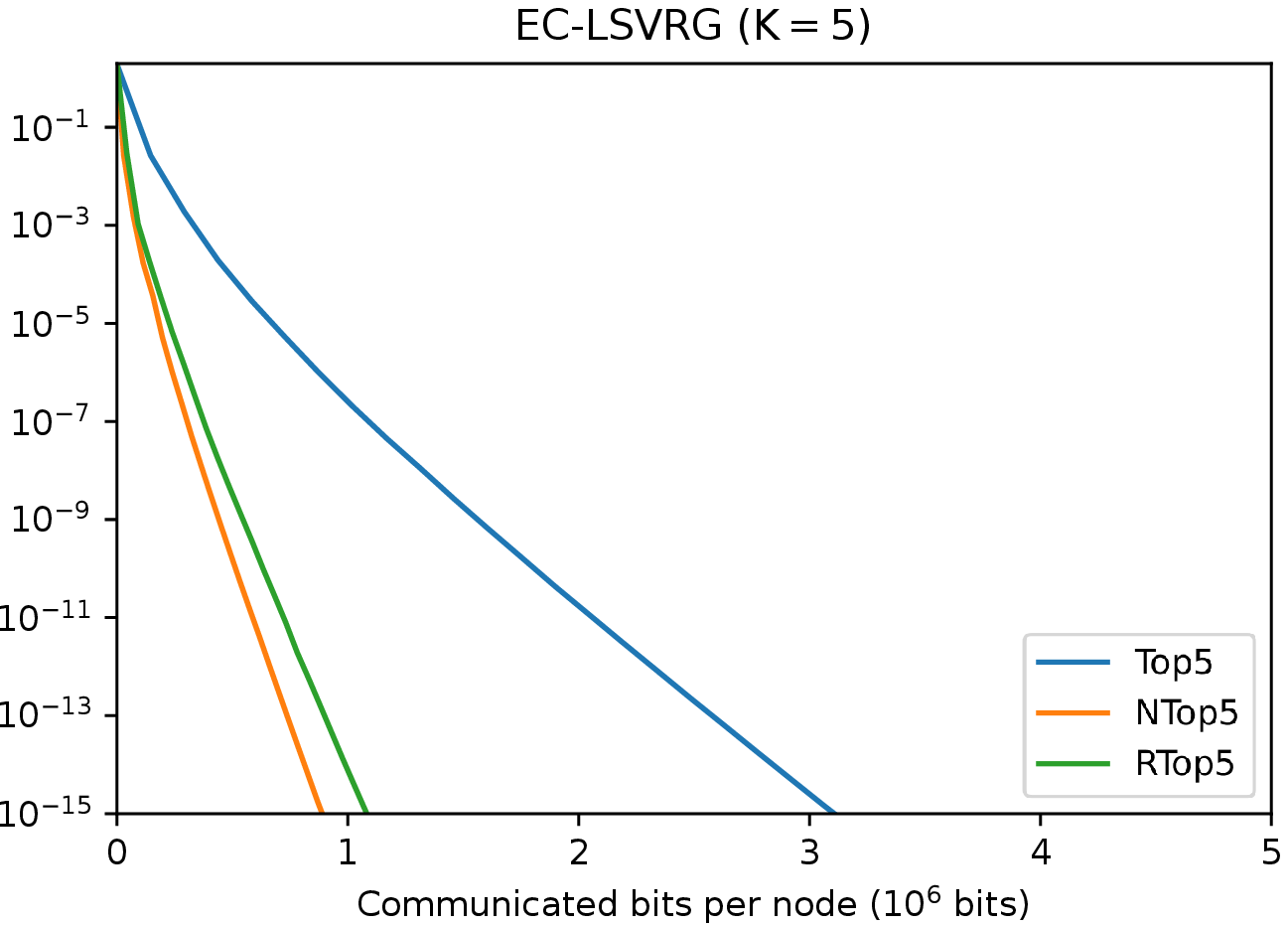}&
		\includegraphics[width=5cm]{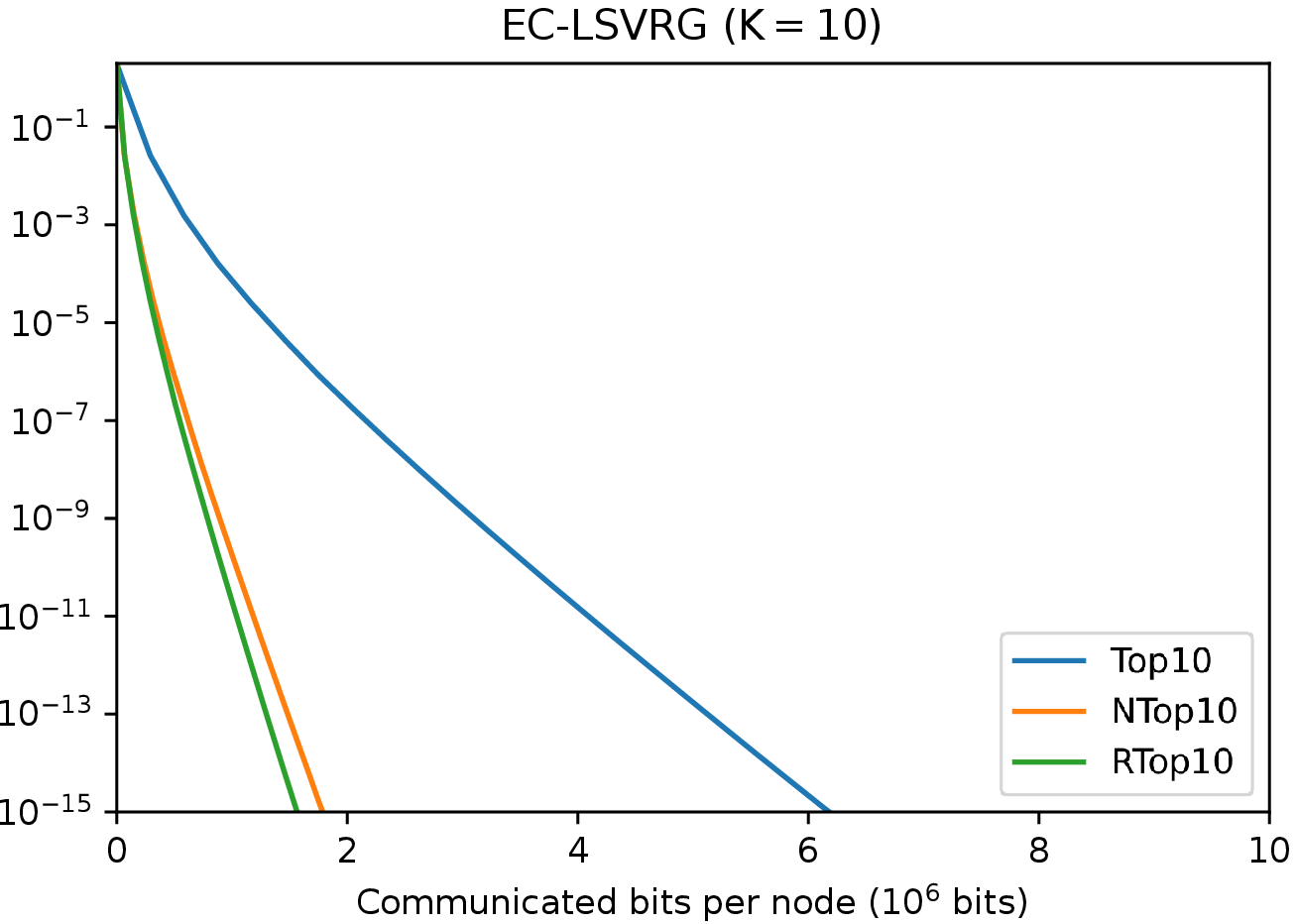}\\
		\includegraphics[width=5cm]{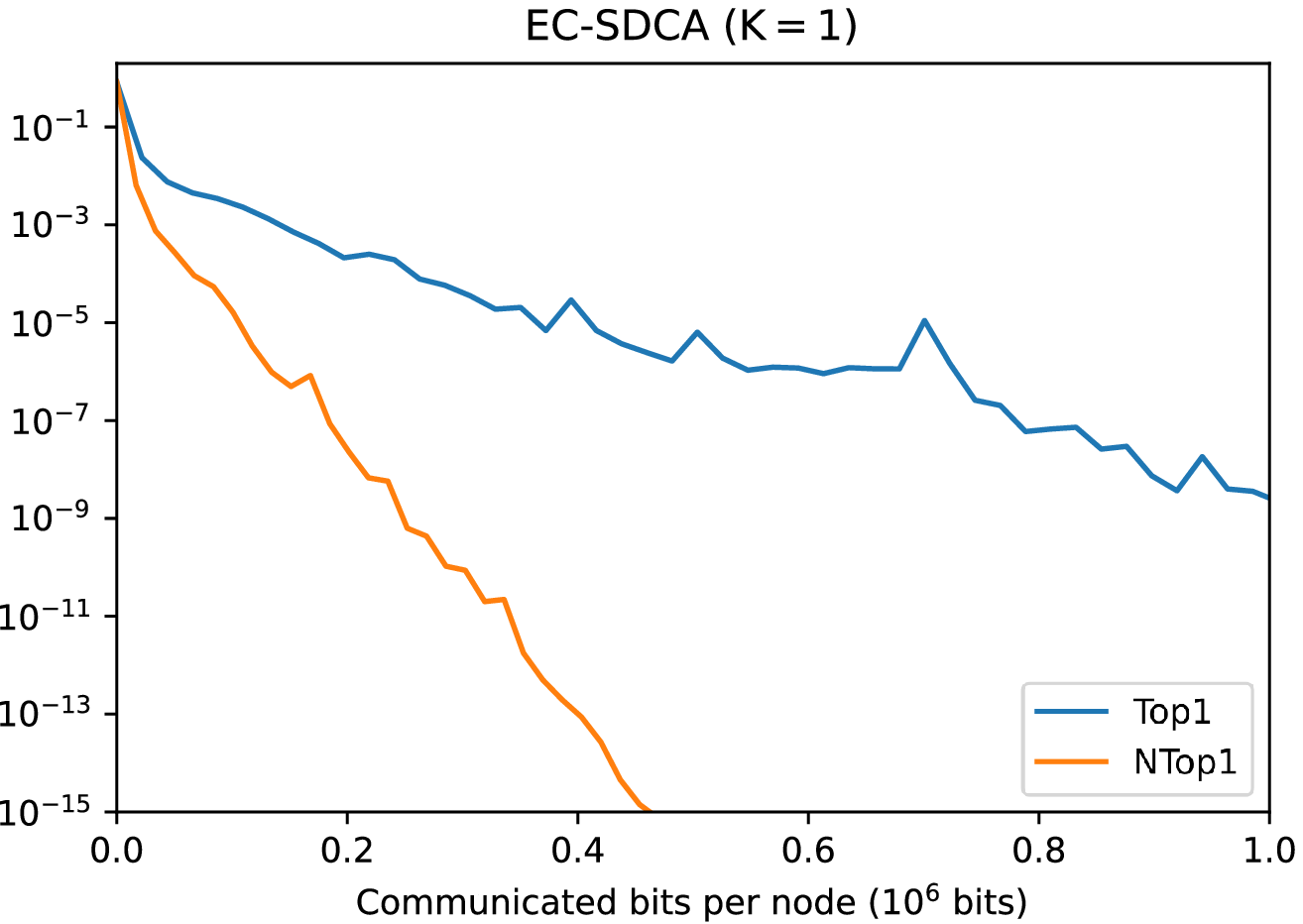}&
		\includegraphics[width=5cm]{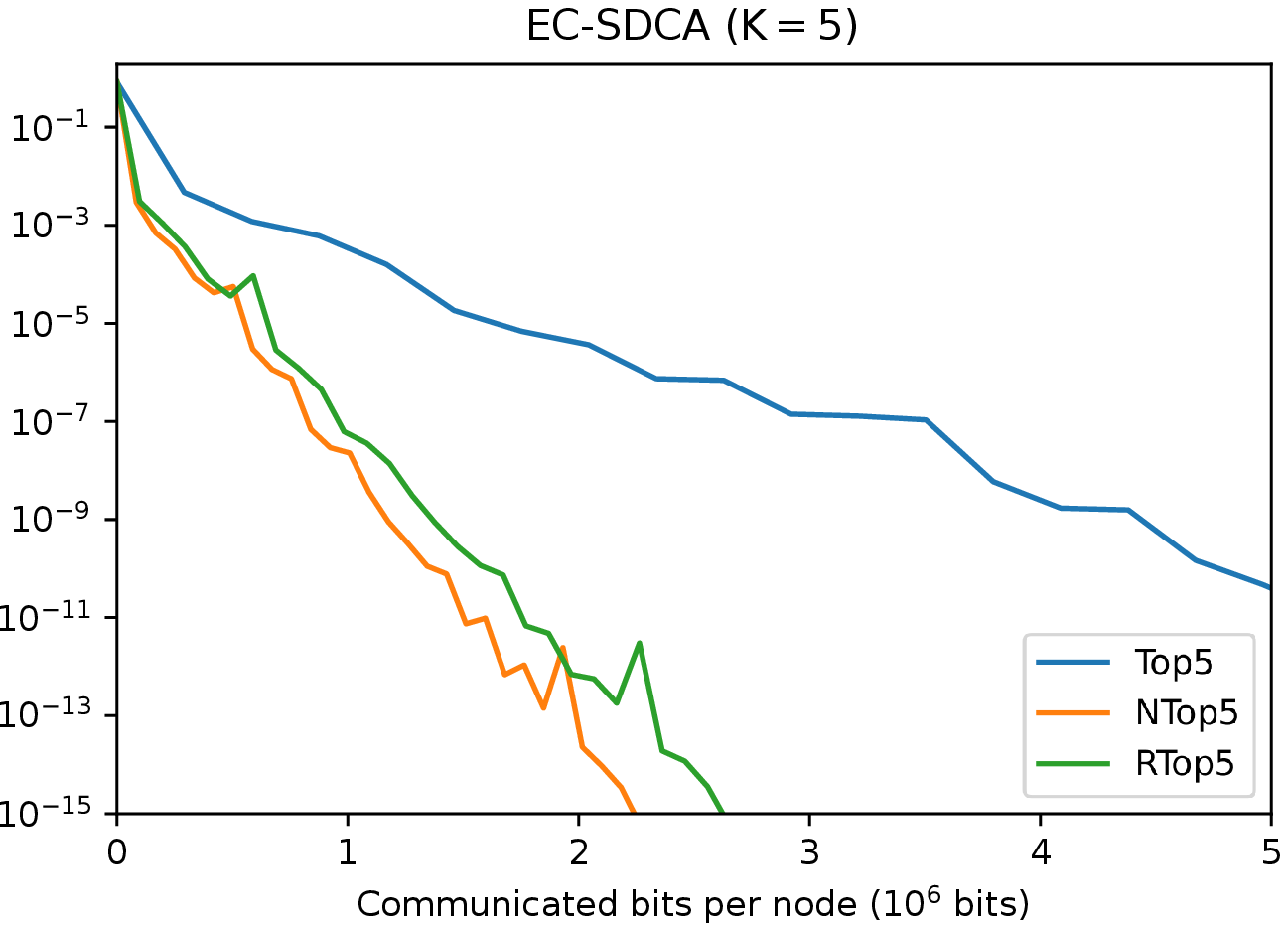}&
		\includegraphics[width=5cm]{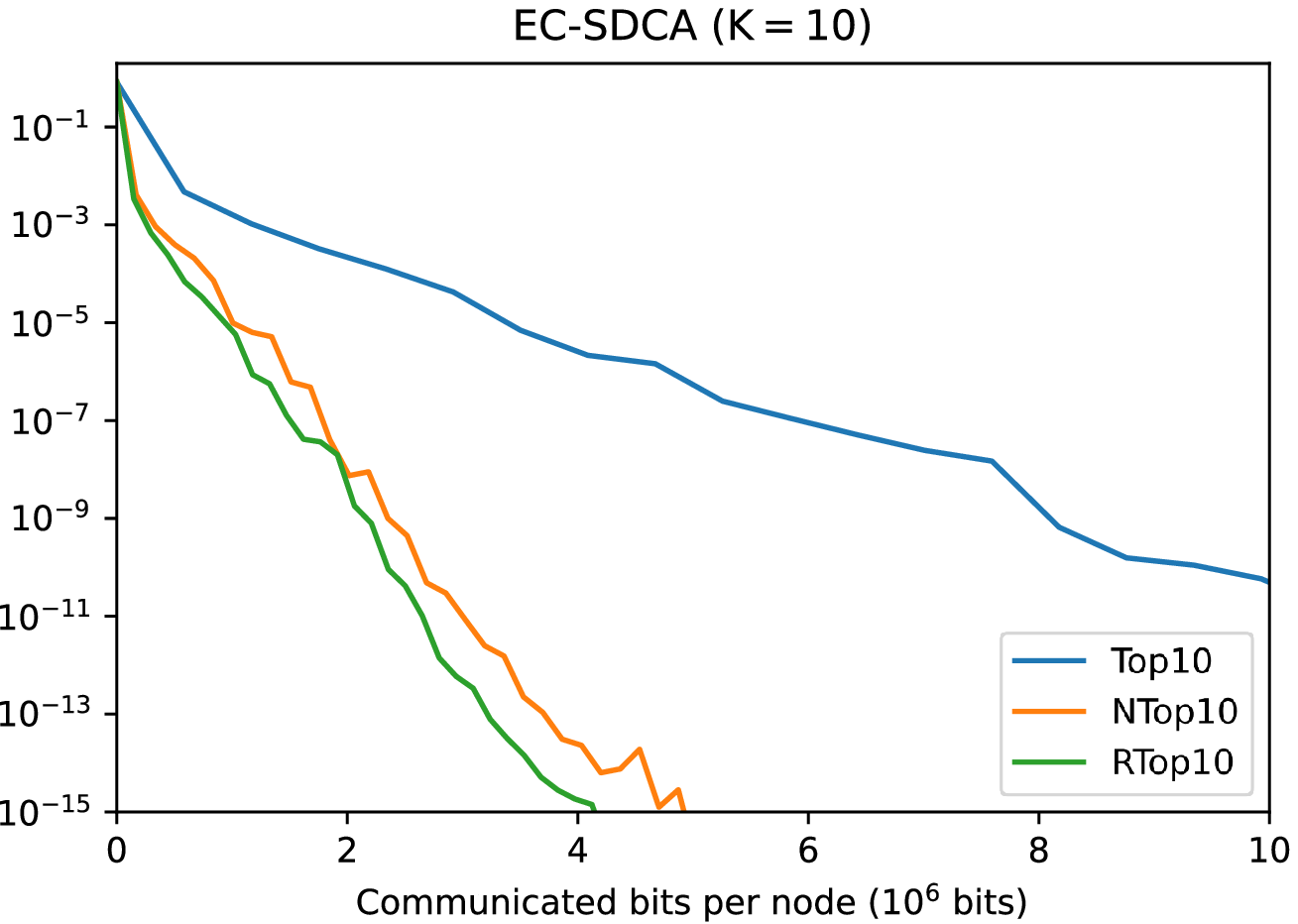}
	\end{tabular}
	\vspace{-0.25cm}
	\caption{Comparison among  TopK, NTopK, and RTopK on \textbf{w6a} }\label{fig:topk_w6a}
	\vspace{-0.25cm}
\end{figure}

\subsection{Impact of the update frequency parameter $p$}

Our default setting in EC-LSVRG is $p=\delta$. In this part, we investigate the impact of the update frequency $p$. From the theoretical results, it is easy to verify that an optimal choice of $p$ is $\Theta(\delta)$, and too large or small $p$ may lead to a slower convergence. By Figure \ref{fig:impact}, when $p=\delta/3$, the convergence is usually much slower (\textbf{mushrooms}, \textbf{w6a}). When $p=1$, the performance is no better than $p = \delta$, generally. In particular, large $p$ makes convergence slower on  \textbf{a5a} and \textbf{a9a}.

\begin{figure}[H]
	\vspace{-0.35cm}
	\centering
	\begin{tabular}{cccc}
		\includegraphics[width=3.8cm]{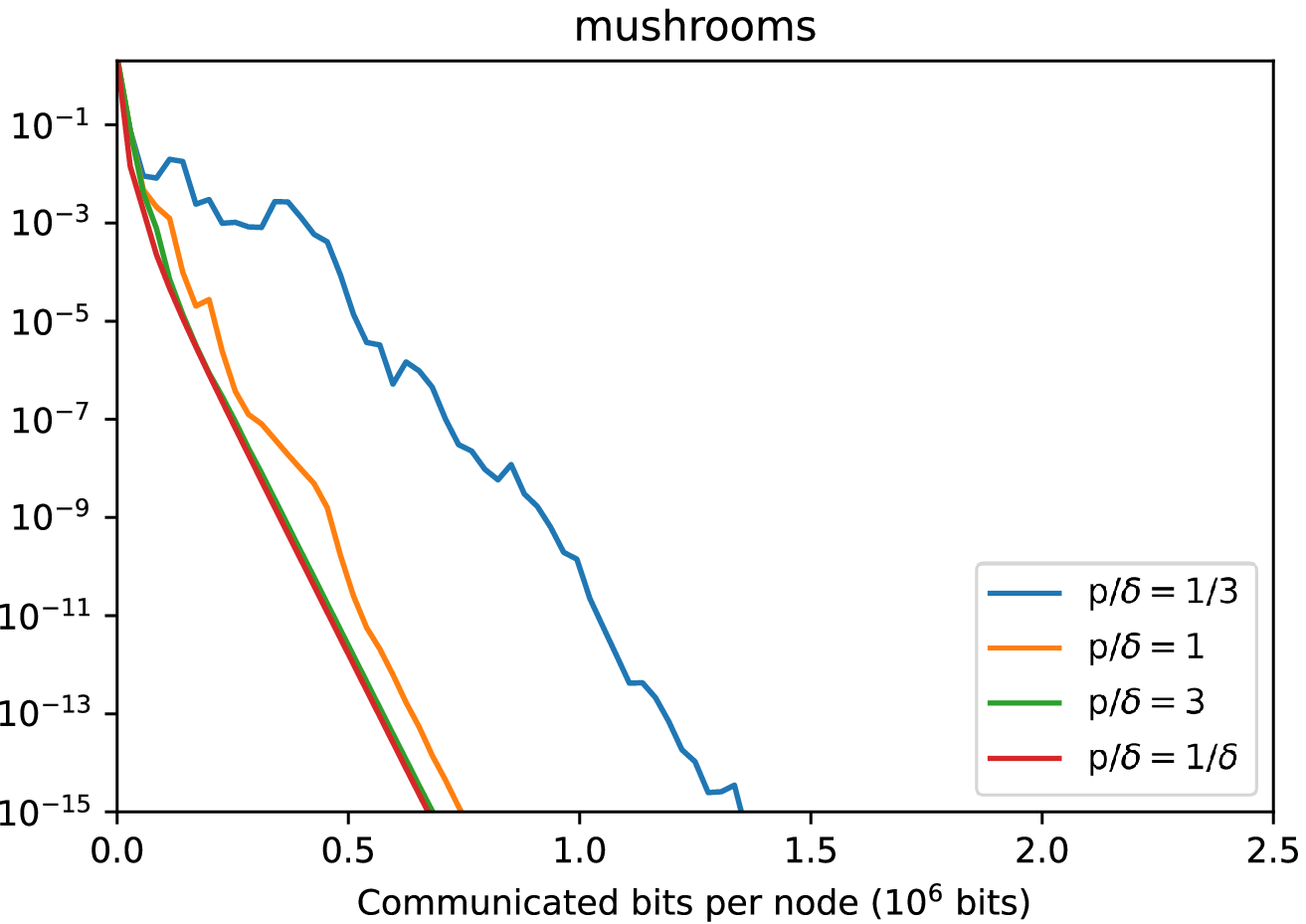}&
		\includegraphics[width=3.8cm]{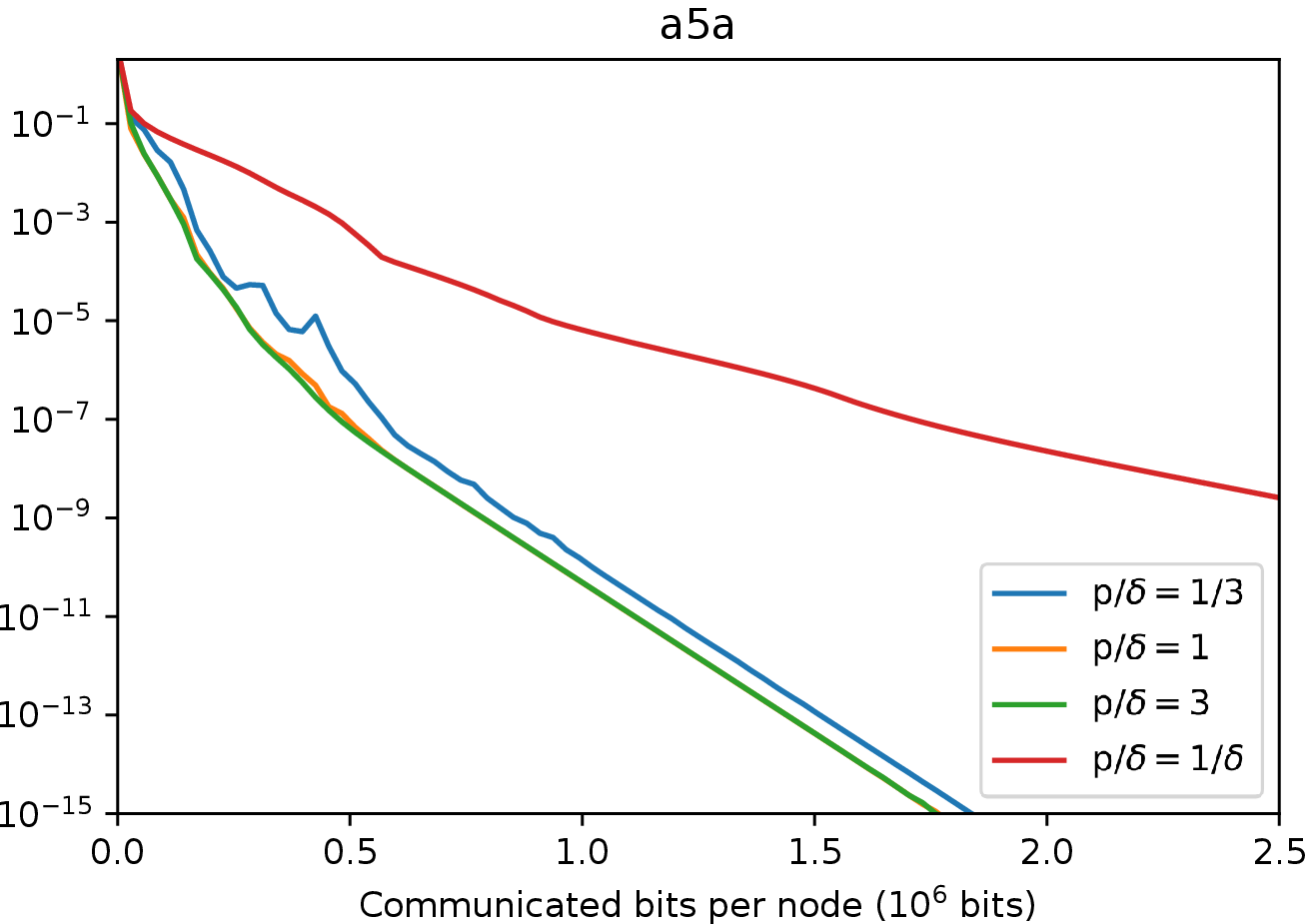}&
				\includegraphics[width=3.8cm]{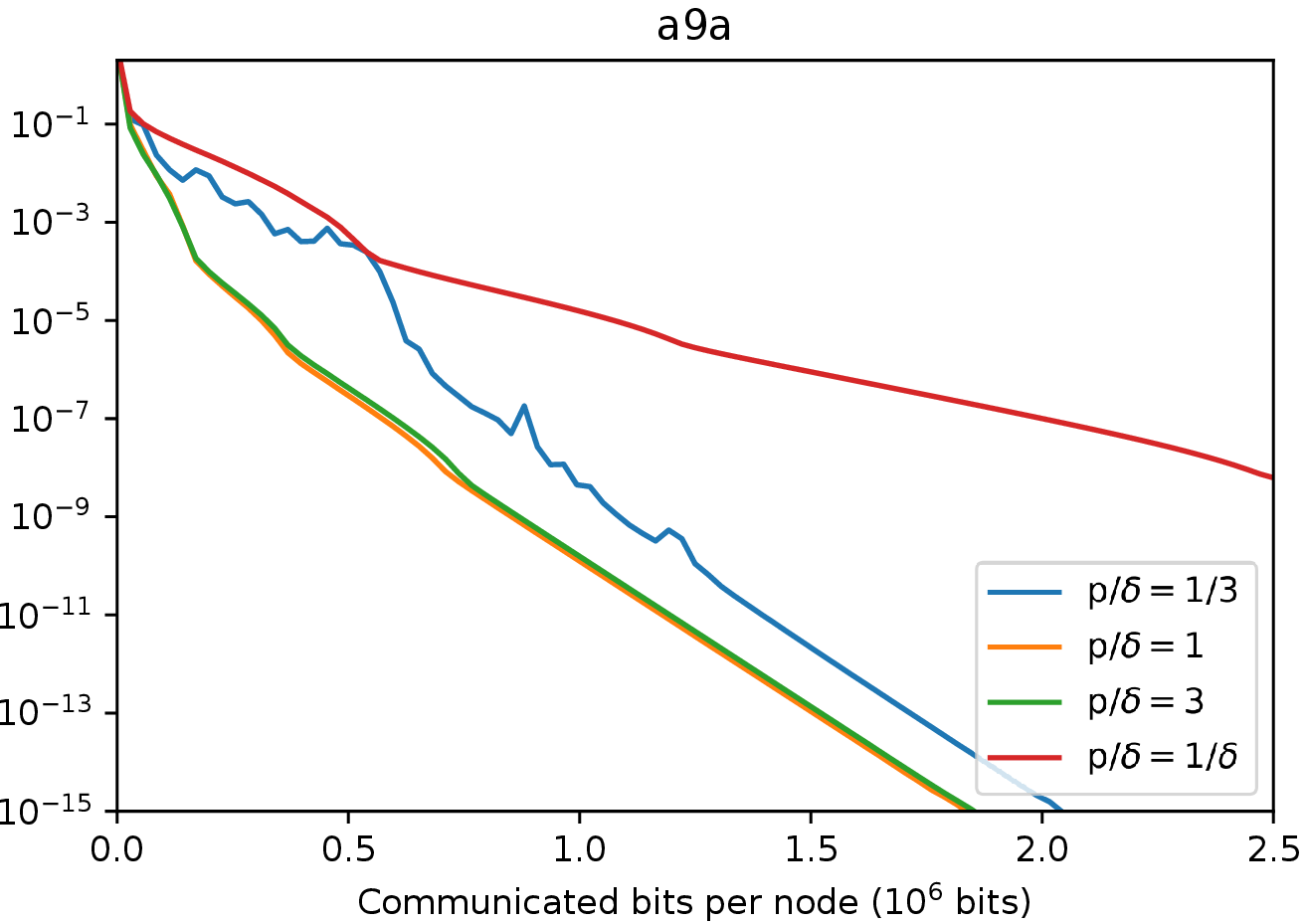}&
		\includegraphics[width=3.8cm]{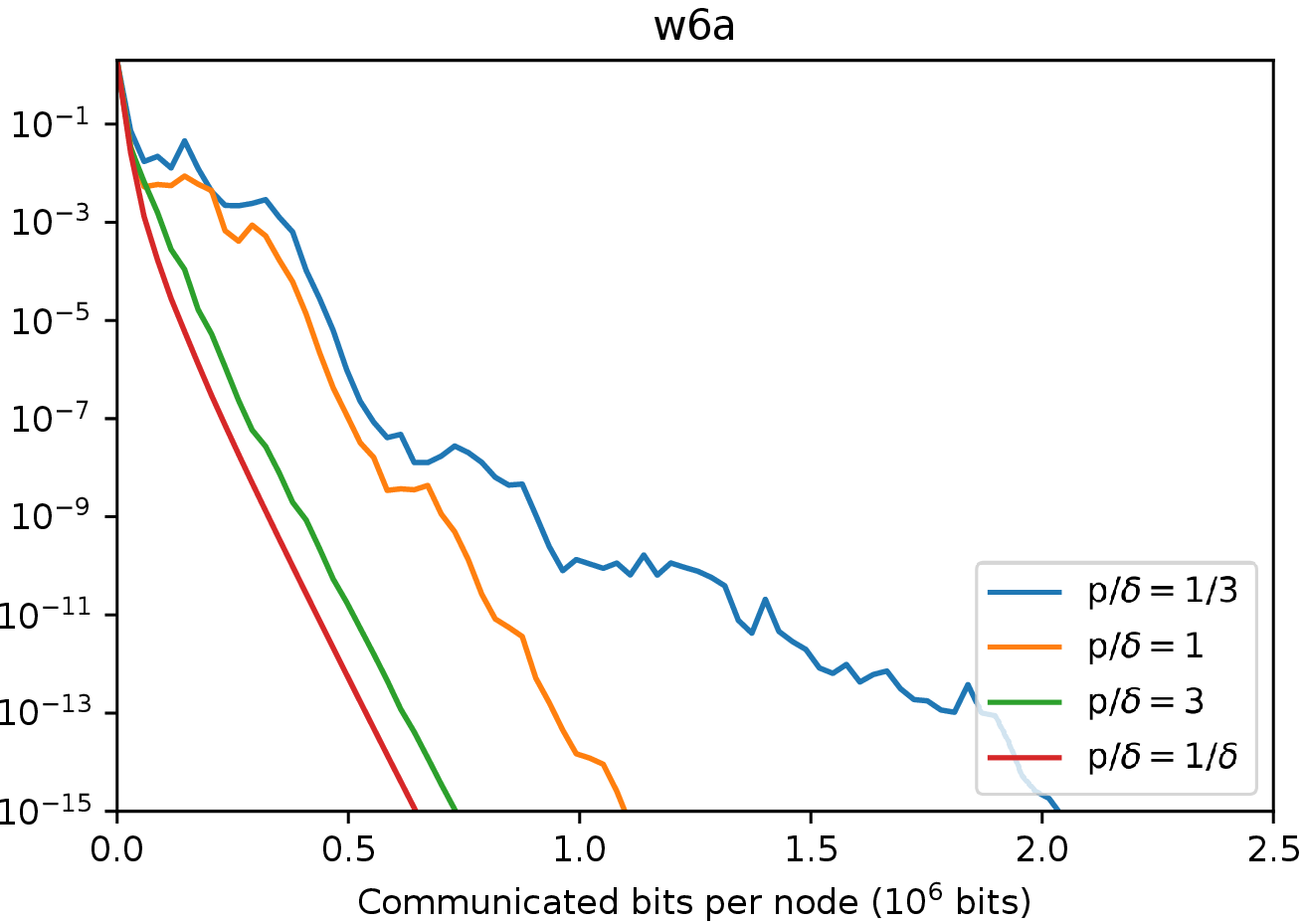}
	\end{tabular}
	\vspace{-0.35cm}
	\caption{Impact of the update frequency $p$}\label{fig:impact}
	\vspace{-0.35cm}
\end{figure}

\subsection{EC-SDCA vs EC-Quartz}

Although in theory, EC-SDCA and EC-Quartz have the same iteration complexities, they perform differently in actual data. Figures \ref{fig:quartz} $-$ \ref{fig:quartz_w6a} show that EC-SDCA is usually comparable to EC-Quartz or better than EC-Quartz, and sometimes much better, especially for Top1 compressor. Thus, we prefer EC-SDCA for more general scenarios.

\begin{figure}[H]

	\centering
	\begin{tabular}{cccc}
		\includegraphics[width=3.8cm]{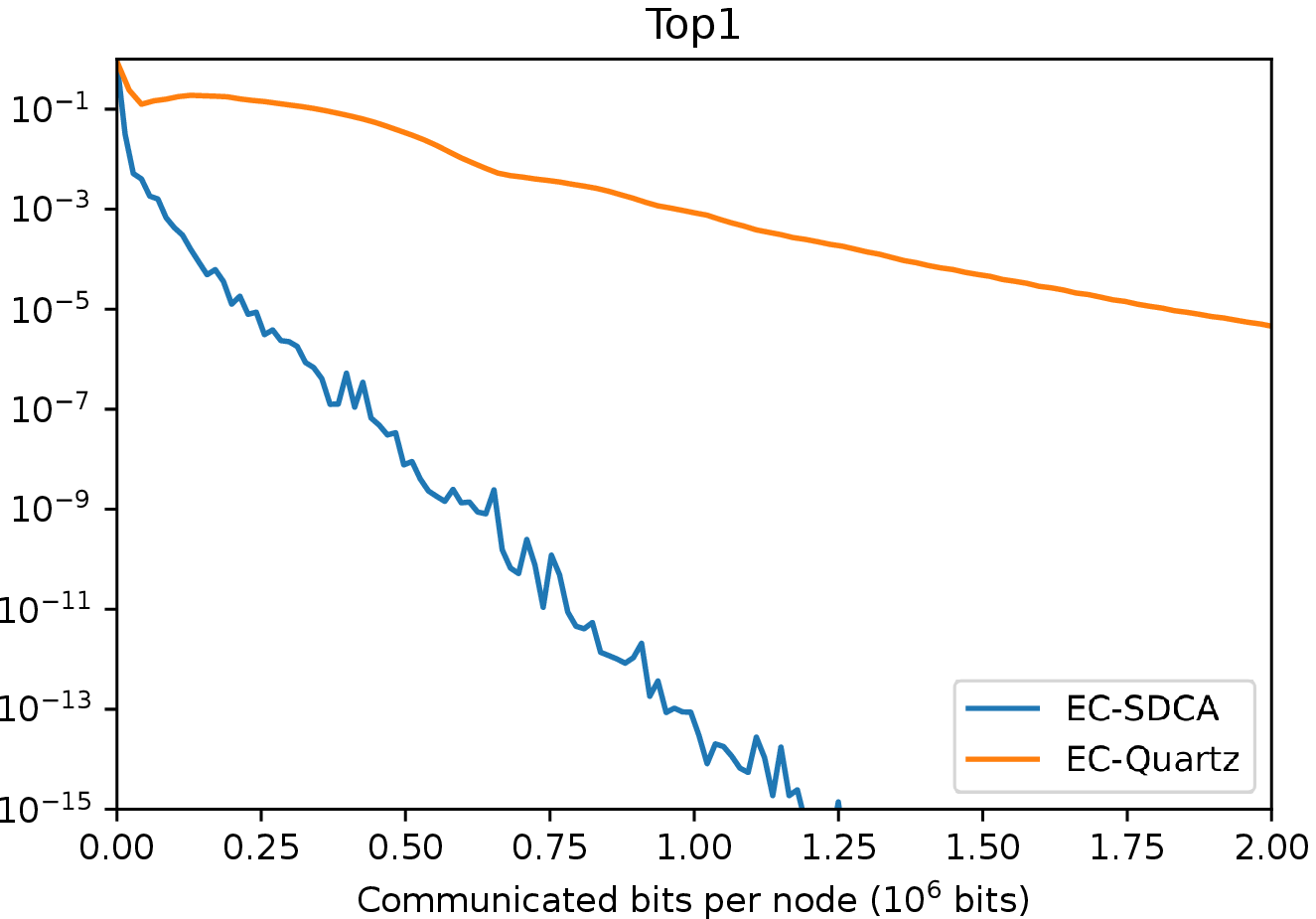}&
		\includegraphics[width=3.8cm]{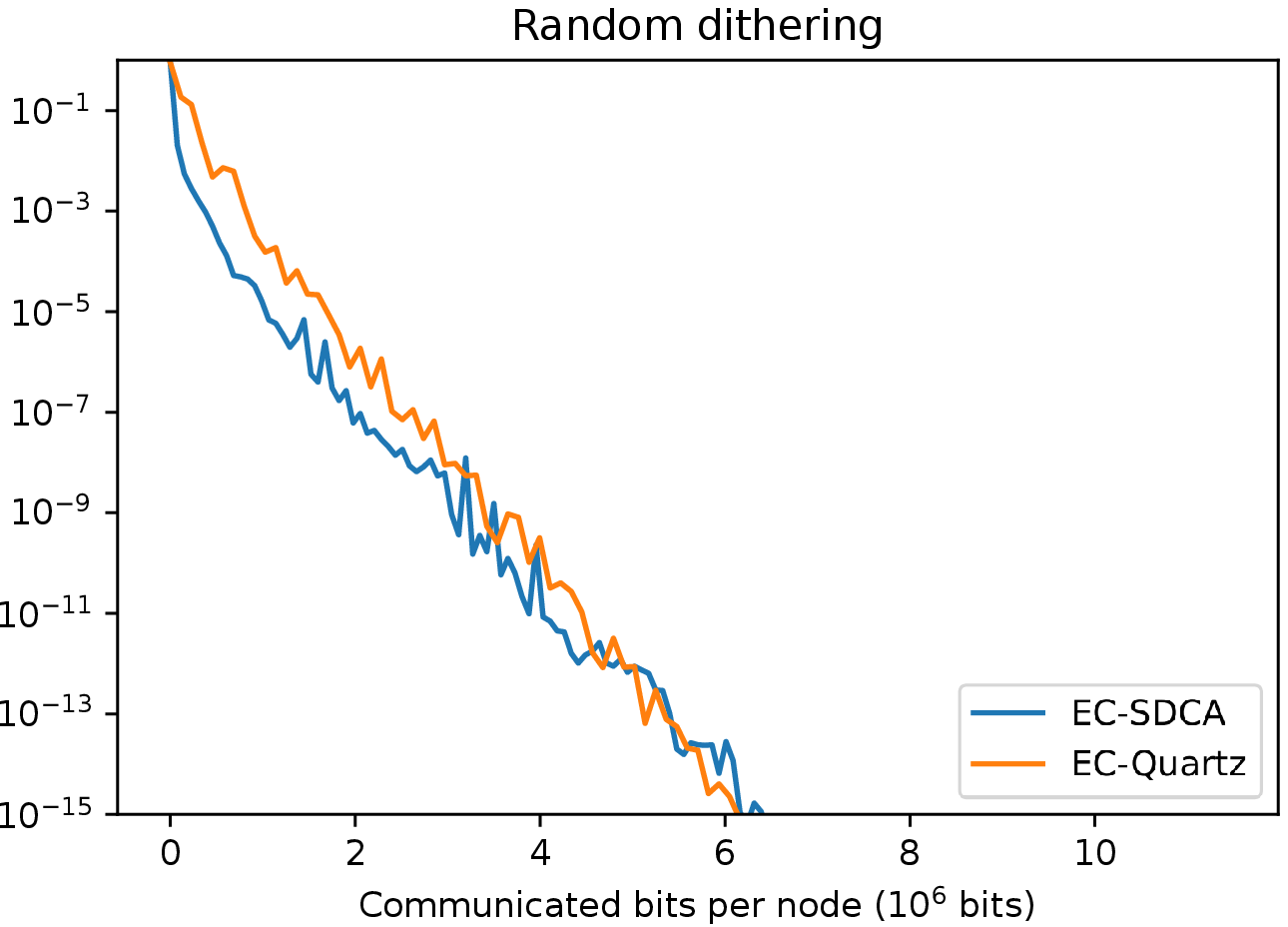}&
		\includegraphics[width=3.8cm]{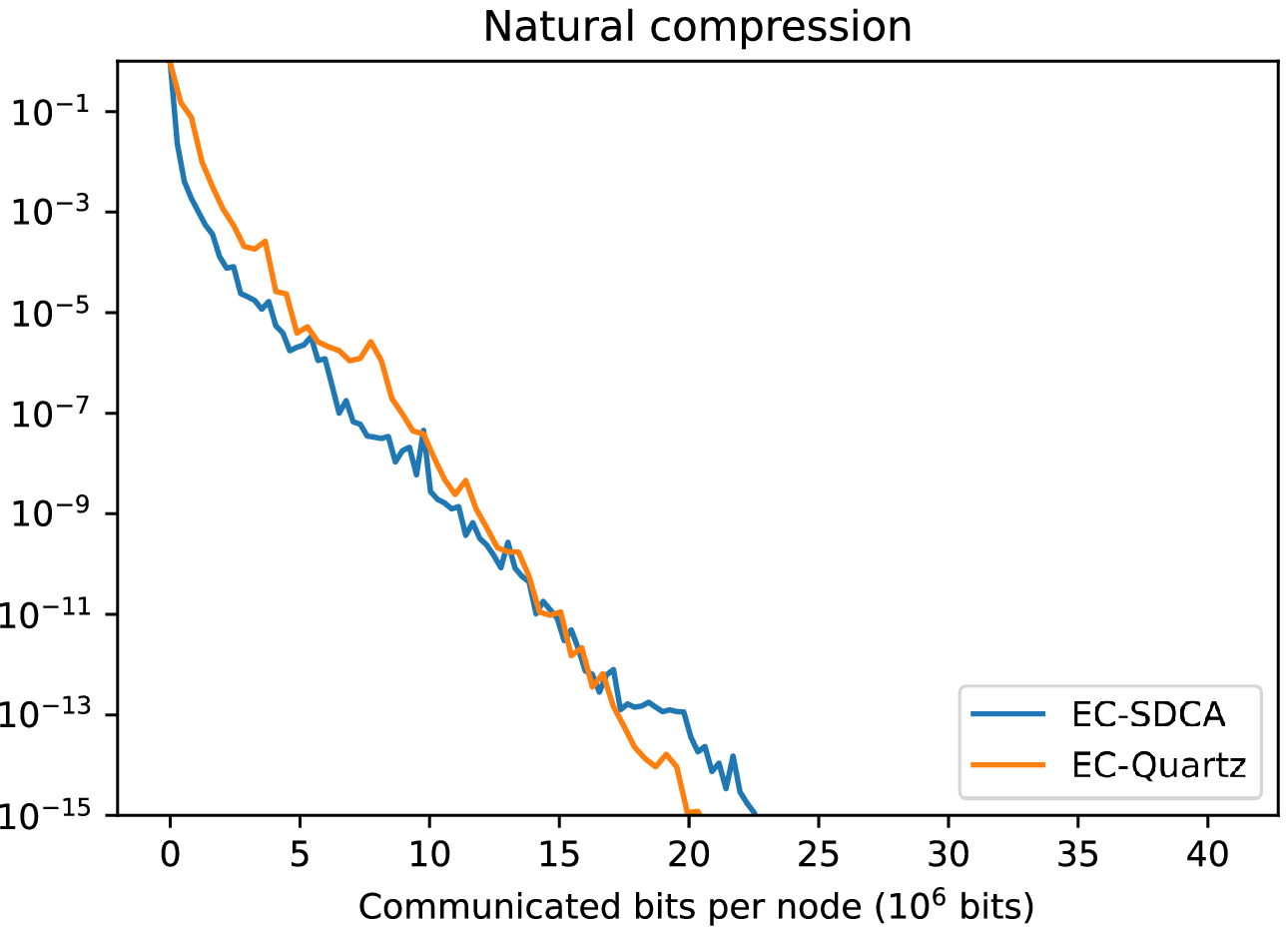}&
		\includegraphics[width=3.8cm]{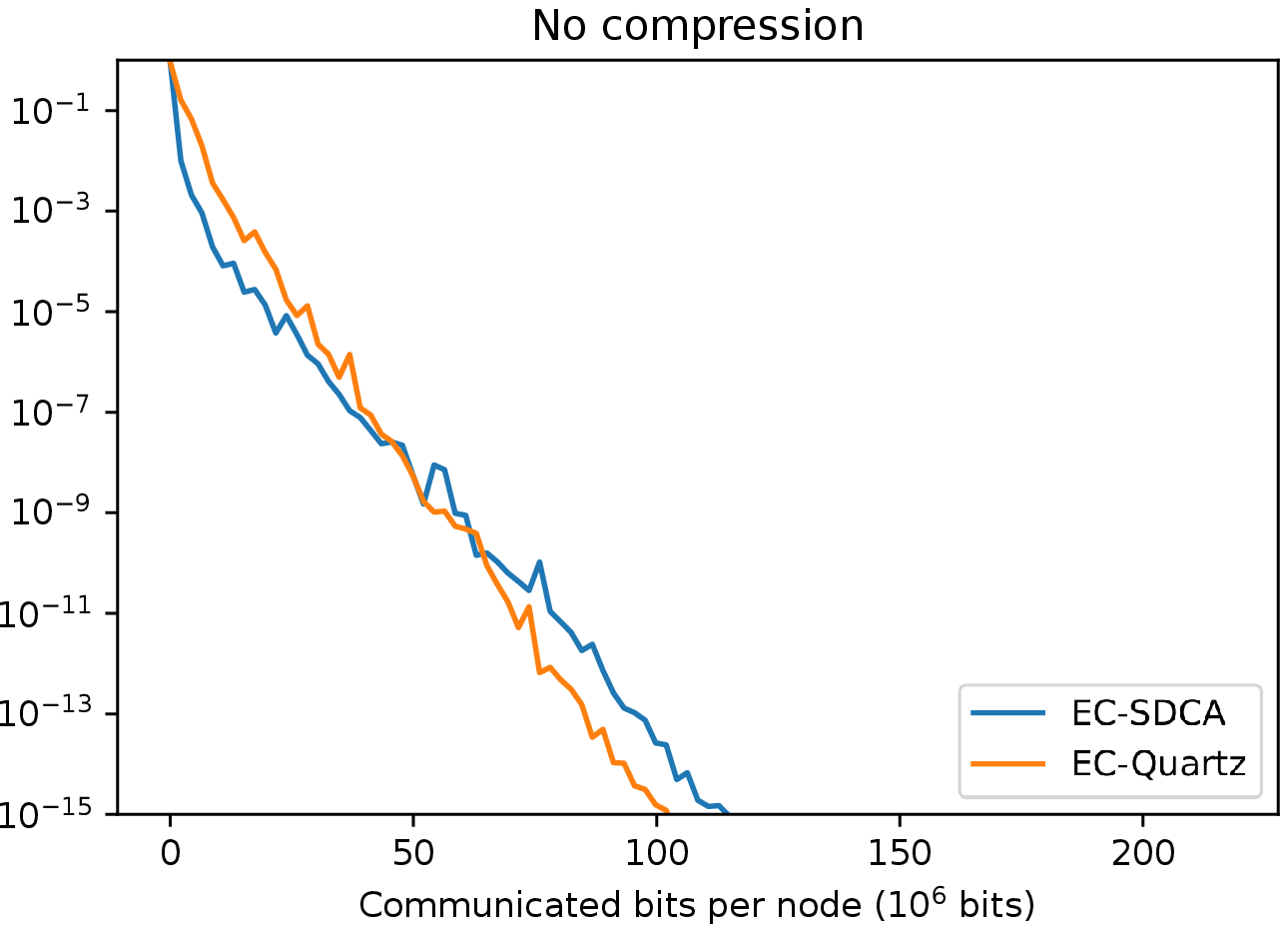}
	\end{tabular}

	\caption{EC-SDCA vs EC-Quartz on  \textbf{mushrooms}}\label{fig:quartz}

\end{figure}

\begin{figure}[H]

	\centering
	\begin{tabular}{cccc}
		\includegraphics[width=3.8cm]{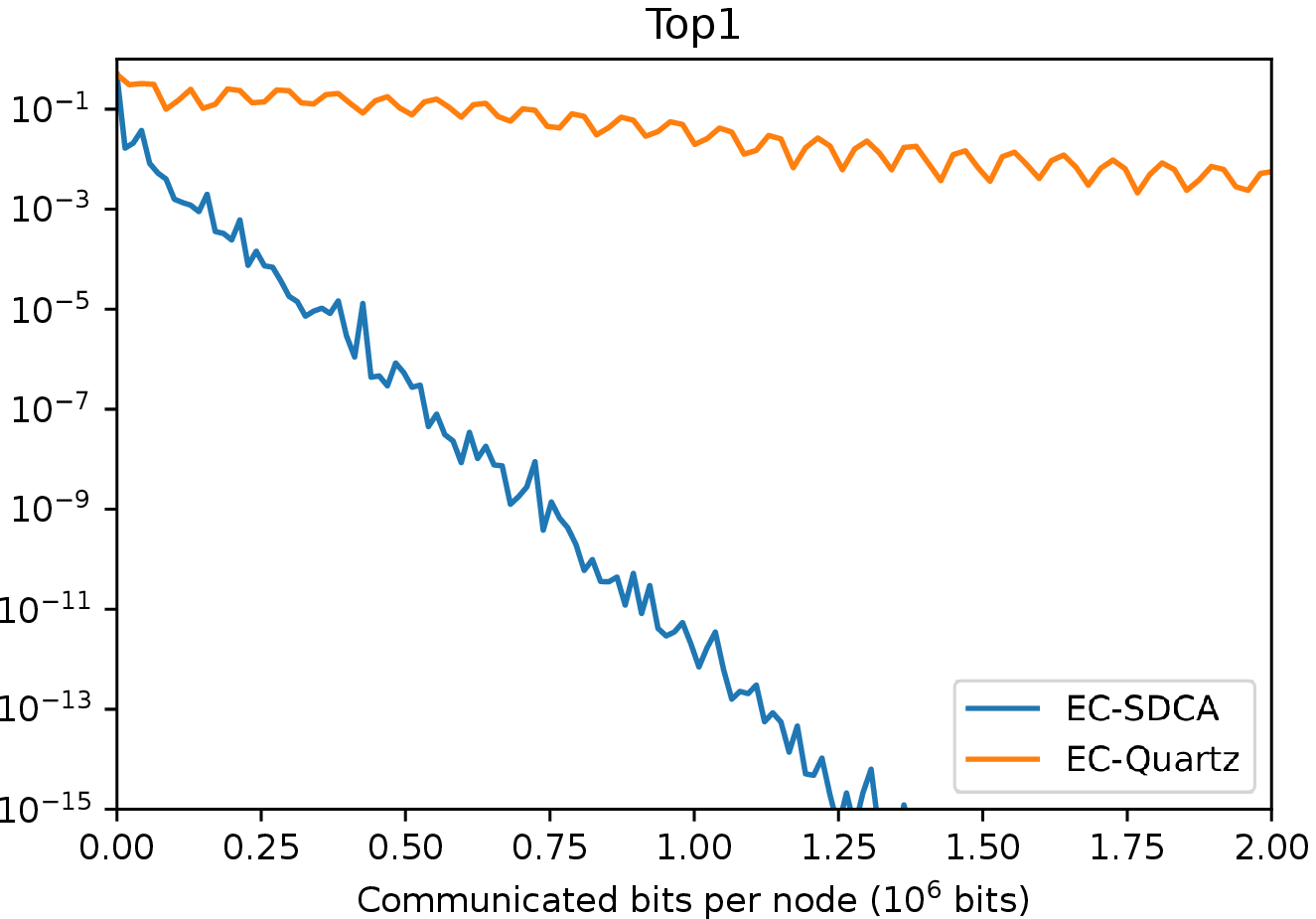}&
		\includegraphics[width=3.8cm]{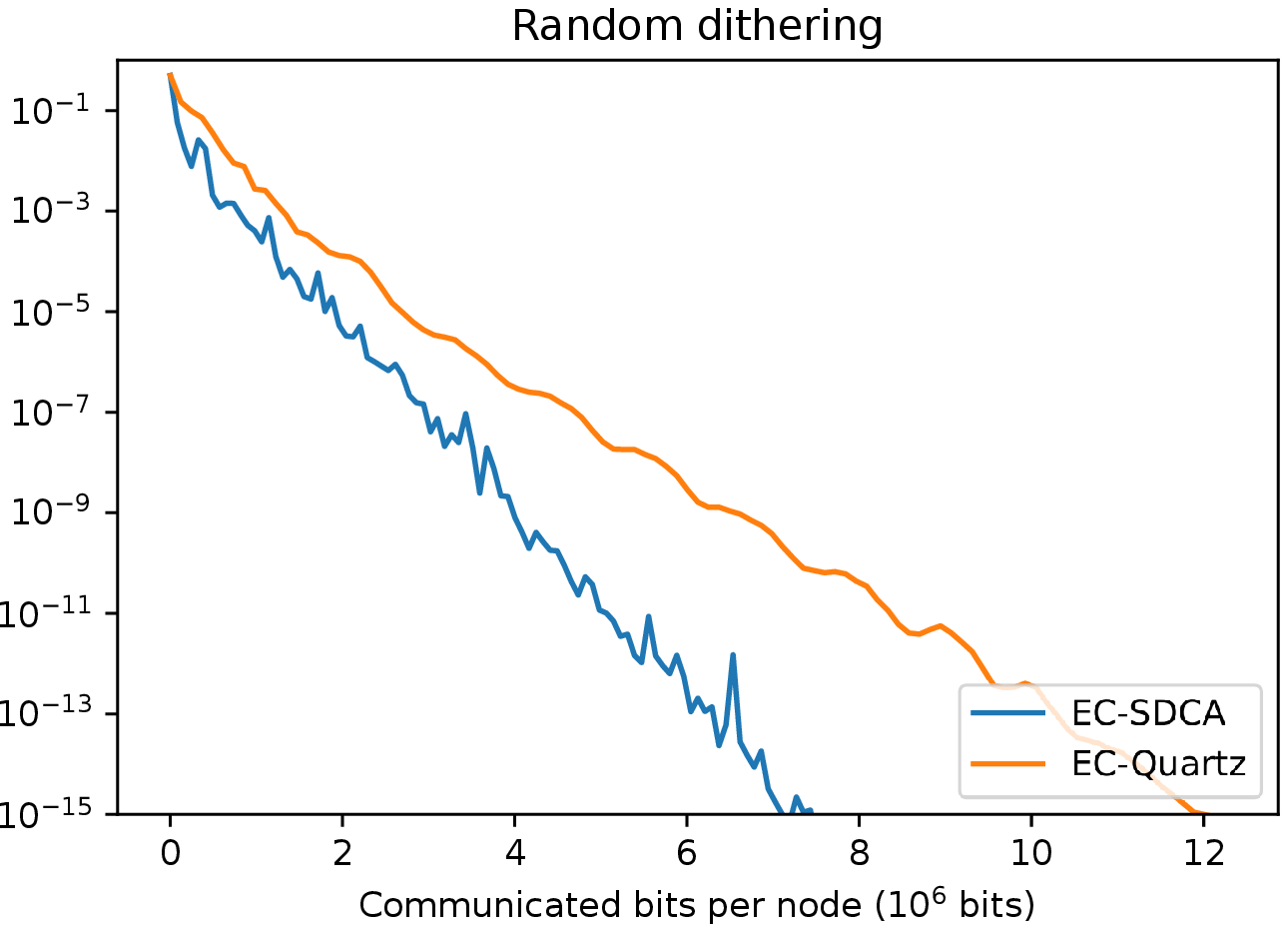}&
		\includegraphics[width=3.8cm]{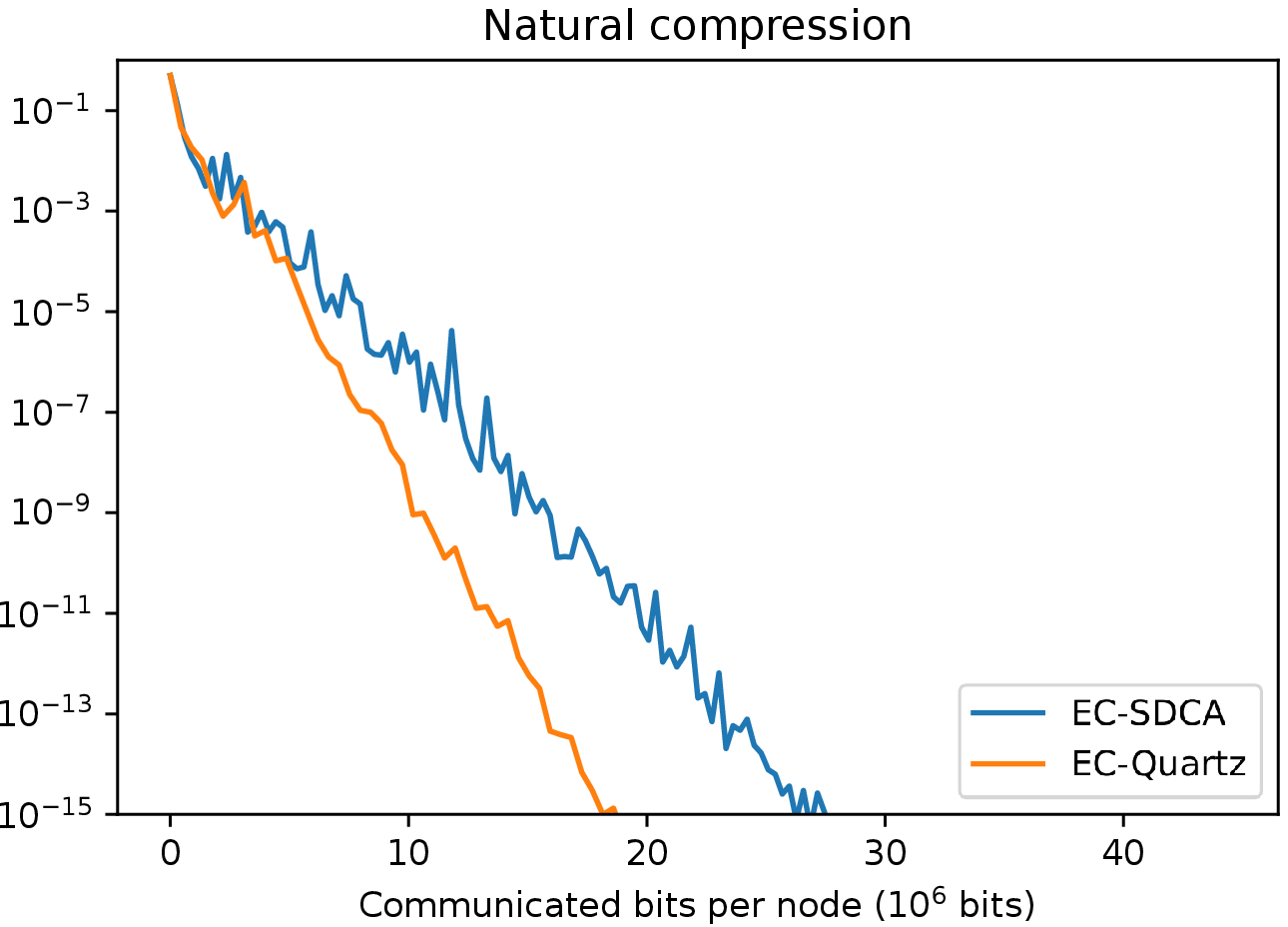}&
		\includegraphics[width=3.8cm]{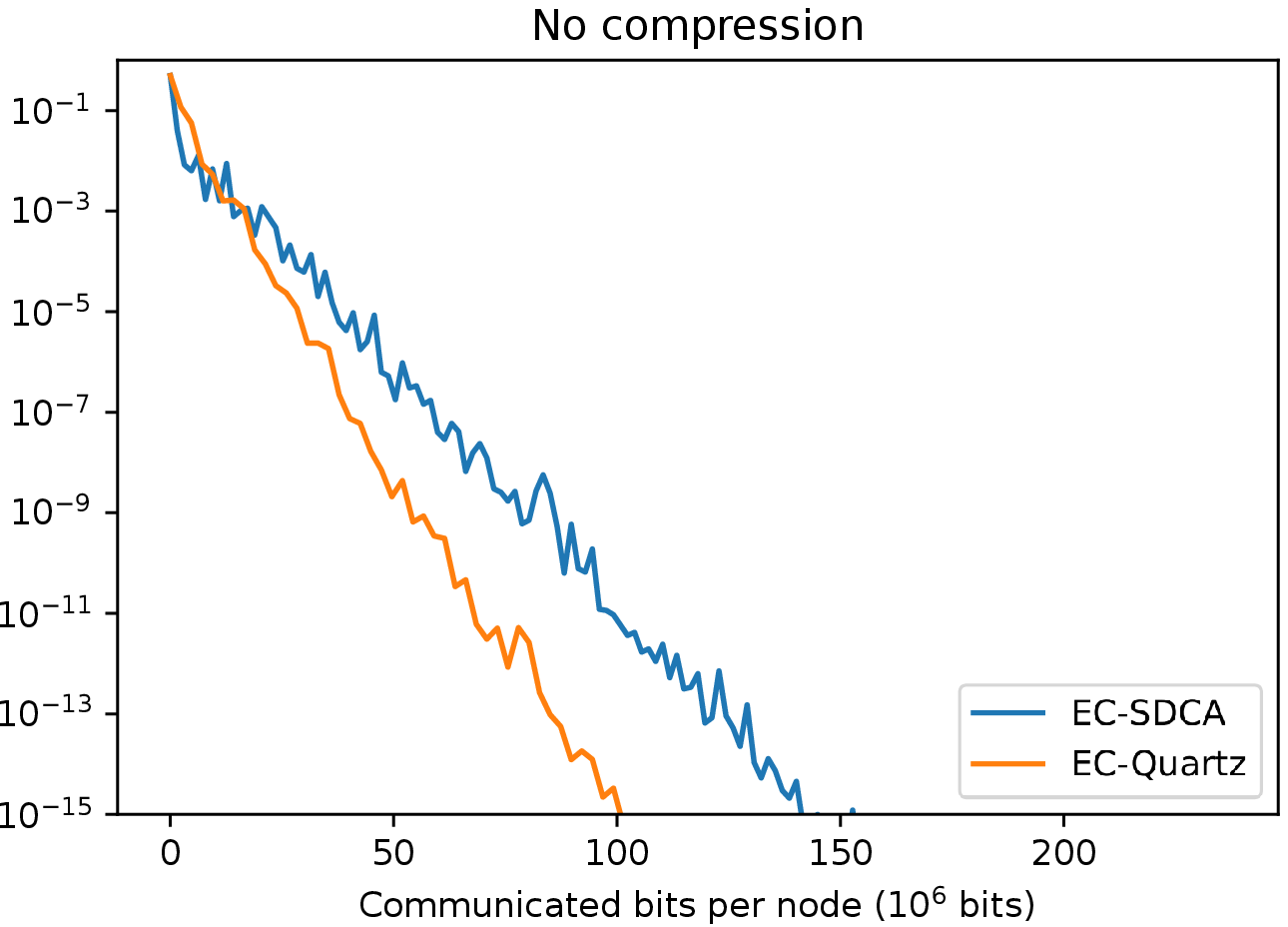}
	\end{tabular}

	\caption{EC-SDCA vs EC-Quartz on  \textbf{a5a}}\label{fig:quartz_a5a}

\end{figure}

\begin{figure}[H]

	\centering
	\begin{tabular}{cccc}
		\includegraphics[width=3.8cm]{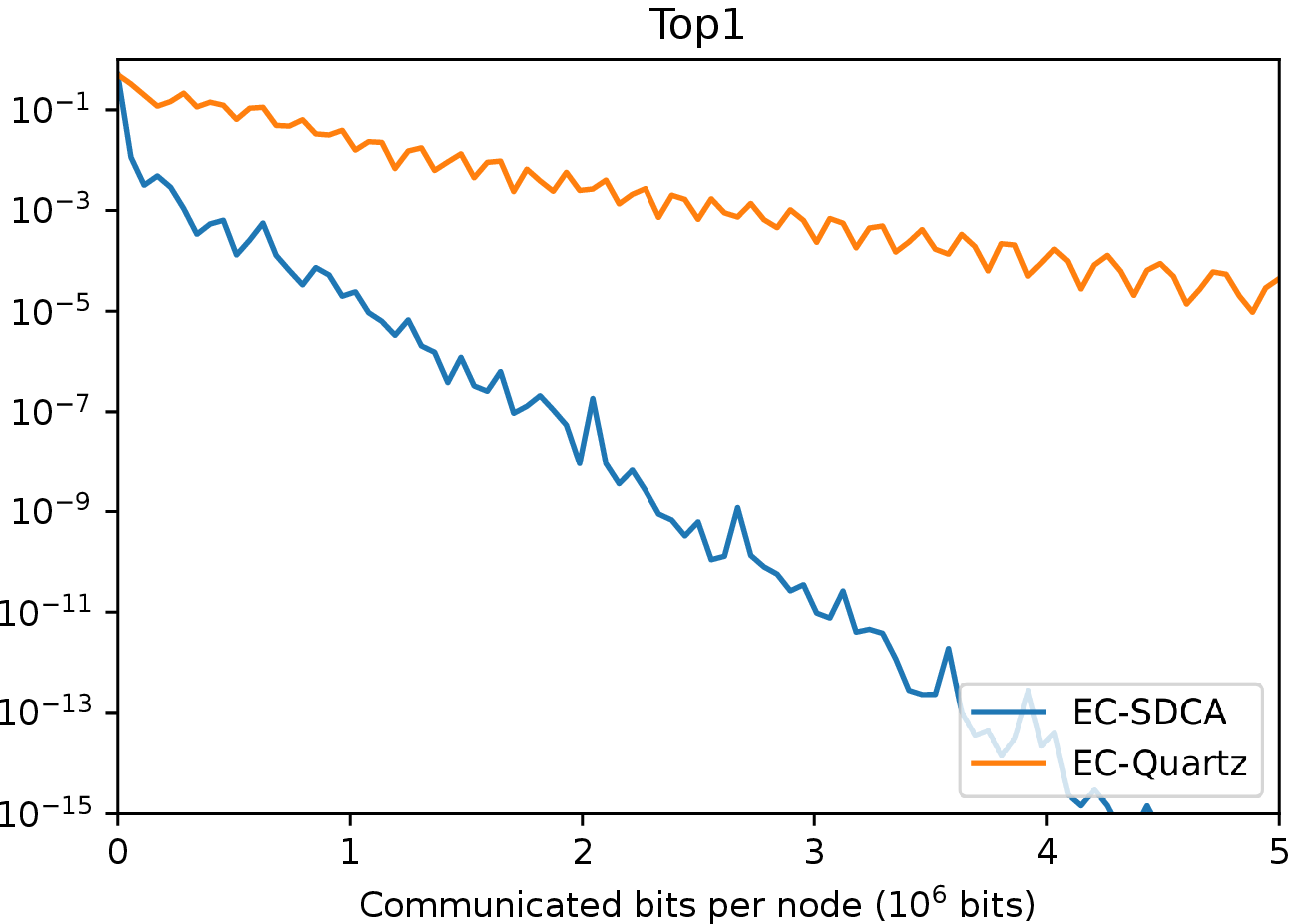}&
		\includegraphics[width=3.8cm]{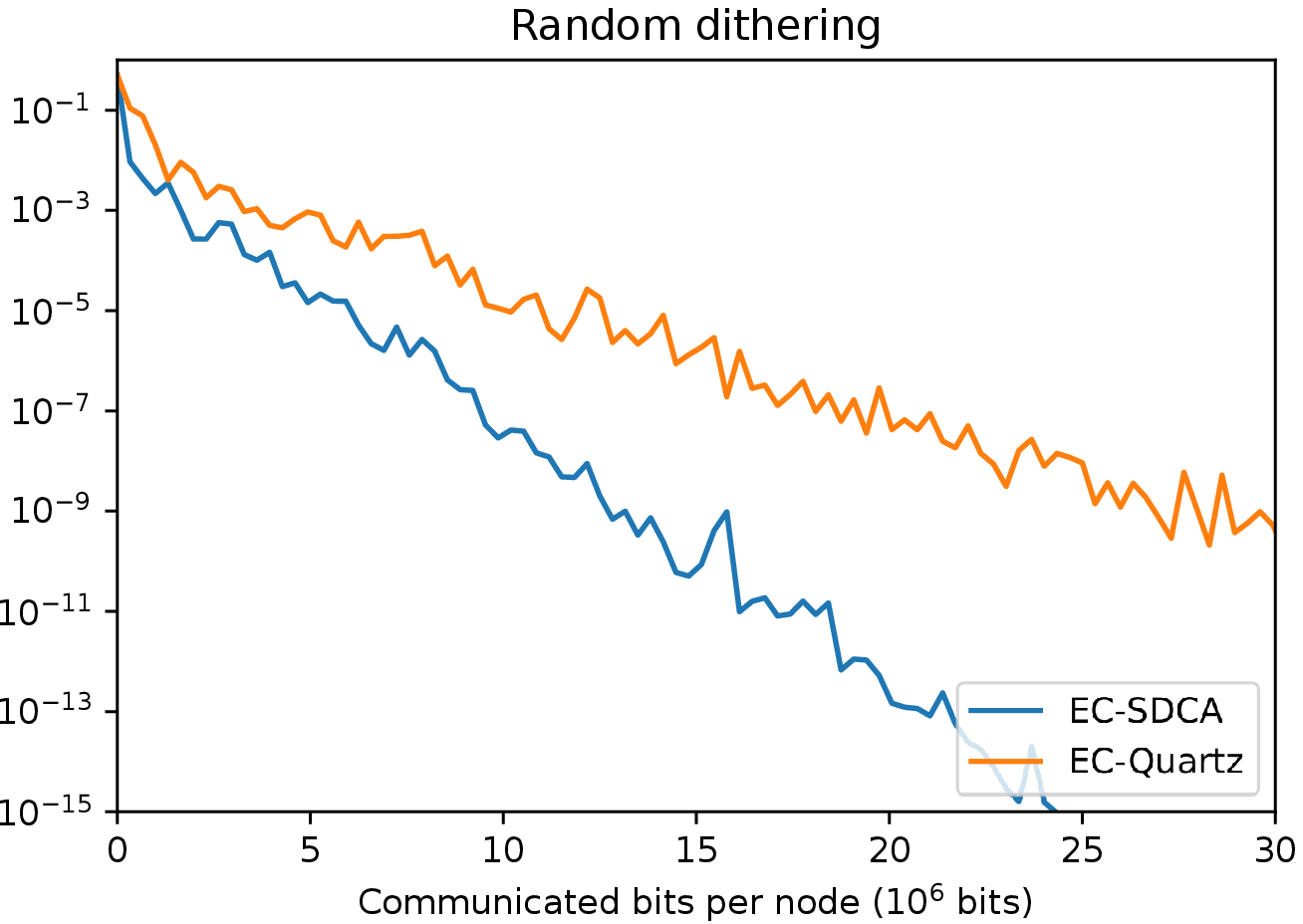}&
		\includegraphics[width=3.8cm]{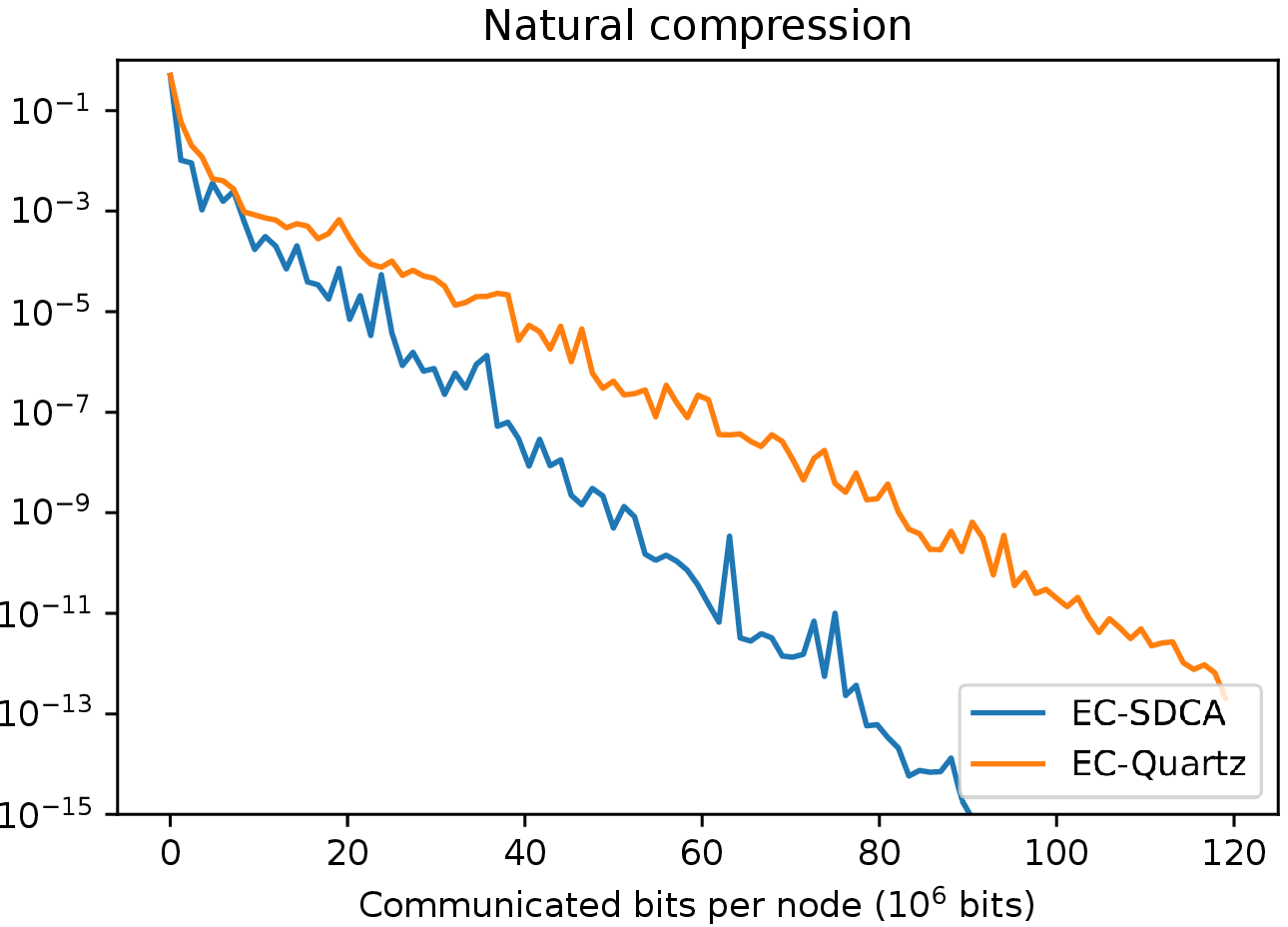}&
		\includegraphics[width=3.8cm]{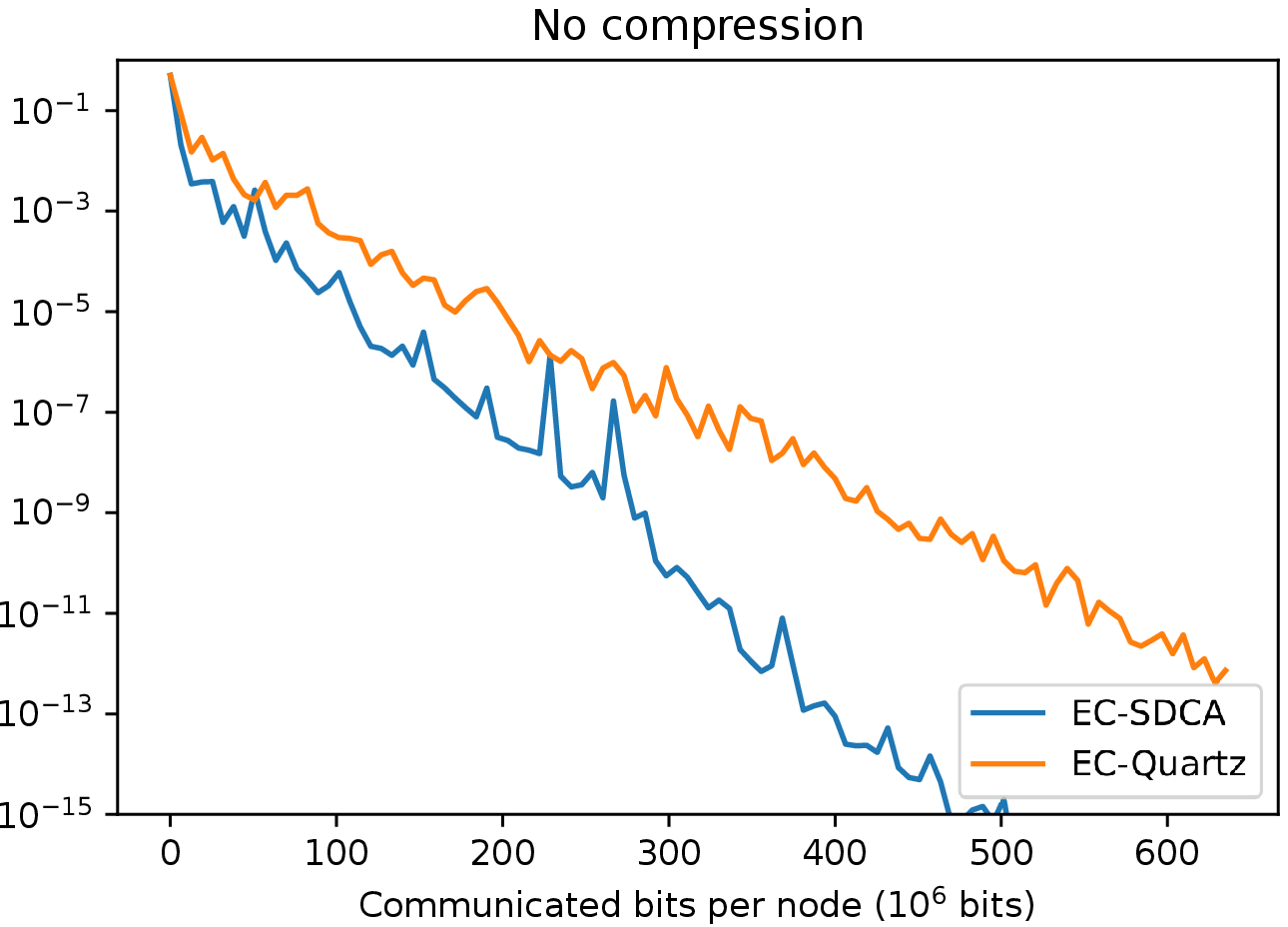}
	\end{tabular}

	\caption{EC-SDCA vs EC-Quartz on  \textbf{a9a}}\label{fig:quartz_a9a}

\end{figure}

\begin{figure}[H]

	\centering
	\begin{tabular}{cccc}
		\includegraphics[width=3.8cm]{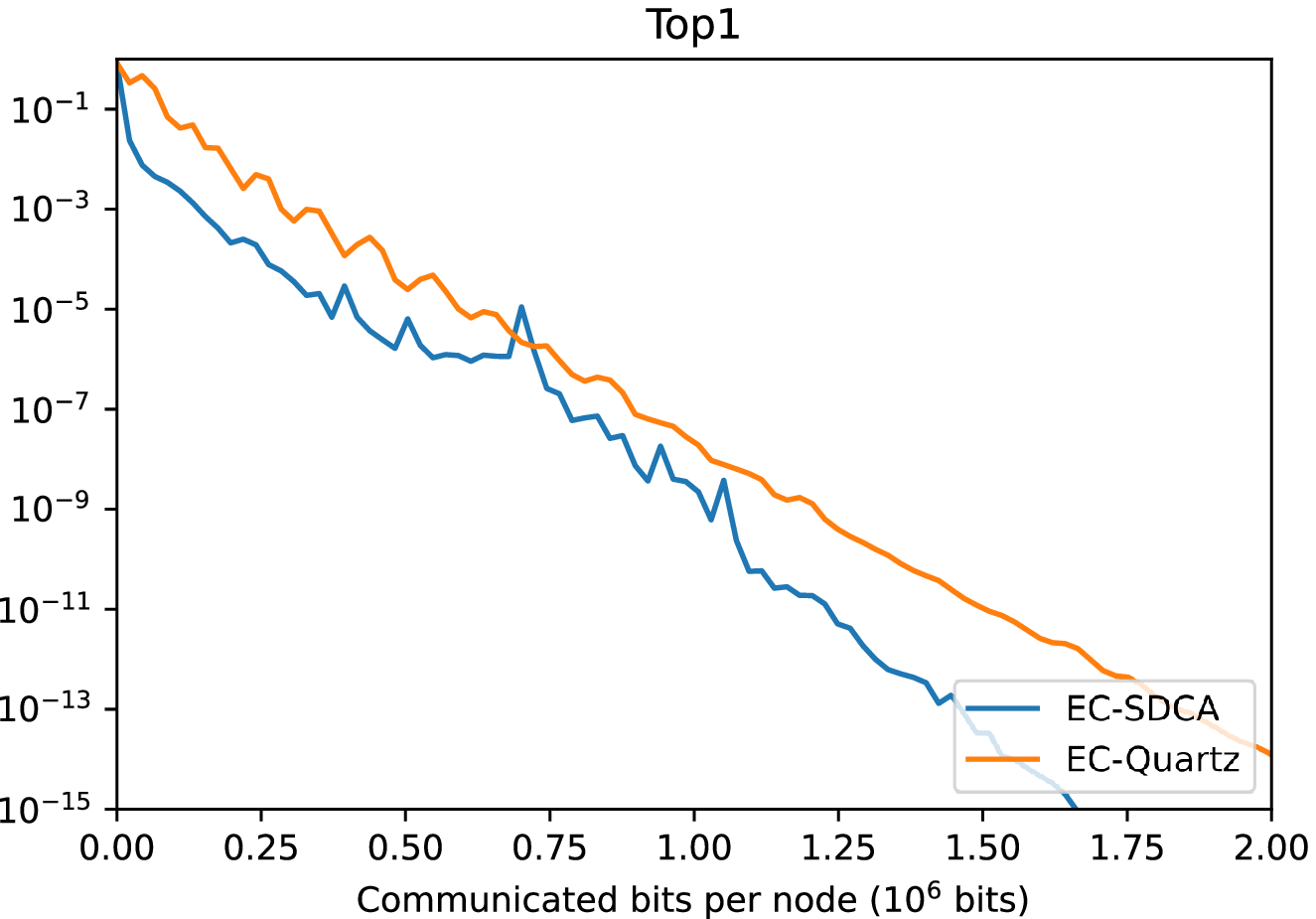}&
		\includegraphics[width=3.8cm]{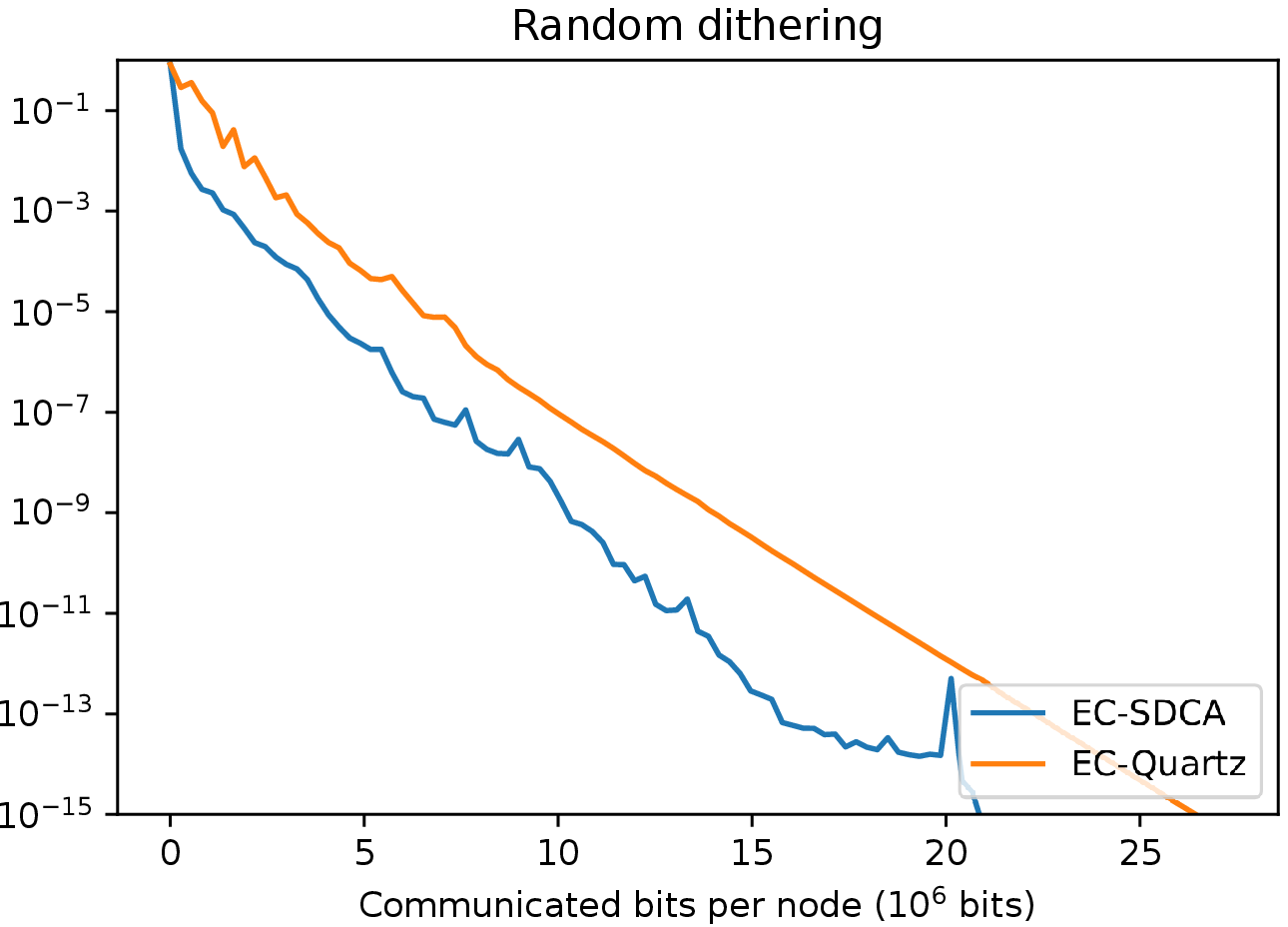}&
		\includegraphics[width=3.8cm]{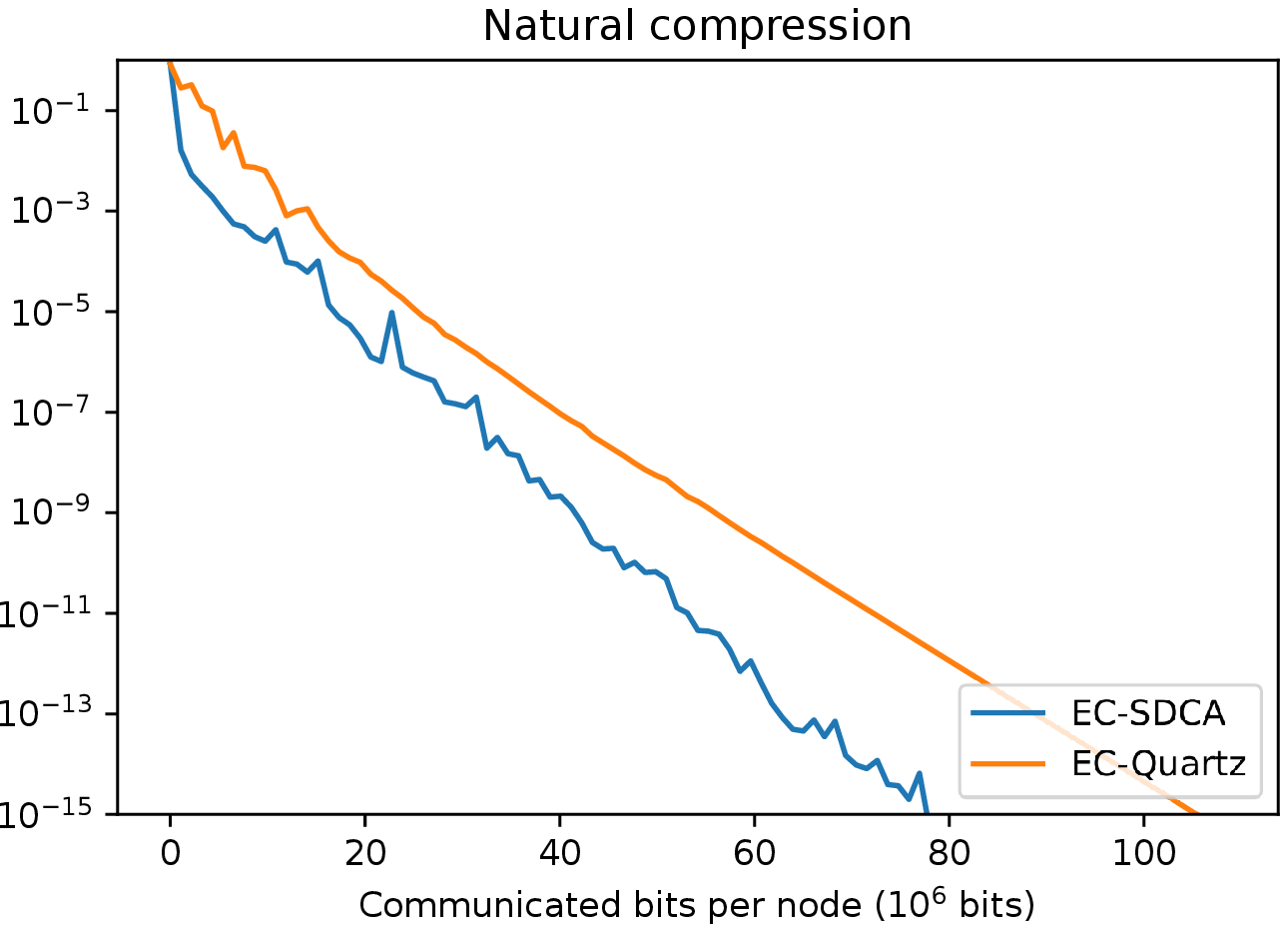}&
		\includegraphics[width=3.8cm]{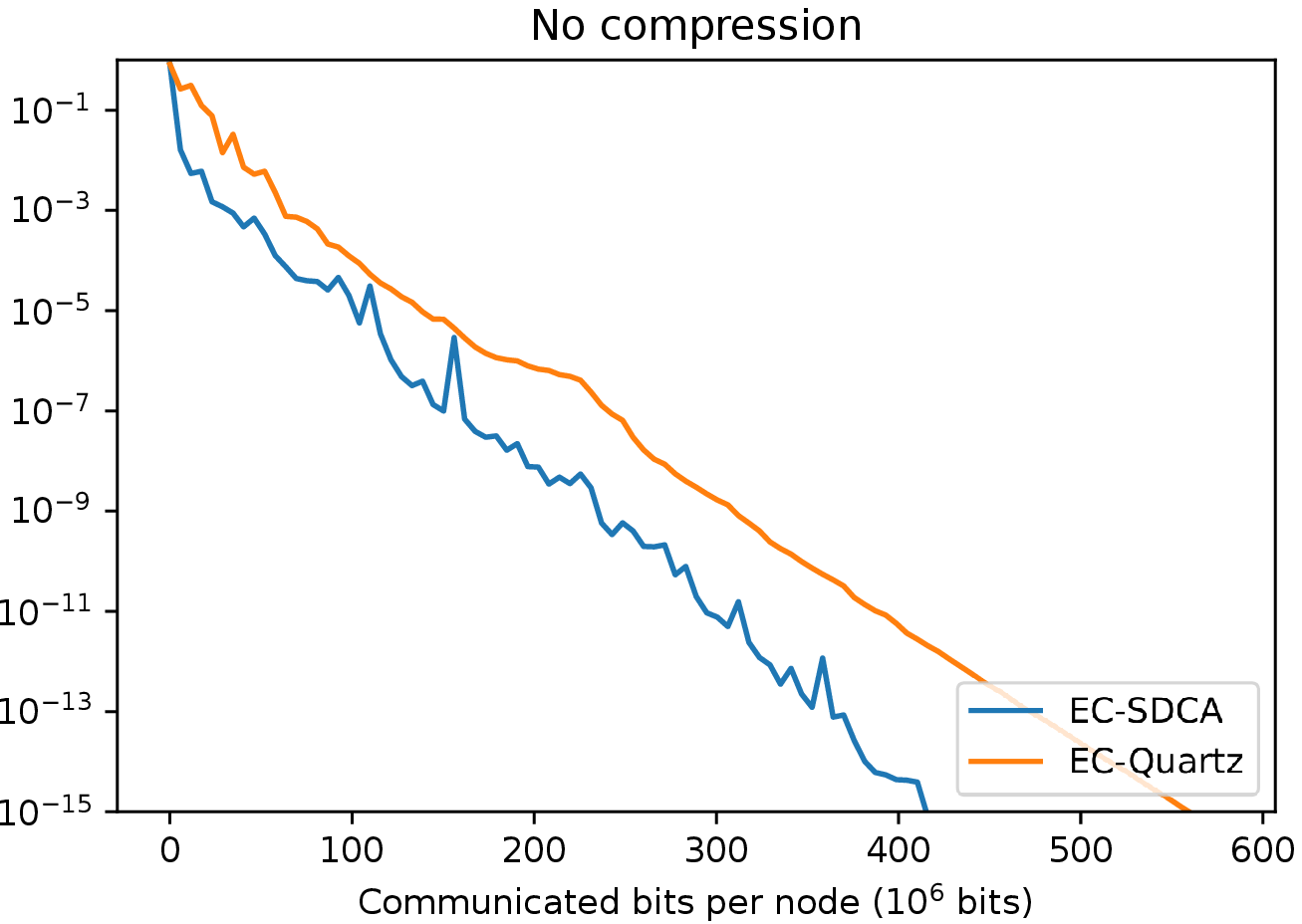}
	\end{tabular}

	\caption{EC-SDCA vs EC-Quartz on  \textbf{w6a}}\label{fig:quartz_w6a}

\end{figure}

\clearpage

\bibliography{error_compensated_LSVRG_and_SDCA}

\begin{thebibliography}{30}
\providecommand{\natexlab}[1]{#1}
\providecommand{\url}[1]{\texttt{#1}}
\expandafter\ifx\csname urlstyle\endcsname\relax
  \providecommand{\doi}[1]{doi: #1}\else
  \providecommand{\doi}{doi: \begingroup \urlstyle{rm}\Url}\fi

\bibitem[Agarwal and Duchi(2011)]{Agarwal11}
A.~Agarwal and J.~C. Duchi.
\newblock Distributed delayed stochastic optimization.
\newblock \emph{Advances in Neural Information Processing Systems}, pages
  873--881, 2011.

\bibitem[Alistarh et~al.(2017)Alistarh, Grubic, Li, Tomioka, and
  Vojnovic]{Alistarh17}
D.~Alistarh, D.~Grubic, J.~Li, R.~Tomioka, and M.~Vojnovic.
\newblock {QSGD}: Communication-efficient {SGD} via gradient quantization and
  encoding.
\newblock \emph{Advances in Neural Information Processing Systems}, pages
  1709--1720, 2017.

\bibitem[Alistarh et~al.(2018)Alistarh, Hoefler, Johansson, Konstantinov,
  Khirirat, and Renggli]{Alistarh18}
D.~Alistarh, T.~Hoefler, M.~Johansson, N.~Konstantinov, S.~Khirirat, and
  C.~Renggli.
\newblock The convergence of sparsified gradient methods.
\newblock \emph{Advances in Neural Information Processing Systems}, pages
  5973--5983, 2018.

\bibitem[Bernstein et~al.(2018)Bernstein, Wang, Azizzadenesheli, and
  Anandkumar]{Bernstein18}
J.~Bernstein, Y.~X. Wang, K.~Azizzadenesheli, and A.~Anandkumar.
\newblock Sign{SGD}: Compressed optimisation for non-convex problems.
\newblock \emph{The 35th International Conference on Machine Learning}, pages
  560--569, 2018.

\bibitem[Beznosikov et~al.(2020)Beznosikov, Horv\'{a}th, Richt\'{a}rik, and
  Safaryan]{biased2020}
A.~Beznosikov, S.~Horv\'{a}th, P.~Richt\'{a}rik, and M.~Safaryan.
\newblock On biased compression for distributed learning.
\newblock \emph{arXiv:2002.12410}, 2020.

\bibitem[Chang and Lin(2011)]{chang2011libsvm}
Chih-Chung Chang and Chih-Jen Lin.
\newblock {LIBSVM}: A library for support vector machines.
\newblock \emph{{ACM} {T}ransactions on {I}ntelligent {S}ystems and
  {T}echnology (TIST)}, 2\penalty0 (3):\penalty0 1--27, 2011.

\bibitem[Gorbunov et~al.(2020)Gorbunov, Kovalev, Makarenko, and
  Richt{\'a}rik]{gorbunov2020linearly}
Eduard Gorbunov, Dmitry Kovalev, Dmitry Makarenko, and Peter Richt{\'a}rik.
\newblock Linearly converging error compensated {SGD}.
\newblock \emph{arXiv preprint arXiv:2010.12292}, 2020.

\bibitem[Goyal et~al.(2017)Goyal, Doll\'{a}r, Girshick, Noordhuis, Wesolowski,
  Kyrola, Tulloch, Jia, and He]{Goyal17}
P.~Goyal, P.~Doll\'{a}r, R.~Girshick, P.~Noordhuis, L.~Wesolowski, A.~Kyrola,
  A.~Tulloch, Y.~Jia, and K.~He.
\newblock Accurate, large minibatch {SGD}: Training imagenet in 1 hour.
\newblock \emph{arXiv: 1706.2677}, 2017.

\bibitem[Horv\'{a}th et~al.(2019{\natexlab{a}})Horv\'{a}th, Kovalev,
  Mishchenko, Stich, and Richt\'{a}rik]{Samuel19}
S.~Horv\'{a}th, D.~Kovalev, K.~Mishchenko, S.~Stich, and P.~Richt\'{a}rik.
\newblock Stochastic distributed learning with gradient quantization and
  variance reduction.
\newblock \emph{arXiv: 1904.05115}, 2019{\natexlab{a}}.

\bibitem[Horv\'{a}th et~al.(2019{\natexlab{b}})Horv\'{a}th, Ho, Horvath, Sahu,
  Canini, and Richt\'{a}rik]{horvath2019natural}
Samuel Horv\'{a}th, Chen-Yu Ho, \v{L}udov\'{i}t Horvath, Atal~Narayan Sahu,
  Marco Canini, and Peter Richt\'{a}rik.
\newblock Natural compression for distributed deep learning.
\newblock \emph{arXiv preprint arXiv:1905.10988}, 2019{\natexlab{b}}.

\bibitem[Karimireddy et~al.(2019)Karimireddy, Rebjock, Stich, and
  Jaggi]{karimireddy2019error}
Sai~Praneeth Karimireddy, Quentin Rebjock, Sebastian~U Stich, and Martin Jaggi.
\newblock Error feedback fixes {SignSGD} and other gradient compression
  schemes.
\newblock \emph{arXiv preprint arXiv:1901.09847}, 2019.

\bibitem[Kovalev et~al.(2019)Kovalev, Horv\'{a}th, and Richt\'{a}rik]{LSVRG}
D.~Kovalev, S.~Horv\'{a}th, and P.~Richt\'{a}rik.
\newblock Don't jump through hoops and remove those loops: Svrg and katyusha
  are better without the outer loop.
\newblock \emph{arXiv: 1901.08689}, 2019.

\bibitem[Lian et~al.(2015)Lian, Huang, Li, and Liu]{Lian15}
X.~Lian, Y.~Huang, Y.~Li, and J.~Liu.
\newblock Asynchronous parallel stochastic gradient for nonconvex optimization.
\newblock \emph{Advances in Neural Information Processing Systems}, pages
  2737--2745, 2015.

\bibitem[Ma et~al.(2017)Ma, Kone\v{c}n\'{y}, Jaggi, Smith, Jordan,
  Richt\'{a}rik, and Tak\'{a}\v{c}]{COCOA+journal}
Chenxin Ma, Jakub Kone\v{c}n\'{y}, Martin Jaggi, Virginia Smith, Michael~I.
  Jordan, Peter Richt\'{a}rik, and Martin Tak\'{a}\v{c}.
\newblock Distributed optimization with arbitrary local solvers.
\newblock \emph{Optimization Methods and Software}, 32\penalty0 (4):\penalty0
  813--848, 2017.

\bibitem[Mishchenko et~al.(2019)Mishchenko, Gorbunov, Tak\'{a}\v{c}, and
  Richt\'{a}rik]{Mish19}
K.~Mishchenko, E.~Gorbunov, M.~Tak\'{a}\v{c}, and P.~Richt\'{a}rik.
\newblock Distributed learning with compressed gradient differences.
\newblock \emph{arXiv: 1901.09269}, 2019.

\bibitem[Nesterov(2004)]{NesterovBook}
Yurii Nesterov.
\newblock \emph{Introductory Lectures on Convex Optimization: A Basic Course
  (Applied Optimization)}.
\newblock Kluwer Academic Publishers, 2004.

\bibitem[Qian et~al.(2019{\natexlab{a}})Qian, Qu, and Richt{\'{a}}rik]{SAGA-AS}
Xun Qian, Zheng Qu, and Peter Richt{\'{a}}rik.
\newblock {SAGA} with arbitrary sampling.
\newblock In \emph{Proceedings of the 36th International Conference on Machine
  Learning, {ICML} 2019, 9-15 June 2019, Long Beach, California, {USA}}, pages
  5190--5199, 2019{\natexlab{a}}.
\newblock URL \url{http://proceedings.mlr.press/v97/qian19a.html}.

\bibitem[Qian et~al.(2019{\natexlab{b}})Qian, Qu, and
  Richt{\'a}rik]{qian2019svrg}
Xun Qian, Zheng Qu, and Peter Richt{\'a}rik.
\newblock L-svrg and l-katyusha with arbitrary sampling.
\newblock \emph{arXiv preprint arXiv:1906.01481}, 2019{\natexlab{b}}.

\bibitem[Qu et~al.(2015)Qu, Richt\'{a}rik, and Zhang]{Quartz}
Zheng Qu, Peter Richt\'{a}rik, and Tong Zhang.
\newblock Quartz: Randomized dual coordinate ascent with arbitrary sampling.
\newblock In C.~Cortes, N.~D. Lawrence, D.~D. Lee, M.~Sugiyama, and R.~Garnett,
  editors, \emph{Advances in Neural Information Processing Systems 28}, pages
  865--873. Curran Associates, Inc., 2015.

\bibitem[Recht et~al.(2011)Recht, Re, Wright, and Niu]{Recht11}
B.~Recht, C.~Re, S.~Wright, and F.~Niu.
\newblock Hogwild: A lock-free approach to parallelizing stochastic gradient
  descent.
\newblock \emph{Advances in Neural Information Processing Systems}, pages
  693--701, 2011.

\bibitem[Seide et~al.(2014)Seide, Fu, Droppo, Li, and Yu]{Seide14}
F.~Seide, H.~Fu, J.~Droppo, G.~Li, and D.~Yu.
\newblock 1-bit stochastic gradient descent and its application to data-
  parallel distributed training of speech dnns.
\newblock \emph{Fifteenth Annual Conference of the International Speech
  Communication Association}, 2014.

\bibitem[Shalev-Shwartz and Zhang(2013)]{SDCA}
Shai Shalev-Shwartz and Tong Zhang.
\newblock Stochastic dual coordinate ascent methods for regularized loss.
\newblock \emph{Journal of Machine Learning Research}, 14\penalty0
  (1):\penalty0 567--599, 2013.

\bibitem[Stich and Karimireddy(2019)]{Stich19}
S.~U. Stich and S.~P. Karimireddy.
\newblock The error-feedback framework: Better rates for {SGD} with delayed
  gradients and compressed communication.
\newblock \emph{arXiv: 1909.05350}, 2019.

\bibitem[Stich et~al.(2018)Stich, Cordonnier, and Jaggi]{Stich18}
S.~U. Stich, J.~B. Cordonnier, and M.~Jaggi.
\newblock Sparsified {SGD} with memory.
\newblock \emph{Advances in Neural Information Processing Systems}, pages
  4447--4458, 2018.

\bibitem[Stich(2020)]{localSGD-Stich}
Sebastian~U. Stich.
\newblock Local {SGD} converges fast and communicates little.
\newblock In \emph{International Conference on Learning Representations
  (ICLR)}, 2020.

\bibitem[Tang et~al.(2018)Tang, Gan, Zhang, Zhang, and Liu]{Tang18}
H.~Tang, S.~Gan, C.~Zhang, T.~Zhang, and J.~Liu.
\newblock Communication compression for decentralized training.
\newblock \emph{Advances in Neural Information Processing Systems}, pages
  7652--7662, 2018.

\bibitem[Tang et~al.(2019)Tang, Lian, Zhang, and Liu]{Tang19}
H.~Tang, X.~Lian, T.~Zhang, and J.~Liu.
\newblock Doublesqueeze: Parallel stochastic gradient descent with double-pass
  error-compensated compression.
\newblock \emph{The 36th International Conference on Machine Learning}, pages
  6155--6165, 2019.

\bibitem[Wen et~al.(2017)Wen, Xu, Yan, Wu, Wang, and Li]{Wen17}
W.~Wen, C.~Xu, F.~Yan, C.~Wu, Y.~Wang, and H.~Li.
\newblock Terngrad: Ternary gradients to reduce communication in distributed
  deep learning.
\newblock \emph{Advances in Neural Information Processing Systems}, pages
  1509--1519, 2017.

\bibitem[Wu et~al.(2018)Wu, Huang, Huang, and Zhang]{Wu18}
J.~Wu, W.~Huang, J.~Huang, and T.~Zhang.
\newblock Error compensated quantized {SGD} and its applications to large-scale
  distributed optimization.
\newblock \emph{The 35th International Conference on Machine Learning}, pages
  5321--5329, 2018.

\bibitem[You et~al.(2017)You, Gitman, and Ginsburg]{You17}
Y.~You, I.~Gitman, and B.~Ginsburg.
\newblock Scaling {SGD} batch size to 32k for imagenet training.
\newblock \emph{arXiv: 1708.03888}, 2017.

\end{thebibliography}
\bibliographystyle{plainnat}

\clearpage
\appendix
\part*{Appendix}

\section{ESO Estimation for Arbitrary Sampling for Quartz}

For simplicity, in this section we consider problem (\ref{primal-sdca}) with $m=1$, and replace $\phi_{i\tau}$ and $A_{i\tau}$ with $\phi_\tau$ and $A_{\tau}$ respectively. We consider arbitrary proper set sampling, i.e., $S \in [n]$ with $p_i \eqdef \mathbb{P}[i\in S] >0$ for all $i\in [n] \eqdef \{  1, ..., n  \}$. 

\begin{assumption}\label{as:AB}
	There exist constants ${\cal A}_i\geq 0$ for  each $i \in [n]$ and $0\leq {\cal B}\leq 1$ such that for any matrix ${\bf M} \in \R^{t\times n}$ and sampling $S$,
	\begin{equation}\label{eq:AB}
	\mathbb{E}\left [\left\| \sum_{i \in S} \frac{1}{p_i}{{\bf M}_{:i}} \right\|^2\right ] \leq \sum_{i=1}^n {\cal A}_i  \left\|{\bf M}_{: i}\right\|^2 + {\cal B}\left\| \sum_{i=1}^n {\bf M}_{:i}\right\|^2,
	\end{equation}
	where ${\bf M}_{: i}$ denotes the $i$th column vector of ${\bf M}$.
\end{assumption}

Assumption~\ref{as:AB} appeared in~\citep{SAGA-AS} and \citep{qian2019svrg} for the convergence analysis of SAGA and L-SVRG. The estimations of ${\cal A}_i$ and ${\cal B}$ for arbitrary set sampling, $\tau$-nice sampling, and group sampling can be found in \citep{qian2019svrg}. 

\begin{lemma}\label{lm:esoAS}
Under Assumption~\ref{as:AB}, the following inequality holds for all $h_i \in \R^t$ with $i\in [n]$, 
$$
\mathbb{E} \left[  \left\| \sum_{i\in S} A_i h_i \right\|^2  \right] \leq \sum_{i=1}^n p_i v_i \|h_i\|^2, 
$$
where $v_i = {\cal A}_i p_i R_m^2  + {\cal B} n p_i R^2$. 
\end{lemma}

\begin{proof}

By applying Assumption~\ref{as:AB} with ${\bf M}_{:i} = p_iA_ih_i$, we have 
\begin{align*}
\mathbb{E} \left[  \left\| \sum_{i\in S} A_i h_i \right\|^2  \right] & = \mathbb{E} \left[  \left\| \sum_{i\in S} \frac{1}{p_i} \cdot p_i A_i h_i \right\|^2  \right] \\
& \leq \sum_{i=1}^n {\cal A}_i \|p_i A_ih_i\|^2 + {\cal B} \left\| \sum_{i=1}^n p_iA_ih_i\right\|^2 \\ 
& \leq \sum_{i=1}^n {\cal A}_i p_i^2 \|A_i\|^2 \|h_i\|^2 + {\cal B} \left\| \sum_{i=1}^n p_iA_ih_i\right\|^2 \\ 
& \leq \sum_{i=1}^n {\cal A}_i p_i^2 R_m^2 \|h_i\|^2 + {\cal B} \left\| \sum_{i=1}^n p_iA_ih_i\right\|^2, 
\end{align*}
where we use $R_m = \max_i \{ \|A_i\|  \}$ in the last inequality. 

For $\left\| \sum_{i=1}^n p_iA_ih_i\right\|^2$, since $\sum_{i=1}^n p_iA_ih_i = [A_1, ..., A_n] [p_1h_1^\top, ..., p_nh_n^\top]^\top$, we have 
\begin{align*}
\left\| \sum_{i=1}^n p_iA_ih_i\right\|^2 &\leq \| [A_1, ..., A_n]\|^2 \cdot \| [p_1h_1^\top, ..., p_nh_n^\top]^\top\|^2 \\ 
& = \lambda_{\rm max} ([A_1, ..., A_n] [A_1, ..., A_n]^\top) \cdot \sum_{i=1}^n p_i^2 \|h_i\|^2 \\ 
& = \lambda_{\rm max} \left(\sum_{i=1}^n A_iA_i^\top \right) \cdot \sum_{i=1}^n p_i^2 \|h_i\|^2 \\
& \leq \sum_{i=1}^n nR^2 p_i^2 \|h_i\|^2, 
\end{align*}
where we use $R^2 = \frac{1}{n} \lambda_{\rm max}\left(\sum_{i=1}^n A_iA_i^\top \right)$ in the last inequality. Combining the above two inequalities, we arrive at 
\begin{align*}
\mathbb{E} \left[  \left\| \sum_{i\in S} A_i h_i \right\|^2  \right] &  \leq \sum_{i=1}^n {\cal A}_i p_i^2 R_m^2 \|h_i\|^2 + {\cal B} \sum_{i=1}^n nR^2 p_i^2 \|h_i\|^2\\ 
& =   \sum_{i=1}^n p_i \left({\cal A}_i p_i R_m^2  + {\cal B} n p_i R^2 \right)\|h_i\|^2. 
\end{align*}

\end{proof}

\section{Proofs for EC-LSVRG in the Composite Case }

\subsection{Lemmas}

Let $\mathbb{E}_k[\cdot]$ denote the expectation conditional on $x^k$, $w^k$, $h^k$, $u^k$, and $e^k_\tau$. 

The following lemma shows the progress at iteration $k$ for the auxiliary points ${\tilde x}^k$ and ${\tilde x}^{k+1}$.

\begin{lemma}\label{lm:itereclsvrg}
	If $\eta \leq \tfrac{1}{4L_f}$, then 
	\begin{eqnarray*}
		\left(1 + \frac{\eta \mu}{2} \right)\mathbb{E}_k \|\tilde{x}^{k+1} - x^*\|^2 &\leq& \|\tilde{x}^k - x^*\|^2 + 2\eta \mathbb{E}_k [P(x^*) - P(x^{k+1})] + \|e^k\|^2 \\ 
		&& + (1 + \eta \mu) \mathbb{E}_k \|e^{k+1}\|^2 + 4\eta^2 \mathbb{E}_k [\|g^k + h^k - \nabla f(x^k)\|^2]. 
	\end{eqnarray*}
\end{lemma}

\begin{proof}

Since $\tilde{x}^{k+1} = \tilde{x}^k - \eta(g^k + h^k + \partial \psi(x^{k+1}))$, we have 

\begin{eqnarray*}
	&& \langle \eta ( g^k + h^k), x^* - x^{k+1} \rangle \\ 
	&=& \langle \tilde{x}^k - \tilde{x}^{k+1} - \eta\partial \psi(x^{k+1}), x^* - x^{k+1} \rangle \\ 
	&=& \langle \tilde{x}^k - x^{k+1}, x^*-x^{k+1} \rangle + \langle x^{k+1} - \tilde{x}^{k+1}, x^* - x^{k+1} \rangle - \eta \langle \partial \psi(x^{k+1}), x^* - x^{k+1} \rangle \\ 
	&\geq& \frac{1}{2} \left(  -\|\tilde{x}^k - x^*\|^2 + \|\tilde{x}^k - x^{k+1}\|^2 + \|x^{k+1} - x^*\|^2  \right)  + \frac{1}{2}\left(  \|\tilde{x}^{k+1}  - x^*\|^2 \right. \\ 
	&& \left. - \|x^{k+1} - \tilde{x}^{k+1}\|^2 - \|x^{k+1} - x^*\|^2  \right) + \eta \left(  \psi(x^{k+1}) - \psi(x^*) + \frac{\mu}{2}\|x^{k+1}-x^*\|^2  \right) \\ 
	&=& \frac{1}{2}\|\tilde{x}^{k+1}-x^*\|^2 - \frac{1}{2}\|\tilde{x}^k - x^*\|^2 + \frac{1}{2}\|\tilde{x}^k - x^{k+1}\|^2 - \frac{1}{2}\|\tilde{x}^{k+1} - x^{k+1}\|^2 \\
	&& + \eta \left(  \psi(x^{k+1}) - \psi(x^*) + \frac{\mu}{2}\|x^{k+1}-x^*\|^2  \right). 
\end{eqnarray*}

From $\|\tilde{x}^k - x^{k+1}\|^2 \geq \frac{1}{2}\|x^{k+1} - x^k\|^2 - \|\tilde{x}^k - x^k\|^2$, and $\|x^{k+1}-x^*\|^2 \geq \frac{1}{2}\|\tilde{x}^{k+1}-x^*\|^2 - \|\tilde{x}^{k+1} - x^{k+1}\|^2$, we arrive at 
\begin{eqnarray}
&& \langle \eta(g^k +  h^k), x^* - x^{k+1} \rangle \nonumber \\ 
&\geq& \frac{1 + \eta\mu/2}{2} \|\tilde{x}^{k+1}-x^*\|^2 - \frac{1}{2}\|\tilde{x}^k - x^*\|^2 + \frac{1}{4}\|x^{k+1}-x^k\|^2 - \frac{1}{2}\|\tilde{x}^k - x^k\|^2 \nonumber \\ 
&& - \frac{1 + \eta \mu}{2}\|\tilde{x}^{k+1} - x^{k+1}\|^2 + \eta (\psi(x^{k+1}) - \psi(x^*))  \label{eq:nfxstar-xk+1}. 
\end{eqnarray}

Since $f$ is convex and $\mathbb{E}_k[g^k + h^k] = \nabla f(x^k)$, we have 
\begin{eqnarray*}
	f(x^*) &\geq& f(x^k) + \langle \nabla f(x^k), x^* - x^k\rangle \\ 
	&=& f(x^k) + \mathbb{E}_k[ \langle g^k + h^k, x^* - x^{k+1} + x^{k+1} -x^k \rangle ] \\ 
	&=& f(x^k) + \mathbb{E}_k [\langle g^k + h^k, x^* - x^{k+1} \rangle ] + \mathbb{E}_k[\langle g^k + h^k - \nabla f(x^k), x^{k+1}-x^k \rangle] \\ 
	&& + \mathbb{E}_k[\langle \nabla f(x^k), x^{k+1}-x^k \rangle] \\ 
	&\geq& \mathbb{E}_k[f(x^{k+1})] - \frac{L_f}{2}\mathbb{E}_k[\|x^{k+1}-x^k\|^2] + \mathbb{E}_k [\langle g^k + h^k, x^* - x^{k+1} \rangle ]  \\ 
	&& + \mathbb{E}_k[\langle g^k + h^k - \nabla f(x^k), x^{k+1}-x^k \rangle]   \\ 
	&\geq& \mathbb{E}_k[f(x^{k+1})] - \frac{L_f}{2}\mathbb{E}_k[\|x^{k+1}-x^k\|^2] + \mathbb{E}_k [\langle g^k + h^k, x^* - x^{k+1} \rangle ]  \\ 
	&& -\frac{1}{2\beta} \mathbb{E}_k [\|g^k + h^k - \nabla f(x^k)\|^2] - \frac{\beta}{2} \mathbb{E}_k[\|x^{k+1} - x^k\|^2], 
\end{eqnarray*}
where the second inequality comes from that $f$ is $L_f$-smooth and the last inequality comes from Young's inequality with any $\beta>0$. 

By choosing $\beta = \frac{1}{4\eta}$, we have 
\begin{eqnarray*}
	&&  f(x^*)  \\ 
	&\geq& \mathbb{E}_k[f(x^{k+1})] - \left(\frac{L_f}{2} + \frac{1}{8\eta} \right)\mathbb{E}_k[\|x^{k+1}-x^k\|^2] + \mathbb{E}_k [\langle g^k + h^k, x^* - x^{k+1} \rangle ] \\ 
	&& - 2\eta \mathbb{E}_k [\|g^k + h^k - \nabla f(x^k)\|^2]   \\ 
	&\overset{(\ref{eq:nfxstar-xk+1})}{\geq}& \mathbb{E}_k[f(x^{k+1})] + \left( \frac{1}{4\eta} - \frac{L_f}{2} - \frac{1}{8\eta} \right)\mathbb{E}_k[\|x^{k+1}-x^k\|^2]  + \frac{1+\eta\mu/2}{2\eta} \mathbb{E}_k\|\tilde{x}^{k+1} - x^*\|^2 \\ 
	&& - \frac{1}{2\eta}\|\tilde{x}^k-x^*\|^2 - \frac{1}{2\eta}\|\tilde{x}^k - x^k\|^2 - \frac{1+\eta\mu}{2\eta} \mathbb{E}_k\|\tilde{x}^{k+1} - x^{k+1}\|^2 \\ 
	&& + \mathbb{E}_k[\psi(x^{k+1})] - \psi(x^*) - 2\eta \mathbb{E}_k [\|g^k + h^k - \nabla f(x^k)\|^2].  
\end{eqnarray*}

Noticing that $\frac{1}{4\eta} - \frac{L_f}{2}-\frac{1}{8\eta} \geq 0$ if $\eta \leq \frac{1}{4L_f}$, we can get the result after rearrangement. 

\end{proof}

\begin{lemma}\label{lm:gk2}
	We have 
	\begin{equation}
	\frac{1}{n} \sum_{\tau=1}^n \mathbb{E}_k \left\|\nabla f_{i_k^\tau}^{(\tau)}(x^k) - \nabla f_{i_k^\tau}^{(\tau)}(w^k) \right\|^2  \leq 4L [P(x^k) - P(x^*)] + 4L [P(w^k) - P(x^*)], \label{eq:gktau2}
	\end{equation}
	and 
	\begin{equation}
	\frac{1}{n} \sum_{\tau=1}^n  \left\| \nabla f^{(\tau)}(x^k) - \nabla f^{(\tau)}(w^k) \right\|^2 \leq 4{\bar L}  [P(x^k) - P(x^*)]  + 4{\bar L} [P(w^k) - P(x^*)], \label{eq:nftau}
	\end{equation}
	and 
	\begin{eqnarray}
	\mathbb{E}_k \left\| \frac{1}{n} \sum_{\tau=1}^n \left(  \nabla f_{i_k^\tau}^{(\tau)}(x^k) - \nabla f_{i_k^\tau}^{(\tau)}(w^k)  \right) \right\|^2 &\leq& 4\left(L_f + \frac{L}{n} \right) [P(x^k) - P(x^*)] \nonumber \\ 
	&& + 4\left(L_f + \frac{L}{n} \right) [P(w^k) - P(x^*)].  \label{eq:gk2}
	\end{eqnarray}
	and 
	\begin{eqnarray}
	&& \mathbb{E}_k \left\|\frac{1}{n} \sum_{\tau=1}^n \left(  \nabla f_{i_k^\tau}^{(\tau)}(x^k) - \nabla f_{i_k^\tau}^{(\tau)}(w^k)  \right) + \nabla f(w^k) - \nabla f(x^k) \right\|^2 \nonumber \\ 
	&\leq& \frac{4L}{n} [P(x^k) - P(x^*) + P(w^k) - P(x^*)] .  \label{eq:gk2variance}
	\end{eqnarray}
\end{lemma}

\begin{proof}

Since $f_i^{(\tau)}$ is $L$-smooth and $f$ is $L_f$-smooth, we have (\citep{NesterovBook}, Theorem 2.1.5)
$$
\|\nabla f_i^{(\tau)}(x) - \nabla f_i^{(\tau)}(y)\|^2 \leq 2L (f_i^{(\tau)}(x) - f_i^{(\tau)}(y) - \langle \nabla f_i^{(\tau)}(y), x-y\rangle ), 
$$
and 
$$
\| \nabla f(x) - \nabla f(y) \|^2 \leq 2L_f \left(  f(x) - f(y) - \langle \nabla f(y), x-y \rangle  \right), 
$$
for any $x, y \in \R^d$. Therefore, 
\begin{eqnarray*}
	 \mathbb{E}_k \|\nabla f_{i_k^\tau}^{(\tau)}(x^k) - \nabla f_{i_k^\tau}^{(\tau)}(w^k) \|^2 
	&\leq& 2\mathbb{E}_k \|\nabla f_{i_k^\tau}^{(\tau)}(x^k) - \nabla f_{i_k^\tau}^{(\tau)} (x^*)\|^2 + 2\mathbb{E}_k \|\nabla f_{i_k^\tau}^{(\tau)}(w^k) - \nabla f_{i_k^\tau}^{(\tau)}(x^*)\|^2 \\ 
	&\leq& 4L [f^{(\tau)}(x^k) - f^{(\tau)}(x^*) - \langle \nabla f^{(\tau)}(x^*), x^k-x^* \rangle] \\ 
	&& + 4L [f^{(\tau)}(w^k) - f^{(\tau)}(x^*) - \langle \nabla f^{(\tau)}(x^*), w^k-x^* \rangle ], 
\end{eqnarray*}

and 

\begin{eqnarray*}
	&& \mathbb{E}_k \left\| \frac{1}{n} \sum_{\tau=1}^n \left(  \nabla f_{i_k^\tau}^{(\tau)}(x^k) - \nabla f_{i_k^\tau}^{(\tau)}(w^k)  \right) \right\|^2 \\ 
	&=& \mathbb{E}_k \left\| \frac{1}{n} \sum_{\tau=1}^n q^k_\tau \right\|^2 \\
	&=& \frac{1}{n^2} \mathbb{E}_k \left \langle \sum_{\tau=1}^n q^k_{\tau},  \sum_{\tau=1}^n q^k_{\tau} \right \rangle \\ 
	&=& \frac{1}{n^2} \sum_{\tau_1, \tau_2=1}^n \mathbb{E}_k \left\langle  q^k_{\tau_1}, q^k_{\tau_2} \right\rangle \\ 
	&=& \frac{1}{n^2} \sum_{\tau=1}^n \mathbb{E}_k \|q^k_{\tau}\|^2 + \frac{1}{n^2} \sum_{\tau_1 \neq \tau_2}  \left \langle \nabla f^{(\tau_1)}(x^k) - \nabla f^{(\tau_1)}(w^k),  \nabla f^{(\tau_2)}(x^k) - \nabla f^{(\tau_2)}(w^k) \right \rangle \\
	&=& \frac{1}{n^2} \sum_{\tau=1}^n \mathbb{E}_k \|q^k_{\tau}\|^2 +  \| \nabla f(x^k) - \nabla f(w^k) \|^2 - \frac{1}{n^2} \sum_{\tau=1}^n  \|\nabla f^{(\tau)}(x^k) - \nabla f^{(\tau)}(w^k) \|^2 \\ 
	&\leq&  \frac{1}{n^2} \sum_{\tau=1}^n \mathbb{E}_k \|q^k_{\tau}\|^2 + 2 \| \nabla f(x^k) - \nabla f(x^*)\|^2 + 2 \|\nabla f(w^k) - \nabla f(x^*)\|^2 \\ 
	&\leq& \left(  \frac{4L}{n} + 4L_f  \right) [f(x^k) - f(x^*) - \langle \nabla f(x^*), x^k-x^* \rangle ] \\ 
	&& + \left(  \frac{4L}{n} + 4L_f  \right) [f(w^k) - f(x^*) - \langle \nabla f(x^*), w^k-x^* \rangle ], 
\end{eqnarray*}
where we denote $q^k_{\tau} =  \nabla f_{i_k^\tau}^{(\tau)}(x^k) - \nabla f_{i_k^\tau}^{(\tau)}(w^k)$. 

Since $x^*$ is an optimal solution, we have $-\nabla f(x^*) \in \partial \psi(x^*)$, which implies that 
\begin{equation}\label{eq:fP}
f(x^k) - f(x^*) - \langle \nabla f(x^*), x^k-x^* \rangle \leq P(x^k) - P(x^*). 
\end{equation}

Thus, 
$$
\frac{1}{n} \sum_{\tau=1}^n \mathbb{E}_k \left\| \nabla f_{i_k^\tau}^{(\tau)}(x^k) - \nabla f_{i_k^\tau}^{(\tau)}(w^k)  \right\|^2 \leq 4L  [P(x^k) - P(x^*)]  + 4L [P(w^k) - P(x^*)], 
$$
and 
\begin{eqnarray*}
\mathbb{E}_k \left\| \frac{1}{n} \sum_{\tau=1}^n \left(  \nabla f_{i_k^\tau}^{(\tau)}(x^k) - \nabla f_{i_k^\tau}^{(\tau)}(w^k)  \right) \right\|^2 &\leq& 4\left(L_f + \frac{L}{n} \right) [P(x^k) - P(x^*)]   + 4\left(L_f + \frac{L}{n} \right)  [P(w^k) - P(x^*)].  
\end{eqnarray*}

For $ \mathbb{E}_k \left\|\frac{1}{n} \sum_{\tau=1}^n \left(  \nabla f_{i_k^\tau}^{(\tau)}(x^k) - \nabla f_{i_k^\tau}^{(\tau)}(w^k)  \right) + \nabla f(w^k) - \nabla f(x^k) \right\|^2$, we have 
\begin{eqnarray*}
	 &&\mathbb{E}_k \left\|\frac{1}{n} \sum_{\tau=1}^n \left(  \nabla f_{i_k^\tau}^{(\tau)}(x^k) - \nabla f_{i_k^\tau}^{(\tau)}(w^k)  \right) + \nabla f(w^k) - \nabla f(x^k) \right\|^2  \\ 
	 &=& \mathbb{E}_k \left\|\frac{1}{n} \sum_{\tau=1}^n \left(  \nabla f_{i_k^\tau}^{(\tau)}(x^k) - \nabla f_{i_k^\tau}^{(\tau)}(w^k)  \right) \right\|^2 - \|\nabla f(x^k) - \nabla f(w^k)\|^2 \\
	&=& \frac{1}{n^2} \sum_{\tau=1}^n \mathbb{E}_k \|q^k_{\tau}\|^2 - \frac{1}{n^2} \sum_{\tau=1}^n \|\nabla f^{(\tau)}(x^k) - \nabla f^{(\tau)}(w^k) \|^2 \\ 
	&\leq& \frac{1}{n^2} \sum_{\tau=1}^n \mathbb{E}_k \|q^k_{\tau}\|^2 \\ 
	&\leq& \frac{4L}{n} [P(x^k) - P(x^*)]  + \frac{4L}{n} [P(w^k) - P(x^*)].
\end{eqnarray*}

Since $f^{(\tau)}$ is ${\bar L}$-smooth, we have 
$$
\|\nabla f^{(\tau)}(x) - \nabla f^{(\tau)}(y)\|^2 \leq 2{\bar L} (f^{(\tau)}(x) - f^{(\tau)}(y) - \langle \nabla f^{(\tau)}(y), x-y\rangle ). 
$$

Then similarly, we can get
$$
\frac{1}{n} \sum_{\tau=1}^n \| \nabla f^{(\tau)}(x^k) - \nabla f^{(\tau)}(w^k) \|^2 \leq 4{\bar L}  [P(x^k) - P(x^*)]  + 4{\bar L} [P(w^k) - P(x^*)]. 
$$

\end{proof}

The following two lemmas show the evolution of $e^k_\tau$ and $e^k$, which will be used to construct the Lyapulov functions.

\begin{lemma}\label{lm:ek+1-1eclsvrg}
	We have 
	\begin{eqnarray*}
		\frac{1}{n} \sum_{\tau=1}^n \mathbb{E}_k [\|e^{k+1}_{\tau}\|^2] &\leq& \left(  1 - \frac{\delta}{2}  \right) \frac{1}{n} \sum_{\tau=1}^n \|e^k_\tau\|^2 +  \frac{4(1-\delta)\eta^2}{\delta n} \sum_{\tau=1}^n  \|\nabla f^{(\tau)}(w^k) - h^k_\tau\|^2 \\
		&&  + 4(1-\delta) \eta^2 \left(  \frac{4\bar L}{\delta} + L \right) \left(  P(x^k) - P(x^*)  + P(w^k) - P(x^*) \right). 
	\end{eqnarray*}
	
\end{lemma}

\begin{proof}

First, we have 
\begin{eqnarray*}
	&& \mathbb{E}_k [ \|e^{k+1}_\tau\|^2 ] \\ 
	&\overset{(\ref{eq:contractor})}{\leq}& (1-\delta)\mathbb{E}_k \|e^k_\tau + \eta g^k_\tau\|^2 \\
	&=& (1-\delta)\mathbb{E}_k \|e^k_\tau + \eta (\nabla f^{(\tau)}(x^k) - \nabla f^{(\tau)}(w^k) ) + \eta g^k_\tau - \eta (\nabla f^{(\tau)}(x^k) - \nabla f^{(\tau)}(w^k) ) \|^2 \\
	&=& (1-\delta) \mathbb{E}_k \|e^k_\tau +  \eta (\nabla f^{(\tau)}(x^k) - \nabla f^{(\tau)}(w^k) + \nabla f^{(\tau)}(w^k) - h^k_{\tau} )\|^2 \\ 
	&& + (1-\delta)\eta^2 \mathbb{E}_k \| \nabla f_{i_k^\tau}^{(\tau)}(x^k) - \nabla f_{i_k^\tau}^{(\tau)}(w^k)  - (\nabla f^{(\tau)}(x^k) - \nabla f^{(\tau)}(w^k) ) \|^2 \\ 
	&\leq& (1-\delta) \|e^k_\tau +  \eta (\nabla f^{(\tau)}(x^k) - h^k_{\tau} )\|^2 + (1-\delta) \eta^2\mathbb{E}_k \| \nabla f_{i_k^\tau}^{(\tau)}(x^k) - \nabla f_{i_k^\tau}^{(\tau)}(w^k)  \|^2 \\ 
	&\leq& (1-\delta)(1+\beta) \|e^k_\tau\|^2 + (1-\delta)\left(1+ \frac{1}{\beta} \right)\eta^2 \|\nabla f^{(\tau)}(x^k) -h^k_{\tau}\|^2 \\ 
	&& + (1-\delta) \eta^2\mathbb{E}_k \|\nabla f_{i_k^\tau}^{(\tau)}(x^k) - \nabla f_{i_k^\tau}^{(\tau)}(w^k) \|^2  \\ 
	&\leq& \left(1-\frac{\delta}{2} \right) \|e^k_\tau\|^2 + \frac{2(1-\delta)}{\delta}\eta^2 \|\nabla f^{(\tau)}(x^k) - h^k_\tau\|^2 + (1-\delta) \eta^2\mathbb{E}_k \|\nabla f_{i_k^\tau}^{(\tau)}(x^k) - \nabla f_{i_k^\tau}^{(\tau)}(w^k) \|^2, 
\end{eqnarray*}
where we use Young's inequality in the third inequality and choose $\beta = \frac{\delta}{2(1-\delta)}$ when $\delta<1$. When $\delta = 1$, it is easy to see that the above inequality also holds. 

Then from Young's inequality, we can get 

\begin{eqnarray*}
	&& \frac{1}{n} \sum_{\tau=1}^n \mathbb{E}_k [\|e^{k+1}_\tau \|^2 ] \\ 
	&\leq& \left(  1 - \frac{\delta}{2}  \right) \frac{1}{n} \sum_{\tau=1}^n \|e^k_\tau\|^2 +  \frac{4(1-\delta) \eta^2}{\delta n}  \sum_{\tau=1}^n \|\nabla f^{(\tau)}(x^k) - \nabla f^{(\tau)}(w^k) \|^2 \\ 
	&& +  \frac{4(1-\delta)\eta^2}{\delta n} \sum_{\tau=1}^n\|\nabla f^{(\tau)}(w^k) - h^k_\tau\|^2 + (1-\delta) \eta^2 \frac{1}{n} \sum_{\tau=1}^n \mathbb{E}_k \|\nabla f_{i_k^\tau}^{(\tau)}(x^k) - \nabla f_{i_k^\tau}^{(\tau)}(w^k) \|^2 \\ 
	&\overset{(\ref{eq:nftau})}{\leq}& \left(  1 - \frac{\delta}{2}  \right) \frac{1}{n} \sum_{\tau=1}^n \|e^k_\tau\|^2 + \frac{16(1-\delta)\eta^2 {\bar L}}{\delta} \left(  P(x^k) - P(x^*)  + P(w^k) - P(x^*) \right)  \\
	&& +   \frac{4(1-\delta)\eta^2}{\delta n} \sum_{\tau=1}^n \|\nabla f^{(\tau)}(w^k) - h^k_\tau\|^2 + (1-\delta) \eta^2 \frac{1}{n} \sum_{\tau=1}^n \mathbb{E}_k \|\nabla f_{i_k^\tau}^{(\tau)}(x^k) - \nabla f_{i_k^\tau}^{(\tau)}(w^k) \|^2 \\ 
	&\overset{(\ref{eq:gktau2})}{\leq}& \left(  1 - \frac{\delta}{2}  \right) \frac{1}{n} \sum_{\tau=1}^n \|e^k_\tau\|^2 +  \frac{4(1-\delta)\eta^2}{\delta n} \sum_{\tau=1}^n \|\nabla f^{(\tau)}(w^k) - h^k_\tau\|^2 \\
	&&  + 4(1-\delta) \eta^2 \left(  \frac{4\bar L}{\delta} + L \right) \left(  P(x^k) - P(x^*)  + P(w^k) - P(x^*) \right). 
\end{eqnarray*}

\end{proof}

\begin{lemma}\label{lm:ek+1-2eclsvrg}
	Under Assumption \ref{as:expcompressor}, we have 
	\begin{eqnarray*}
		\mathbb{E}_k \|e^{k+1}\|^2 &\leq&   \left(  1 - \frac{\delta}{2}  \right) \|e^k\|^2 + \frac{2(1-\delta)\delta}{n^2} \sum_{\tau=1}^n  \|e^k_{\tau} \|^2  + \frac{4(1-\delta)\delta \eta^2 }{n^2 } \sum_{\tau=1}^n  \| \nabla f^{(\tau)}(w^k) - h^k_\tau \|^2  \\  
		&& + 4(1-\delta)\eta^2 \left(  \frac{4L_f}{\delta} + \frac{5L}{n}  \right) [P(x^k)- P(x^*) + P(w^k) - P(x^*)]  + \frac{4(1-\delta)\eta^2}{\delta }  \|\nabla f(w^k) - h^k \|^2.  
	\end{eqnarray*}
	
\end{lemma}

\begin{proof}

Under Assumption \ref{as:expcompressor}, we have $\mathbb{E}[Q(x)] = \delta x$, and 
\begin{eqnarray*}
	\mathbb{E}_k \|e^{k+1}\|^2 &=& \mathbb{E}_k \left\| \frac{1}{n} \sum_{\tau=1}^n e^{k+1}_{\tau} \right\|^2 \\
	&=& \frac{1}{n^2} \sum_{i, j} \mathbb{E}_k \langle e^{k+1}_i, e^{k+1}_j \rangle \\ 
	&=& \frac{1}{n^2} \sum_{\tau=1}^n \mathbb{E}_k \|e^{k+1}_{\tau}\|^2 + \frac{1}{n^2} \sum_{i\neq j} \mathbb{E}_k \langle e^{k+1}_i, e^{k+1}_j \rangle \\
	&\overset{(\ref{eq:contractor})}{\leq}& \frac{1-\delta}{n^2} \sum_{\tau=1}^n \mathbb{E}_k \left\|e^k_{\tau} + \eta g^k_{\tau} \right\|^2 + \frac{(1-\delta)^2}{n^2} \sum_{i\neq j} \mathbb{E}_k \left\langle e^k_i + \eta  g^k_i , e^k_j + \eta g^k_j \right\rangle \\ 
	&=& \frac{(1-\delta)^2}{n^2} \mathbb{E}_k \left\|\sum_{\tau=1}^n (e^k_{\tau} + \eta g^k_{\tau}) \right\|^2 + \frac{(1-\delta)\delta}{n^2} \sum_{\tau=1}^n \mathbb{E}_k \left\|e^k_{\tau} + \eta g^k_{\tau} \right\|^2 \\ 
	&\leq& (1-\delta) \mathbb{E}_k \left\|e^k + \eta g^k \right\|^2 + \frac{(1-\delta)\delta}{n^2} \sum_{\tau=1}^n \mathbb{E}_k \left\|e^k_{\tau} +\eta g^k_{\tau} \right\|^2, 
\end{eqnarray*}
where we use the definitions of $e^k$ and $g^k$ in the last inequality. Then we can obtain 

\begin{eqnarray}
&& \mathbb{E}_k \|e^{k+1}\|^2 \nonumber \\ 
&\leq&  (1-\delta) \mathbb{E}_k \left\|e^k + \eta g^k \right\|^2 + \frac{(1-\delta)\delta}{n^2} \sum_{\tau=1}^n \mathbb{E}_k \left\|e^k_{\tau} + \eta g^k_{\tau} \right\|^2 \nonumber \\ 
&\leq& (1-\delta) \mathbb{E}_k \left\|e^k + \eta g^k \right\|^2 + \frac{2(1-\delta)\delta}{n^2} \sum_{\tau=1}^n  \|e^k_{\tau} \|^2 + \frac{2(1-\delta)\delta \eta^2 }{n^2 } \sum_{\tau=1}^n \mathbb{E}_k \| g^k_{\tau} \|^2 \nonumber \\ 
&\leq& (1-\delta) \mathbb{E}_k \left\|e^k + \eta g^k \right\|^2 + \frac{2(1-\delta)\delta}{n^2} \sum_{\tau=1}^n \|e^k_{\tau} \|^2 + \frac{4(1-\delta)\delta \eta^2 }{n^2 } \sum_{\tau=1}^n \mathbb{E}_k \| \nabla f_{i_k^\tau}^{(\tau)}(x^k) - \nabla f_{i_k^\tau}^{(\tau)}(w^k) \|^2 \nonumber \\
&& + \frac{4(1-\delta)\delta \eta^2 }{n^2 } \sum_{\tau=1}^n \| \nabla f^{(\tau)}(w^k) - h^k_\tau \|^2 \nonumber \\
&\overset{(\ref{eq:gktau2})}{\leq}& (1-\delta) \mathbb{E}_k \left\|e^k + \eta g^k \right\|^2 + \frac{2(1-\delta)\delta}{n^2} \sum_{\tau=1}^n \|e^k_{\tau} \|^2  + \frac{4(1-\delta)\delta \eta^2 }{n^2 } \sum_{\tau=1}^n  \| \nabla f^{(\tau)}(w^k) - h^k_\tau \|^2  \nonumber \\ 
&& + \frac{16(1-\delta)\delta L\eta^2}{n}  [P(x^k) - P(x^*) + P(w^k) - P(x^*)], \label{eq:ek+1in}
\end{eqnarray}
where in the second and third inequalities we use the Young's inequality. 

For $(1-\delta)\mathbb{E}_k \|e^k + \eta g^k\|^2$, we have 
\begin{eqnarray*}
	&& (1-\delta)\mathbb{E}_k \left\|e^k + \eta g^k \right\|^2 \\ 
	&=& (1-\delta) \mathbb{E}_k \left\|e^k +\eta (\nabla f(x^k) - \nabla f(w^k)) + \eta g^k - \eta (\nabla f(x^k) - \nabla f(w^k)) \right\|^2 \\ 
	&=& (1-\delta) \mathbb{E}_k \left\|e^k + \eta (\nabla f(x^k) - h^k) \right\|^2 \\ 
	&& + (1-\delta)\eta^2 \mathbb{E}_k \left\|  \frac{1}{n} \sum_{\tau=1}^n \left(  \nabla f_{i_k^\tau}^{(\tau)}(x^k) - \nabla f_{i_k^\tau}^{(\tau)}(w^k)  \right) - (\nabla f(x^k) - \nabla f(w^k)) \right\|^2 \\ 
	&\leq& \left(  1 - \frac{\delta}{2}  \right) \|e^k\|^2 + \frac{2(1-\delta)\eta^2}{\delta }  \|\nabla f(x^k) - h^k \|^2 \\ 
	&& + (1-\delta)\eta^2 \mathbb{E}_k \left\| \frac{1}{n} \sum_{\tau=1}^n \left(  \nabla f_{i_k^\tau}^{(\tau)}(x^k) - \nabla f_{i_k^\tau}^{(\tau)}(w^k)  \right) - (\nabla f(x^k) - \nabla f(w^k)) \right\|^2 \\ 
	&\overset{(\ref{eq:gk2variance})}{\leq}&  \left(  1 - \frac{\delta}{2}  \right) \|e^k\|^2 + \frac{2(1-\delta)\eta^2}{\delta }  \|\nabla f(x^k) - h^k \|^2  +  (1-\delta)\frac{4L\eta^2}{n} [P(x^k) - P(x^*) + P(w^k) - P(x^*)]. 
\end{eqnarray*}

Since $f$ is $L_f$-smooth, we have 

\begin{eqnarray}
\|\nabla f(x^k) - \nabla f(w^k)\|^2 &=&  \|\nabla f(x^k) - \nabla f(x^*) + \nabla f(x^*) - \nabla f(w^k)\|^2 \nonumber \\
&\leq&  2 \|\nabla f(x^k) - \nabla f(x^*)\|^2 + 2 \| \nabla f(w^k) - \nabla f(x^*)\|^2 \nonumber \\ 
&\leq& 4L_f \left[ f(x^k) - f(x^*) - \langle \nabla f(x^*), x^k-x^* \rangle  \right] \nonumber \\
&& + 4L_f \left[  f(w^k) - f(x^*) - \langle \nabla f(x^*), w^k-x^* \rangle  \right] \nonumber \\
&\overset{(\ref{eq:fP})}{\leq}& 4L_f [P(x^k)- P(x^*)] + 4L_f [P(w^k) - P(x^*)], \label{eq:nablaf}
\end{eqnarray}

which implies that 

\begin{eqnarray*}
	 \|\nabla f(x^k) - h^k \|^2 &\leq& 2 \| \nabla f(x^k) - \nabla f(w^k)\|^2 + 2 \| \nabla f(w^k) - h^k\|^2  \\
	&\leq& 8L_f [P(x^k)- P(x^*)] + 8L_f [P(w^k) - P(x^*)] + 2 \| \nabla f(w^k) - h^k\|^2. 
\end{eqnarray*}

Hence, we arrive at 
\begin{eqnarray*}
	(1-\delta)\mathbb{E}_k \left\|e^k + \eta g^k \right\|^2 &\leq&  \left(  1 - \frac{\delta}{2}  \right) \|e^k\|^2  + \frac{4(1-\delta)\eta^2}{\delta } \|\nabla f(w^k) - h^k \|^2 \\ 
	&& + 4(1-\delta)\eta^2 \left(  \frac{4L_f}{\delta} + \frac{L}{n}  \right) [P(x^k)- P(x^*) + P(w^k) - P(x^*)]. 
\end{eqnarray*}

Combining (\ref{eq:ek+1in}) and the above inequality, we can get 

\begin{eqnarray*}
	&& \mathbb{E}_k \|e^{k+1}\|^2 \\ 
	&\leq&  \left(  1 - \frac{\delta}{2}  \right) \|e^k\|^2 + \frac{2(1-\delta)\delta}{n^2} \sum_{\tau=1}^n  \|e^k_{\tau} \|^2  + \frac{4(1-\delta)\delta \eta^2 }{n^2 } \sum_{\tau=1}^n  \| \nabla f^{(\tau)}(w^k) - h^k_\tau \|^2  \\ 
	&& + 4(1-\delta)\eta^2 \left(  \frac{4L_f}{\delta} + \frac{L}{n} + \frac{4\delta L}{n}  \right) [P(x^k)- P(x^*) + P(w^k) - P(x^*)]  + \frac{4(1-\delta)\eta^2}{\delta }  \|\nabla f(w^k) - h^k \|^2\\ 
	&\leq&  \left(  1 - \frac{\delta}{2}  \right) \|e^k\|^2 + \frac{2(1-\delta)\delta}{n^2} \sum_{\tau=1}^n  \|e^k_{\tau} \|^2  + \frac{4(1-\delta)\delta \eta^2 }{n^2 } \sum_{\tau=1}^n \| \nabla f^{(\tau)}(w^k) - h^k_\tau \|^2  \\  
	&& + 4(1-\delta)\eta^2 \left(  \frac{4L_f}{\delta} + \frac{5L}{n}  \right) [P(x^k)- P(x^*) + P(w^k) - P(x^*)]  + \frac{4(1-\delta)\eta^2}{\delta } \|\nabla f(w^k) - h^k \|^2. 
\end{eqnarray*}

\end{proof}

\begin{lemma}\label{lm:hk+1-1eclsvrg}
	We have 
	\begin{eqnarray*}
		\frac{1}{n} \sum_{\tau=1}^n \mathbb{E}_k [\| h^{k+1}_\tau - \nabla f^{(\tau)}(w^{k+1}) \|^2] &\leq& \left(  1 - \frac{\delta_1}{2}  \right) \frac{1}{n} \sum_{\tau=1}^n [\|h^k_\tau - \nabla f^{(\tau)}(w^k)\|^2 ] \\ 
		&& + 4{\bar L}p \left(1 + \frac{2p}{\delta_1} \right) [P(x^k)- P(x^*) + P(w^k) - P(x^*)]. 
	\end{eqnarray*}
	
\end{lemma}

\begin{proof}

First, from the update rule of $w^k$, we have 
\begin{eqnarray*}
 && \mathbb{E}_k [\| h^{k+1}_\tau - \nabla f^{(\tau)}(w^{k+1}) \|^2] \\ 
 &=& p\mathbb{E}_k [\|h^{k+1}_\tau - \nabla f^{(\tau)}(x^k)\|^2 ] + (1-p)\mathbb{E}_k [\|h^{k+1}_\tau - \nabla f^{(\tau)}(w^k)\|^2] \\ 
 &\leq& p\left(1 + \frac{2p}{\delta_1} \right) \mathbb{E}_k [\| \nabla f^{(\tau)}(x^k) - \nabla f^{(\tau)}(w^k)\|^2] + p\left(  1 + \frac{\delta_1}{2p}  \right)\mathbb{E}_k [\|h^{k+1}_\tau - \nabla f^{(\tau)}(w^k)\|^2 ] \\ 
 && + (1-p)\mathbb{E}_k [\|h^{k+1}_\tau - \nabla f^{(\tau)}(w^k)\|^2 ] \\ 
 &=& p\left(1 + \frac{2p}{\delta_1} \right) [\| \nabla f^{(\tau)}(x^k) - \nabla f^{(\tau)}(w^k)\|^2] +\left(  1 + \frac{\delta_1}{2}  \right) \mathbb{E}_k [\|h^{k+1}_\tau - \nabla f^{(\tau)}(w^k)\|^2 ] \\ 
 &\leq& p\left(1 + \frac{2p}{\delta_1} \right) [\| \nabla f^{(\tau)}(x^k) - \nabla f^{(\tau)}(w^k)\|^2] + \left(  1 - \frac{\delta_1}{2}  \right) [\|h^k_\tau - \nabla f^{(\tau)}(w^k)\|^2 ], 
\end{eqnarray*}
where the first inequality comes from the Young's inequality and the last inequality comes from the contraction property of $Q_1$. 

Then we can obtain 
\begin{eqnarray*}
&& \frac{1}{n} \sum_{\tau=1}^n \mathbb{E}_k [\| h^{k+1}_\tau - \nabla f^{(\tau)}(w^{k+1}) \|^2] \\
&\leq& \frac{p}{n}\left(1 + \frac{2p}{\delta_1} \right) \sum_{\tau=1}^n [\| \nabla f^{(\tau)}(x^k) - \nabla f^{(\tau)}(w^k)\|^2] + \left(  1 - \frac{\delta_1}{2}  \right) \frac{1}{n} \sum_{\tau=1}^n [\|h^k_\tau - \nabla f^{(\tau)}(w^k)\|^2 ] \\ 
&\overset{(\ref{eq:nftau})}{\leq}&  \left(  1 - \frac{\delta_1}{2}  \right) \frac{1}{n} \sum_{\tau=1}^n  [\|h^k_\tau - \nabla f^{(\tau)}(w^k)\|^2 ] +  4{\bar L}p \left(1 + \frac{2p}{\delta_1} \right) [P(x^k)- P(x^*) + P(w^k) - P(x^*)]. 
\end{eqnarray*}

\end{proof}

\begin{lemma}\label{lm:hk+1-2eclsvrg}
	Under Assumption \ref{as:expcompressor}, we have 
	\begin{eqnarray*}
	\mathbb{E}_k \|h^{k+1} - \nabla f^{(\tau)}(w^{k+1}) \|^2 &\leq& \left(  1 - \frac{\delta_1}{2}  \right) \|h^k - \nabla f(w^k)\|^2 + \frac{(1- \frac{\delta_1}{2})\delta_1}{n^2} \sum_{\tau=1}^n \| h^k_\tau - \nabla f^{(\tau)}(w^k)\|^2 \\ 
	&& + 4pL_f \left(  1 + \frac{2p}{\delta_1}  \right)  [P(x^k)- P(x^*) + P(w^k) - P(x^*)]. 
	\end{eqnarray*}
	
\end{lemma}

\begin{proof}

First, from the update rule of $w^k$, we can obtain 

\begin{eqnarray*}
&& \mathbb{E}_k \|h^{k+1} - \nabla f(w^{k+1})\|^2 \\
&=& p\mathbb{E}_k \|h^{k+1} -  \nabla f(x^k)\|^2 + (1-p)\mathbb{E}_k \|h^{k+1} - \nabla f(w^k)\|^2 \\
&\leq& p\left(  1 + \frac{2p}{\delta_1}  \right)\|\nabla f(x^k) - \nabla f(w^k) \|^2 + p\left(  1 + \frac{\delta_1}{2p}  \right)\mathbb{E}_k \| h^{k+1} - \nabla f(w^k)\|^2 + (1-p)\mathbb{E}_k \|h^{k+1} - \nabla f(w^k)\|^2 \\ 
&=&  p\left(  1 + \frac{2p}{\delta_1}  \right) \|\nabla f(x^k) - \nabla f(w^k) \|^2 + \left(  1 + \frac{\delta_1}{2}  \right) \mathbb{E}_k \|h^{k+1} - \nabla f(w^k)\|^2. 
\end{eqnarray*}

For $\mathbb{E}_k \|h^{k+1} - \nabla f(w^k)\|^2$, under Assumption \ref{as:expcompressor}, same as $\mathbb{E}_k \|e^{k+1}\|^2$, we have 
\begin{eqnarray*}
\mathbb{E}_k \|h^{k+1} - \nabla f(w^k)\|^2 \leq (1-\delta_1) \|h^k - \nabla f(w^k)\|^2 + \frac{(1-\delta_1)\delta_1}{n^2} \sum_{\tau=1}^n \| h^k_\tau - \nabla f^{(\tau)}(w^k)\|^2. 
\end{eqnarray*}

Hence, we arrive at 
\begin{eqnarray*}
&& \mathbb{E}_k \|h^{k+1} - \nabla f(w^{k+1})\|^2 \\
&\leq& p\left(  1 + \frac{2p}{\delta_1}  \right) \|\nabla f(x^k) - \nabla f(w^k) \|^2 + \left(  1 - \frac{\delta_1}{2}  \right) \|h^k - \nabla f(w^k)\|^2 \\ 
&& + \frac{(1- \frac{\delta_1}{2})\delta_1}{n^2} \sum_{\tau=1}^n \| h^k_\tau - \nabla f^{(\tau)}(w^k)\|^2 \\ 
&\overset{(\ref{eq:nablaf})}{\leq}&  \left(  1 - \frac{\delta_1}{2}  \right) \|h^k - \nabla f(w^k)\|^2 + \frac{(1- \frac{\delta_1}{2})\delta_1}{n^2} \sum_{\tau=1}^n \| h^k_\tau - \nabla f^{(\tau)}(w^k)\|^2 \\ 
&& + 4pL_f \left(  1 + \frac{2p}{\delta_1}  \right) [P(x^k)- P(x^*) + P(w^k) - P(x^*)]. 
\end{eqnarray*}

\end{proof}

\subsection{Proof of Theorem \ref{th:eclsvrg-1}}

Let $\eta \leq \frac{1}{4L_f}$. From $\|e^k\|^2 \leq \frac{1}{n} \sum_{\tau=1}^n \|e^k_\tau\|^2$ and Lemma \ref{lm:itereclsvrg}, we have 

\begin{eqnarray*}
	&& \left(1 + \frac{\eta \mu}{2} \right)\mathbb{E}_k \|\tilde{x}^{k+1} - x^*\|^2 - \|\tilde{x}^k - x^*\|^2 - 2\eta \mathbb{E}_k (P(x^*) - P(x^{k+1}))  \\ 
	&\leq& \|e^k\|^2  + (1 + \eta \mu) \mathbb{E}_k \|e^{k+1}\|^2 + 4\eta^2\mathbb{E}_k \|g^k + h^k - \nabla f(x^k) \|^2 \\ 
	&\leq&  \frac{1}{n} \sum_{\tau=1}^n \|e^k_\tau\|^2 +  \frac{5}{4n} \sum_{\tau=1}^n \mathbb{E}_k \|e^{k+1}_\tau\|^2 + 4\eta^2\mathbb{E}_k \|g^k + h^k - \nabla f(x^k)\|^2 \\ 
	&\overset{Lemma \ref{lm:ek+1-1eclsvrg}}{\leq}& \frac{1}{n} \sum_{\tau=1}^n \|e^k_\tau\|^2 + \frac{5}{4} \left(  1 - \frac{\delta}{2}  \right) \frac{1}{n} \sum_{\tau=1}^n \|e^k_\tau\|^2 + \frac{5(1-\delta)\eta^2}{\delta n} \sum_{\tau=1}^n  \|\nabla f^{(\tau)}(w^k) - h^k_\tau\|^2 \\ 
	&& + 5(1-\delta)\eta^2 \left(  \frac{{4\bar L}}{\delta} + L \right) [P(x^k) - P(x^*) + P(w^k) - P(x^*)] + 4\eta^2\mathbb{E}_k \|g^k  + h^k - \nabla f(x^k)\|^2 \\ 
	&\overset{(\ref{eq:gk2variance})}{\leq}&  \frac{9}{4} \cdot \frac{1}{n} \sum_{\tau=1}^n \|e^k_\tau\|^2 + \frac{5(1-\delta)\eta^2}{\delta n} \sum_{\tau=1}^n  \|\nabla f^{(\tau)}(w^k) - h^k_\tau\|^2 \\ 
	&& + \left(  5(1-\delta)\left(  \frac{4{\bar L}}{\delta} + L \right) + \frac{16L}{n}  \right) \eta^2  [P(x^k) - P(x^*) + P(w^k) - P(x^*)]. 
\end{eqnarray*}

From the definition of $w^{k+1}$, we have 

\begin{equation}\label{eq:wk+1}
\mathbb{E}_k [P(w^{k+1}) - P(x^*)] = p [P(x^k) - P(x^*)] + (1-p)  [P(w^k) - P(x^*)]. 
\end{equation}

From Lemma \ref{lm:ek+1-1eclsvrg}, we have 

\begin{eqnarray*}
	&& \left(1 + \frac{\eta \mu}{2} \right)\mathbb{E}_k \|\tilde{x}^{k+1} - x^*\|^2   - \|\tilde{x}^k - x^*\|^2 - 2\eta \mathbb{E}_k [P(x^*) - P(x^{k+1})] + \frac{9}{\delta n} \sum_{\tau=1}^n\mathbb{E}_k \|e^{k+1}_{\tau}\|^2 \\ 
	&\leq&  \frac{9}{\delta n}\left(  1 -  \frac{\delta}{4}  \right) \sum_{\tau=1}^n \|e^k_{\tau}\|^2 + \frac{41(1-\delta)\eta^2}{\delta^2 n}  \sum_{\tau=1}^n  \|\nabla f^{(\tau)}(w^k) - h^k_\tau\|^2 \\ 
	&& + \left(  \frac{41(1-\delta)}{\delta} \left(  \frac{4{\bar L}}{\delta} + L \right) + \frac{16L}{n}  \right) \eta^2 [P(x^k) - P(x^*) + P(w^k) - P(x^*)]. 
\end{eqnarray*} 

Then from Lemma \ref{lm:hk+1-1eclsvrg}, we have 

\begin{eqnarray*}
	&& \left(1 + \frac{\eta \mu}{2} \right)\mathbb{E}_k \|\tilde{x}^{k+1} - x^*\|^2   - \|\tilde{x}^k - x^*\|^2 - 2\eta \mathbb{E}_k(P(x^*) - P(x^{k+1})) \\ 
	&& + \frac{9}{\delta n} \sum_{\tau=1}^n\mathbb{E}_k \|e^{k+1}_{\tau}\|^2 + \frac{164(1-\delta)\eta^2}{\delta^2 \delta_1 n} \sum_{\tau=1}^n \mathbb{E}_k \| h^{k+1}_\tau - \nabla f^{(\tau)}(w^{k+1})\|^2  \\ 
	&\leq&  \frac{9}{\delta n}\left(  1 -  \frac{\delta}{4}  \right) \sum_{\tau=1}^n \|e^k_{\tau}\|^2 + \frac{164(1-\delta)\eta^2}{\delta^2 \delta_1 n} \left(  1 - \frac{\delta_1}{4}  \right)  \sum_{\tau=1}^n  \| h^k_\tau - \nabla f^{(\tau)}(w^k) \|^2 \\ 
	&& + \left(  \frac{41(1-\delta)}{\delta} \left(  \frac{4{\bar L}}{\delta} + L \right) + \frac{16L}{n} + \frac{656(1-\delta){\bar L}p}{\delta^2 \delta_1} \left(  1 + \frac{2p}{\delta_1}  \right) \right) \eta^2 [P(x^k) - P(x^*) + P(w^k) - P(x^*)]. 
\end{eqnarray*}

Combining the above inequality and (\ref{eq:wk+1}), we can get 

\begin{eqnarray*}
	 \mathbb{E}_k [{\Phi}^{k+1}_1] &=& \left(1 + \frac{\eta \mu}{2} \right)\mathbb{E}_k \|\tilde{x}^{k+1} - x^*\|^2 + \frac{9}{\delta n} \sum_{\tau=1}^n\mathbb{E}_k \|e^{k+1}_{\tau}\|^2 + \frac{164(1-\delta)\eta^2}{\delta^2 \delta_1 n} \sum_{\tau=1}^n \mathbb{E}_k \| h^{k+1}_\tau - \nabla f^{(\tau)}(w^{k+1})\|^2 \\ 
	 && + \frac{4\eta^2}{3p} \left(  \frac{41(1-\delta)}{\delta} \left(  \frac{4{\bar L}}{\delta} + L  + \frac{16{\bar L}p}{\delta \delta_1} \left(  1 + \frac{2p}{\delta_1}  \right) \right) + \frac{16L}{n} \right) \mathbb{E}_k [P(w^{k+1}) - P(x^*)] \\ 
	&\leq&  \|\tilde{x}^k - x^*\|^2 + \frac{9}{\delta n}\left(  1 -  \frac{\delta}{4}  \right) \sum_{\tau=1}^n \|e^k_{\tau}\|^2  + \frac{164(1-\delta)\eta^2}{\delta^2 \delta_1 n} \left(  1 - \frac{\delta_1}{4}  \right)  \sum_{\tau=1}^n \| h^k_\tau - \nabla f^{(\tau)}(w^k) \|^2 \\ 
	&& + \frac{4\eta^2}{3p} \left(  \frac{41(1-\delta)}{\delta} \left(  \frac{4{\bar L}}{\delta} + L  + \frac{16{\bar L}p}{\delta \delta_1} \left(  1 + \frac{2p}{\delta_1}  \right) \right) + \frac{16L}{n} \right) \left(  1- \frac{p}{4}  \right)  [P(w^{k}) - P(x^*)] \\ 
	&& + 2\eta \mathbb{E}_k [P(x^*) - P(x^{k+1})] \\ 
	&& +  \left(  \frac{96(1-\delta)}{\delta} \left(  \frac{4{\bar L}}{\delta} + L  + \frac{16{\bar L}p}{\delta \delta_1} \left(  1 + \frac{2p}{\delta_1}  \right) \right) + \frac{38L}{n}  \right) \eta^2 [P(x^k) - P(x^*)] \\ 
	&\leq& \left(  1 - \min\left\{  \frac{\mu \eta}{3}, \frac{\delta}{4}, \frac{\delta_1}{4}, \frac{p}{4}  \right\}  \right) \Phi^k_1 + 2\eta \mathbb{E}_k [P(x^*) - P(x^{k+1})] \\ 
	&& + \left(  \frac{96(1-\delta)}{\delta} \left(  \frac{4{\bar L}}{\delta} + L  + \frac{16{\bar L}p}{\delta \delta_1} \left(  1 + \frac{2p}{\delta_1}  \right) \right) + \frac{38L}{n}  \right) \eta^2 [P(x^k) - P(x^*)], 
\end{eqnarray*}
where we use $\left(  1 + \frac{\eta \mu}{2}  \right)^{-1} \leq 1 - \frac{\mu \eta}{3}$ for $\mu \eta<1$. Taking expectation again and applying the tower property, we can get the result.

\subsection{Proof of Theorem \ref{th:eclsvrg-2}} 

Let $\eta \leq \frac{1}{4L_f}$. From Lemma \ref{lm:itereclsvrg}, we have 

\begin{eqnarray*}
	&& \left(1 + \frac{\eta \mu}{2} \right)\mathbb{E}_k \|\tilde{x}^{k+1} - x^*\|^2 - \|\tilde{x}^k - x^*\|^2 - 2\eta \mathbb{E}_k (P(x^*) - P(x^{k+1}))\\ 
	&\leq&  \|e^k\|^2  + (1 + \eta \mu) \mathbb{E}_k \|e^{k+1}\|^2 + 4\eta^2\mathbb{E}_k \|g^k + h^k - \nabla f(x^k) \|^2 \\ 
	&\leq& \|e^k\|^2  + \frac{5}{4} \mathbb{E}_k \|e^{k+1}\|^2 + 4\eta^2\mathbb{E}_k \|g^k + h^k - \nabla f(x^k)\|^2 \\ 
	&\overset{Lemma \ref{lm:ek+1-2eclsvrg}}{\leq}&  \frac{9}{4} \|e^k\|^2 + \frac{5(1-\delta)\delta}{2n^2} \sum_{\tau=1}^n  \|e^k_\tau\|^2 + \frac{5(1-\delta)\delta \eta^2}{n^2} \sum_{\tau=1}^n \|h^k_\tau - \nabla f^{(\tau)}(w^k)\|^2 \\ 
	&& + \frac{5(1-\delta)\eta^2}{\delta} \|h^k - \nabla f(w^k)\|^2 + 4\eta^2 \mathbb{E}_k \|g^k + h^k -\nabla f(x^k)\|^2 \\ 
	&& + 5(1-\delta)\eta^2 \left(  \frac{4L_f}{\delta} + \frac{5L}{n}  \right) [P(x^k)- P(x^*) + P(w^k) - P(x^*)] \\ 
	&\overset{(\ref{eq:gk2variance})}{\leq}&  \frac{9}{4} \|e^k\|^2 + \frac{5(1-\delta)\delta}{2n^2} \sum_{\tau=1}^n \|e^k_\tau\|^2  + \frac{5(1-\delta)\delta \eta^2}{n^2} \sum_{\tau=1}^n \|h^k_\tau - \nabla f^{(\tau)}(w^k)\|^2 \\ 
	&& + \frac{5(1-\delta)\eta^2}{\delta} \|h^k - \nabla f(w^k)\|^2  \\ 
	&& + \left(  5(1-\delta) \left(  \frac{4L_f}{\delta} + \frac{5L}{n}  \right) + \frac{16L}{n} \right) \eta^2 [P(x^k)- P(x^*) + P(w^k) - P(x^*)]. 
\end{eqnarray*}

Then from Lemmas \ref{lm:ek+1-1eclsvrg} and \ref{lm:ek+1-2eclsvrg}, we have 

\begin{eqnarray*}
	&& \left(1 + \frac{\eta \mu}{2} \right)\mathbb{E}_k \|\tilde{x}^{k+1} - x^*\|^2 + \frac{9}{\delta} \mathbb{E}_k \|e^{k+1}\|^2 + \frac{84(1-\delta)}{\delta n^2} \sum_{\tau=1}^n \mathbb{E}_k \|e^{k+1}_\tau\|^2\\ 
	&&  - \mathbb{E}\|\tilde{x}^k - x^*\|^2 - 2\eta \mathbb{E}[P(x^*) - P(x^{k+1})]\\ 
	&\overset{lemma \ref{lm:ek+1-2eclsvrg}}{\leq}&  \left(  1 - \frac{\delta}{2} + \frac{\delta}{4}  \right) \frac{9}{\delta} \|e^k\|^2  + \frac{41(1-\delta) \eta^2}{n^2} \sum_{\tau=1}^n \|h^k_\tau - \nabla f^{(\tau)}(w^k)\|^2 + \frac{41(1-\delta)\eta^2}{\delta^2} \|h^k - \nabla f(w^k)\|^2   \\ 
	&& + \left(  \frac{18(1-\delta)}{n^2} + \frac{5(1-\delta)\delta}{2n^2} \right) \sum_{\tau=1}^n \|e^k_\tau\|^2 + \frac{84(1-\delta)}{\delta n^2} \sum_{\tau=1}^n \mathbb{E}_k \|e^{k+1}_\tau\|^2 \\ 
	&& +  \left( \frac{41(1-\delta)}{\delta} \left(  \frac{4L_f}{\delta} + \frac{5L}{n}  \right)  + \frac{16L}{n} \right) \eta^2 [P(x^k)- P(x^*) + P(w^k) - P(x^*)] \\ 
	&\overset{Lemma \ref{lm:ek+1-1eclsvrg}}{\leq}&  \left(  1 -  \frac{\delta}{4}  \right) \frac{9}{\delta} \|e^k\|^2 + \left(  1 - \frac{\delta}{4}  \right) \frac{84(1-\delta)}{\delta n^2} \sum_{\tau=1}^n \|e^{k}_\tau\|^2 + \frac{336(1-\delta) \eta^2}{\delta^2 n^2}\sum_{\tau=1}^n \|h^k_\tau - \nabla f^{(\tau)}(w^k)\|^2  \\ 
	&& + \frac{41(1-\delta)\eta^2}{\delta^2} \|h^k - \nabla f(w^k)\|^2  + \left(  \frac{41(1-\delta)}{\delta} \left(  \frac{4L_f}{\delta} + \frac{5L}{n}  \right) + \frac{16L}{n} + \frac{336(1-\delta)^2}{\delta n} \left(  \frac{4{\bar L}}{\delta} + L  \right)  \right) \\ 
	&& \cdot \eta^2  [P(x^k)- P(x^*) + P(w^k) - P(x^*)]\\ 
	&\leq&  \left(  1 -  \frac{\delta}{4}  \right) \frac{9}{\delta} \|e^k\|^2 + \left(  1 - \frac{\delta}{4}  \right) \frac{84(1-\delta)}{\delta n^2} \sum_{\tau=1}^n \|e^{k}_\tau\|^2 + \frac{336(1-\delta)\eta^2}{\delta^2 n^2}\sum_{\tau=1}^n \|h^k_\tau - \nabla f^{(\tau)}(w^k)\|^2  \\ 
	&& + \frac{41(1-\delta)\eta^2}{\delta^2} \|h^k - \nabla f(w^k)\|^2  \\ 
	&& + \left(  \frac{(1-\delta)}{\delta} \left(  \frac{164L_f}{\delta} + \frac{1344{\bar L}}{\delta n} + \frac{541L}{n}  \right) + \frac{16L}{n}  \right)  \eta^2 [P(x^k)- P(x^*) + P(w^k) - P(x^*)]. 
\end{eqnarray*}

Then from Lemma \ref{lm:hk+1-1eclsvrg} and Lemma \ref{lm:hk+1-2eclsvrg}, we can get 
\begin{eqnarray*}
	&& \left(1 + \frac{\eta \mu}{2} \right)\mathbb{E}_k \|\tilde{x}^{k+1} - x^*\|^2 + \frac{9}{\delta} \mathbb{E}_k \|e^{k+1}\|^2 + \frac{84(1-\delta)}{\delta n^2} \sum_{\tau=1}^n \mathbb{E}_k \|e^{k+1}_\tau\|^2 \\
	&& + \frac{164(1-\delta)\eta^2}{\delta^2 \delta_1} \mathbb{E}_k \|h^{k+1} - \nabla f(w^{k+1})\|^2 + \frac{2000(1-\delta)\eta^2}{\delta^2 \delta_1 n^2} \sum_{\tau=1}^n \mathbb{E}_k \|h^{k+1}_\tau - \nabla f^{(\tau)}(w^{k+1})\|^2 \\ 
	&& - \|\tilde{x}^k - x^*\|^2 - 2\eta \mathbb{E}_k [P(x^*) - P(x^{k+1})]\\ 
	&\overset{Lemma~\ref{lm:hk+1-2eclsvrg}}{\leq}& \left(  1 -  \frac{\delta}{4}  \right) \frac{9}{\delta} \|e^k\|^2 + \left(  1 - \frac{\delta}{4}  \right) \frac{84(1-\delta)}{\delta n^2} \sum_{\tau=1}^n \|e^{k}_\tau\|^2  + \left(  1 - \frac{\delta_1}{4}  \right) \frac{164(1-\delta)\eta^2}{\delta^2 \delta_1} \|h^k - \nabla f(w^k)\|^2 \\ 
	&& + \frac{500(1-\delta)\eta^2}{\delta^2 n^2}\sum_{\tau=1}^n \|h^k_\tau - \nabla f^{(\tau)}(w^k)\|^2 + \frac{2000(1-\delta)\eta^2}{\delta^2 \delta_1 n^2} \sum_{\tau=1}^n \mathbb{E}_k \|h^{k+1}_\tau - \nabla f^{(\tau)}(w^{k+1})\|^2  \\ 
	&& + \left(  \frac{(1-\delta)}{\delta} \left(  \frac{164L_f}{\delta} + \frac{1344{\bar L}}{\delta n} + \frac{541L}{n}  \right) + \frac{16L}{n} + \frac{656(1-\delta)pL_f}{\delta^2 \delta_1}\left(  1 + \frac{2p}{\delta_1}  \right) \right) \\ 
	&& \cdot  \eta^2  [P(x^k)- P(x^*) + P(w^k) - P(x^*)] \\ 
	&\overset{Lemma~\ref{lm:hk+1-1eclsvrg}}{\leq}& \left(  1 -  \frac{\delta}{4}  \right) \frac{9}{\delta} \|e^k\|^2 + \left(  1 - \frac{\delta}{4}  \right) \frac{84(1-\delta)}{\delta n^2} \sum_{\tau=1}^n \|e^{k}_\tau\|^2  + \left(  1 - \frac{\delta_1}{4}  \right) \frac{164(1-\delta)\eta^2}{\delta^2 \delta_1} \|h^k - \nabla f(w^k)\|^2 \\ 
	&& + \left(  1 - \frac{\delta_1}{4}  \right) \frac{2000(1-\delta)\eta^2}{\delta^2 \delta_1 n^2} \sum_{\tau=1}^n \|h^{k}_\tau - \nabla f^{(\tau)}(w^{k})\|^2 \\ 
	&& + \left(  \frac{(1-\delta)}{\delta} \left(  \frac{164L_f}{\delta} + \frac{1344{\bar L}}{\delta n} + \frac{541L}{n}  \right) + \frac{16L}{n} + \frac{(656L_f + \frac{8000{\bar L}}{n})(1-\delta)p}{\delta^2 \delta_1}\left(  1 + \frac{2p}{\delta_1}  \right) \right) \\ 
	&& \cdot  \eta^2 [P(x^k)- P(x^*) + P(w^k) - P(x^*)]. 
\end{eqnarray*}

Combining the above inequality and (\ref{eq:wk+1}), we can obtain 

\begin{eqnarray*}
	&& \mathbb{E}_k [{\Phi_2^{k+1}}] \\ 
	&=& \left(1 + \frac{\eta \mu}{2} \right)\mathbb{E}_k \|\tilde{x}^{k+1} - x^*\|^2 + \frac{9}{\delta} \mathbb{E}_k \|e^{k+1}\|^2 + \frac{84(1-\delta)}{\delta n^2} \sum_{\tau=1}^n \mathbb{E}_k\|e^{k+1}_\tau\|^2 \\ 
	&&  + \frac{164(1-\delta)\eta^2}{\delta^2 \delta_1} \mathbb{E}_k \|h^{k+1} - \nabla f(w^{k+1})\|^2 + \frac{2000(1-\delta)\eta^2}{\delta^2 \delta_1 n^2} \sum_{\tau=1}^n \mathbb{E}_k \|h^{k+1}_\tau - \nabla f^{(\tau)}(w^{k+1})\|^2 \\ 
	&& + \frac{4\eta^2}{3p}  \left(  \frac{(1-\delta)}{\delta} \left(  \frac{164L_f}{\delta} + \frac{1344{\bar L}}{\delta n} + \frac{541L}{n}  + \frac{(656L_f + \frac{8000{\bar L}}{n})p}{\delta \delta_1}\left(  1 + \frac{2p}{\delta_1}  \right)  \right) + \frac{16L}{n} \right) \\ 
	&& \cdot \mathbb{E}_k [P(w^{k+1}) - P(x^*)] \\ 
	&\leq& \left(  1 - \min\left\{  \frac{\mu \eta}{3}, \frac{\delta}{4}, \frac{\delta_1}{4}, \frac{p}{4}  \right\}  \right) {\Phi}_2^k + 2\eta \mathbb{E}_k [P(x^*) - P(x^{k+1})] \\ 
	&& +  \left(  \frac{(1-\delta)}{\delta} \left(  \frac{383L_f}{\delta} + \frac{3136{\bar L}}{\delta n} + \frac{1263L}{n} + \frac{(1531L_f + \frac{18667{\bar L}}{n})p}{\delta \delta_1}\left(  1 + \frac{2p}{\delta_1}  \right)  \right) + \frac{38L}{n} \right)  \eta^2 [P(x^k)- P(x^*) ], 
\end{eqnarray*}
where we use $\left(  1 + \frac{\eta \mu}{2}  \right)^{-1} \leq 1 - \frac{\mu \eta}{3}$ for $\mu \eta<1$. Taking expectation again and applying the tower property, we can have the result.

\subsection{Proof of Theorem \ref{th:eclsvrg-11}}	

Let $\eta \leq \frac{1}{4L_f}$. From Theorem \ref{th:eclsvrg-1}, we have 

\begin{eqnarray*}
	&& \mathbb{E}[\Phi^{k}_1] \\ 
	&\leq& \left(  1 - \min\left\{  \frac{\mu \eta}{3}, \frac{\delta}{4}, \frac{\delta_1}{4}, \frac{p}{4}  \right\}  \right) \mathbb{E} [\Phi^{k-1}_1] + 2\eta \mathbb{E}[P(x^*) - P(x^{k})] \\ 
	&& + \left(  \frac{96(1-\delta)}{\delta} \left(  \frac{4{\bar L}}{\delta} + L  + \frac{16{\bar L}p}{\delta \delta_1} \left(  1 + \frac{2p}{\delta_1}  \right) \right) + \frac{38L}{n}  \right) \eta^2 \mathbb{E} [P(x^{k-1}) - P(x^*)] \\ 
	&\leq&  \left(  1 - \min\left\{  \frac{\mu \eta}{3}, \frac{\delta}{4}, \frac{\delta_1}{4}, \frac{p}{4}  \right\}  \right)^k \Phi_1^0 - 2\eta \sum_{i=1}^k  \left(  1 - \min\left\{  \frac{\mu \eta}{3}, \frac{\delta}{4}, \frac{\delta_1}{4}, \frac{p}{4}  \right\}  \right)^{k-i} \mathbb{E} [P(x^i) - P(x^*)] \\ 
	&& + \left(  \frac{96(1-\delta)}{\delta} \left(  \frac{4{\bar L}}{\delta} + L  + \frac{16{\bar L}p}{\delta \delta_1} \left(  1 + \frac{2p}{\delta_1}  \right) \right) + \frac{38L}{n}  \right) \eta^2 \\ 
	&& \cdot \sum_{i=0}^{k-1}  \left(  1 - \min\left\{  \frac{\mu \eta}{3}, \frac{\delta}{4}, \frac{\delta_1}{4}, \frac{p}{4}  \right\}  \right)^{k-1-i} \mathbb{E} [P(x^i) - P(x^*)] \\ 
	&=& \frac{1}{w_k} \Phi_1^0 - \frac{2\eta}{w_k} \sum_{i=1}^k w_i \mathbb{E} [P(x^i) - P(x^*)] \\ 
	&& +\left(  \frac{96(1-\delta)}{\delta} \left(  \frac{4{\bar L}}{\delta} + L  + \frac{16{\bar L}p}{\delta \delta_1} \left(  1 + \frac{2p}{\delta_1}  \right) \right) + \frac{38L}{n}  \right) \frac{w_1 \eta^2}{w_k} \sum_{i=0}^{k-1} w_i \mathbb{E} [P(x^i) - P(x^*)] \\ 
	&\leq& \frac{1}{w_k} \Phi_1^0 - \frac{2\eta}{w_k} \sum_{i=1}^k w_i \mathbb{E} [P(x^i) - P(x^*)] \\ 
	&&  + \left(  \frac{105(1-\delta)}{\delta} \left(  \frac{4{\bar L}}{\delta} + L  + \frac{16{\bar L}p}{\delta \delta_1} \left(  1 + \frac{2p}{\delta_1}  \right) \right) + \frac{42L}{n}  \right) \frac{\eta^2}{w_k} \sum_{i=0}^{k} w_i \mathbb{E} [P(x^i) - P(x^*)], 
\end{eqnarray*}
where we use $w_1 \leq \frac{12}{11}$ in the last inequality. Rearranging the above inequality, we can get 
\begin{eqnarray*}
	&& \frac{2}{w_k} \sum_{i=0}^k w_i \mathbb{E} [P(x^i) - P(x^*)] \\ 
	&\leq&  \frac{1}{\eta w_k} \Phi_1^0 - \frac{1}{\eta} \mathbb{E} [\Phi_1^k] + \frac{2(P(x^0) - P(x^*))}{w_k}  \\ 
	&& + \left(  \frac{105(1-\delta)}{\delta} \left(  \frac{4{\bar L}}{\delta} + L  + \frac{16{\bar L}p}{\delta \delta_1} \left(  1 + \frac{2p}{\delta_1}  \right) \right) + \frac{42L}{n}  \right) \frac{\eta}{w_k} \sum_{i=0}^{k} w_i \mathbb{E} [P(x^i) - P(x^*)] \\ 
	&\leq&  \frac{1}{\eta w_k} \Phi_1^0 + \frac{2(P(x^0) - P(x^*))}{w_k} \\ 
	&& + \left(  \frac{105(1-\delta)}{\delta} \left(  \frac{4{\bar L}}{\delta} + L  + \frac{16{\bar L}p}{\delta \delta_1} \left(  1 + \frac{2p}{\delta_1}  \right) \right) + \frac{42L}{n}  \right) \frac{\eta}{w_k} \sum_{i=0}^{k} w_i \mathbb{E} [P(x^i) - P(x^*)]. 
\end{eqnarray*}

Hence, if 
$$
\eta \leq \frac{1}{\left(  \frac{105(1-\delta)}{\delta} \left(  \frac{4{\bar L}}{\delta} + L  + \frac{16{\bar L}p}{\delta \delta_1} \left(  1 + \frac{2p}{\delta_1}  \right) \right) + 4L_f + \frac{42L}{n}  \right)},
$$
then 
\begin{equation}\label{eq:sumPx-1-eclsvrg}
\sum_{i=0}^k w_i \mathbb{E} [P(x^i) - P(x^*)] \leq \frac{1}{\eta} \Phi_1^0 + 2(P(x^0) - P(x^*)). 
\end{equation}

When $\mu>0$, since 
$$
W_k = \sum_{i=0}^k w_i = \frac{1 - \frac{1}{(1 - \min\{   \frac{\mu \eta}{3}, \frac{\delta}{4}, \frac{\delta_1}{4}, \frac{p}{4}   \})^{k+1}}}{1 - \frac{1}{1 - \min\{   \frac{\mu \eta}{3}, \frac{\delta}{4}, \frac{\delta_1}{4}, \frac{p}{4}    \}}} = \frac{1 - (1 - \min\{   \frac{\mu \eta}{3}, \frac{\delta}{4}, \frac{\delta_1}{4}, \frac{p}{4}   \})^{k+1}}{\min\{   \frac{\mu \eta}{3}, \frac{\delta}{4}, \frac{\delta_1}{4}, \frac{p}{4}   \} (1 - \min\{   \frac{\mu \eta}{3}, \frac{\delta}{4}, \frac{\delta_1}{4}, \frac{p}{4}   \})^k}, 
$$
we can get 
\begin{eqnarray*}
	&& \frac{1}{W_k} \sum_{i=0}^k w_i \mathbb{E} [P(x^i) - P(x^*)] \\ 
	&\overset{(\ref{eq:sumPx-1-eclsvrg})}{\leq}& \frac{1}{W_k} \left(  \frac{1}{\eta} \Phi_1^0 + 2(P(x^0) - P(x^*))  \right) \\ 
	&=& \frac{\min\{   \frac{\mu \eta}{3}, \frac{\delta}{4}, \frac{\delta_1}{4}, \frac{p}{4}   \} }{1 - (1 - \min\{   \frac{\mu \eta}{3}, \frac{\delta}{4}, \frac{\delta_1}{4}, \frac{p}{4}   \})^{k+1}} \left(  \frac{1}{\eta} \Phi_1^0 + 2(P(x^0) - P(x^*))  \right) \left(  1 - \min\left\{  \frac{\mu \eta}{3}, \frac{\delta}{4}, \frac{\delta_1}{4}, \frac{p}{4}   \right\}  \right)^k. 
\end{eqnarray*}

From the definition of $\Phi_1^k$ and $e_\tau^0=0$, we have 
\begin{eqnarray*}
	&& \min\{   \frac{\mu \eta}{3}, \frac{\delta}{4}, \frac{\delta_1}{4}, \frac{p}{4}   \} \cdot \frac{1}{\eta} \Phi_1^0 \\ 
	&=& \min\{   \frac{\mu \eta}{3}, \frac{\delta}{4}, \frac{\delta_1}{4}, \frac{p}{4}   \} \left(  \frac{1}{\eta} \left(  1 + \frac{\mu \eta}{2}  \right)\|x^0 - x^*\|^2 + \frac{164(1-\delta)\eta}{\delta^2 \delta_1 n} \sum_{\tau=1}^n \| \nabla f^{(\tau)}(x^0) - h_\tau^0\|^2 \right. \\ 
	&& \left. + \frac{4\eta}{3p} \left(  \frac{41(1-\delta)}{\delta} \left(  \frac{4{\bar L}}{\delta} + L  + \frac{16{\bar L}p}{\delta \delta_1} \left(  1 + \frac{2p}{\delta_1}  \right) \right) + \frac{16L}{n} \right)  [P(x^0) - P(x^*)].  \right) \\ 
	&\leq& \frac{\mu}{3} \left(  1 + \frac{\mu \eta}{2}  \right) \|x^0-x^*\|^2 + \min\{   \frac{\mu \eta}{3}, \frac{\delta}{4}, \frac{\delta_1}{4}, \frac{p}{4}   \} \cdot \frac{1}{10 {\bar L} p n} \sum_{\tau=1}^n \| \nabla f^{(\tau)}(x^0)  - h_\tau^0\|^2 \\ 
	&& +  \min\{   \frac{\mu \eta}{3}, \frac{\delta}{4}, \frac{\delta_1}{4}, \frac{p}{4}   \} \cdot \frac{2}{3p} [P(x^0) - P(x^*)] \\ 
	&\leq& \frac{\mu}{2} \|x^0 - x^*\|^2 + \frac{1}{3} [P(x^0) - P(x^*)] + \frac{1}{40{\bar L}n}\sum_{\tau=1}^n \| \nabla f^{(\tau)}(x^0)  - h_\tau^0\|^2  . 
\end{eqnarray*}

Therefore, we arrive at 
\begin{eqnarray*}
&& \frac{1}{W_k} \sum_{i=0}^k w_i \mathbb{E} [P(x^i) - P(x^*)] \\ 
&\leq& \frac{\frac{\mu}{2} \|x^0 - x^*\|^2 + \frac{1}{2} (P(x^0) - P(x^*))  + \frac{1}{40{\bar L}n}\sum_{\tau=1}^n \| \nabla f^{(\tau)}(x^0)  - h_\tau^0\|^2 }{1 - (1 - \min\{   \frac{\mu \eta}{3}, \frac{\delta}{4}, \frac{\delta_1}{4},  \frac{p}{4}   \})^{k+1}} \left(  1 - \min\left\{  \frac{\mu \eta}{3}, \frac{\delta}{4}, \frac{\delta_1}{4}, \frac{p}{4}   \right\}  \right)^k. 
\end{eqnarray*}

For ${\bar x}^k = \frac{1}{W_k} \sum_{i=0}^k w_ix^i$, from the convexity of $P$, we have 
\begin{eqnarray*}
	&& \mathbb{E}[P({\bar x}^k)- P(x^*)] \leq  \frac{1}{W_k} \sum_{i=0}^k w_i \mathbb{E} [P(x^i) - P(x^*)] \\ 
	&\leq& \frac{\frac{\mu}{2} \|x^0 - x^*\|^2 + \frac{1}{2} (P(x^0) - P(x^*))  + \frac{1}{40{\bar L}n}\sum_{\tau=1}^n \| \nabla f^{(\tau)}(x^0)  - h_\tau^0\|^2 }{1 - (1 - \min\{   \frac{\mu \eta}{3}, \frac{\delta}{4}, \frac{\delta_1}{4},  \frac{p}{4}   \})^{k+1}} \left(  1 - \min\left\{  \frac{\mu \eta}{3}, \frac{\delta}{4}, \frac{\delta_1}{4}, \frac{p}{4}   \right\}  \right)^k. 
\end{eqnarray*}

If we choose $\eta = \frac{1}{\left(  \frac{105(1-\delta)}{\delta} \left(  \frac{4{\bar L}}{\delta} + L  + \frac{16{\bar L}p}{\delta \delta_1} \left(  1 + \frac{2p}{\delta_1}  \right) \right) +4L_f + \frac{42L}{n}  \right)}$, then in order to guarantee $\mathbb{E} [P({\bar x}^k) - P(x^*)] \leq \epsilon$, we first let 
$$
\left(  1 - \min\left\{  \frac{\mu \eta}{3}, \frac{\delta}{4}, \frac{\delta_1}{4}, \frac{p}{4}  \right\}  \right)^{k+1} \leq \frac{1}{2}, 
$$
which implies that 
$$
\mathbb{E}[P({\bar x}^k)- P(x^*)] \leq \left(  \mu \|x^0 - x^*\|^2 + P(x^0) - P(x^*) + \frac{1}{20{\bar L}n}\sum_{\tau=1}^n \| \nabla f^{(\tau)}(x^0)  - h_\tau^0\|^2  \right) \left(  1 - \min\left\{  \frac{\mu \eta}{3}, \frac{\delta}{4}, \frac{\delta_1}{4}, \frac{p}{4}   \right\}  \right)^k. 
$$

Hence, when $\epsilon \leq \frac{\mu}{2}\|x^0-x^*\|^2 + \frac{1}{2} (P(x^0) - P(x^*)) + \frac{1}{40{\bar L}n}\sum_{\tau=1}^n \| \nabla f^{(\tau)}(x^0)  - h_\tau^0\|^2$, $\mathbb{E}[P({\bar x}^k)- P(x^*)] \leq \epsilon$ as long as 
$$
\left(  1 - \min\left\{  \frac{\mu \eta}{3}, \frac{\delta}{4}, \frac{\delta_1}{4}, \frac{p}{4}  \right\}  \right)^{k} \leq \frac{\epsilon}{ \mu \|x^0 - x^*\|^2 + P(x^0) - P(x^*)  + \frac{1}{20{\bar L}n}\sum_{\tau=1}^n \| \nabla f^{(\tau)}(x^0)  - h_\tau^0\|^2}, 
$$
which is equivalent to 
$$
k \geq \frac{1}{- \ln (1 - \min\left\{  \frac{\mu \eta}{3}, \frac{\delta}{4}, \frac{\delta_1}{4}, \frac{p}{4}  \right\} )} \ln\left(  \frac{\mu \|x^0 - x^*\|^2 + P(x^0) - P(x^*) + \frac{1}{20{\bar L}n}\sum_{\tau=1}^n \| \nabla f^{(\tau)}(x^0)  - h_\tau^0\|^2}{\epsilon}  \right). 
$$
Since $-\ln(1 - x) \geq x$ for $x \in[0, 1)$, if $p\leq O(\delta_1)$, we have $\mathbb{E}[P({\bar x}^k)- P(x^*)] \leq \epsilon$ as long as 
\begin{eqnarray*}
k \geq O\left(  \left(  \frac{1}{\delta} + \frac{1}{p}  + \frac{L_f}{\mu} + \frac{L}{n \mu}  +  \frac{(1-\delta){\bar L}}{\delta^2 \mu} + \frac{(1-\delta)L}{\delta \mu} \right)  \ln \frac{1}{\epsilon}  \right). 
\end{eqnarray*}

When $\mu=0$, from (\ref{eq:sumPx-1-eclsvrg}) and the convexity of $P$, we have 
$$
\mathbb{E}[P({\bar x}^k) - P(x^*)] \leq \frac{1}{k+1} \sum_{i=0}^k \mathbb{E}[P(x^i) - P(x^*)] \leq \frac{1}{k+1} \left(   \frac{1}{\eta} \Phi_1^0 + 2(P(x^0) - P(x^*))  \right). 
$$

From the definition of $\Phi_1^k$ and $e_\tau^0=0$, we have 
\begin{eqnarray*}
	 \frac{1}{\eta} \Phi_1^0 &=&  \frac{1}{\eta} \left(  1 + \frac{\mu \eta}{2}  \right)\|x^0 - x^*\|^2 + \frac{164(1-\delta)\eta}{\delta^2 \delta_1 n} \sum_{\tau=1}^n \| \nabla f^{(\tau)}(x^0)  - h_\tau^0\|^2 \\ 
	&&  + \frac{4\eta}{3p} \left(  \frac{41(1-\delta)}{\delta} \left(  \frac{4{\bar L}}{\delta} + L  + \frac{16{\bar L}p}{\delta \delta_1} \left(  1 + \frac{2p}{\delta_1}  \right) \right) + \frac{16L}{n} \right)  [P(x^0) - P(x^*)].  \\ 
	&\leq& \frac{9}{8\eta} \|x^0-x^*\|^2 + \frac{1}{10 {\bar L} p n} \sum_{\tau=1}^n \| \nabla f^{(\tau)}(x^0)  - h_\tau^0\|^2 + \frac{2}{3p} (P(x^0) - P(x^*)). 
\end{eqnarray*}

Hence, we arrive at 
$$
\mathbb{E}[P({\bar x}^k) - P(x^*)] \leq \frac{1}{k+1} \left( \frac{9}{8\eta} \|x^0-x^*\|^2 + \frac{1}{10 {\bar L} p n} \sum_{\tau=1}^n \| \nabla f^{(\tau)}(x^0)  - h_\tau^0\|^2 + \frac{8}{3p} (P(x^0) - P(x^*))  \right). 
$$

In particular, if we choose $\eta = \frac{1}{\left(  \frac{105(1-\delta)}{\delta} \left(  \frac{4{\bar L}}{\delta} + L  + \frac{16{\bar L}p}{\delta \delta_1} \left(  1 + \frac{2p}{\delta_1}  \right) \right) +4L_f + \frac{42L}{n}  \right)}$ and $p\leq O(\delta_1)$, we have $\mathbb{E}[P({\bar x}^k)- P(x^*)] \leq \epsilon$ as long as 
$$
k \geq O\left(  \left( \frac{1}{p} + L_f + \frac{L}{n} + \frac{(1-\delta){\bar L}}{\delta^2} + \frac{(1-\delta) L}{\delta}   \right) \frac{1}{\epsilon} \right). 
$$

\subsection{Proof of Theorem \ref{th:eclsvrg-22}}

Same as the proof of Theorem \ref{th:eclsvrg-11}, if 
$$
\eta \leq \frac{1}{ \left(  \tfrac{(1-\delta)}{\delta} \left(  \tfrac{418L_f}{\delta} + \tfrac{3422{\bar L}}{\delta n} + \tfrac{1349L}{n} + \tfrac{(1671L_f + \tfrac{20364{\bar L}}{n})p}{\delta \delta_1}\left(  1 + \tfrac{2p}{\delta_1}  \right)  \right) + 4L_f + \tfrac{42L}{n} \right) },
$$
then 
\begin{equation}\label{eq:sumPx-2-eclsvrg}
\sum_{i=0}^k w_i \mathbb{E} [P(x^i) - P(x^*)] \leq \frac{1}{\eta} \Phi_2^0 + 2(P(x^0) - P(x^*)). 
\end{equation}

When $\mu>0$, notice that 

\begin{eqnarray*}
	&& \min\{   \frac{\mu \eta}{3}, \frac{\delta}{4}, \frac{\delta_1}{4}, \frac{p}{4}   \} \cdot \frac{1}{\eta} \Phi_2^0 \\ 
	&=& \min\{   \frac{\mu \eta}{3}, \frac{\delta}{4}, \frac{\delta_1}{4}, \frac{p}{4}   \} \left(  \frac{1}{\eta} \left(  1 + \frac{\mu \eta}{2}  \right)\|x^0 - x^*\|^2 + \frac{164(1-\delta)\eta}{\delta^2 \delta_1} \|\nabla f(x^0)\|^2 + \frac{2000(1-\delta)\eta}{\delta^2 \delta_1 n} \sum_{\tau=1}^n \| \nabla f^{(\tau)}(x^0)  - h_\tau^0\|^2 \right. \\ 
	&& \left. + \frac{4\eta}{3p} \left(  \frac{(1-\delta)}{\delta} \left(  \frac{164L_f}{\delta} + \frac{1344{\bar L}}{\delta n} + \frac{541L}{n}  + \frac{(656L_f + \frac{8000{\bar L}}{n})p}{\delta \delta_1}\left(  1 + \frac{2p}{\delta_1}  \right)  \right) + \frac{16L}{n} \right)  [P(x^0) - P(x^*)]  \right) \\ 
	&\leq& \frac{\mu}{3} \left(  1 + \frac{\mu \eta}{2}  \right) \|x^0-x^*\|^2 + \min\{   \frac{\mu \eta}{3}, \frac{\delta}{4}, \frac{\delta_1}{4}, \frac{p}{4}   \} \cdot \frac{1}{10L_fp} \|\nabla f(x^0) \|^2 \\ 
	&& + \min\{   \frac{\mu \eta}{3}, \frac{\delta}{4}, \frac{\delta_1}{4}, \frac{p}{4}   \} \cdot \frac{6}{5 L_f p n} \sum_{\tau=1}^n \| \nabla f^{(\tau)}(x^0)  - h_\tau^0\|^2  +  \min\{   \frac{\mu \eta}{3}, \frac{\delta}{4}, \frac{\delta_1}{4}, \frac{p}{4}   \} \cdot \frac{2}{3p} [P(x^0) - P(x^*)] \\ 
	&\leq& \frac{\mu}{2} \|x^0 - x^*\|^2 + \frac{1}{3} [P(x^0) - P(x^*)] + \frac{1}{40L_f}\|\nabla f(x^0)\|^2 + \frac{3}{10 L_fn}\sum_{\tau=1}^n \| \nabla f^{(\tau)}(x^0)  - h_\tau^0\|^2. 
\end{eqnarray*}

Hence, 

\begin{eqnarray*}
	 && \frac{1}{W_k} \sum_{i=0}^k w_i \mathbb{E} [P(x^i) - P(x^*)] \leq  \left(  1 - \min\left\{  \frac{\mu \eta}{3}, \frac{\delta}{4}, \frac{\delta_1}{4}, \frac{p}{4}   \right\}  \right)^k \\ 
	&& \cdot \frac{ \frac{\mu}{2} \|x^0 - x^*\|^2 + \frac{1}{2} (P(x^0) - P(x^*)) + \frac{1}{40L_f}\|\nabla f(x^0)\|^2 + \frac{3}{10L_f n}\sum_{\tau=1}^n \| \nabla f^{(\tau)}(x^0)  - h_\tau^0\|^2 }{1 - (1 - \min\{   \frac{\mu \eta}{3}, \frac{\delta}{4}, \frac{\delta_1}{4},  \frac{p}{4}   \})^{k+1}}. 
\end{eqnarray*}

When $\mu=0$, notice that 

\begin{eqnarray*}
	 \frac{1}{\eta} \Phi_2^0 &=&  \frac{1}{\eta} \left(  1 + \frac{\mu \eta}{2}  \right)\|x^0 - x^*\|^2 + \frac{164(1-\delta)\eta}{\delta^2 \delta_1} \|\nabla f(x^0)\|^2 + \frac{2000(1-\delta)\eta}{\delta^2 \delta_1 n} \sum_{\tau=1}^n \| \nabla f^{(\tau)}(x^0)  - h_\tau^0\|^2  \\ 
	&&  + \frac{4\eta}{3p} \left(  \frac{(1-\delta)}{\delta} \left(  \frac{164L_f}{\delta} + \frac{1344{\bar L}}{\delta n} + \frac{541L}{n}  + \frac{(656L_f + \frac{8000{\bar L}}{n})p}{\delta \delta_1}\left(  1 + \frac{2p}{\delta_1}  \right)  \right) + \frac{16L}{n} \right)  [P(x^0) - P(x^*)] \\ 
	&\leq& \frac{9}{8\eta} \|x^0-x^*\|^2 +  \frac{1}{10L_fp} \|\nabla f(x^0)\|^2  +  \frac{6}{5 L_f p n} \sum_{\tau=1}^n \| \nabla f^{(\tau)}(x^0)  - h_\tau^0\|^2  + \frac{2}{3p} \left(P(x^0) - P(x^*) \right). 
\end{eqnarray*}

From $\frac{{\bar L}}{n} \leq L_f$, the rest is the same as that of Theorem \ref{th:eclsvrg-11}.

\section{Proofs for EC-LSVRG in the Smooth Case}

\subsection{A lemma} 

Thanks to the following lemma, we can get better results than the composite case. The main difference between Lemma \ref{lm:itereclsvrg} and Lemma \ref{lm:itereclsvrgsmooth} is that there is an additional stepsize $\eta$ before $\mathbb{E}\|e^k\|^2$. The following lemma is similar to Lemma 7 in \citep{Stich19}. However, for completeness, we give the proof.  

\begin{lemma}\label{lm:itereclsvrgsmooth}
	If $\eta\leq \tfrac{1}{4L_f + 8L/n}$, then 
	\begin{eqnarray*}
		\mathbb{E}_k \|{\tilde x}^{k+1} - x^*\|^2 &\leq& \left(1-\tfrac{\mu\eta}{2} \right) \|{\tilde x}^k - x^*\|^2 - \tfrac{\eta}{2} [f(x^k) - f(x^*)] \\
		&&  + 3L_f \eta \|e^k\|^2 + \tfrac{4L}{n}\eta^2 [f(w^k) - f(x^*)]. 
	\end{eqnarray*}
\end{lemma}

\begin{proof}

Since $\psi =0$, we have ${\tilde x}^{k+1} = {\tilde x}^k - \eta (g^k + h^k)$. Hence 
\begin{eqnarray*}
	&& \mathbb{E}_k \|{\tilde x}^{k+1} - x^*\|^2 \\ 
	&=& \mathbb{E}_k \| {\tilde x}^k -x^* - \eta(g^k + h^k)\|^2 \\ 
	&=& \|{\tilde x}^k - x^*\|^2 - 2\eta \langle {\tilde x}^k - x^*, \nabla f(x^k) \rangle + \eta^2 \mathbb{E}_k \|g^k +h^k\|^2 \\ 
	&=&  \|{\tilde x}^k - x^*\|^2 - 2\eta \langle x^k - x^*, \nabla f(x^k) \rangle + 2\eta \langle x^k - {\tilde x}^k, \nabla f(x^k) \rangle  + \eta^2 \mathbb{E}_k \|g^k + h^k\|^2 \\ 
	&\leq& \|{\tilde x}^k - x^*\|^2 - 2\eta (f(x^k) - f(x^*)) - \mu\eta \|x^k-x^*\|^2 + 2\eta \langle e^k, \nabla f(x^k) \rangle  + \eta^2 \mathbb{E}_k \|g^k + h^k\|^2, 
\end{eqnarray*}
where the last inequality comes from the $\mu$-strongly convexity of $f$. 

For $\|x^k - x^*\|^2$, we have 
$$
\|{\tilde x}^k - x^*\|^2 \leq 2\|x^k - x^*\|^2 + 2\|e^k\|^2.
$$
For $2\langle e^k, \nabla f(x^k) \rangle$, we have 
$$
2\langle e^k, \nabla f(x^k) \rangle \leq \frac{1}{2L_f} \|\nabla f(x^k)\|^2 + 2L_f\|e^k\|^2 \leq f(x^k) - f(x^*) + 2L_f\|e^k\|^2. 
$$
Thus, we arrive at 
\begin{eqnarray*}
	&& \mathbb{E}_k \|{\tilde x}^{k+1} - x^*\|^2 \\ 
	&\leq& \left(1-\frac{\mu\eta}{2} \right) \|{\tilde x}^k - x^*\|^2 - \eta (f(x^k) - f(x^*)) + (2L_f+\mu)\eta \|e^k\|^2 + \eta^2 \mathbb{E}_k \|g^k + h^k\|^2 \\ 
	&\leq& \left(1-\frac{\mu\eta}{2} \right) \|{\tilde x}^k - x^*\|^2 - \eta (f(x^k) - f(x^*)) + 3L_f\eta \|e^k\|^2 + \eta^2 \mathbb{E}_k \|g^k + h^k\|^2. 
\end{eqnarray*}

Finally, for $\mathbb{E}_k \|g^k+ h^k\|^2$, we have 
\begin{eqnarray*}
	\mathbb{E}_k \|g^k+ h^k\|^2 &=& \mathbb{E}_k \|\frac{1}{n} \sum_{\tau=1}^n \left(  \nabla f_{i_k^\tau}^{(\tau)}(x^k) - \nabla f_{i_k^\tau}^{(\tau)}(w^k)  \right)  + \nabla f(w^k) - \nabla f(x^k) + \nabla f(x^k) - \nabla f(x^*)\|^2 \\ 
	&=& \mathbb{E}_k \|\frac{1}{n} \sum_{\tau=1}^n \left(  \nabla f_{i_k^\tau}^{(\tau)}(x^k) - \nabla f_{i_k^\tau}^{(\tau)}(w^k)  \right)  + \nabla f(w^k) - \nabla f(x^k)\|^2 + \mathbb{E}\|\nabla f(x^k) - \nabla f(x^*)\|^2 \\ 
	&\leq&  \mathbb{E}_k \|\frac{1}{n} \sum_{\tau=1}^n \left(  \nabla f_{i_k^\tau}^{(\tau)}(x^k) - \nabla f_{i_k^\tau}^{(\tau)}(w^k)  \right)  + \nabla f(w^k) - \nabla f(x^k)\|^2 + 2L_f (f(x^k) - f(x^*)) \\ 
	&\overset{(\ref{eq:gk2variance})}{\leq}& \left(2L_f + \frac{4L}{n} \right) [f(x^k) - f(x^*)] + \frac{4L}{n} [f(w^k) - f(x^*)]. 
\end{eqnarray*}
Thereofore, 
\begin{eqnarray*}
	\mathbb{E}_k \|{\tilde x}^{k+1} - x^*\|^2 &\leq& \left(1-\frac{\mu\eta}{2} \right) \|{\tilde x}^k - x^*\|^2 - \eta \left(1 - \left(2L_f + \frac{4L}{n} \right)\eta \right)[f(x^k) - f(x^*)] \\ 
	&& + 3L_f \eta \|e^k\|^2 + \frac{4L}{n}\eta^2  [f(w^k) - f(x^*)]. 
\end{eqnarray*}
By choosing $\eta \leq \frac{1}{4L_f + 8L/n}$, we can get the reslut. 

\end{proof}

\subsection{Proof of Theorem \ref{th:eclsvrgsmooth-1}} 

Let $\eta\leq \frac{1}{4L_f + 8L/n}$. From Lemma \ref{lm:itereclsvrgsmooth}, Lemma \ref{lm:ek+1-1eclsvrg}, and $\|e^k\|^2 \leq \frac{1}{n} \sum_{\tau=1}^n \|e^k_{\tau}\|^2$, we can obtain 

\begin{eqnarray*}
	&& \mathbb{E}_k \|{\tilde x}^{k+1} - x^*\|^2 + \frac{12L_f \eta}{n \delta } \sum_{\tau=1}^n \mathbb{E}_k \|e^{k+1}_\tau\|^2 - \left(1-\frac{\mu\eta}{2} \right) \|{\tilde x}^k - x^*\|^2 + \frac{\eta}{2} [f(x^k) - f(x^*)] \\ 
	&\leq&  \frac{3L_f \eta}{n} \sum_{\tau=1}^n \|e^k_\tau\|^2 + \frac{4L}{n}\eta^2 [f(w^k) - f(x^*)]  + \frac{48(1-\delta)L_f \eta^3}{\delta^2 n} \sum_{\tau=1}^n \|h^k_\tau - \nabla f^{(\tau)}(w^k)\|^2 \\ 
	&& + \frac{12L_f \eta}{n \delta } \left(  1 - \frac{\delta}{2}  \right) \sum_{\tau=1}^n \|e^{k}_\tau\|^2 + \frac{48(1-\delta) L_f \eta^3}{\delta} \left(  \frac{4{\bar L}}{\delta} + L  \right)  [f(x^k) - f(x^*) + f(w^k) - f(x^*)] \\ 
	&=&  \frac{12L_f \eta}{n \delta } \left(  1 - \frac{\delta}{4}  \right) \sum_{\tau=1}^n \|e^{k}_\tau\|^2 + \frac{48(1-\delta)L_f \eta^3}{\delta^2 n} \sum_{\tau=1}^n \|h^k_\tau - \nabla f^{(\tau)}(w^k)\|^2 + \frac{4L}{n}\eta^2 [f(w^k) - f(x^*)]  \\ 
	&& + \frac{48(1-\delta) L_f \eta^3}{\delta} \left(  \frac{4{\bar L}}{\delta} + L  \right)  [f(x^k) - f(x^*) +  f(w^k) - f(x^*)]. 
\end{eqnarray*}

Then from Lemma \ref{lm:hk+1-1eclsvrg}, we can get 
\begin{eqnarray*}
	&& \mathbb{E}_k \|{\tilde x}^{k+1} - x^*\|^2 + \frac{12L_f \eta}{n \delta } \sum_{\tau=1}^n \mathbb{E}_k \|e^{k+1}_\tau\|^2 + \frac{192(1-\delta)L_f \eta^3}{\delta^2 \delta_1 n}  \sum_{\tau=1}^n \mathbb{E}_k \|h^{k+1}_\tau - \nabla f^{(\tau)}(w^{k+1})\|^2 \\ 
	&&  -  \left(1-\frac{\mu\eta}{2} \right) \|{\tilde x}^k - x^*\|^2 + \frac{\eta}{2} [f(x^k) - f(x^*)] -  \frac{4L}{n}\eta^2 [f(w^k) - f(x^*)]  \\ 
	&\leq&  \left(  1 - \frac{\delta}{4}  \right)\frac{12L_f \eta}{n \delta }  \sum_{\tau=1}^n \|e^{k}_\tau\|^2 + \left(  1 - \frac{\delta_1}{4}  \right) \frac{192(1-\delta)L_f \eta^3}{\delta^2 \delta_1 n} \sum_{\tau=1}^n \|h^k_\tau - \nabla f^{(\tau)}(w^k)\|^2 \\
	&& + \frac{48(1-\delta) L_f \eta^3}{\delta} \left(  \frac{4{\bar L}}{\delta} + L + \frac{16{\bar L}p}{\delta \delta_1} \left(  1 + \frac{2p}{\delta_1}  \right)  \right) [f(x^k) - f(x^*) +  f(w^k) - f(x^*)]. 
\end{eqnarray*}

Combining (\ref{eq:wk+1}) and the above inequality, we can obtain 
\begin{eqnarray*}
	&& \mathbb{E}_k [\Phi_3^{k+1}] \\ 
	&=& \mathbb{E}_k \|{\tilde x}^{k+1} - x^*\|^2 + \frac{12L_f \eta}{n \delta } \sum_{\tau=1}^n \mathbb{E}_k \|e^{k+1}_\tau\|^2 + \frac{192(1-\delta)L_f \eta^3}{\delta^2 \delta_1 n}  \sum_{\tau=1}^n \mathbb{E}_k \|h^{k+1}_\tau - \nabla f^{(\tau)}(w^{k+1})\|^2 \\ 
	&& + \frac{4}{3p}  \left(   \frac{48(1-\delta) L_f \eta^3}{\delta} \left(  \frac{4{\bar L}}{\delta} + L  + \frac{16{\bar L}p}{\delta \delta_1} \left(  1 + \frac{2p}{\delta_1}  \right) \right) + \frac{4L\eta^2}{n}  \right) \mathbb{E}_k [f(w^{k+1}) - f(x^*)] \\ 
	&\leq&  \left(1-\frac{\mu\eta}{2} \right) \|{\tilde x}^k - x^*\|^2 + \left(  1 - \frac{\delta}{4}  \right) \frac{12L_f \eta}{n \delta }  \sum_{\tau=1}^n \|e^{k}_\tau\|^2 + \left(  1 - \frac{\delta_1}{4}  \right) \frac{192(1-\delta)L_f \eta^3}{\delta^2 \delta_1 n} \sum_{\tau=1}^n  \|h^k_\tau - \nabla f^{(\tau)}(w^k)\|^2 \\ 
	&& + \frac{4}{3p}  \left(   \frac{48(1-\delta) L_f \eta^3}{\delta} \left(  \frac{4{\bar L}}{\delta} + L  + \frac{16{\bar L}p}{\delta \delta_1} \left(  1 + \frac{2p}{\delta_1}  \right)   \right) + \frac{4L\eta^2}{n}  \right) \left(  1 - \frac{p}{4}  \right) [f(w^{k}) - f(x^*)] \\ 
	&& - \frac{\eta}{2} \left(  1 - \frac{224(1-\delta)L_f \eta^2}{\delta} \left(  \frac{4{\bar L}}{\delta} + L  + \frac{16{\bar L}p}{\delta \delta_1} \left(  1 + \frac{2p}{\delta_1}  \right)  \right) - \frac{11L\eta}{n}   \right) [f(x^k) - f(x^*)] \\ 
	&\leq& \left(  1 - \min\left\{  \frac{\mu \eta}{2}, \frac{\delta}{4}, \frac{\delta_1}{4}, \frac{p}{4}  \right\}  \right) \Phi_3^k \\ 
	&& - \frac{\eta}{2} \left(  1 - \frac{224(1-\delta)L_f \eta^2}{\delta} \left(  \frac{4{\bar L}}{\delta} + L + \frac{16{\bar L}p}{\delta \delta_1} \left(  1 + \frac{2p}{\delta_1}  \right)   \right) - \frac{11L\eta}{n}   \right) [f(x^k) - f(x^*)]. 
\end{eqnarray*}

Taking expectation again and applying the tower property, we can get the result.

\subsection{Proof of Theorem \ref{th:eclsvrgsmooth-2}} 

Let $\eta \leq \frac{1}{4L_f + 8L/n}$. From Lemma \ref{lm:itereclsvrgsmooth}, we have 

\begin{eqnarray*}
	&& \mathbb{E}_k \|{\tilde x}^{k+1} - x^*\|^2 + \frac{12L_f \eta}{\delta} \mathbb{E}_k \|e^{k+1}\|^2 +  \frac{96(1-\delta) L_f \eta}{n^2 \delta } \sum_{\tau=1}^n \mathbb{E}_k \|e^{k+1}_\tau\|^2 \\ 
	&& - \left(1-\frac{\mu\eta}{2} \right) \|{\tilde x}^k - x^*\|^2 +  \frac{\eta}{2} [f(x^k) - f(x^*)] - \frac{4L}{n}\eta^2 [f(w^k) - f(x^*)] \\ 
	&\leq& 3L_f \eta \|e^k\|^2 +  \frac{12L_f \eta}{\delta} \mathbb{E}_k \|e^{k+1}\|^2 +  \frac{96(1-\delta) L_f \eta}{n^2 \delta } \sum_{\tau=1}^n \mathbb{E}_k \|e^{k+1}_\tau\|^2 \\ 
	&\overset{Lemma~\ref{lm:ek+1-2eclsvrg}}{\leq}&  \frac{12L_f \eta}{\delta} \left(  1 - \frac{\delta}{4}  \right) \|e^k\|^2 + \frac{24(1-\delta)L_f \eta }{n^2} \sum_{\tau=1}^n \|e^k_\tau\|^2 +  \frac{96(1-\delta) L_f \eta}{n^2 \delta } \sum_{\tau=1}^n \mathbb{E}_k \|e^{k+1}_\tau\|^2 \\ 
	&& + \frac{48(1-\delta) L_f \eta^3}{\delta} \left(  \frac{4L_f}{\delta} + \frac{5L}{n}  \right) [f(x^k) - f(x^*) + f(w^k) - f(x^*)] \\ 
	&& + \frac{48(1-\delta)L_f \eta^3}{n^2} \sum_{\tau=1}^n \|h^k_\tau - \nabla f^{(\tau)}(w^k)\|^2 + \frac{48(1-\delta)L_f \eta^3}{\delta^2} \|h^k - \nabla f(w^k)\|^2 \\
	&\overset{Lemma~\ref{lm:ek+1-1eclsvrg}}{\leq}&  \frac{12L_f \eta}{\delta} \left(  1 -  \frac{\delta}{4}  \right)\|e^k\|^2 +  \frac{96(1-\delta) L_f \eta}{n^2 \delta } \left( 1 - \frac{\delta}{4} \right) \sum_{\tau=1}^n \|e^{k}_\tau\|^2 \\ 
	&& + \frac{48(1-\delta)L_f \eta^3}{\delta} \left(  \frac{4L_f}{\delta} + \frac{32{\bar L}}{n \delta} + \frac{13L}{n}  \right) [f(x^k) - f(x^*) + f(w^k) - f(x^*)] \\ 
	&& + \frac{384(1-\delta)L_f \eta^3}{\delta^2 n^2} \sum_{\tau=1}^n \|h^k_\tau - \nabla f^{(\tau)}(w^k)\|^2 + \frac{48(1-\delta)L_f \eta^3}{\delta^2} \|h^k - \nabla f(w^k)\|^2. 
\end{eqnarray*}

Then from Lemma \ref{lm:hk+1-1eclsvrg} and Lemma \ref{lm:hk+1-2eclsvrg}, we can obtain 
\begin{eqnarray*}
	&& \mathbb{E}_k \|{\tilde x}^{k+1} - x^*\|^2 + \frac{12L_f \eta}{\delta} \mathbb{E}_k \|e^{k+1}\|^2 +  \frac{96(1-\delta) L_f \eta}{n^2 \delta } \sum_{\tau=1}^n \mathbb{E}_k \|e^{k+1}_\tau\|^2 - \left(  1 -  \frac{\delta}{4}  \right)\frac{12L_f \eta}{\delta}  \|e^k\|^2 \\ 
	&& + \frac{192(1-\delta)L_f \eta^3}{\delta^2 \delta_1} \mathbb{E}_k \|h^{k+1} - \nabla f(w^{k+1})\|^2 + \frac{2304(1-\delta)L_f \eta^3}{\delta^2 \delta_1 n^2} \sum_{\tau=1}^n \mathbb{E}_k \|h^{k+1}_\tau - \nabla f^{(\tau)}(w^{k+1})\|^2  \\ 
	&& - \left(1-\frac{\mu\eta}{2} \right) \|{\tilde x}^k - x^*\|^2 +  \frac{\eta}{2} [f(x^k) - f(x^*)] - \frac{4L}{n}\eta^2 [f(w^k) - f(x^*)] \\ 
	&\overset{Lemma \ref{lm:hk+1-2eclsvrg}}{\leq}& \left( 1 - \frac{\delta}{4} \right) \frac{96(1-\delta) L_f \eta}{n^2 \delta }  \sum_{\tau=1}^n  \|e^{k}_\tau\|^2 + \frac{2304(1-\delta)L_f \eta^3}{\delta^2 \delta_1 n^2} \sum_{\tau=1}^n \mathbb{E}_k \|h^{k+1}_\tau - \nabla f^{(\tau)}(w^{k+1})\|^2  \\ 
	&& + \left(  1 - \frac{\delta_1}{4}  \right) \frac{192(1-\delta)L_f \eta^3}{\delta^2 \delta_1}\|h^{k} - \nabla f(w^{k})\|^2 + \frac{576(1-\delta)L_f \eta^3}{\delta^2 n^2} \sum_{\tau=1}^n \|h^k_\tau - \nabla f^{(\tau)}(w^k)\|^2 \\ 
	&& + \frac{48(1-\delta)L_f \eta^3}{\delta} \left(  \frac{4L_f}{\delta}  + \frac{32{\bar L}}{n \delta} + \frac{13L}{n}  + \frac{16pL_f}{\delta\delta_1} \left(  1 + \frac{2p}{\delta_1}  \right)\right) [f(x^k) - f(x^*) + f(w^k) - f(x^*)] \\ 
	&\overset{Lemma \ref{lm:hk+1-1eclsvrg}}{\leq}&  \left( 1 - \frac{\delta}{4} \right)\frac{96(1-\delta) L_f \eta}{n^2 \delta } \sum_{\tau=1}^n \|e^{k}_\tau\|^2 + \left(  1 - \frac{\delta_1}{4}  \right) \frac{2304(1-\delta)L_f \eta^3}{\delta^2 \delta_1 n^2} \sum_{\tau=1}^n \|h^{k}_\tau - \nabla f^{(\tau)}(w^{k})\|^2 \\ 
	&& + \left(  1 - \frac{\delta_1}{4}  \right) \frac{192(1-\delta)L_f \eta^3}{\delta^2 \delta_1} \|h^{k} - \nabla f(w^{k})\|^2 +  [f(x^k) - f(x^*) + f(w^k) - f(x^*)]  \\ 
	&& \cdot \frac{48(1-\delta)L_f \eta^3}{\delta} \left(  \frac{4L_f}{\delta} + \frac{32{\bar L}}{n \delta} + \frac{13L}{n}  + \frac{16p(L_f + \frac{12{\bar L}}{n})}{\delta\delta_1} \left(  1 + \frac{2p}{\delta_1}  \right)\right). 
\end{eqnarray*}

Combining (\ref{eq:wk+1}) and the above inequality, we arrive at 

\begin{eqnarray*}
	&& \mathbb{E}_k [\Phi_4^{k+1}] \\ 
	&=& \mathbb{E}_k \|{\tilde x}^{k+1} - x^*\|^2 + \frac{12L_f \eta}{\delta} \mathbb{E}_k \|e^{k+1}\|^2 +  \frac{96(1-\delta) L_f \eta}{n^2 \delta } \sum_{\tau=1}^n \mathbb{E}_k \|e^{k+1}_\tau\|^2 \\ 
	&& + \frac{192(1-\delta)L_f \eta^3}{\delta^2 \delta_1} \mathbb{E}_k \|h^{k+1} - \nabla f(w^{k+1})\|^2  + \frac{2304(1-\delta)L_f \eta^3}{\delta^2 \delta_1 n^2} \sum_{\tau=1}^n \mathbb{E}_k \|h^{k+1}_\tau - \nabla f^{(\tau)}(w^{k+1})\|^2  \\ 
	&& + \frac{4}{3p} \left( \frac{48(1-\delta)L_f \eta^3}{\delta} \left(  \frac{4L_f}{\delta} + \frac{32{\bar L}}{n \delta} + \frac{13L}{n}  + \frac{16p(L_f + \frac{12{\bar L}}{n})}{\delta\delta_1} \left(  1 + \frac{2p}{\delta_1}  \right)\right) + \frac{4L\eta^2}{n}  \right) \mathbb{E}_k [f(w^{k+1}) - f(x^*)] \\ 
	&\leq& \left(  1 - \min\left\{  \frac{\mu\eta}{2}, \frac{\delta}{4}, \frac{\delta_1}{4}, \frac{p}{4}  \right\}  \right) \Phi_4^k  \\ 
	&& -  \frac{\eta}{2} \left(  1 - \frac{224(1-\delta) L_f \eta^2}{\delta} \left(  \frac{4L_f}{\delta} + \frac{32{\bar L}}{n \delta} + \frac{13L}{n}  + \frac{16p(L_f + \frac{12{\bar L}}{n})}{\delta\delta_1} \left(  1 + \frac{2p}{\delta_1}  \right)\right) -  \frac{11L\eta}{n} \right) [f(x^k) - f(x^*)]. 
\end{eqnarray*}

Taking expectation again and applying the tower property, we can get the result.

\subsection{Proof of Theorem \ref{th:eclsvrgsmooth-11}}

Let $\eta \leq \min\left\{  \frac{1}{4L_f + 33L/n}, \frac{\delta}{60\sqrt{(1-\delta)L_f{\bar L}}}, \frac{\sqrt{\delta}}{64\sqrt{(1-\delta) L_fL}}, \frac{\delta \sqrt{\delta_1}}{120\sqrt{(1-\delta)L_f{\bar L}p\left(1 + \frac{2p}{\delta_1} \right)}}  \right\}$. Then we have 
$$
\frac{11L\eta}{n} \leq \frac{1}{3}, \ \ \frac{896(1-\delta)L_f {\bar L}\eta^2}{\delta^2} \leq \frac{1}{4}, \ \ \frac{224(1-\delta)L_f L\eta^2}{\delta} \leq \frac{1}{18}, \ \ {\rm and} \ \ \frac{3584(1-\delta)L_f{\bar L}p\eta^2 \left(  1 + \frac{2p}{\delta_1}  \right)}{\delta^2 \delta_1} \leq \frac{1}{4}. 
$$

Hence, from Theorem \ref{th:eclsvrgsmooth-1}, we have 
\begin{eqnarray*}
	\mathbb{E} [\Phi_3^{k+1}] &\leq& \left(  1 - \min\left\{  \frac{\mu \eta}{2}, \frac{\delta}{4}, \frac{\delta_1}{4}, \frac{p}{4}  \right\}  \right) \mathbb{E} [\Phi_3^{k}] - \frac{\eta}{18} \mathbb{E} [f(x^{k}) - f(x^*)] \\ 
	&\leq& \left(  1 - \min\left\{  \frac{\mu \eta}{2}, \frac{\delta}{4}, \frac{\delta_1}{4}, \frac{p}{4}  \right\}  \right)^{k+1} \Phi^0_3 - \frac{\eta}{18} \sum_{i=0}^{k} \left(  1 - \min\left\{  \frac{\mu \eta}{2}, \frac{\delta}{4}, \frac{\delta_1}{4}, \frac{p}{4}  \right\}  \right)^{k-i} \mathbb{E}[f(x^i) - f(x^*)] \\ 
	&\leq& \left(  1 - \min\left\{  \frac{\mu \eta}{2}, \frac{\delta}{4}, \frac{\delta_1}{4}, \frac{p}{4}  \right\}  \right)^{k} \Phi^0_3 - \frac{\eta}{18} \sum_{i=0}^{k} \left(  1 - \min\left\{  \frac{\mu \eta}{2}, \frac{\delta}{4}, \frac{\delta_1}{4}, \frac{p}{4}  \right\}  \right)^{k-i} \mathbb{E}[f(x^i) - f(x^*)] \\ 
	&=& \frac{1}{w_k} \Phi_3^0 - \frac{\eta}{18w_k} \sum_{i=0}^k w_i \mathbb{E}[f(x^i) - f(x^*)], 
\end{eqnarray*}

which implies that 

\begin{equation}\label{eq:sumf-1-eclsvrgsmooth}
	\frac{1}{W_k} \sum_{i=0}^k w_i \mathbb{E}[f(x^i) - f(x^*)] \leq \frac{18}{\eta W_k} \Phi_3^0 - \frac{18w_k}{\eta W_k} \mathbb{E}[\Phi_3^{k+1}] \leq \frac{18}{\eta W_k} \Phi_3^0. 
\end{equation}

When $\mu>0$, from (\ref{eq:sumf-1-eclsvrgsmooth}) and 
\begin{equation}\label{eq:Wkeclsvrg}
W_k = \sum_{i=0}^k w_i = \frac{1 - \frac{1}{(1 - \min\{   \frac{\mu \eta}{2}, \frac{\delta}{4}, \frac{\delta_1}{4}, \frac{p}{4}   \})^{k+1}}}{1 - \frac{1}{1 - \min\{   \frac{\mu \eta}{2}, \frac{\delta}{4}, \frac{\delta_1}{4}, \frac{p}{4}    \}}} = \frac{1 - (1 - \min\{   \frac{\mu \eta}{2}, \frac{\delta}{4}, \frac{\delta_1}{4}, \frac{p}{4}   \})^{k+1}}{\min\{   \frac{\mu \eta}{2}, \frac{\delta}{4}, \frac{\delta_1}{4}, \frac{p}{4}   \} (1 - \min\{   \frac{\mu \eta}{2}, \frac{\delta}{4}, \frac{\delta_1}{4}, \frac{p}{4} \})^k}, 
\end{equation}
we can obtain 
\begin{eqnarray*}
	\frac{1}{W_k} \sum_{i=0}^k w_i \mathbb{E}[f(x^i) - f(x^*)] \leq \frac{\min\{   \frac{\mu \eta}{2}, \frac{\delta}{4}, \frac{\delta_1}{4}, \frac{p}{4}   \} }{1 - (1 - \min\{   \frac{\mu \eta}{2}, \frac{\delta}{4}, \frac{\delta_1}{4}, \frac{p}{4}   \})^{k+1}} \cdot \frac{18}{\eta} \Phi_3^0 \left(1 - \min\left\{   \frac{\mu \eta}{2}, \frac{\delta}{4}, \frac{\delta_1}{4}, \frac{p}{4}   \right\} \right)^k. 
\end{eqnarray*}

From the definition of $\Phi_3^k$ and $e^0_\tau=0$, we have 

\begin{eqnarray*}
	&& \min\left\{   \frac{\mu \eta}{2}, \frac{\delta}{4}, \frac{\delta_1}{4}, \frac{p}{4}   \right\} \cdot \frac{1}{\eta} \Phi_3^0 \\ 
	&=& \min\left\{   \frac{\mu \eta}{2}, \frac{\delta}{4}, \frac{\delta_1}{4}, \frac{p}{4}   \right\} \left(  \frac{1}{\eta} \|x^0-x^*\|^2 + \frac{192(1-\delta)L_f\eta^2}{\delta^2 \delta_1 n}\sum_{\tau=1}^n \|\nabla f^{(\tau)}(x^0)  - h_\tau^0\|^2 \right. \\ 
	&& \left.  + \frac{4}{3p}  \left(   \frac{48(1-\delta) L_f \eta^2}{\delta} \left(  \frac{4{\bar L}}{\delta} + L + \frac{16{\bar L}p}{\delta \delta_1} \left(  1 + \frac{2p}{\delta_1}  \right)  \right) + \frac{4L\eta}{n}  \right)  [f(x^{0}) - f(x^*)]  \right) \\ 
	&\leq& \frac{\mu}{2} \|x^0 - x^*\|^2 + \min\left\{   \frac{\mu \eta}{2}, \frac{\delta}{4}, \frac{\delta_1}{4}, \frac{p}{4}   \right\}  \frac{1}{70{\bar L}p n} \sum_{\tau=1}^n \|\nabla f^{(\tau)}(x^0)  - h_\tau^0\|^2 \\ 
	&&  + \min\left\{   \frac{\mu \eta}{2}, \frac{\delta}{4}, \frac{\delta_1}{4}, \frac{p}{4}   \right\} \frac{4}{3p} \left(  \frac{1}{18} + \frac{1}{84} + \frac{1}{18} + \frac{1}{8}  \right) [f(x^0) - f(x^*)] \\ 
	&\leq& \frac{\mu}{2} \|x^0 - x^*\|^2 + \frac{1}{280{\bar L}n} \sum_{\tau=1}^n \|\nabla f^{(\tau)}(x^0)  - h_\tau^0\|^2  + \frac{1}{9} [f(x^0) - f(x^*)]. 
\end{eqnarray*}

Therefore, we can get 
\begin{eqnarray*}
&& \frac{1}{W_k} \sum_{i=0}^k w_i \mathbb{E}[f(x^i) - f(x^*)] \\
&\leq& \frac{ 9\mu\|x^0-x^*\|^2 + 2(f(x^0) - f(x^*)) + \frac{1}{15{\bar L}n} \sum_{\tau=1}^n \|\nabla f^{(\tau)}(x^0)  - h_\tau^0\|^2 }{1 - (1 - \min\{   \frac{\mu \eta}{2}, \frac{\delta}{4}, \frac{\delta_1}{4}, \frac{p}{4}   \})^{k+1}} \left(1 - \min\left\{   \frac{\mu \eta}{2}, \frac{\delta}{4}, \frac{\delta_1}{4}, \frac{p}{4}   \right\} \right)^k. 
\end{eqnarray*}

For ${\bar x}^k = \frac{1}{W_k} \sum_{i=0}^k w_ix^i$, from the convexity of $f$ anf the above inequality, we have 
\begin{eqnarray*}
&& \mathbb{E} [f({\bar x}^k) - f(x^*)] \\ 
&\leq& \frac{ 9\mu\|x^0-x^*\|^2 + 2(f(x^0) - f(x^*)) + \frac{1}{15{\bar L}n} \sum_{\tau=1}^n \|\nabla f^{(\tau)}(x^0)  - h_\tau^0\|^2 }{1 - (1 - \min\{   \frac{\mu \eta}{2}, \frac{\delta}{4}, \frac{\delta_1}{4}, \frac{p}{4}   \})^{k+1}} \left(1 - \min\left\{   \frac{\mu \eta}{2}, \frac{\delta}{4}, \frac{\delta_1}{4}, \frac{p}{4}   \right\} \right)^k. 
\end{eqnarray*}

If we choose  $\eta = \min\left\{  \frac{1}{4L_f + 33L/n}, \frac{\delta}{60\sqrt{(1-\delta)L_f{\bar L}}}, \frac{\sqrt{\delta}}{64\sqrt{(1-\delta) L_fL}}, \frac{\delta \sqrt{\delta_1}}{120\sqrt{(1-\delta)L_f{\bar L}p\left(1 + \frac{2p}{\delta_1} \right)}}  \right\}$, then in order to guarantee $\mathbb{E} [f({\bar x}^k) - f(x^*)] \leq \epsilon$, we first let 
$$
\left(1 - \min\left\{   \frac{\mu \eta}{2}, \frac{\delta}{4}, \frac{\delta_1}{4}, \frac{p}{4}   \right\} \right)^{k+1} \leq \frac{1}{2}, 
$$
which implies that 
\begin{align*}
& \mathbb{E} [f({\bar x}^k) - f(x^*)] \\ 
& \leq  \left(  18\mu\|x^0-x^*\|^2 + 4(f(x^0) - f(x^*)) + \frac{2}{15{\bar L}n} \sum_{\tau=1}^n \|\nabla f^{(\tau)}(x^0)  - h_\tau^0\|^2  \right) \left(1 - \min\left\{   \frac{\mu \eta}{2}, \frac{\delta}{4}, \frac{\delta_1}{4}, \frac{p}{4}   \right\} \right)^k. 
\end{align*}
Hence, when $\epsilon \leq  9\mu\|x^0-x^*\|^2 + 2(f(x^0) - f(x^*)) + \frac{1}{15{\bar L}n} \sum_{\tau=1}^n \|\nabla f^{(\tau)}(x^0)  - h_\tau^0\|^2 $, $\mathbb{E} [f({\bar x}^k) - f(x^*)] \leq \epsilon$ as long as 
$$
\left(1 - \min\left\{   \frac{\mu \eta}{2}, \frac{\delta}{4}, \frac{p}{2}   \right\} \right)^k \leq \frac{\epsilon}{18\mu\|x^0-x^*\|^2 + 4(f(x^0) - f(x^*)) + \frac{2}{15{\bar L}n} \sum_{\tau=1}^n \|\nabla f^{(\tau)}(x^0)  - h_\tau^0\|^2  }, 
$$
which is equivalen to 
$$
k \geq \frac{1}{-\ln (1 - \min\left\{   \frac{\mu \eta}{2}, \frac{\delta}{4}, \frac{\delta_1}{4}, \frac{p}{4}   \right\} )} \ln \left(  \frac{18\mu\|x^0-x^*\|^2 + 4(f(x^0) - f(x^*)) + \frac{2}{15{\bar L}n} \sum_{\tau=1}^n \|\nabla f^{(\tau)}(x^0)  - h_\tau^0\|^2 }{\epsilon}  \right). 
$$

Since $-\ln(1-x) \geq x$ for $x\in [0, 1)$, if $p\leq O(\delta_1)$, we have $\mathbb{E} [f({\bar x}^k) - f(x^*)] \leq \epsilon$ as long as 
\begin{eqnarray*}
k &\geq& O\left( \left( \frac{1}{\delta} + \frac{1}{p} + \frac{L_f}{\mu} + \frac{L}{n \mu}  +  \frac{\sqrt{(1-\delta) L_f{\bar L}}}{\delta \mu}  + \frac{\sqrt{(1-\delta)L_f L}}{\sqrt{\delta}\mu}   \right) \right. \\ 
&& \left. \cdot  \ln \left(  \frac{18\mu\|x^0-x^*\|^2 + 4(f(x^0) - f(x^*)) + \frac{2}{15{\bar L}n} \sum_{\tau=1}^n \|\nabla f^{(\tau)}(x^0)  - h_\tau^0\|^2 }{\epsilon}  \right) \right). 
\end{eqnarray*}

When $\mu=0$, from (\ref{eq:sumf-1-eclsvrgsmooth}) and the convexity of $f$, we have 

$$
\mathbb{E}[f({\bar x}^k) - f(x^*)] \leq \frac{1}{k+1}\sum_{i=0}^k \mathbb{E}[f(x^i) - f(x^*)] \leq \frac{18}{\eta(k+1)}\Phi_3^0. 
$$

From the definition of $\Phi_3^k$ and $e^0_\tau=0$, we have 

\begin{eqnarray*}
	\frac{1}{\eta} \Phi_3^0 &=&  \frac{1}{\eta} \|x^0-x^*\|^2 + \frac{192(1-\delta)L_f\eta^2}{\delta^2 \delta_1 n}\sum_{\tau=1}^n \|\nabla f^{(\tau)}(x^0)  - h_\tau^0\|^2  \\ 
	&&  + \frac{4}{3p}  \left(   \frac{48(1-\delta) L_f \eta^2}{\delta} \left(  \frac{4{\bar L}}{\delta} + L + \frac{16{\bar L}p}{\delta \delta_1} \left(  1 + \frac{2p}{\delta_1}  \right)  \right) + \frac{4L\eta}{n}  \right)  [f(x^{0}) - f(x^*)]  \\ 
	&\leq& \frac{1}{\eta} \|x^0 - x^*\|^2 +  \frac{1}{70{\bar L}p n} \sum_{\tau=1}^n \|\nabla f^{(\tau)}(x^0)  - h_\tau^0\|^2 + \frac{4}{3p} \left(  \frac{1}{18} + \frac{1}{84} + \frac{1}{18} + \frac{1}{8}  \right) [f(x^0) - f(x^*)] \\ 
	&\leq& \frac{1}{\eta} \|x^0 - x^*\|^2 + \frac{1}{70{\bar L}p n} \sum_{\tau=1}^n \|\nabla f^{(\tau)}(x^0)  - h_\tau^0\|^2  + \frac{1}{3p} [f(x^0) - f(x^*)]. 
\end{eqnarray*}

Hence, we arrive at 
$$
\mathbb{E}[f({\bar x}^k) - f(x^*)] \leq  \frac{1}{k+1} \left( \frac{18}{\eta} \|x^0 - x^*\|^2 + \frac{1}{3{\bar L}p n} \sum_{\tau=1}^n \|\nabla f^{(\tau)}(x^0)  - h_\tau^0\|^2  + \frac{6}{p} \left(f(x^0) - f(x^*) \right)  \right). 
$$

In particular, if we choose $\eta= \min\left\{  \frac{1}{4L_f + 33L/n}, \frac{\delta}{60\sqrt{(1-\delta)L_f{\bar L}}}, \frac{\sqrt{\delta}}{64\sqrt{(1-\delta) L_fL}}, \frac{\delta \sqrt{\delta_1}}{120\sqrt{(1-\delta)L_f{\bar L}p\left(1 + \frac{2p}{\delta_1} \right)}}  \right\}$ and $p\leq O(\delta_1)$, we have $\mathbb{E}[f({\bar x}^k) - f(x^*)] \leq \epsilon$ as long as 
$$
k \geq O\left(  \left( \frac{1}{p} + L_f + \frac{L}{n} +  \frac{\sqrt{(1-\delta) L_f{\bar L}}}{\delta}  + \frac{\sqrt{(1-\delta)L_f L}}{\sqrt{\delta}}    \right) \frac{1}{\epsilon}  \right). 
$$

\subsection{Proof of Theorem \ref{th:eclsvrgsmooth-22}}

Let $\eta \leq \min \left\{  \frac{1}{4L_f + 33L/n}, \frac{\delta}{60\sqrt{1-\delta}L_f}, \frac{\sqrt{n \delta}}{229\sqrt{(1-\delta)L_fL}}, \frac{\sqrt{n} \delta}{360\sqrt{(1-\delta)L_f {\bar L}}} , \frac{\delta\sqrt{\delta_1}}{120\sqrt{(1-\delta)pL_f\left(  L_f + \frac{12{\bar L}}{n}  \right) \left(  1 + \frac{2p}{\delta_1}  \right)}} \right\}$. Then we have 
$$
\frac{896(1-\delta)L_f^2 \eta^2}{\delta^2} \leq \frac{1}{4}, \ \ \frac{7168(1-\delta)L_f {\bar L} \eta^2}{n\delta^2} \leq \frac{1}{18}, \ \ \frac{2912(1-\delta)L_f L \eta^2}{n \delta} \leq \frac{1}{18}, \ \ \frac{11L\eta}{n}\leq \frac{1}{3}, 
$$
and $\frac{3584(1-\delta)p\eta^2 L_f (L_f + \frac{12{\bar L}}{n})(1 + \frac{2p}{\delta_1})}{\delta^2 \delta_1} \leq \frac{1}{4}$.  

Therefore, from Theorem \ref{th:eclsvrgsmooth-2}, we have 

\begin{eqnarray*}
	\mathbb{E} [\Phi_4^{k+1}] &\leq& \left(  1 - \min\left\{  \frac{\mu \eta}{2}, \frac{\delta}{4}, \frac{\delta_1}{4}, \frac{p}{4}  \right\}  \right) \mathbb{E} [\Phi_4^{k}] - \frac{\eta}{36} \mathbb{E} [f(x^{k}) - f(x^*)] \\ 
	&\leq& \left(  1 - \min\left\{  \frac{\mu \eta}{2}, \frac{\delta}{4}, \frac{\delta_1}{4}, \frac{p}{4}  \right\}  \right)^{k+1} \Phi^0_4 - \frac{\eta}{36} \sum_{i=0}^{k} \left(  1 - \min\left\{  \frac{\mu \eta}{2}, \frac{\delta}{4}, \frac{\delta_1}{4}, \frac{p}{4}  \right\}  \right)^{k-i} \mathbb{E}[f(x^i) - f(x^*)] \\ 
	&\leq& \left(  1 - \min\left\{  \frac{\mu \eta}{2}, \frac{\delta}{4}, \frac{\delta_1}{4}, \frac{p}{4}  \right\}  \right)^{k} \Phi^0_4 - \frac{\eta}{36} \sum_{i=0}^{k} \left(  1 - \min\left\{  \frac{\mu \eta}{2}, \frac{\delta}{4}, \frac{\delta_1}{4}, \frac{p}{4}  \right\}  \right)^{k-i} \mathbb{E}[f(x^i) - f(x^*)] \\ 
	&=& \frac{1}{w_k} \Phi_4^0 - \frac{\eta}{36w_k} \sum_{i=0}^k w_i \mathbb{E}[f(x^i) - f(x^*)], 
\end{eqnarray*}

When $\mu>0$, notice that 
\begin{eqnarray*}
	&& \min\left\{   \frac{\mu \eta}{2}, \frac{\delta}{4}, \frac{\delta_1}{4}, \frac{p}{4}   \right\} \cdot \frac{1}{\eta} \Phi_4^0 \\ 
	&=& \min\left\{   \frac{\mu \eta}{2}, \frac{\delta}{4}, \frac{\delta_1}{4}, \frac{p}{4}   \right\} \left(  \frac{1}{\eta} \|x^0-x^*\|^2 + \frac{192(1-\delta)L_f \eta^2}{\delta^2 \delta_1} \|\nabla f(x^0)\|^2 + \frac{2304(1-\delta)L_f\eta^2}{\delta^2 \delta_1 n}\sum_{\tau=1}^n \|\nabla f^{(\tau)}(x^0)  - h_\tau^0\|^2 \right. \\ 
	&& \left.  + \frac{4}{3p}  \left(   \frac{48(1-\delta) L_f \eta^2}{\delta} \left(  \frac{4L_f}{\delta} + \frac{32{\bar L}}{n \delta} + \frac{13L}{n}  + \frac{16p(L_f + \frac{12{\bar L}}{n})}{\delta\delta_1} \left(  1 + \frac{2p}{\delta_1}  \right)  \right) + \frac{4L\eta}{n}  \right)  [f(x^{0}) - f(x^*)]  \right) \\ 
	&\leq& \frac{\mu}{2} \|x^0 - x^*\|^2 +  \min\left\{   \frac{\mu \eta}{2}, \frac{\delta}{4}, \frac{\delta_1}{4}, \frac{p}{4}   \right\} \frac{1}{70L_fp}\|\nabla f(x^0)\|^2 + \min\left\{   \frac{\mu \eta}{2}, \frac{\delta}{4}, \frac{\delta_1}{4}, \frac{p}{4}   \right\}  \frac{1}{6L_f p n} \sum_{\tau=1}^n \|\nabla f^{(\tau)}(x^0)  - h_\tau^0\|^2 \\ 
	&&  + \min\left\{   \frac{\mu \eta}{2}, \frac{\delta}{4}, \frac{\delta_1}{4}, \frac{p}{4}   \right\} \frac{4}{3p} \left(  \frac{1}{18} + \frac{1}{84}  + \frac{1}{84} + \frac{1}{18} + \frac{1}{8}  \right) [f(x^0) - f(x^*)] \\ 
	&\leq& \frac{\mu}{2} \|x^0 - x^*\|^2 + \frac{1}{280L_f}\|\nabla f(x^0) - \nabla f(x^*)\|^2 + \frac{1}{24L_f n} \sum_{\tau=1}^n \|\nabla f^{(\tau)}(x^0)  - h_\tau^0\|^2  + \frac{1}{9} [f(x^0) - f(x^*)] \\ 
	&\leq& \frac{\mu}{2} \|x^0 - x^*\|^2 + \frac{1}{24L_f n} \sum_{\tau=1}^n \|\nabla f^{(\tau)}(x^0)  - h_\tau^0\|^2  + \frac{1}{8} [f(x^0) - f(x^*)]. 
\end{eqnarray*}

Then same as the proof of Theorem \ref{th:eclsvrgsmooth-11}, we have 
\begin{eqnarray*}
\mathbb{E} [f({\bar x}^k) - f(x^*)] &\leq& \frac{ 18\mu\|x^0-x^*\|^2  + \frac{2}{L_f n}\sum_{\tau=1}^n \|\nabla f^{(\tau)}(x^0)  - h_\tau^0\|^2 +  5(f(x^0) - f(x^*)) }{1 - (1 - \min\{   \frac{\mu \eta}{2}, \frac{\delta}{4}, \frac{\delta_1}{4}, \frac{p}{4}   \})^{k+1}}\\ 
&& \cdot \left(1 - \min\left\{   \frac{\mu \eta}{2}, \frac{\delta}{4}, \frac{\delta_1}{4}, \frac{p}{4}   \right\} \right)^k, 
\end{eqnarray*}

and if we choose 
$$
\eta = \min \left\{  \tfrac{1}{4L_f + 33L/n}, \tfrac{\delta}{60\sqrt{1-\delta}L_f}, \tfrac{\sqrt{n \delta}}{229\sqrt{(1-\delta)L_fL}}, \tfrac{\sqrt{n} \delta}{360\sqrt{(1-\delta)L_f {\bar L}}} , \tfrac{\delta\sqrt{\delta_1}}{120\sqrt{(1-\delta)pL_f\left(  L_f + \tfrac{12{\bar L}}{n}  \right) \left(  1 + \tfrac{2p}{\delta_1}  \right)}} \right\}, 
$$ 
and $p\leq O(\delta_1)$, then $\mathbb{E} [f({\bar x}^k) - f(x^*)] \leq \epsilon$ with  $\epsilon \leq 9\mu\|x^0-x^*\|^2 + \frac{1}{L_f n}\sum_{\tau=1}^n \|\nabla f^{(\tau)}(x^0)  - h_\tau^0\|^2 +  3(f(x^0) - f(x^*))  $ as long as 
\begin{eqnarray*} 
	k &\geq& O\left( \left( \frac{1}{\delta}  + \frac{1}{p}  + \frac{L_f}{\mu} + \frac{L}{n \mu}  + \frac{\sqrt{(1-\delta)}L_f}{\mu \delta} + \frac{\sqrt{(1-\delta) L_f{\bar L}}}{\mu \sqrt{n}\delta}  + \frac{\sqrt{(1-\delta)L_f L}}{\mu \sqrt{n\delta}}  \right) \ln  \frac{1 }{\epsilon}  \right), 
\end{eqnarray*}
which is equivalent to 
\begin{eqnarray*}
k &\geq& O\left( \left( \frac{1}{\delta} + \frac{1}{p}  + \frac{L_f}{\mu} + \frac{L}{n \mu} + \frac{\sqrt{(1-\delta)}L_f}{\mu \delta}   \right) \ln  \frac{1  }{\epsilon}   \right), 
\end{eqnarray*}
since $\frac{{\bar L}}{n} \leq L_f$, and 
$$
2\sqrt{\frac{(1-\delta)L_f L}{n\delta}} \leq \frac{\sqrt{1-\delta}L_f}{\delta} + \frac{\sqrt{1-\delta}L}{n} \leq \frac{\sqrt{1-\delta}L_f}{\delta} + \frac{L}{n}. 
$$

When $\mu=0$, notice that 

\begin{eqnarray*}
	\frac{1}{\eta} \Phi_4^0 &=&  \frac{1}{\eta} \|x^0-x^*\|^2 + \frac{192(1-\delta)L_f \eta^2}{\delta^2 \delta_1} \|\nabla f(x^0)\|^2 + \frac{2304(1-\delta)L_f\eta^2}{\delta^2 \delta_1 n}\sum_{\tau=1}^n \|\nabla f^{(\tau)}(x^0)  - h_\tau^0\|^2  \\ 
	&&   + \frac{4}{3p}  \left(   \frac{48(1-\delta) L_f \eta^2}{\delta} \left(  \frac{4L_f}{\delta} + \frac{32{\bar L}}{n \delta} + \frac{13L}{n}  + \frac{16p(L_f + \frac{12{\bar L}}{n})}{\delta\delta_1} \left(  1 + \frac{2p}{\delta_1}  \right)  \right) + \frac{4L\eta}{n}  \right)  [f(x^{0}) - f(x^*)]  \\ 
	&\leq& \frac{1}{\eta} \|x^0 - x^*\|^2 +  \frac{1}{70L_fp}\|\nabla f(x^0)\|^2 +  \frac{1}{6L_f p n} \sum_{\tau=1}^n \|\nabla f^{(\tau)}(x^0)  - h_\tau^0\|^2 \\ 
	&&  + \frac{4}{3p} \left(  \frac{1}{18} + \frac{1}{84}  + \frac{1}{84} + \frac{1}{18} + \frac{1}{8}  \right) [f(x^0) - f(x^*)] \\ 
	&\leq& \frac{1}{\eta} \|x^0 - x^*\|^2 + \frac{1}{70L_f p}\|\nabla f(x^0) - \nabla f(x^*)\|^2 + \frac{1}{6L_f p n} \sum_{\tau=1}^n \|\nabla f^{(\tau)}(x^0)  - h_\tau^0\|^2  + \frac{4}{9p} [f(x^0) - f(x^*)] \\ 
	&\leq& \frac{1}{\eta} \|x^0 - x^*\|^2 + \frac{1}{6L_f p n} \sum_{\tau=1}^n \|\nabla f^{(\tau)}(x^0)  - h_\tau^0\|^2  + \left( \frac{4}{9p} + \frac{1}{35p} \right) [f(x^0) - f(x^*)]. 
\end{eqnarray*}

Then same as the proof of Theorem \ref{th:eclsvrgsmooth-11}, we have 
$$
\mathbb{E}[f({\bar x}^k) - f(x^*)] \leq \frac{1}{k+1} \left(  \frac{36}{\eta} \|x^0 - x^*\|^2  + \frac{6}{L_f p n} \sum_{\tau=1}^n \|\nabla f^{(\tau)}(x^0)  - h_\tau^0\|^2  + \frac{18}{p} \left(f(x^0) - f(x^*)\right)  \right). 
$$

In particular, if we choose 
$$
\eta = \min \left\{  \tfrac{1}{4L_f + 33L/n}, \tfrac{\delta}{60\sqrt{1-\delta}L_f}, \tfrac{\sqrt{n \delta}}{229\sqrt{(1-\delta)L_fL}}, \tfrac{\sqrt{n} \delta}{360\sqrt{(1-\delta)L_f {\bar L}}} , \tfrac{\delta\sqrt{\delta_1}}{120\sqrt{(1-\delta)pL_f\left(  L_f + \tfrac{12{\bar L}}{n}  \right) \left(  1 + \tfrac{2p}{\delta_1}  \right)}} \right\}, 
$$ 
and $p\leq O(\delta_1)$, we have $\mathbb{E}[f({\bar x}^k) - f(x^*)] \leq \epsilon$ as long as 
$$
k \geq O\left( \left(  \frac{1}{p} +  L_f + \frac{L}{n} + \frac{\sqrt{(1-\delta)}L_f}{\delta} \right)  \frac{1}{\epsilon}  \right). 
$$

\section{Proofs for EC-Quartz}

\subsection{Lemmas}

For brevity, let $$A = [ A_{11}, ..., A_{m1}, A_{12}, ..., A_{m2}, ..., A_{n1}, ..., A_{mn} ] \in \R^{d \times tN}.$$ Let $S = \{  (i^\tau, \tau) | \ i^\tau \mbox{ is chosen from $[m]$ uniformly and independently for all } \tau \in [n]  \}$. For any vector $h \in \R^{tN}$, let $h_{[S]} \in \R^{tN}$ be defined by  
$$
(h_{[S]})_{i\tau} = \left\{ \begin{array}{rl}
h_{i\tau} & \mbox{ if $(i, \tau) \in S$} \\
0 &\mbox{ otherwise }
\end{array} \right.
$$

\begin{lemma}\label{lm:eso}[ESO] 
	The following inequality holds for all $h \in \R^{tN}$: 
	\begin{equation}\label{eq:esouni}
	\mathbb{E} [\| Ah_{[S]}\|^2 ] \leq \sum_{\tau=1}^n \sum_{i=1}^m p_{i\tau}v_{i\tau} \|h_{i\tau}\|^2, 
	\end{equation}
	where $p_{i\tau} = \frac{1}{m}$ and $v_{i\tau} = R_m^2 + nR^2$. 
\end{lemma}

\begin{proof}

We give two proofs. 

{\bf First proof.} Notice that $S$ can be regarded as a group sampling \citep{qian2019svrg} where the index set on each node is a group and $\mathbb{P}[(i, \tau) \in S] = p_{i\tau} = \frac{1}{m}$ for all $i\in [m]$ and $\tau \in [n]$. Hence, from Lemma 6.6 in \citep{qian2019svrg}, we have 
\begin{align*}
m^2  \mathbb{E} \left[  \left\| \sum_{(i, \tau) \in S} A_{i\tau} h_{i\tau} \right\|^2  \right]  & = \mathbb{E} \left[  \left\| \sum_{(i, \tau) \in S} \frac{1}{p_{i\tau}} A_{i\tau} h_{i\tau} \right\|^2  \right] \\ 
& \leq \sum_{\tau=1}^n \sum_{i=1}^m \frac{1}{p_{i\tau}} \|A_{i\tau} h_{i\tau}\|^2 + \left\| \sum_{\tau=1}^n \sum_{i=1}^m A_{i\tau} h_{i\tau} \right\|^2 \\ 
& = m \sum_{\tau=1}^n \sum_{i=1}^m \|A_{i\tau} h_{i\tau}\|^2 + \| Ah\|^2 \\ 
& \leq m \sum_{\tau=1}^n \sum_{i=1}^m \|A_{i\tau}\|^2 \|h_{i\tau}\|^2 + \lambda_{\rm max}(A^\top A) \|h\|^2 \\ 
& = m \sum_{\tau=1}^n \sum_{i=1}^m \|A_{i\tau}\|^2 \|h_{i\tau}\|^2 + \lambda_{\rm max}(A A^\top) \|h\|^2 \\ 
& \leq (m R_m^2 + NR^2) \sum_{\tau=1}^n \sum_{i=1}^m \|h_{i\tau}\|^2,
\end{align*}
where in the last inequality, we use $\|A_{i\tau}\| \leq \max_{i, \tau} \|A_{i\tau}\| = R_m$ and $$\frac{1}{N}\lambda_{\rm max}(A A^\top)  = \frac{1}{N} \lambda_{\rm max} (\sum_{\tau=1}^n \sum_{i=1}^m A_{i\tau} A_{i\tau}^\top) = R^2. $$ 
Then we arrive at
\begin{align*}
\mathbb{E} \left[\| Ah_{[S]}\|^2 \right] & =  \mathbb{E} \left[  \left\| \sum_{(i, \tau) \in S} A_{i\tau} h_{i\tau} \right\|^2  \right] 
 \leq \frac{1}{m} (R_m^2 + nR^2)  \sum_{\tau=1}^n \sum_{i=1}^m \|h_{i\tau}\|^2  =  \sum_{\tau=1}^n \sum_{i=1}^m p_{i\tau} v_{i\tau} \|h_{i\tau}\|^2. 
\end{align*}

{\bf Second proof.} From the definition of $S$, we have 
\begin{align*}
\mathbb{E} [\| Ah_{[S]}\|^2 ] & =  \mathbb{E} \left[  \left\| \sum_{(i, \tau) \in S} A_{i\tau} h_{i\tau} \right\|^2  \right] \\
& =  \mathbb{E} \left[  \left\| \sum_{\tau=1}^n A_{i^\tau \tau} h_{i^\tau \tau} \right\|^2  \right] \\ 
& = \mathbb{E} \left[ \sum_{\tau_1 \neq \tau_2} \langle A_{i^{\tau_1} \tau_1} h_{i^{\tau_1} \tau_1}, A_{i^{\tau_2} \tau_2} h_{i^{\tau_2} \tau_2} \rangle  \right] + \mathbb{E} \left[  \sum_{\tau=1}^n \| A_{i^\tau \tau} h_{i^\tau \tau}\|^2 \right] \\ 
& = \sum_{\tau_1 \neq \tau_2} \left\langle \frac{1}{m} \sum_{i=1}^m A_{i \tau_1} h_{i \tau_1}, \frac{1}{m} \sum_{j=1}^m A_{j \tau_2} h_{j \tau_2} \right\rangle + \frac{1}{m} \sum_{\tau=1}^n \sum_{i=1}^m \|A_{i\tau} h_{i\tau}\|^2 \\ 
& = \frac{1}{m^2} \sum_{\tau_1 \neq \tau_2} \sum_{i, j=1}^m \langle  A_{i \tau_1} h_{i \tau_1}, A_{j \tau_2} h_{j \tau_2}  \rangle + \frac{1}{m} \sum_{\tau=1}^n \sum_{i=1}^m \|A_{i\tau} h_{i\tau}\|^2 \\ 
& =  \frac{1}{m^2} \sum_{\tau_1, \tau_2=1}^n  \sum_{i, j=1}^m \langle  A_{i \tau_1} h_{i \tau_1}, A_{j \tau_2} h_{j \tau_2}  \rangle - \frac{1}{m^2} \sum_{\tau=1}^n \sum_{i, j=1}^m \langle  A_{i \tau} h_{i \tau}, A_{j \tau} h_{j \tau}  \rangle \\ 
& \quad + \frac{1}{m} \sum_{\tau=1}^n \sum_{i=1}^m \|A_{i\tau} h_{i\tau}\|^2 \\ 
& = \frac{1}{m^2} \|Ah\|^2 - \frac{1}{m^2} \sum_{\tau=1}^n \left\| \sum_{i=1}^m A_{i\tau}h_{i\tau} \right\|^2 + \frac{1}{m} \sum_{\tau=1}^n \sum_{i=1}^m \|A_{i\tau} h_{i\tau}\|^2  \\ 
& \leq  \frac{1}{m^2} \|Ah\|^2 + \frac{1}{m} \sum_{\tau=1}^n \sum_{i=1}^m \|A_{i\tau} h_{i\tau}\|^2 \\ 
& \leq \frac{NR^2}{m^2} \sum_{\tau=1}^n \sum_{i=1}^m \| h_{i\tau}\|^2 + \frac{1}{m} R_m^2  \sum_{\tau=1}^n \sum_{i=1}^m \| h_{i\tau}\|^2 \\ 
& =  \sum_{\tau=1}^n \sum_{i=1}^m p_{i\tau} v_{i\tau} \|h_{i\tau}\|^2, 
\end{align*}
where in the fourth equality, we use the fact that $i^{\tau_1}$ is indpendent of $i^{\tau_2}$ for $\tau_1 \neq \tau_2$. 

\end{proof}

\begin{lemma}\label{lm:tildef}[Lemma 18 in \citep{Quartz}]
	Function ${\tilde f} : \R^{tN} \to \R$ satisfies the following inequality: 
	\begin{equation}\label{eq:tildef}
	{\tilde f} (\alpha + h) \leq {\tilde f}(\alpha) + \langle \nabla {\tilde f} (\alpha), h \rangle + \frac{1}{2\lambda N^2} h^\top A^\top A h, \quad \forall \alpha, h \in \R^{tN}. 
	\end{equation}
\end{lemma}

\begin{lemma}\label{lm:Dalpha}[Lemma 19 in \citep{Quartz}] 
	For all $\alpha$, $h\in \R^{tN}$, the following holds: 
	\begin{align}
	&\mathbb{E} [- D(\alpha + h_{[S]}) ] \nonumber \\
	& \leq {\tilde f}(\alpha) + \sum_{\tau=1}^n \sum_{i=1}^m p_{i\tau} \left\langle  \frac{1}{N} A_{i\tau}^\top \nabla g^* \left(\frac{1}{\lambda N} A\alpha \right), h_{i\tau}  \right\rangle + \frac{1}{2\lambda N^2} \sum_{\tau=1}^n \sum_{i=1}^m p_{i\tau} v_{i\tau} \|h_{i\tau}\|^2 \nonumber \\ 
	& \quad + \frac{1}{N} \sum_{\tau=1}^n \sum_{i=1}^m [(1 - p_{i\tau})\phi^*_{i\tau} (-\alpha_{i\tau}) + p_{i\tau} \phi^*_{i\tau} (-\alpha_{i\tau} - h_{i\tau}) ].     \label{eq:Dalpha}
	\end{align}
	
\end{lemma}

\begin{lemma}\label{lm:Dalpha-2}
	Fixing $\alpha \in \R^{tN}$ and $x \in \R^d$, let ${\tilde u} = \frac{1}{\lambda N} A \alpha$ and $h\in \R^{tN}$ be defined by: 
	$$
	h_{i\tau} = - \theta p_{i\tau}^{-1} (\alpha_{i\tau} + \nabla \phi_{i\tau} (A_{i\tau}^\top x) ), \ i\in [m], \ \tau \in [n], 
	$$
	where $\theta>0$. Then 
	\begin{align*}
	& {\tilde f}(\alpha) + \sum_{\tau=1}^n \sum_{i=1}^m p_{i\tau} \left\langle  \frac{1}{N} A_{i\tau}^\top \nabla g^* \left(\frac{1}{\lambda N} A\alpha \right), h_{i\tau}  \right\rangle + \frac{1}{2\lambda N^2} \sum_{\tau=1}^n \sum_{i=1}^m p_{i\tau} v_{i\tau} \|h_{i\tau}\|^2 \nonumber \\ 
	& \quad + \frac{1}{N} \sum_{\tau=1}^n \sum_{i=1}^m [(1 - p_{i\tau})\phi^*_{i\tau} (-\alpha_{i\tau}) + p_{i\tau} \phi^*_{i\tau} (-\alpha_{i\tau} - h_{i\tau}) ]\\ 
	& \leq -(1-\theta) D(\alpha) - \theta \lambda g(\nabla g^*({u})) -  \frac{1}{N} \sum_{\tau=1}^n \sum_{i=1}^m \langle \theta \nabla g^*(u), A_{i\tau} \nabla \phi_{i\tau} (A_{i\tau}^\top x) \rangle \\ 
	& \quad + \frac{\theta}{N} \sum_{\tau=1}^n \sum_{i=1}^m \phi^*_{i\tau} (\nabla \phi_{i\tau}(A_{i\tau}^\top x)) + \frac{\rho + \theta\lambda}{2} \|{\tilde u} - u\|^2 \\ 
	& \quad + \frac{1}{2\lambda N^2} \sum_{\tau=1}^n \sum_{i=1}^m \left( p_{i\tau}v_{i\tau} + \frac{ N\lambda p_{i\tau}^2 R^2}{\rho} -  \frac{N \lambda \gamma p_{i\tau}^2 (1-\theta p_{i\tau}^{-1})}{\theta}  \right) \|h_{i\tau}\|^2, 
	\end{align*}
	for any $u\in \R^d$ and $\rho>0$. 
\end{lemma}

\begin{proof}
	
	First, for any $u\in \R^d$ and $\rho>0$, we have 
	\begin{align*}
	& {\tilde f}(\alpha) + \sum_{\tau=1}^n \sum_{i=1}^m p_{i\tau} \left\langle  \frac{1}{N} A_{i\tau}^\top \nabla g^* \left(\frac{1}{\lambda N} A\alpha \right), h_{i\tau}  \right\rangle \\
	& =  {\tilde f}(\alpha) + \frac{1}{N} \sum_{\tau=1}^n \sum_{i=1}^m p_{i\tau} \left\langle  A_{i\tau}^\top \nabla g^* \left({\tilde u} \right), h_{i\tau}  \right\rangle \\ 
	& = {\tilde f}(\alpha) + \frac{1}{N} \sum_{\tau=1}^n \sum_{i=1}^m p_{i\tau}  \langle A_{i\tau}^\top \nabla g^*(u), h_{i\tau} \rangle + \frac{1}{N} \sum_{\tau=1}^n \sum_{i=1}^m p_{i\tau}  \langle  (\nabla g^*({\tilde u}) - \nabla g^*(u)) , A_{i\tau} h_{i\tau} \rangle \\ 
	& =  {\tilde f}(\alpha) + \frac{1}{N} \sum_{\tau=1}^n \sum_{i=1}^m p_{i\tau}  \langle A_{i\tau}^\top \nabla g^*(u), h_{i\tau} \rangle +  \langle  (\nabla g^*({\tilde u}) - \nabla g^*(u)) , \frac{1}{N} Ah^p \rangle \\ 
	& \leq  {\tilde f}(\alpha) + \frac{1}{N} \sum_{\tau=1}^n \sum_{i=1}^m p_{i\tau}  \langle A_{i\tau}^\top \nabla g^*(u), h_{i\tau} \rangle +  \frac{\rho}{2} \| \nabla g^*({\tilde u}) - \nabla g^*(u)\|^2 + \frac{1}{2\rho N^2}\|A h^p\|^2  \\ 
	& \leq {\tilde f}(\alpha) + \frac{1}{N} \sum_{\tau=1}^n \sum_{i=1}^m p_{i\tau}  \langle A_{i\tau}^\top \nabla g^*(u), h_{i\tau} \rangle + \frac{\rho}{2} \|{\tilde u} - u\|^2 + \frac{1}{N} \sum_{\tau=1}^n \sum_{i=1}^m \frac{p_{i\tau}^2 R^2}{2\rho}\|h_{i\tau}\|^2, 
	\end{align*}
	where we denote $h^p \in \R^{tn}$ such that $h^p_{i\tau} = {p_{i\tau}} h_{i\tau}$ in the third equality, in the first inequality we use the Young's inequality and the last inequality comes from $g^*$ is $1$-smooth since $g$ is $1$-strongly convex. For the first two terms in the above inequality, we have 
	
	\begin{align*}
	& {\tilde f}(\alpha) + \frac{1}{N} \sum_{\tau=1}^n \sum_{i=1}^m p_{i\tau}  \langle A_{i\tau}^\top \nabla g^*(u), h_{i\tau} \rangle \\
	& = \lambda g^*({\tilde u}) - \frac{\theta}{N} \sum_{\tau=1}^n \sum_{i=1}^m \langle \nabla g^*(u), A_{i\tau} \alpha_{i\tau} + A_{i\tau} \nabla \phi_{i\tau} (A_{i\tau}^\top x) \rangle \\ 
	& = \lambda g^*({\tilde u}) - \theta \lambda \langle \nabla g^*(u), {\tilde u} \rangle -  \frac{1}{N} \sum_{\tau=1}^n \sum_{i=1}^m \langle \theta \nabla g^*(u), A_{i\tau} \nabla \phi_{i\tau} (A_{i\tau}^\top x) \rangle \\ 
	& = (1-\theta) \lambda g^*({\tilde u}) + \theta \lambda g^*({\tilde u}) - \theta \lambda g^*(u) - \theta \lambda \langle \nabla g^*(u), {\tilde u} - u \rangle + \theta \lambda g^*(u) - \theta \lambda \langle \nabla g^*(u), u \rangle \\ 
	& \quad  -  \frac{1}{N} \sum_{\tau=1}^n \sum_{i=1}^m \langle \theta \nabla g^*(u), A_{i\tau} \nabla \phi_{i\tau} (A_{i\tau}^\top x) \rangle \\ 
	& \leq (1-\theta) \lambda g^*({\tilde u}) + \frac{\theta \lambda}{2}\|{\tilde u} - u\|^2 + \theta \lambda (g^*(u) - \langle \nabla g^*(u), u \rangle) \\ 
	& \quad -  \frac{1}{N} \sum_{\tau=1}^n \sum_{i=1}^m \langle \theta \nabla g^*(u), A_{i\tau} \nabla \phi_{i\tau} (A_{i\tau}^\top x) \rangle \\ 
	& = (1-\theta) \lambda g^*({\tilde u}) + \frac{\theta \lambda}{2}\|{\tilde u} - u\|^2 - \theta \lambda g(\nabla g^*({u})) -  \frac{1}{N} \sum_{\tau=1}^n \sum_{i=1}^m \langle \theta \nabla g^*(u), A_{i\tau} \nabla \phi_{i\tau} (A_{i\tau}^\top x) \rangle, 
	\end{align*}
	where the first inequality comes from $g^*$ is $1$-smooth and the last equality comes from the definition of conjugate functions. 
	
	From (50) in \citep{Quartz}, we also have 
	
	\begin{align*}
	& \frac{1}{N} \sum_{\tau=1}^n \sum_{i=1}^m [(1 - p_{i\tau})\phi^*_{i\tau} (-\alpha_{i\tau}) + p_{i\tau} \phi^*_{i\tau} (-\alpha_{i\tau} - h_{i\tau}) ] \\ 
	& \leq (1-\theta) {\tilde \psi}(\alpha) + \frac{\theta}{N} \sum_{\tau=1}^n \sum_{i=1}^m \phi^*_{i\tau} (\nabla \phi_{i\tau}(A_{i\tau}^\top x)) - \frac{1}{2\lambda N^2} \sum_{\tau=1}^n \sum_{i=1}^m \frac{N \lambda \gamma p_{i\tau}^2 (1-\theta p_{i\tau}^{-1})}{\theta} \|h_{i\tau}\|^2. 
	\end{align*}
	
	Combining the above three inequalities, we arrive at 
	\begin{align*}
	& {\tilde f}(\alpha) + \sum_{\tau=1}^n \sum_{i=1}^m p_{i\tau} \left\langle  \frac{1}{N} A_{i\tau}^\top \nabla g^* \left(\frac{1}{\lambda N} A\alpha \right), h_{i\tau}  \right\rangle + \frac{1}{2\lambda N^2} \sum_{\tau=1}^n \sum_{i=1}^m p_{i\tau} v_{i\tau} \|h_{i\tau}\|^2 \nonumber \\ 
	& \quad + \frac{1}{N} \sum_{\tau=1}^n \sum_{i=1}^m [(1 - p_{i\tau})\phi^*_{i\tau} (-\alpha_{i\tau}) + p_{i\tau} \phi^*_{i\tau} (-\alpha_{i\tau} - h_{i\tau}) ]\\ 
	& \leq -(1-\theta) D(\alpha) - \theta \lambda g(\nabla g^*({u})) -  \frac{1}{N} \sum_{\tau=1}^n \sum_{i=1}^m \langle \theta \nabla g^*(u), A_{i\tau} \nabla \phi_{i\tau} (A_{i\tau}^\top x) \rangle \\ 
	& \quad + \frac{\theta}{N} \sum_{\tau=1}^n \sum_{i=1}^m \phi^*_{i\tau} (\nabla \phi_{i\tau}(A_{i\tau}^\top x)) + \frac{\rho + \theta\lambda}{2} \|{\tilde u} - u\|^2 \\ 
	& \quad + \frac{1}{2\lambda N^2} \sum_{\tau=1}^n \sum_{i=1}^m \left( p_{i\tau}v_{i\tau} + \frac{ N\lambda p_{i\tau}^2 R^2}{\rho} -  \frac{N \lambda \gamma p_{i\tau}^2 (1-\theta p_{i\tau}^{-1})}{\theta}  \right) \|h_{i\tau}\|^2. 
	\end{align*}
	
\end{proof}

Let $\mathbb{E}_k[\cdot]$ denote the expectation conditional on $x^k$, $\alpha^k$, $u^k$, and $e^k_\tau$. Define $\Delta \alpha_{i\tau}^{k+1} \eqdef - \theta p_{i\tau}^{-1} \alpha_{i\tau}^k - \theta p_{i\tau}^{-1} \nabla \phi_{i\tau} (A_{i\tau}^\top x^{k+1}) $ for $k\geq 0$. Notice that $\mathbb{E}_k[x^{k+1}] = x^{k+1}$. Hence, $\mathbb{E}_k[\Delta \alpha_{i\tau}^{k+1} ] = \Delta \alpha_{i\tau}^{k+1}$.

\begin{lemma}\label{lm:ek+1-1ecQuartz}
	We have 
	\begin{eqnarray*}
		\frac{1}{n} \sum_{\tau=1}^n  \mathbb{E}_k [\|e^{k+1}_{\tau}\|^2] \leq  \left(1-\frac{\delta}{2} \right) \frac{1}{n} \sum_{\tau=1}^n \|e^k_\tau\|^2 + \frac{(1-\delta)}{\lambda^2 Nm^2} \left(  \frac{2{\bar R}^2}{\delta} + R_m^2  \right) \sum_{\tau=1}^n \sum_{i=1}^m \| \Delta \alpha_{i\tau}^{k+1}\|^2. 
	\end{eqnarray*}
	
\end{lemma}

\begin{proof}
	
	First, from the contraction property of $Q$, we have 
	\begin{eqnarray*}
		&& \mathbb{E}_k[ \|e^{k+1}_\tau\|^2 ] \\ 
		&\overset{(\ref{eq:contractor})}{\leq}& (1-\delta)\mathbb{E}_k \|e^k_\tau + \frac{1}{\lambda m}A_{i_k^\tau \tau} \Delta\alpha_{i^\tau_{k} \tau}^{k+1} \|^2 \\ 
		&=& (1-\delta) \mathbb{E}_k [\|e_\tau^k + \frac{1}{\lambda m^2} \sum_{i=1}^m A_{i\tau} \Delta \alpha_{i\tau}^{k+1} +  \frac{1}{\lambda m}A_{i_k^\tau \tau} \Delta\alpha_{i^\tau_{k} \tau}^{k+1} -   \frac{1}{\lambda m^2} \sum_{i=1}^m A_{i\tau} \Delta \alpha_{i\tau}^{k+1} \|^2 ] \\ 
		&=& (1-\delta)\mathbb{E}_k [ \|e_\tau^k + \frac{1}{\lambda m^2} \sum_{i=1}^m A_{i\tau} \Delta \alpha_{i\tau}^{k+1}\|^2 ] + (1-\delta) \mathbb{E}_k [ \| \frac{1}{\lambda m}A_{i_k^\tau \tau} \Delta\alpha_{i^\tau_{k} \tau}^{k+1} -   \frac{1}{\lambda m^2} \sum_{i=1}^m A_{i\tau} \Delta \alpha_{i\tau}^{k+1} \|^2 ] \\
		&\leq& (1-\delta)\mathbb{E}_k [ \|e_\tau^k + \frac{1}{\lambda m^2} \sum_{i=1}^m A_{i\tau} \Delta \alpha_{i\tau}^{k+1}\|^2 ] + (1-\delta) \mathbb{E}_k [ \| \frac{1}{\lambda m}A_{i_k^\tau \tau} \Delta\alpha_{i^\tau_{k} \tau}^{k+1} \|^2 ] \\ 
		&\leq&  (1-\delta)\mathbb{E}_k [ \|e_\tau^k + \frac{1}{\lambda m^2} \sum_{i=1}^m A_{i\tau} \Delta \alpha_{i\tau}^{k+1}\|^2 ] + \frac{(1-\delta)R_m^2}{\lambda^2 m^3}\sum_{i=1}^m \| \Delta \alpha_{i\tau}^{k+1}\|^2 \\ 
		&\leq& (1-\delta)(1+\beta) \|e^k_\tau\|^2 + (1-\delta)\left(1+ \frac{1}{\beta} \right) \frac{1}{\lambda^2 m^4} \left\| \sum_{i=1}^m A_{i\tau} \Delta \alpha_{i\tau}^{k+1} \right\|^2 +  \frac{(1-\delta)R_m^2}{\lambda^2 m^3}\sum_{i=1}^m \| \Delta \alpha_{i\tau}^{k+1}\|^2  \\
		&\leq& \left(1-\frac{\delta}{2} \right) \|e^k_\tau\|^2 + \frac{2(1-\delta)}{\delta} \frac{{\bar R}^2}{\lambda^2m^3} \sum_{i=1}^m \|\Delta \alpha_{i\tau}^{k+1}\|^2  +  \frac{(1-\delta)R_m^2}{\lambda^2 m^3}\sum_{i=1}^m \| \Delta \alpha_{i\tau}^{k+1}\|^2 \\ 
		&=&  \left(1-\frac{\delta}{2} \right) \|e^k_\tau\|^2 + \frac{(1-\delta)}{\lambda^2 m^3} \left(  \frac{2{\bar R}^2}{\delta} + R_m^2  \right) \sum_{i=1}^m \| \Delta \alpha_{i\tau}^{k+1}\|^2, 
	\end{eqnarray*}
	where we use Young's inequality in the third inequality and choose $\beta = \frac{\delta}{2(1-\delta)}$ when $\delta<1$. When $\delta = 1$, it is easy to see that the above inequality also holds. 
	
	Taking the average of the above inequality from $\tau=1$ to $n$, we can get the result. 
	
\end{proof}

\begin{lemma}\label{lm:ek+1-2ecQuartz}
	Let $e^k = \frac{1}{n}\sum_{\tau=1}^n e_\tau^k$ for $k\geq 0$. Under Assumption \ref{as:expcompressor}, we have 
	\begin{eqnarray*}
		&& \mathbb{E}_k \|e^{k+1}\|^2 \\ 
		&\leq&  \left(  1 - \frac{\delta}{2}  \right) \|e^k\|^2 + \frac{2(1-\delta)\delta}{n^2} \sum_{\tau=1}^n \|e^k_{\tau} \|^2 + \frac{(1-\delta)}{\lambda^2 N m^2} \left(  \frac{(2\delta+1) R_m^2}{n}  + \frac{2R^2}{\delta}  \right) \sum_{\tau=1}^n \sum_{i=1}^m \| \Delta \alpha_{i\tau}^{k+1}\|^2. 
	\end{eqnarray*}
	
\end{lemma}

\begin{proof}
	
	Under Assumption \ref{as:expcompressor}, we have $\mathbb{E}[Q(x)] = \delta x$, and 
	\begin{eqnarray*}
		\mathbb{E}_k \|e^{k+1}\|^2 &=& \mathbb{E}_k \left\| \frac{1}{n} \sum_{\tau=1}^n e^{k+1}_{\tau} \right\|^2 \\
		&=& \frac{1}{n^2} \sum_{j, s} \mathbb{E}_k \langle e^{k+1}_j, e^{k+1}_s \rangle \\ 
		&=& \frac{1}{n^2} \sum_{\tau=1}^n \mathbb{E}_k \|e^{k+1}_{\tau}\|^2 + \frac{1}{n^2} \sum_{j\neq s} \mathbb{E}_k \langle e^{k+1}_j, e^{k+1}_s \rangle \\
		&\overset{(\ref{eq:contractor})}{\leq}& \frac{1-\delta}{n^2} \sum_{\tau=1}^n \mathbb{E}_k \left\|e^k_{\tau} + \frac{1}{\lambda m}A_{i_k^\tau \tau} \Delta\alpha_{i^\tau_{k} \tau}^{k+1} \right\|^2 \\ 
		&& + \frac{(1-\delta)^2}{n^2} \sum_{j\neq s} \mathbb{E}_k \left\langle e^k_j +\frac{1}{\lambda m}A_{i_k^j j} \Delta\alpha_{i^j_{k} j}^{k+1}  , e^k_s + \frac{1}{\lambda m}A_{i_k^s \tau} \Delta\alpha_{i^s_{k} s}^{k+1}  \right\rangle \\ 
		&=& \frac{(1-\delta)^2}{n^2} \mathbb{E}_k \left\|\sum_{\tau=1}^n (e^k_{\tau} + \frac{1}{\lambda m}A_{i_k^\tau \tau} \Delta\alpha_{i^\tau_{k} \tau}^{k+1}) \right\|^2 + \frac{(1-\delta)\delta}{n^2} \sum_{\tau=1}^n \mathbb{E}_k \left\|e^k_{\tau} + \frac{1}{\lambda m}A_{i_k^\tau \tau} \Delta\alpha_{i^\tau_{k} \tau}^{k+1} \right\|^2 \\ 
		&\leq& (1-\delta) \mathbb{E}_k \left\|e^k + \frac{1}{n} \sum_{\tau=1}^n \frac{1}{\lambda m}A_{i_k^\tau \tau} \Delta\alpha_{i^\tau_{k} \tau}^{k+1}  \right\|^2 + \frac{(1-\delta)\delta}{n^2} \sum_{\tau=1}^n \mathbb{E}_k \left\|e^k_{\tau} + \frac{1}{\lambda m}A_{i_k^\tau \tau} \Delta\alpha_{i^\tau_{k} \tau}^{k+1} \right\|^2, 
	\end{eqnarray*}
	where we use the definitions of $e^k$ in the last inequality. \\
	Let $S_k = \{  (i^\tau_k, \tau) | \ i^\tau_k \mbox{ is chosen from $[m]$ uniformly and independently for all } \tau \in [n]  \}$.  Then $ \sum_{\tau=1}^n A_{i_k^\tau \tau} \Delta\alpha_{i^\tau_{k} \tau}^{k+1} = \sum_{(i, \tau) \in S_k} A_{i\tau} \Delta \alpha^{k+1}_{i\tau} $, and we can obtain 
	
	\begin{eqnarray}
	&& \mathbb{E}_k \|e^{k+1}\|^2 - (1-\delta) \mathbb{E}_k \left\|e^k + \frac{1}{\lambda N }\sum_{(i, \tau) \in S_k} A_{i\tau} \Delta \alpha^{k+1}_{i\tau}  \right\|^2 \nonumber \\ 
	&\leq&   \frac{(1-\delta)\delta}{n^2} \sum_{\tau=1}^n \mathbb{E}_k \left\|e^k_{\tau} + \frac{1}{\lambda m}A_{i_k^\tau \tau} \Delta\alpha_{i^\tau_{k} \tau}^{k+1}\right\|^2 \nonumber \\ 
	&\leq&  \frac{2(1-\delta)\delta}{n^2} \sum_{\tau=1}^n \|e^k_{\tau} \|^2 + \frac{2(1-\delta)\delta  }{\lambda^2 m^2 n^2 } \sum_{\tau=1}^n \mathbb{E}_k \| A_{i_k^\tau \tau} \Delta\alpha_{i^\tau_{k} \tau}^{k+1} \|^2 \nonumber \\ 
	&=& \frac{2(1-\delta)\delta}{n^2} \sum_{\tau=1}^n \|e^k_{\tau} \|^2 + \frac{2(1-\delta)\delta  }{\lambda^2 m^3 n^2 } \sum_{\tau=1}^n \sum_{i=1}^m  \| A_{i \tau} \Delta\alpha_{i \tau}^{k+1} \|^2 \nonumber \\ 
	&\leq& \frac{2(1-\delta)\delta}{n^2} \sum_{\tau=1}^n \|e^k_{\tau} \|^2 + \frac{2(1-\delta)\delta R_m^2 }{\lambda^2 N^2 m } \sum_{\tau=1}^n \sum_{i=1}^m  \| \Delta\alpha_{i \tau}^{k+1} \|^2, \label{eq:ek+1in-ecQuartz}
	\end{eqnarray}
	where in the second and third inequalities we use the Young's inequality. 
	
	For $(1-\delta) \mathbb{E}_k \left\|e^k + \frac{1}{\lambda N }\sum_{(i, \tau) \in S_k} A_{i\tau} \Delta \alpha^{k+1}_{i\tau}  \right\|^2$, we have 
	\begin{eqnarray*}
		&& (1-\delta) \mathbb{E}_k \left\|e^k + \frac{1}{\lambda N }\sum_{(i, \tau) \in S_k} A_{i\tau} \Delta \alpha^{k+1}_{i\tau}  \right\|^2 \\ 
		&=& (1-\delta) \mathbb{E}_k \left\|e^k + \frac{1}{\lambda N m} \sum_{\tau=1}^n \sum_{i=1}^m A_{i\tau} \Delta \alpha_{i\tau}^{k+1}  + \frac{1}{\lambda N }\sum_{(i, \tau) \in S_k} A_{i\tau} \Delta \alpha^{k+1}_{i\tau}  - \frac{1}{\lambda N m} \sum_{\tau=1}^n \sum_{i=1}^m A_{i\tau} \Delta \alpha_{i\tau}^{k+1}  \right\|^2 \\ 
		&=& (1-\delta) \mathbb{E}_k \left\|e^k + \frac{1}{\lambda N m} \sum_{\tau=1}^n \sum_{i=1}^m A_{i\tau} \Delta \alpha_{i\tau}^{k+1} \right\|^2 \\ 
		&& + (1-\delta) \mathbb{E}_k \left\| \frac{1}{\lambda N } \sum_{(i, \tau) \in S_k} A_{i\tau} \Delta \alpha^{k+1}_{i\tau}  - \frac{1}{\lambda N m} \sum_{\tau=1}^n \sum_{i=1}^m A_{i\tau} \Delta \alpha_{i\tau}^{k+1}  \right\|^2 \\
		&\leq& \left(  1 - \frac{\delta}{2}  \right) \|e^k\|^2 + \frac{2(1-\delta)}{\delta \lambda^2N^2m^2} \left\|  \sum_{\tau=1}^n \sum_{i=1}^m A_{i\tau} \Delta \alpha_{i\tau}^{k+1} \right\|^2 + \frac{(1-\delta)}{\lambda^2 N^2 } \mathbb{E}_k \left\| \sum_{(i, \tau) \in S_k} A_{i\tau} \Delta \alpha^{k+1}_{i\tau} \right\|^2 \\ 
		&& - \frac{(1-\delta)}{\lambda^2N^2m^2} \left\|  \sum_{\tau=1}^n \sum_{i=1}^m A_{i\tau} \Delta \alpha_{i\tau}^{k+1} \right\|^2 \\ 
		&\leq&  \left(  1 - \frac{\delta}{2}  \right) \|e^k\|^2 + \frac{(1-\delta) R^2}{ \lambda^2N m^2}\left( \frac{2}{\delta} -1 \right) \sum_{\tau=1}^n \sum_{i=1}^m \| \Delta \alpha_{i\tau}^{k+1}\|^2 +  \frac{(1-\delta)}{\lambda^2 N^2 } \mathbb{E}_k \left\| \sum_{(i, \tau) \in S_k} A_{i\tau} \Delta \alpha^{k+1}_{i\tau} \right\|^2 \\
		&\overset{Lemma~\ref{lm:eso}}{\leq}&  \left(  1 - \frac{\delta}{2}  \right) \|e^k\|^2 +\frac{(1-\delta) R^2}{ \lambda^2N m^2}\left( \frac{2}{\delta} -1 \right) \sum_{\tau=1}^n \sum_{i=1}^m \| \Delta \alpha_{i\tau}^{k+1}\|^2 +  \frac{(1-\delta)}{\lambda^2 N^2 } \sum_{\tau=1}^n \sum_{i=1}^m p_{i\tau}v_{i\tau}  \| \Delta \alpha_{i\tau}^{k+1}\|^2. 
	\end{eqnarray*}
	
	Recall that $p_{i\tau} = \frac{1}{m}$ and $v_{i\tau} = R_m^2 + nR^2$, we have 
	\begin{eqnarray*}
		&& (1-\delta) \mathbb{E}_k \left\|e^k + \frac{1}{\lambda N }\sum_{(i, \tau) \in S_k} A_{i\tau} \Delta \alpha^{k+1}_{i\tau}  \right\|^2 \\ 
		&\leq&  \left(  1 - \frac{\delta}{2}  \right) \|e^k\|^2 + \frac{(1-\delta)}{\lambda^2 N m^2} \left(  \frac{2R^2}{\delta} + \frac{R_m^2}{n}  \right) \sum_{\tau=1}^n \sum_{i=1}^m \| \Delta \alpha_{i\tau}^{k+1}\|^2. 
	\end{eqnarray*}
	
	Combining (\ref{eq:ek+1in-ecQuartz}) and the above inequality, we can get 
	
	\begin{eqnarray*}
		&& \mathbb{E}_k \|e^{k+1}\|^2 \\ 
		&\leq&  \left(  1 - \frac{\delta}{2}  \right) \|e^k\|^2 + \frac{2(1-\delta)\delta}{n^2} \sum_{\tau=1}^n \|e^k_{\tau} \|^2 + \frac{(1-\delta)}{\lambda^2 N m^2} \left(  \frac{2\delta R_m^2}{n}  + \frac{2R^2}{\delta} + \frac{R_m^2}{n}  \right) \sum_{\tau=1}^n \sum_{i=1}^m \| \Delta \alpha_{i\tau}^{k+1}\|^2 \\ 
		&=&  \left(  1 - \frac{\delta}{2}  \right) \|e^k\|^2 + \frac{2(1-\delta)\delta}{n^2} \sum_{\tau=1}^n \|e^k_{\tau} \|^2 + \frac{(1-\delta)}{\lambda^2 N m^2} \left(  \frac{(2\delta+1) R_m^2}{n}  + \frac{2R^2}{\delta}  \right) \sum_{\tau=1}^n \sum_{i=1}^m \| \Delta \alpha_{i\tau}^{k+1}\|^2. 
	\end{eqnarray*}
	
\end{proof}

\subsection{Proof of Theorem \ref{th:ecQuartz-1}}

Let $h^k \in \R^{tN}$ be defined by: 
$$
h^k_{i\tau} = \Delta \alpha_{i\tau}^{k+1} = -\theta p_{i\tau}^{-1} (\alpha_{i\tau}^k + \nabla \phi_{i\tau}(A_{i\tau}^\top x^{k+1}) ), \ i\in [m], \ \tau\in[n], 
$$
for $k\geq 0$. Then we have $\alpha^{k+1} = \alpha^k + h^k_{[S_k]}$. By Lemma \ref{lm:Dalpha}, 
\begin{align*}
& \mathbb{E}_k[-D(\alpha^{k+1})]  \\ 
& \leq {\tilde f}(\alpha^k) + \sum_{\tau=1}^n \sum_{i=1}^m p_{i\tau} \left\langle  \frac{1}{N} A_{i\tau}^\top \nabla g^* \left(\frac{1}{\lambda N} A\alpha^k \right), h^k_{i\tau}  \right\rangle + \frac{1}{2\lambda N^2} \sum_{\tau=1}^n \sum_{i=1}^m p_{i\tau} v_{i\tau} \|h^k_{i\tau}\|^2 \nonumber \\ 
& \quad + \frac{1}{N} \sum_{\tau=1}^n \sum_{i=1}^m [(1 - p_{i\tau})\phi^*_{i\tau} (-\alpha^k_{i\tau}) + p_{i\tau} \phi^*_{i\tau} (-\alpha^k_{i\tau} - h^k_{i\tau}) ]. 
\end{align*}

Define ${\tilde u}^k = u^k + e^k$ for $k\geq 0$. Then we have 
\begin{align*}
{\tilde u}^{k+1} &= u^{k+1} + e^{k+1} \\
& = u^k + \frac{1}{n}\sum_{\tau=1}^n y_{\tau}^k + \frac{1}{n} \sum_{\tau=1}^n e_{\tau}^{k+1} \\ 
& = u^k +  \frac{1}{n}\sum_{\tau=1}^n y_{\tau}^k +  \frac{1}{n}\sum_{\tau=1}^n \left(  e_\tau^k +  \frac{1}{\lambda m}A_{i_k^\tau \tau} \Delta\alpha_{i^\tau_{k} \tau}^{k+1} - y_\tau^k   \right) \\ 
& = u^k + e^k + \frac{1}{\lambda N} \sum_{\tau=1}^n A_{i_k^\tau \tau} \Delta\alpha_{i^\tau_{k} \tau}^{k+1} \\ 
& = {\tilde u}^k + \frac{1}{\lambda N} \sum_{\tau=1}^n A_{i_k^\tau \tau} \Delta\alpha_{i^\tau_{k} \tau}^{k+1}. 
\end{align*}

Moreover, since ${\tilde u}^0 = u^0 = \frac{1}{\lambda N} \sum_{\tau=1}^n \sum_{i=1}^m A_{i\tau} \alpha_{i\tau}^0$, we have 
\begin{equation}\label{eq:tildeuk}
{\tilde u}^k = \frac{1}{\lambda N} \sum_{\tau=1}^n \sum_{i=1}^m A_{i\tau} \alpha_{i\tau}^k, 
\end{equation}
for $k\geq 0$.

Next we use Lemma \ref{lm:Dalpha-2} to further bound $\mathbb{E}_k[-D(\alpha^{k+1})]$ by choosing $\alpha = \alpha^k$, $x = x^{k+1}$, and $u=u^k$. We have 

\begin{align*}
&\mathbb{E}_k[-D(\alpha^{k+1})] \\ 
& \leq -(1-\theta) D(\alpha^k) - \theta \lambda g(\nabla g^*({u^k})) -  \frac{1}{N} \sum_{\tau=1}^n \sum_{i=1}^m \langle \theta \nabla g^*(u^k), A_{i\tau} \nabla \phi_{i\tau} (A_{i\tau}^\top x^{k+1}) \rangle \\ 
& \quad + \frac{\theta}{N} \sum_{\tau=1}^n \sum_{i=1}^m \phi^*_{i\tau} (\nabla \phi_{i\tau}(A_{i\tau}^\top x^{k+1})) + \frac{\rho + \theta\lambda}{2} \|e^k\|^2 \\ 
& \quad + \frac{1}{2\lambda N^2} \sum_{\tau=1}^n \sum_{i=1}^m \left( p_{i\tau}v_{i\tau} + \frac{ N\lambda p_{i\tau}^2 R^2}{\rho} -  \frac{N \lambda \gamma p_{i\tau}^2 (1-\theta p_{i\tau}^{-1})}{\theta}  \right) \|h^k_{i\tau}\|^2. 
\end{align*}

By convexity of $g$, 
\begin{align*}
P(x^{k+1}) & = \frac{1}{N} \sum_{\tau=1}^n \sum_{i=1}^m \phi_{i\tau}(A_{i\tau}^\top x^{k+1}) + \lambda g((1-\theta)x^k + \theta \nabla g^*(u^k)) \\ 
& \leq  \frac{1}{N} \sum_{\tau=1}^n \sum_{i=1}^m \phi_{i\tau}(A_{i\tau}^\top x^{k+1}) + (1-\theta) \lambda g(x^k) + \theta \lambda g(\nabla g^*(u^k)). 
\end{align*}

By combining the above two inequalities, we can get 
\begin{align*}
& \mathbb{E}_k[P(x^{k+1}) - D(\alpha^{k+1})] \\ 
& \leq  \frac{1}{N} \sum_{\tau=1}^n \sum_{i=1}^m \phi_{i\tau}(A_{i\tau}^\top x^{k+1}) + (1-\theta) \lambda g(x^k) - (1-\theta) D(\alpha^k) \\ 
& \quad -  \frac{1}{N} \sum_{\tau=1}^n \sum_{i=1}^m \langle \theta \nabla g^*(u^k), A_{i\tau} \nabla \phi_{i\tau} (A_{i\tau}^\top x^{k+1}) \rangle + \frac{\theta}{N} \sum_{\tau=1}^n \sum_{i=1}^m \phi^*_{i\tau} (\nabla \phi_{i\tau}(A_{i\tau}^\top x^{k+1})) + \frac{\rho + \theta\lambda}{2} \|e^k\|^2 \\
& \quad + \frac{1}{2\lambda N^2} \sum_{\tau=1}^n \sum_{i=1}^m \left( p_{i\tau}v_{i\tau} + \frac{ N\lambda p_{i\tau}^2 R^2}{\rho} -  \frac{N \lambda \gamma p_{i\tau}^2 (1-\theta p_{i\tau}^{-1})}{\theta}  \right) \|h^k_{i\tau}\|^2. 
\end{align*}

Since $\nabla g^*(u^k) = x^{k+1} - (1-\theta)x^k$, same as the proof of Theorem 9 in \citep{Quartz}, we can simply the above inequality to the following form 
\begin{align}
& \mathbb{E}_k[P(x^{k+1}) - D(\alpha^{k+1})] \nonumber \\ 
& \leq (1-\theta) (P(x^k) - D(\alpha^k)) +  \frac{\rho + \theta\lambda}{2} \|e^k\|^2 \nonumber \\ 
& \quad + \frac{1}{2\lambda N^2} \sum_{\tau=1}^n \sum_{i=1}^m \left( p_{i\tau}v_{i\tau} + \frac{ N\lambda p_{i\tau}^2 R^2}{\rho} -  \frac{N \lambda \gamma p_{i\tau}^2 (1-\theta p_{i\tau}^{-1})}{\theta}  \right) \|h^k_{i\tau}\|^2. \label{eq:PDk+1}
\end{align}

Since $\|e^k\|^2 \leq \frac{1}{n} \sum_{\tau=1}^n \|e_\tau^k\|^2$ from the convexity of the Euclidean norm $\|\cdot\|^2$, from (\ref{eq:PDk+1}) we have 
\begin{align*}
\mathbb{E}_k[\Psi_1^{k+1}] & \leq  (1-\theta) (P(x^k) - D(\alpha^k)) + \frac{2(\rho + \theta \lambda)}{\delta n} \sum_{\tau=1}^n \mathbb{E}_k \|e_\tau^{k+1}\|^2 + \frac{\rho + \theta \lambda}{2n} \sum_{\tau=1}^n \|e_\tau^k\|^2 \\ 
& \quad + \frac{1}{2\lambda N^2} \sum_{\tau=1}^n \sum_{i=1}^m \left( p_{i\tau}v_{i\tau} + \frac{ N\lambda p_{i\tau}^2 R^2}{\rho} -  \frac{N \lambda \gamma p_{i\tau}^2 (1-\theta p_{i\tau}^{-1})}{\theta}  \right) \|h^k_{i\tau}\|^2 \\ 
& \overset{Lemma~\ref{lm:ek+1-1ecQuartz}}{\leq}  (1-\theta) (P(x^k) - D(\alpha^k)) + \left(  1 - \frac{\delta}{2} + \frac{\delta}{4}  \right) \frac{2(\rho + \theta \lambda)}{\delta n} \sum_{\tau=1}^n \mathbb{E}_k \|e_\tau^{k}\|^2 \\ 
& \quad + \frac{1}{2\lambda N^2} \sum_{\tau=1}^n \sum_{i=1}^m \left( \frac{4(1-\delta) (\rho+ \theta \lambda)n}{\delta \lambda m} \left(  \frac{2{\bar R}^2}{\delta} + R_m^2  \right) + p_{i\tau}v_{i\tau} \right. \\ 
& \quad \left. + \frac{ N\lambda p_{i\tau}^2 R^2}{\rho} -  \frac{N \lambda \gamma p_{i\tau}^2 (1-\theta p_{i\tau}^{-1})}{\theta}  \right) \|h^k_{i\tau}\|^2. 
\end{align*}

Recall that $p_{i\tau} = \frac{1}{m}$, by choosing $\rho = \frac{\delta \lambda R}{2 \sqrt{a_1}}$, where $a_1=  (1-\delta) (2{\bar R}^2 + \delta R_m^2) $, the coefficient of $\|h^k_{i\tau}\|^2$ becomes 
\begin{align*}
\frac{4np_{i\tau}R\sqrt{a_1 }}{\delta}    +     \frac{4 \theta  n a_1}{\delta^2 m}    + p_{i\tau}v_{i\tau} -  \frac{N \lambda \gamma p_{i\tau}^2 (1-\theta p_{i\tau}^{-1})}{\theta}. 
\end{align*}
In order to guarantee the above coefficient to be nonpositive, we let 
$$
\frac{4\theta n a_1}{\delta^2  m}  \leq \frac{1}{3} \cdot \frac{N \lambda \gamma p_{i\tau}^2 (1-\theta p_{i\tau}^{-1})}{\theta}, \ \ p_{i\tau}v_{i\tau} \leq \frac{1}{3} \cdot \frac{N \lambda \gamma p_{i\tau}^2 (1-\theta p_{i\tau}^{-1})}{\theta}, 
$$
and 
$$
\frac{4np_{i\tau}R\sqrt{a_1}}{\delta} \leq \frac{1}{3} \cdot \frac{N \lambda \gamma p_{i\tau}^2 (1-\theta p_{i\tau}^{-1})}{\theta}, 
$$
which is equivalent to 

$$
\theta \leq \min\left\{  \tfrac{2\delta \lambda \gamma}{\delta \lambda \gamma m + \sqrt{\delta^2 \lambda^2 \gamma^2 m^2 + 48\lambda \gamma a_1 }},  \tfrac{N\lambda \gamma p_{i\tau}}{3v_{i\tau} + N\lambda \gamma}, \tfrac{\delta \lambda \gamma}{\delta \lambda \gamma m + 12R \sqrt{a_1} }  \right\}. 
$$ 
By choosing the upper bound in the above inequality for $\theta$, we arrive at 
\begin{align*}
\mathbb{E}_k[\Psi_1^{k+1}] & \leq  (1-\theta) (P(x^k) - D(\alpha^k)) + \left(  1 - \frac{\delta}{2} + \frac{\delta}{4}  \right) \frac{2(\rho + \theta \lambda)}{\delta n} \sum_{\tau=1}^n \mathbb{E}_k \|e_\tau^{k}\|^2 \\ 
& = \left(  1 - \min\left\{  \theta, \frac{\delta}{4}  \right\}  \right) \Psi_1^k. 
\end{align*}

By using the tower property, we can obtain 
$$
\mathbb{E}[\Psi_1^k] \leq \left(  1 - \min\left\{  \theta, \frac{\delta}{4}  \right\}  \right)^k \Psi_1^0. 
$$
Therefore, $\mathbb{E}[\Psi_1^k] \leq \epsilon$ as long as 
\begin{align*}
k & \geq O \left( \left(  \frac{1}{\theta} + \frac{1}{\delta}  \right) \ln\frac{1}{\epsilon}  \right) \\ 
& = O \left( \left(  \frac{1}{\delta} + m + \frac{R_m^2}{n\lambda \gamma} + \frac{R^2}{\lambda \gamma} + \frac{1}{\delta} \sqrt{\frac{(1-\delta) ({\bar R}^2 + \delta R_m^2) }{\lambda \gamma } }    +  \frac{R \sqrt{ (1-\delta) ({\bar R}^2 + \delta R_m^2)  }}{\delta \lambda \gamma}  \right) \ln\frac{1}{\epsilon}  \right) \\ 
& = O \left( \left(  \frac{1}{\delta} + m + \frac{R_m^2}{n\lambda \gamma} + \frac{R^2}{\lambda \gamma}   +  \frac{R \sqrt{ (1-\delta) ({\bar R}^2 + \delta R_m^2)  }}{\delta \lambda \gamma}  \right) \ln\frac{1}{\epsilon}  \right) \\ 
& = O \left( \left(  \frac{1}{\delta} + m + \frac{R_m^2}{n\lambda \gamma} + \frac{R^2}{\lambda \gamma}   +   \frac{\sqrt{1-\delta} R{\bar R}}{\delta \lambda \gamma}  + \frac{\sqrt{1-\delta} RR_m}{\lambda \gamma \sqrt{\delta}}  \right) \ln\frac{1}{\epsilon}  \right), 
\end{align*}
where we use $\frac{R^2}{\gamma} \geq \lambda$ in the second equality.

\subsection{Proof of Theorem \ref{th:ecQuartz-2}} 

First, from (\ref{eq:PDk+1}) and Lemma \ref{lm:ek+1-2ecQuartz} we have 
\begin{align*}
& \mathbb{E}_k[P(x^{k+1}) - D(\alpha^{k+1}) + \frac{2(\rho + \theta \lambda)}{\delta} \|e^{k+1}\|^2 ] \\ 
& \leq (1-\theta) (P(x^k) - D(\alpha^k)) + \left(  1- \frac{\delta}{4}  \right) \frac{2(\rho + \theta \lambda)}{\delta} \|e^{k}\|^2 + \frac{4(1-\delta) (\rho + \theta\lambda)}{n^2} \sum_{\tau=1}^n\|e_\tau^k\|^2  \\ 
&  \quad + \frac{1}{2\lambda N^2} \sum_{\tau=1}^n \sum_{i=1}^m \left(  \frac{4(\rho + \theta\lambda)(1-\delta)n}{\delta \lambda m} \left(  \frac{(2\delta+1) R_m^2}{n} + \frac{2R^2}{\delta}  \right)   +    p_{i\tau}v_{i\tau} + \frac{ N\lambda p_{i\tau}^2 R^2}{\rho} -  \frac{N \lambda \gamma p_{i\tau}^2 (1-\theta p_{i\tau}^{-1})}{\theta}  \right) \|h^k_{i\tau}\|^2. 
\end{align*}

Combining the above inequality and Lemma \ref{lm:ek+1-1ecQuartz} yields 
\begin{align*}
\mathbb{E}_k[\Psi_2^{k+1}] &\leq \left(  1 - \min\left\{  \theta, \frac{\delta}{4}  \right\}  \right) \Psi_2^k    +    \frac{1}{2\lambda N^2} \sum_{\tau=1}^n \sum_{i=1}^m \left(   \frac{32(1-\delta)^2 (\rho + \theta\lambda)}{\delta \lambda m} \left(  \frac{2{\bar R}^2}{\delta}  + R_m^2  \right)  \right. \\ 
& \quad \left.  +   \frac{4(\rho + \theta\lambda)(1-\delta)n}{\delta \lambda m} \left(  \frac{(2\delta+1) R_m^2}{n} + \frac{2R^2}{\delta}  \right)   +    p_{i\tau}v_{i\tau} + \frac{ N\lambda p_{i\tau}^2 R_m^2}{\rho} -  \frac{N \lambda \gamma p_{i\tau}^2 (1-\theta p_{i\tau}^{-1})}{\theta}  \right) \|h^k_{i\tau}\|^2 \\ 
& \leq  \left(  1 - \min\left\{  \theta, \frac{\delta}{4}  \right\}  \right) \Psi_2^k    +    \frac{1}{2\lambda N^2} \sum_{\tau=1}^n \sum_{i=1}^m \left(   \frac{4(1-\delta) (\rho + \theta\lambda)n}{\delta \lambda m} \left(  \frac{9R_m^2}{n} +   \frac{16{\bar R}^2}{\delta n}  + \frac{2R^2}{\delta}  \right)  \right. \\ 
& \quad \left.   +    p_{i\tau}v_{i\tau} + \frac{ N\lambda p_{i\tau}^2 R^2}{\rho} -  \frac{N \lambda \gamma p_{i\tau}^2 (1-\theta p_{i\tau}^{-1})}{\theta}  \right) \|h^k_{i\tau}\|^2. 
\end{align*}

By choosing $\rho = \frac{\delta \lambda R}{2\sqrt{a_2}}$, where $a_2 = (1-\delta) (2R^2 + \frac{16{\bar R}^2}{n} + \frac{9\delta R_m^2}{n})$, the coefficient of $\|h^k_{i\tau}\|^2$ becomes 
\begin{align*}
\frac{4np_{i\tau}R\sqrt{a_2 }}{\delta}    +     \frac{4 \theta  n a_2}{\delta^2 m}    + p_{i\tau}v_{i\tau} -  \frac{N \lambda \gamma p_{i\tau}^2 (1-\theta p_{i\tau}^{-1})}{\theta}. 
\end{align*}
Same as the proof of Theorem \ref{th:ecQuartz-1}, we can choose 
$$
\theta = \min\left\{  \tfrac{2\delta \lambda \gamma}{\delta \lambda \gamma m + \sqrt{\delta^2 \lambda^2 \gamma^2 m^2 + 48\lambda \gamma a_2 }},  \tfrac{N\lambda \gamma p_{i\tau}}{3v_{i\tau} + N\lambda \gamma}, \tfrac{\delta \lambda \gamma}{\delta \lambda \gamma m + 12R \sqrt{a_2} }  \right\} , 
$$ 
and get $\mathbb{E}_k[\Psi_2^{k+1}] \leq \left(  1 - \min\left\{  \theta, \frac{\delta}{4}  \right\}  \right) \Psi_2^k$. 

By using the tower property, we can obtain 
$$
\mathbb{E}[\Psi_2^k] \leq \left(  1 - \min\left\{  \theta, \frac{\delta}{4}  \right\}  \right)^k \Psi_2^0. 
$$
Therefore, $\mathbb{E}[\Psi_2^k] \leq \epsilon$ as long as 
\begin{align*}
k & \geq O \left( \left(  \frac{1}{\theta} + \frac{1}{\delta}  \right) \ln\frac{1}{\epsilon}  \right) \\ 
& = O \left( \left(  \frac{1}{\delta} + m + \frac{R_m^2}{n\lambda \gamma} + \frac{R^2}{\lambda \gamma} + \frac{1}{\delta} \sqrt{\frac{(1-\delta) R^2  }{\lambda \gamma } }   + \frac{\sqrt{(1-\delta) R_m^2}}{\sqrt{\delta n \lambda \gamma}}   +  \frac{\sqrt{1-\delta}}{\delta} \frac{R^2}{\lambda \gamma}  + \frac{\sqrt{1-\delta}}{\sqrt{\delta n}} \frac{RR_m}{\lambda \gamma} \right) \ln\frac{1}{\epsilon}  \right) \\ 
& = O \left( \left(  \frac{1}{\delta} + m + \frac{R_m^2}{n\lambda \gamma} + \frac{R^2}{\lambda \gamma}  +  \frac{\sqrt{1-\delta}}{\delta} \frac{R^2}{\lambda \gamma}  + \frac{\sqrt{1-\delta}}{\sqrt{\delta n}} \frac{RR_m}{\lambda \gamma} \right) \ln\frac{1}{\epsilon}  \right) \\ 
& = O \left( \left(  \frac{1}{\delta} + m + \frac{R_m^2}{n\lambda \gamma} + \frac{R^2}{\lambda \gamma}     +  \frac{\sqrt{1-\delta}}{\delta} \frac{R^2}{\lambda \gamma}   \right) \ln\frac{1}{\epsilon}  \right), 
\end{align*}
where in the first equality we use $\frac{{\bar R}^2}{n} \leq R^2$, in the second equality we use 
$$
\frac{2}{\delta} \sqrt{\frac{(1-\delta) R^2  }{\lambda \gamma } } \leq \frac{\sqrt{1-\delta}}{\delta} + \frac{\sqrt{1-\delta}}{\delta} \frac{R^2}{\lambda \gamma}, 
$$
and 
$$
2\frac{\sqrt{(1-\delta) R_m^2}}{\sqrt{\delta n \lambda \gamma}} \leq \frac{1-\delta}{\delta} + \frac{R_m^2}{n\lambda \gamma}, 
$$
and in the last equality, we use 
$$
2\frac{\sqrt{1-\delta}}{\sqrt{\delta n}} \frac{RR_m}{\lambda \gamma}  \leq \frac{\sqrt{1-\delta}}{\delta} \frac{R^2}{\lambda \gamma} + \frac{\sqrt{1-\delta} R_m^2}{n\lambda \gamma}. 
$$

\section{Proofs for EC-SDCA} 

\subsection{ A lemma}

\begin{lemma}\label{lm:ecsdca}
For error compensated SDCA, we have 
\begin{align}
\mathbb{E}_k [\epsilon_D^{k+1}] & \leq (1-\theta) \epsilon_D^k - \theta \epsilon_P^{k+1} + \frac{\rho + \theta \lambda }{2} \|e^k\|^2 \nonumber \\ 
& \quad - \frac{1}{2\lambda N^2} \sum_{\tau=1}^n \sum_{i=1}^m \left(  \frac{\lambda N\gamma p_{i\tau}^2(1-\theta p_{i\tau}^{-1})}{\theta} - p_{i\tau}v_{i\tau} - \frac{\lambda Np_{i\tau}^2 R^2}{\rho}  \right) \|\Delta \alpha_{i\tau}^{k+1}\|^2, \label{eq:epsilonD-ecsdca}
\end{align}
for any $\rho>0$. 
\end{lemma}

\begin{proof}

Denote $s_\tau^k = \theta p_{i_k^\tau \tau}^{-1}$ and $z_\tau^k = -\nabla \phi_{i^\tau_k \tau}(A_{i_k^\tau \tau}^\top x^{k+1})$.  Then from (\ref{eq:tildeuk}) we have 
\begin{align*}
N[D(\alpha^{k+1}) - D(\alpha^k)] &=  \left(  \sum_{\tau=1}^n -\phi_{i^\tau_k \tau}^*(-\alpha_{i^\tau_k \tau}^{k+1}) - \lambda N g^*({\tilde u}^{k+1})  \right) - \left(  \sum_{\tau=1}^n  -\phi_{i^\tau_k \tau}^*(-\alpha_{i^\tau_k \tau}^{k}) - \lambda N g^*({\tilde u}^{k})   \right)  \\ 
& = \sum_{\tau=1}^n \left[ \phi_{i^\tau_k \tau}^*(-\alpha_{i^\tau_k \tau}^{k}) -  \phi_{i^\tau_k \tau}^*(-(1-s_\tau^k) \alpha_{i^\tau_k \tau} - s_\tau^k z_\tau^k) \right] - \lambda N \left(  g^*({\tilde u}^{k+1}) - g^*({\tilde u}^k)  \right)   \\ 
& \geq \sum_{\tau=1}^n \left[  \phi_{i^\tau_k \tau}^*(-\alpha_{i^\tau_k \tau}^{k}) - (1-s_\tau^k) \phi_{i^\tau_k \tau}^*(-\alpha_{i^\tau_k \tau}^{k}) - s_\tau^k \phi_{i^\tau_k \tau}^*(-z_\tau^k)  + \frac{\gamma s_\tau^k (1-s_\tau^k)}{2} \|z_\tau^k - \alpha_{i^\tau_k \tau}^{k}\|^2  \right] \\ 
& \quad -  \lambda N \left(  \langle \nabla g^*({\tilde u}^k), {\tilde u}^{k+1} - {\tilde u}^k \rangle + \frac{1}{2} \|{\tilde u}^{k+1} - {\tilde u}^k \|^2   \right) \\ 
& = \sum_{\tau=1}^n \left[  s_\tau^k \phi_{i^\tau_k \tau}^*(-\alpha_{i^\tau_k \tau}^{k}) - s_\tau^k \phi_{i^\tau_k \tau}^*(-z_\tau^k)  + \frac{\gamma s_\tau^k (1-s_\tau^k)}{2} \|z_\tau^k - \alpha_{i^\tau_k \tau}^{k}\|^2  \right] - \lambda N\langle x^{k+1}, {\tilde u}^{k+1} - {\tilde u}^k \rangle \\
& \quad -  \lambda N \left(  \langle \nabla g^*({\tilde u}^k) - x^{k+1}, {\tilde u}^{k+1} - {\tilde u}^k \rangle + \frac{1}{2} \|{\tilde u}^{k+1} - {\tilde u}^k \|^2   \right), 
\end{align*}
where in the first inequality, we use that $\phi^*_{i\tau}$ is $\gamma$-strongly convex and $g^*$ is $1$-smooth. From (\ref{eq:tildeuk}) and the update of $\alpha^k$, we know ${\tilde u}^{k+1} - {\tilde u}^k = \frac{1}{\lambda N} \sum_{\tau=1}^n  A_{i_k^\tau \tau} \Delta \alpha^{k+1}_{i_k^\tau \tau} = \frac{1}{\lambda N} \sum_{\tau=1}^n A_{i_k^\tau \tau}(-s_\tau^k \alpha_{i^\tau_k \tau}^k + s_\tau^k z_\tau^k)$. Then we have 
\begin{align}
& \quad N[D(\alpha^{k+1}) - D(\alpha^k)] \nonumber \\ 
&\geq \sum_{\tau=1}^n \left[  s_\tau^k \left( \phi_{i^\tau_k \tau}^*(-\alpha_{i^\tau_k \tau}^{k}) + \langle A_{i_k^\tau \tau}^\top x^{k+1}, \alpha_{i^\tau_k \tau}^{k}\rangle \right)  - s_\tau^k \left( \phi_{i^\tau_k \tau}^*(-z_\tau^k) + \langle A_{i_k^\tau \tau}^\top x^{k+1}, z_\tau^k \rangle \right) + \frac{\gamma  (1-s_\tau^k)}{2s_\tau^k} \| \Delta \alpha^{k+1}_{i_k^\tau \tau}\|^2  \right] \nonumber \\ 
& \quad - \left\langle \nabla g^*({\tilde u}^k) - x^{k+1},  \sum_{\tau=1}^n  A_{i_k^\tau \tau} \Delta \alpha^{k+1}_{i_k^\tau \tau} \right\rangle - \frac{1}{2\lambda N} \left\|  \sum_{\tau=1}^n  A_{i_k^\tau \tau} \Delta \alpha^{k+1}_{i_k^\tau \tau} \right\|^2 \nonumber \\
& = \sum_{\tau=1}^n \left[  s_\tau^k \left(  \phi_{i^\tau_k \tau}(A_{i_k^\tau \tau}^\top x^{k+1}) + \phi_{i^\tau_k \tau}^*(-\alpha_{i^\tau_k \tau}^{k}) + \langle A_{i_k^\tau \tau}^\top x^{k+1}, \alpha_{i^\tau_k \tau}^{k}\rangle \right)  + \frac{\gamma  (1-s_\tau^k)}{2s_\tau^k} \| \Delta \alpha^{k+1}_{i_k^\tau \tau}\|^2  \right] \nonumber \\ 
& \quad - \left\langle \nabla g^*({\tilde u}^k) - x^{k+1},  \sum_{\tau=1}^n  A_{i_k^\tau \tau} \Delta \alpha^{k+1}_{i_k^\tau \tau} \right\rangle - \frac{1}{2\lambda N} \left\|  \sum_{\tau=1}^n  A_{i_k^\tau \tau} \Delta \alpha^{k+1}_{i_k^\tau \tau} \right\|^2, \label{eq:diffD-ecsdca}
\end{align}
where in the last equality we use $\phi_{i^\tau_k \tau}^*(-z_\tau^k) + \langle A_{i_k^\tau \tau}^\top x^{k+1}, z_\tau^k \rangle = - \phi_{i^\tau_k \tau}(A_{i_k^\tau \tau}^\top x^{k+1})$ which comes from $z_\tau^k = - \nabla \phi_{i^\tau_k \tau} (A_{i_k^\tau \tau}^\top x^{k+1})$.

Since $x^{k+1} = \nabla g^*(u^k)$, we have $g(x^{k+1}) + g^*(u^k) = \langle x^{k+1}, u^k \rangle$. Therefore, 
\begin{align*}
P(x^{k+1}) - D(\alpha^k) & = \frac{1}{N} \sum_{\tau=1}^n \sum_{i=1}^m \phi_{i \tau}(A_{i\tau}^\top x^{k+1}) + \lambda g(x^{k+1}) + \frac{1}{N} \sum_{\tau=1}^n \sum_{i=1}^m \phi_{i \tau}^*(-\alpha_{i\tau}^k) + \lambda g^*({\tilde u}^k) \\ 
& = \frac{1}{N} \sum_{\tau=1}^n \sum_{i=1}^m \left( \phi_{i \tau}(A_{i\tau}^\top x^{k+1}) + \phi_{i \tau}^*(-\alpha_{i\tau}^k) \right) + \lambda g^*({\tilde u}^k) - \lambda g^*(u^k) + \lambda \langle x^{k+1}, u^k \rangle, 
\end{align*}
which indicates that 
\begin{align}
& \quad \mathbb{E}_k \left[  \sum_{\tau=1}^n  s_\tau^k \left(  \phi_{i^\tau_k \tau}(A_{i_k^\tau \tau}^\top x^{k+1}) + \phi_{i^\tau_k \tau}^*(-\alpha_{i^\tau_k \tau}^{k}) + \langle A_{i_k^\tau \tau}^\top x^{k+1}, \alpha_{i^\tau_k \tau}^{k}\rangle \right)  \right] \nonumber \\ 
& = \sum_{\tau=1}^n \sum_{i=1}^m \theta \left(  \phi_{i\tau}(A_{i\tau}^\top x^{k+1}) + \phi_{i\tau}^*(-\alpha_{i\tau}^k) + \langle A_{i\tau}^\top x^{k+1}, \alpha_{i\tau}^k \rangle \right) \nonumber \\ 
& \overset{(\ref{eq:tildeuk})}{=}  \sum_{\tau=1}^n \sum_{i=1}^m \theta \left(  \phi_{i\tau}(A_{i\tau}^\top x^{k+1}) + \phi_{i\tau}^*(-\alpha_{i\tau}^k) \right) + \theta\lambda N \langle x^{k+1}, {\tilde u}^k \rangle \nonumber \\ 
& = \theta N \left(  P(x^{k+1}) - D(\alpha^k)  \right) - \theta \lambda N \left(  g^*({\tilde u}^k) - g^*(u^k) - \langle x^{k+1}, {\tilde u}^k - u^k \rangle  \right) \nonumber \\ 
& = \theta N \left(  P(x^{k+1}) - D(\alpha^k)  \right) - \theta \lambda N \left(  g^*({\tilde u}^k) - g^*(u^k) - \langle \nabla g^*(u^k), {\tilde u}^k - u^k \rangle  \right) \nonumber \\ 
& \geq \theta N \left(  P(x^{k+1}) - D(\alpha^k)  \right) - \frac{\theta \lambda N}{2}\|e^k\|^2, \label{eq:inter-ecsdca}
\end{align}
where in the last equality we use $x^{k+1} = \nabla g^*(u^k)$, and in the last inequality we use that $g^*$ is $1$-smooth and ${\tilde u}^k = u^k + e^k$. 

From (\ref{eq:diffD-ecsdca}) and (\ref{eq:inter-ecsdca}), we arrive at 
\begin{align*}
\mathbb{E}_k [\epsilon_D^k - \epsilon_D^{k+1}] & = \mathbb{E}_k[D(\alpha^{k+1}) - D(\alpha^k)] \\
& \geq \theta  \left(  P(x^{k+1}) - D(\alpha^k)  \right) - \frac{\theta \lambda }{2} \|e^k\|^2 - \frac{1}{N} \mathbb{E}_k \left\langle \nabla g^*({\tilde u}^k) - x^{k+1},  \sum_{\tau=1}^n  A_{i_k^\tau \tau} \Delta \alpha^{k+1}_{i_k^\tau \tau} \right\rangle  \\ 
& \quad + \frac{1}{N} \mathbb{E}_k \left[  \sum_{\tau=1}^n \frac{\gamma  (1-s_\tau^k)}{2s_\tau^k} \| \Delta \alpha^{k+1}_{i_k^\tau \tau}\|^2  \right] - \frac{1}{2\lambda N^2}\mathbb{E}_k \left\|  \sum_{\tau=1}^n  A_{i_k^\tau \tau} \Delta \alpha^{k+1}_{i_k^\tau \tau} \right\|^2. 
\end{align*}

\noindent For $\mathbb{E}_k \left[  \sum_{\tau=1}^n \frac{\gamma  (1-s_\tau^k)}{2s_\tau^k} \| \Delta \alpha^{k+1}_{i_k^\tau \tau}\|^2  \right]$, we have 
$$
\mathbb{E}_k \left[  \sum_{\tau=1}^n \frac{\gamma  (1-s_\tau^k)}{2s_\tau^k} \| \Delta \alpha^{k+1}_{i_k^\tau \tau}\|^2  \right] = \sum_{\tau=1}^n \sum_{i=1}^m \frac{\gamma p_{i\tau}^2(1-\theta p_{i\tau}^{-1})}{2\theta} \|\Delta \alpha_{i\tau}^{k+1}\|^2. 
$$
\noindent For $\mathbb{E}_k \left\|  \sum_{\tau=1}^n  A_{i_k^\tau \tau} \Delta \alpha^{k+1}_{i_k^\tau \tau} \right\|^2$, from Lemma \ref{lm:eso}, we have 
$$
\mathbb{E}_k \left\|  \sum_{\tau=1}^n  A_{i_k^\tau \tau} \Delta \alpha^{k+1}_{i_k^\tau \tau} \right\|^2 \leq \sum_{\tau=1}^n \sum_{i=1}^m p_{i\tau} v_{i\tau} \|\Delta \alpha_{i\tau}^{k+1}\|^2. 
$$

\noindent For $\mathbb{E}_k \left\langle \nabla g^*({\tilde u}^k) - x^{k+1},  \sum_{\tau=1}^n  A_{i_k^\tau \tau} \Delta \alpha^{k+1}_{i_k^\tau \tau} \right\rangle$, we have 
\begin{align*}
\mathbb{E}_k \left\langle \nabla g^*({\tilde u}^k) - x^{k+1},  \sum_{\tau=1}^n  A_{i_k^\tau \tau} \Delta \alpha^{k+1}_{i_k^\tau \tau} \right\rangle & = \left\langle  \nabla g^*({\tilde u}^k) - \nabla g^*(u^k), \sum_{\tau=1}^n \sum_{i=1}^m A_{i\tau} p_{i\tau} \Delta \alpha_{i\tau}^{k+1}   \right\rangle \\ 
& \leq \frac{\rho N}{2} \|e^k\|^2 + \frac{1}{2\rho N} \left\| \sum_{\tau=1}^n \sum_{i=1}^m A_{i\tau} p_{i\tau} \Delta \alpha_{i\tau}^{k+1}  \right\|^2\\ 
& \leq \frac{\rho N}{2} \|e^k\|^2  +  \frac{\|A\|^2}{2\rho N} \sum_{\tau=1}^n \sum_{i=1}^m p_{i\tau}^2 \|\Delta \alpha_{i\tau}^{k+1}\|^2 \\ 
& = \frac{\rho N}{2} \|e^k\|^2 + \frac{R^2}{2\rho} \sum_{\tau=1}^n \sum_{i=1}^m p_{i\tau}^2 \|\Delta \alpha_{i\tau}^{k+1}\|^2, 
\end{align*}
where in the first inequality we use Young's inequality for any $\rho>0$ and that $g^*$ is $1$-smooth, in the last equality we use the fact that $R^2 = \frac{1}{N} \lambda_{\rm max}\left(  \sum_{\tau=1}^n \sum_{i=1}^m A_{i\tau}A_{i\tau}^\top   \right) = \frac{1}{N} \|A\|^2$.  

Then we can obtain 
\begin{align*}
\mathbb{E}_k [\epsilon_D^k - \epsilon_D^{k+1}] & \geq \theta  \left(  P(x^{k+1}) - D(\alpha^k)  \right) - \frac{\rho + \theta \lambda }{2} \|e^k\|^2 \\ 
& \quad + \frac{1}{2\lambda N^2} \sum_{\tau=1}^n \sum_{i=1}^m \left(  \frac{\lambda N\gamma p_{i\tau}^2(1-\theta p_{i\tau}^{-1})}{\theta} - p_{i\tau}v_{i\tau} - \frac{\lambda Np_{i\tau}^2 R^2}{\rho}  \right) \|\Delta \alpha_{i\tau}^{k+1}\|^2 \\ 
& = \theta \epsilon_P^{k+1} + \theta \epsilon_D^k - \frac{\rho + \theta \lambda }{2} \|e^k\|^2 \\ 
& \quad + \frac{1}{2\lambda N^2} \sum_{\tau=1}^n \sum_{i=1}^m \left(  \frac{\lambda N\gamma p_{i\tau}^2(1-\theta p_{i\tau}^{-1})}{\theta} - p_{i\tau}v_{i\tau} - \frac{\lambda Np_{i\tau}^2 R^2}{\rho}  \right) \|\Delta \alpha_{i\tau}^{k+1}\|^2. 
\end{align*}
After rearrangement, we can get the result. 

\end{proof}

\subsection{Proof of Theorem \ref{th:ecSDCA-1}}

First, notice that (\ref{eq:epsilonD-ecsdca}) in Lemma \ref{lm:ecsdca} is the same as (\ref{eq:PDk+1}) except that $\epsilon_D^{k+1}$ and $\epsilon_D^k$ are replaced by $P(x^{k+1}) - D(\alpha^{k+1})$ and $P(x^k)-D(\alpha^k)$ respectively,  and there is an additional term $-\theta \epsilon_P^{k+1}$. Hence, same as the proof in Theorem \ref{th:ecQuartz-1}, we can get 
\begin{equation}\label{eq:Psi3-ecsdca}
\mathbb{E}_k[\Psi_3^{k+1}] \leq \left(  1 - \min\left\{  \theta, \frac{\delta}{4}  \right\}  \right) \Psi_3^k - \theta \epsilon_P^{k+1}, 
\end{equation}
by when $\theta$ satisfies (\ref{eq:theta-ecQuartz-1}). Since $\epsilon_P^{k+1} \geq 0$, by using the tower property, we can obtain $\mathbb{E}[\Psi_3^k] \leq \left(  1 - \min\left\{  \theta, \frac{\delta}{4}  \right\}  \right)^k \Psi_3^0$. 

From (\ref{eq:Psi3-ecsdca}) and the tower property, we have 
$$
\mathbb{E}[\Psi_3^k] \leq \left(  1 - \min\left\{  \theta, \frac{\delta}{4}  \right\}  \right) \mathbb{E}[\Psi_3^{k-1}] - \theta \mathbb{E}[\epsilon_P^k], 
$$
for $k\geq 1$. Let $w_k = \left(  1 - \min\left\{  \theta, \frac{\delta}{4}  \right\}  \right)^{-k}$ and $W_k = \sum_{i=1}^k w_i$. By multiplying $w_k$ on the both sides of the above inequality, we can get 
\begin{align*}
\mathbb{E}[w_k \Psi_3^k] & \leq w_{k-1}\Psi_3^{k-1} - \theta \mathbb{E}[w_k \epsilon_P^k] \\ 
& \leq w_0 \Psi_3^0 - \theta \sum_{i=1}^k w_i \mathbb{E}[\epsilon_P^i]\\ 
& = \Psi_3^0 - \theta \sum_{i=1}^k w_i \mathbb{E}[\epsilon_P^i], 
\end{align*}
which implies that 
\begin{align*}
\frac{1}{W_k} \sum_{i=1}^k w_i \mathbb{E}[\epsilon_P^i] & \leq \frac{1}{\theta W_k} \Psi_3^0 \\ 
& = \frac{\min\left\{  \theta, \frac{\delta}{4}  \right\} }{\theta \left(  (1-\min\left\{  \theta, \frac{\delta}{4}  \right\} )^{-k} -1  \right)}\Psi_3^0 \\ 
& \leq \frac{\left(  1 - \min\left\{  \theta, \frac{\delta}{4}  \right\}  \right)^k \epsilon_D^0}{1 - \left(  1 - \min\left\{  \theta, \frac{\delta}{4}  \right\}  \right)^k}, 
\end{align*}
where we use $\Psi_3^0 = \epsilon_D^0$. Then from the convexity of $P$, we have 
$$
\mathbb{E}[P({\bar x}^k) - P(x^*)] \leq \frac{1}{W_k} \sum_{i=1}^k w_i \mathbb{E}[\epsilon_P^i] \leq \frac{\left(  1 - \min\left\{  \theta, \frac{\delta}{4}  \right\}  \right)^k \epsilon_D^0}{1 - \left(  1 - \min\left\{  \theta, \frac{\delta}{4}  \right\}  \right)^k}. 
$$

In order to guarantee $\mathbb{E}[P({\bar x}^k) - P(x^*)] \leq \epsilon$, we first let 
$$
\left(  1 - \min\left\{  \theta, \frac{\delta}{4}  \right\}  \right)^k \leq \frac{1}{2}, 
$$
which indicates that 
$$
\mathbb{E}[P({\bar x}^k) - P(x^*)] \leq \left(  1 - \min\left\{  \theta, \frac{\delta}{4}  \right\}  \right)^k \cdot 2\epsilon_D^0. 
$$
Thus, when $\epsilon \leq \epsilon_D^0$, $\mathbb{E}[P({\bar x}^k) - P(x^*)] \leq \epsilon$ as long as 
$$
\left(  1 - \min\left\{  \theta, \frac{\delta}{4}  \right\}  \right)^k \leq \frac{\epsilon}{2\epsilon_D^0}, 
$$
which is equivalent to 
$$
k \geq \frac{1}{-\ln\left(  1 - \min\left\{  \theta, \frac{\delta}{4}  \right\}  \right) } \ln \left(  \frac{2\epsilon_D^0}{\epsilon}  \right). 
$$
Finally, from $-\ln(1-x) \geq x$ for $x\in[0,1)$, we have $\mathbb{E}[P({\bar x}^k) - P(x^*)] \leq \epsilon$ as long as 
\begin{align*}
k & \geq O\left(  \left(  \frac{1}{\delta} + \frac{1}{\theta}  \right) \ln \left(  \frac{2\epsilon_D^0}{\epsilon}  \right)  \right) \\ 
& = O \left( \left(  \frac{1}{\delta} + m + \frac{R_m^2}{n\lambda \gamma} + \frac{R^2}{\lambda \gamma}   +   \frac{\sqrt{1-\delta} R{\bar R}}{\delta \lambda \gamma}  + \frac{\sqrt{1-\delta} RR_m}{\lambda \gamma \sqrt{\delta}}  \right) \ln \left(  \frac{2\epsilon_D^0}{\epsilon}  \right)  \right). 
\end{align*}

\subsection{Proof of Theorem \ref{th:ecSDCA-2}}

The proof is the same as that of Theorem \ref{th:ecSDCA-1}. Hence we omit it.

\end{document}